\documentclass{amsart}
\usepackage{amsmath,amscd,xypic,amssymb,combelow,color,enumitem,graphicx,float}

\usepackage{tikz-cd}
\usepackage{mathtools}
\usepackage{setspace}

\usepackage{mathrsfs}
\usepackage[margin=1 in]{geometry}

\usepackage{xr-hyper}
\usepackage[
pdftex,
bookmarks=true,
colorlinks=true,
debug=true,
pdfnewwindow=true]{hyperref}
\hypersetup{bookmarksdepth = 3}
\newtheorem{theorem}{Theorem}[section]

\newtheorem{prop}[theorem]{Proposition}
\newtheorem{lemma}[theorem]{Lemma}
\newtheorem{cor}[theorem]{Corollary}

\newtheorem{definition}[theorem]{Definition}
\newtheorem{remark}[theorem]{Remark}

\theoremstyle{remark}

\numberwithin{equation}{section}

\usepackage{amsthm}
\makeatletter
\def\th@plain{%
  \thm@notefont{}% same as heading font
  \itshape % body font
}
\def\th@definition{%
  \thm@notefont{}% same as heading font
  \normalfont % body font
}
\makeatother
\newcommand{\bb}[1]{\mathbb{#1}}

\newcommand{\cal}[1]{\mathcal{#1}}
\renewcommand{\frak}[1]{\mathfrak{#1}}
\newcommand{\scr}[1]{\mathscr{#1}}
\newcommand{\what}[1]{\widehat{#1}}
\renewcommand{\bf}[1]{\mathbf{#1}}

\newcommand{\qbinom}{\genfrac{[}{]}{0pt}{}}
\usepackage{nicematrix}
\usepackage{accents}

\usepackage{etoolbox}
\usepackage{tikz}
\usepackage{latexsym}
\usepackage{verbatim}
\usepackage{cases}
\usepackage{theoremref}

\DeclareMathOperator{\End}{End}

\DeclareMathOperator{\diag}{diag}

\DeclareMathOperator{\id}{id}

\DeclareMathOperator{\hgt}{ht}

\newcommand{\fg}{\frak{g}}

\newcommand{\fn}{\frak{n}}

\newcommand{\BZ}{\mathbb{Z}}
\newcommand{\BC}{\mathbb{C}}
\newcommand{\gl}{\mathfrak{gl}}
\newcommand{\ssl}{\mathfrak{sl}}
\newcommand{\sso}{\mathfrak{so}}
\newcommand{\ssp}{\mathfrak{sp}}
\newcommand{\uu}{U_{r,s}(\fg)}
\newcommand{\UU}{U_{r,s}(\widehat{\fg})}

\newcommand{\Id}{\mathrm{Id}}

\newcommand\iso{\,\vphantom{j^{X^2}}\smash{\overset{\sim}{\vphantom{\rule{0pt}{0.20em}}\smash{\longrightarrow}}}\,}

\newcommand{\ol}{\overline}
\newcommand{\wtd}{\widetilde}

\begin{document}

\title[Bicharacter twists of quantum groups]
      {\Large{\textbf{Bicharacter twists of quantum groups}}}

\author[Ian Martin and Alexander Tsymbaliuk]{Ian Martin and Alexander Tsymbaliuk}

\address{I.M.: University of Washington, Department of Mathematics, Seattle, WA, USA}
\email{ianm723@uw.edu}

\address{A.T.: Purdue University, Department of Mathematics, West Lafayette, IN, USA}
\email{sashikts@gmail.com}

\begin{abstract}
We apply the general construction of a twist of bigraded Hopf algebras by skew bicharacters to obtain
two-parameter quantum groups in the Drinfeld-Jimbo, new Drinfeld (for affine types), and FRT (for both finite and affine)
presentations from their standard one-parameter versions. This yields new elementary proofs of the fundamental results on
two-parameter quantum groups that appeared in the literature over the last two decades, and also leads to natural
generalizations in the super and multiparameter cases.
\end{abstract}

\maketitle
\tableofcontents

   %%%%%%%%%%%%%%%%%%%%%%%%%%%%%%%%%%%%%%%%%%%%%%%%%%%%%%%%%%%%%%%%%%%%%%%%%%%%%%%%%%%%%
   %%%%%%%%%%%%%%%%%%%%%%%%%%%%%%%%%% INTRODUCTION %%%%%%%%%%%%%%%%%%%%%%%%%%%%%%%%%%%%%
   %%%%%%%%%%%%%%%%%%%%%%%%%%%%%%%%%%%%%%%%%%%%%%%%%%%%%%%%%%%%%%%%%%%%%%%%%%%%%%%%%%%%%

\section{Introduction}\label{sec:intro}

   %%%%%%%%%%%%%%%%%%%%%%%%%%%%%%%%%%%%%%%%%%%%%%%%%%%%%%%%%%%%%%%%%%%%%%%%%%%%%%%%%%%%%

\subsection{Summary}\label{ssec:summary}
\

The universal enveloping algebra $U(\fg)$ of a simple Lie algebra admits a famous quantization $U_q(\fg)$, the so-called
\emph{Drinfeld-Jimbo quantum group}. While for general $q$, the categories of their representations are equivalent to those
of $\fg$, the braiding $\hat{R}_{VW}\colon V\otimes W \iso W\otimes V$ is not a mere flip map $\tau$ as it is for $\fg$.
One can actually reverse engineer the above to build a quantum group from any solution $R_{VV}\in \End(V\otimes V)$
satisfying the \emph{Yang-Baxter} relation or equivalently $\hat{R}_{VV}:=R_{VV} \circ \tau$ satisfying the \emph{braid} relation:
\begin{equation*}
  R_{12}R_{13}R_{23}=R_{23}R_{13}R_{12} \qquad \mathrm{and} \qquad
  \hat{R}_{12}\hat{R}_{23}\hat{R}_{12}=\hat{R}_{23}\hat{R}_{12}\hat{R}_{23},
\end{equation*}
which goes back to the \emph{quantum inverse scattering method (QISM)} developed by Faddeev's school in 1980s.
In fact, all of the aforementioned braidings $\hat{R}_{VW}$ arise as images of the universal $R$-matrix
$\mathcal{R}\in U_q(\fg)\what{\otimes} U_q(\fg)$.
As $U_q(\fg)$ is actually a Drinfeld double of Cartan-enhanced subalgebras $U^\pm_q(\fg)$, the key ingredient in the formula
for $\mathcal{R}$ is a pair of dual bases of these subalgebras with respect to a Hopf pairing. To this end, one needs to
define \emph{root vectors} $e_{\pm \gamma}\in U^\pm_q(\fg)$ for every positive root $\gamma\in \Phi^+$ of $\fg$. This is
classically accomplished using Lusztig's braid group action (and depends on a choice of a reduced decomposition of the
longest element $w_0$ in the Weyl group of $\fg$). There is however a purely combinatorial approach to the construction
of $e_{\pm \gamma}$ that goes back to~\cite{K1,K2,L,Ro}. This relies on the Lalonde-Ram's bijection~\cite{LR}:
\begin{equation}\label{eq:LR-bijection}
  \ell \colon \Phi^+ \iso \Big\{\text{standard Lyndon words in}\ \Pi\Big\},
\end{equation}
where standard Lyndon words intrinsically depend on a fixed total order on the set $\Pi$ of simple roots in $\Phi^+$.
This bijection also gives rise to the \emph{lexicographical} order on $\Phi^+$ via:
\begin{equation}\label{eq:lyndon_order}
  \alpha < \beta \quad \Longleftrightarrow \quad \ell(\alpha) < \ell(\beta)  \ \ \mathrm{ lexicographically}.
\end{equation}
There is a canonical way (\emph{costandard factorization}) to split
\begin{equation}\label{eq:cost-factorization}
  \ell(\gamma)=\ell(\alpha)\ell(\beta) \qquad \mathrm{with} \quad \alpha,\beta\in \Phi^+,\ \gamma=\alpha+\beta,
\end{equation}
by maximizing the length of $\ell(\alpha)$. One then defines root vectors $e_{\pm \gamma}$ inductively via
\begin{equation}\label{eq:root_vectors_intro}
  e_\gamma=e_\alpha e_\beta - q^{(\alpha,\beta)}e_\beta e_\alpha \qquad \mathrm{and} \qquad
  e_{-\gamma}=e_{-\beta} e_{-\alpha} - q^{-(\alpha,\beta)} e_{-\alpha} e_{-\beta}.
\end{equation}

\medskip
Replacing simple $\fg$ with an affine Lie algebra $\what{\fg}$, it turns out that the representation theory of $U_q(\what{\fg})$
is best developed using an alternative realization known as the \emph{new Drinfeld realization} $U^D_q(\what{\fg})$. The isomorphism
\begin{equation}\label{eq:DJ=Dr_1param}
  U_q(\what{\fg})\iso U^D_q(\what{\fg})
\end{equation}
was stated by Drinfeld~\cite{D}, while its inverse was constructed by Beck~\cite{B} using the extended affine braid group action.
In the new Drinfeld realization, the infinite set of generators is nicely packed into the currents $X^\pm_i(z),\Omega^\pm_i(z)$.
We note that the above isomorphism does not intertwine the triangular decompositions of both algebras. Furthermore,
no explicit formulas for the Drinfeld-Jimbo coproduct of the generators in $U^D_q(\what{\fg})$ are known.
One can also apply the aforementioned QISM in the affine setup to introduce $U^{\mathrm{RTT}}_q(\what{\fg})$. The isomorphism
\begin{equation}\label{eq:RTT=Dr_1param}
  U^{D}_q(\what{\fg}) \iso U^{\mathrm{RTT}}_q(\what{\fg})
\end{equation}
was explicitly constructed in~\cite{DF} for type $A^{(1)}_n$ through the Gauss decomposition of the generating matrices.
The generalization of this result for other classical Lie algebras goes back to~\cite{HM} and much more recently to~\cite{JLM1, JLM2}.

\medskip
Although the theory of multiparameter quantum groups goes back to the early 1990s (see e.g.~\cite{AST,R,T}), much of the
current interest in the subject stems from the papers~\cite{BW1, BW2, BW3}, which studied the two-parameter quantum group
$U_{r,s}(\gl_n)$ and subsequently gave an application to pointed finite-dimensional Hopf algebras. In~\cite{BW2}, the authors
developed the theory of finite-dimensional representations in a complete analogy with the one-parameter case, computed the
two-parameter $R$-matrix for the first fundamental $U_{r,s}(\gl_n)$-representation, and used it to establish the Schur-Weyl
duality between $U_{r,s}(\gl_n)$ and a two-parameter Hecke algebra. These works of Benkart and Witherspoon stimulated an
increased interest in the theory of two-parameter quantum groups. In particular, the definitions of $\uu$ for other classical
simple Lie algebras $\fg$ were first given in~\cite{BGH1,BGH2}, where basic results on the structure and representation theory
of $\uu$ were also established. Subsequently, these algebras have been treated case-by-case in multiple papers. For a more
uniform treatment and complete references, we refer the reader to~\cite{HP1}. We note that due to the absence of Lusztig's braid
group action on $\uu$, one is forced to use the aforementioned combinatorial approach to PBW bases of $U^\pm_{r,s}(\fg)$.
One can accomplish this along with the construction of orthogonal PBW bases by using the embedding of $U^{+}_{r,s}(\fg)$
into the corresponding quantum shuffle algebra akin to~\cite{L,CHW}, see~\cite{MT2}.

The generalization to affine Lie algebras started with the work~\cite{HRZ} (which however had some gaps in the exposition,
fixed in~\cite{Ts}). Subsequently, some attempts were made to provide a uniform Drinfeld-Jimbo presentation of such algebras,
establishing the triangular decomposition and the Drinfeld double construction for them. More importantly, a new Drinfeld
realization of these algebras $\UU$ was established on a case-by-case basis for $\fg$ being of type $A_n$ (see~\cite{HRZ}),
types $D_n$ and $E_6$ (see~\cite{HZ}), type $G_2$ (see~\cite{GHZ}), and type $C_n$ (see~\cite{HZ2}). However, we note a
caveat in this treatment: while a surjective homomorphism from the Drinfeld-Jimbo to the new Drinfeld realization is constructed
similarly to~\eqref{eq:DJ=Dr_1param}, there is no proper proof of its injectivity. The aforementioned new Drinfeld realization of
$\UU$ was used to construct the vertex representations of $\UU$ in an analogy with the one-parameter case. Finally, the FRT-formalism
for two-parameter quantum groups (with two-parameter analogues of~\eqref{eq:RTT=Dr_1param}) was partially carried out in~\cite{JL2} for type
$A^{(1)}_n$ and was attempted (with major errors) in~\cite{HJZ,HXZ1,HXZ2} for other classical types $B^{(1)}_n,C^{(1)}_n,D^{(1)}_n$.

\medskip
It has been known for a long time that multiparameter quantum groups can often be obtained from the standard one-parameter ones
by twisting the coalgebra structure~\cite{R} or by twisting the algebra structure via a Hopf 2-cocycle~\cite{DT,M} (in fact,
the latter is known to be obtained by dualizing the former construction in the case of co-quasitriangular Hopf algebras).
In the context relevant to the present paper, this has been established in~\cite{HP1} for two-parameter quantum groups and
in~\cite{HPR} for multiparameter quantum groups, in their Drinfeld-Jimbo realization. On the other hand,
~\cite{HP1} noted that this approach is completely inapplicable to the new Drinfeld realization $U^D_{r,s}(\what{\fg})$.
It is also unclear to us how to relate representations of $U_{r,s}(\fg)$ to those of $U_{q}(\fg)$ through such twists.

In this paper, we rather focus on a different type of twist, originating from~\cite{AST}: twisting bigraded Hopf algebras by skew bicharacters.
While this idea has been already utilized in the setup of two-parameter quantum groups back in~\cite{HP} (see also~\cite[\S4]{FL} and~\cite{GG},
as well as~\cite[\S2]{FX} and~\cite[\S2]{CFLW} for multiparameter- and super- analogues only for the ``positive half''), our key observation
is that not only $U_{r,s}(\fg),U_{r,s}(\what{\fg})$ but also all aforementioned realizations\footnote{In the text, we rather use the terminology
``FRT construction'' and the notation $U(R), \wtd{U}(\wtd{R}(z))$ instead of $U^{\mathrm{RTT}}_q(\fg), U^{\mathrm{RTT}}_q(\what{\fg})$.}
$U^D_{r,s}(\what{\fg}), U^{\mathrm{RTT}}_{r,s}(\fg), U^{\mathrm{RTT}}_{r,s}(\what{\fg})$ can be constructed using these kinds of twists.
Moreover, both isomorphisms~\eqref{eq:DJ=Dr_1param} and~\eqref{eq:RTT=Dr_1param} are compatible with these twistings.
This has a number of important consequences (implying most of the properties of two-parameter quantum groups established
via technical calculations over the last two decades):
\begin{itemize}[leftmargin=0.7cm]

\item[--]
the construction of PBW-type bases of $U^\pm_{r,s}(\fg)$ and the proof of their orthogonality (see~\cite{MT2}, and much
of the preceding literature cited there)

\item[--]
the calculation of the finite $R$-matrices arising as standard braidings for the first fundamental representation of $U_{r,s}(\fg)$ for
classical $\fg$, and their affine counterparts (see~\cite{MT1})

\item[--]
the construction of the algebra isomorphisms $U_{r,s}(\fg)\simeq U^{\mathrm{RTT}}_{r,s}(\fg)$ and
$U^D_{r,s}(\what{\fg})\simeq U^{\mathrm{RTT}}_{r,s}(\what{\fg})$ (correcting and upgrading~\cite{JL1,JL2,HJZ,HXZ1,HXZ2})

\item[--]
the classification of finite-dimensional representations of $U^D_{r,s}(\what{\fg})$ in terms of Drinfeld polynomials and the description
of their pseudo-highest weights (see~\cite{JZ}); the construction of level-one vertex representations of $U^D_{r,s}(\what{\fg})$ for
simply-laced $\fg$ (see~\cite{HZ})

\item[--]
the multiparameter and super-case generalizations, cf.~\cite{Z,HJZ2} and~\cite{HT}.

\end{itemize}
We should stress right away two technical aspects of our approach. First, one needs to develop Cartan-doubled versions of all classical
results in the one-parameter setup (which relies on extending the usual one-parameter algebras by appropriate Laurent polynomial rings,
and then realizing the Cartan-doubled versions as subalgebras of these extensions). Second, to perform the bicharacter twists, one needs
to have compatible algebra bigradings on these Cartan-doubled versions.

   %%%%%%%%%%%%%%%%%%%%%%%%%%%%%%%%%%%%%%%%%%%%%%%%%%%%%%%%%%%%%%%%%%%%%%%%%%%%%%%%%%%%%

\subsection{Outline}\label{ssec:outline}
\

\noindent
The structure of the present paper is as follows:
\begin{itemize}[leftmargin=0.5cm]

\item[$\bullet$]
In Section~\ref{sec:notation}, we recall the two-parameter quantum groups $U_{r,s}(\fg)$ for simple finite-dimensional $\fg$
(Definition~\ref{def:general_2param}), the Hopf pairing for those (Proposition~\ref{prop:pairing_2param}) that yields their
Drinfeld double realization, and the standard intertwiners between tensor products of $U_{r,s}(\fg)$-modules in the category
$\mathcal{O}$ (Proposition~\ref{prop:universal-R}).

\item[$\bullet$]
In Section~\ref{sec:twisted_R_matrices}, we crucially utilize the general construction of a twist of a bigraded Hopf algebra
by a skew bicharacter, see~\eqref{eq:twisted_product_general}, which results in relating $U_{r,s}(\fg)$ to its one-parameter
version $U_{q,q^{-1}}(\fg)$ with $q=(r/s)^{1/2}$ (Proposition~\ref{prop:twisted_algebra}). This can be upgraded to
multiparameter quantum groups and associated quantum shuffle algebras (Remarks~\ref{rem:shuffle-twist},~\ref{rem:multiparameter}).
Furthermore, both the Hopf pairing and the intertwiners can be naturally obtained by twisting their one-parameter counterparts
(Propositions~\ref{prop:twisted_pairing},~\ref{prop:R-matrix-twist}). This allows for new short proofs
of~\cite[Theorems 4.4--4.6]{MT1}, which explicitly evaluated $\hat{R}_{r,s}\colon V\otimes V \iso V\otimes V$ when
$V$ is the first fundamental $U_{r,s}(\fg)$-module for classical $\fg$. Consequently, this also yields new proofs
of~\cite[Theorems 6.10--6.12]{MT1}, which evaluated the corresponding affine $R$-matrices $\hat{R}_{r,s}(z)$,
due to~\eqref{eq:Baxterization-B}--\eqref{eq:Baxterization-D}.

\item[$\bullet$]
In Section~\ref{sec:pbw}, we relate the quantum root vectors of $U_{r,s}(\fg)$ from~\cite{MT1,MT2} to their one-parameter counterparts
(Proposition~\ref{prop:pbw-basis-twist}). Combining this with Proposition~\ref{prop:twisted_pairing}, we then derive the orthogonality
of PBW bases of $U^\pm_{r,s}(\fg)$ through the well-known one-parameter version of~\cite{L} (Theorem~\ref{thm:orthogonal PBW bases}).
Finally, evoking the pairing constants for root vectors in one-parameter case due to~\cite{BKM} (Theorem~\ref{thm:pairing_BKM} and
Corollary~\ref{cor:1param_pairing_recursion}), we derive the nonzero constants in the pairing between PBW bases of $U^\pm_{r,s}(\fg)$
(Theorem~\ref{thm:2param_pairing_recursion}). The results of this section establish~\cite[Theorem 5.12]{MT1} and
generalize~\cite[Theorems~7.1--7.2]{MT2}.

\item[$\bullet$]
In Section~\ref{sec:subalgebra-interpret}, we provide another relation between the two-parameter and one-parameter quantum algebras,
realizing the former as subalgebras of certain extensions of the latter (Proposition~\ref{prop:subalgebra_interpretation_qgp}).
We also present a slight generalization of such construction (Proposition~\ref{prop:subalgebra_interpretation_general}, cf.\ the proof
of Proposition~\ref{prop:subalgebra_interpretation_qgp}). We conclude this Section by similarly realizing the Cartan-doubled
one-parameter algebra $U_{q,q^{-1}}(\fg)$ as a subalgebra of a certain extension of $U_q(\fg)$ with a ring of Laurent polynomials
(Proposition~\ref{prop:double_cartan_DJ}).

\item[$\bullet$]
In Section~\ref{sec:FRT construction finite}, we recall the FRT-construction that associates to every solution $R$ of the Yang-Baxter
equation a bialgebra $A(R)$ and a Hopf algebra $U(R)$. We equip these algebras with appropriate bigradings, and show that in the
present setup, the twisting of such algebras associated with the one-parameter $R_{q}$ recovers the same-named algebras associated
with the two-parameter $R_{r,s}$ (Corollaries~\ref{cor:A(R)_2_vs_1_parameter} and~\ref{cor:U(R)_2_vs_1_parameter}).
We use these results to derive the two-parameter generalization of the isomorphism between the Drinfeld-Jimbo and RTT-type
presentations of finite quantum groups in classical types
(Theorems~\ref{thm:DJ=RTT_Btype_2_param},~\ref{thm:DJ=RTT_Ctype_2param},~\ref{thm:DJ=RTT_Dtype_2param}),
by first upgrading the former to the Cartan-doubled setup
(Corollaries~\ref{cor:DJ=RTT_Btype_1param_double_cartan},~\ref{thm:DJ=RTT_Ctype_1param_double_cartan},~\ref{thm:DJ=RTT_Dtype_1_param_double_cartan}).
The latter is crucially based on the embedding of Proposition~\ref{prop:double_cartan_RTT_finite}, the proof of which is actually much more involved
than that of Proposition~\ref{prop:double_cartan_DJ} and requires a short detour into triangular decompositions and Drinfeld doubles
(the latter crucially utilizes the \emph{crossing symmetry} identities of Lemma~\ref{lem:crossing_symmetries}).

\item[$\bullet$]
In Section~\ref{sec:loop-realization}, we recall the new Drinfeld realization $U_{r,s}^{D}(\what{\fg})$ of two-parameter quantum affine algebras.
We equip them with a natural bigrading (Lemma~\ref{lem:Q_hat_bigrading}) and show that they can be realized as bicharacter twists of their
one-parameter counterparts $U_{q,q^{-1}}^{D}(\what{\fg})$
(Theorem~\ref{thm:twisted_algebra_loop}). Combining this with Proposition~\ref{prop:twisted_representation} allows to deduce the classification
of simple finite-dimensional $U_{r,s}^{D}(\what{\fg})$-modules of~\cite{JZ} and to obtain the level-one vertex $U_{r,s}^{D}(\what{\fg})$-modules
of~\cite{HZ} (Remarks~\ref{rem:finite-dimensional-2parameter},~\ref{rem:vertex-2parameter}). We also use Theorem~\ref{thm:twisted_algebra_loop}
to derive the two-parameter generalization of the isomorphism between the Drinfeld-Jimbo and new Drinfeld presentations of affine quantum groups
(Theorem~\ref{thm:DJ=D_2_parameter}), by first upgrading the latter to the Cartan-doubled setup (Theorem~\ref{thm:DJ=D_1_parameter}).

\item[$\bullet$]
In Section~\ref{sec:affine_RTT}, we develop the affine analogues of the results from Section~\ref{sec:FRT construction finite}. To this end,
we recall the FRT-construction that associates to every solution $R(z)$ of the Yang-Baxter equation with a spectral parameter a Hopf algebra $U(R(z))$.
We equip these algebras (and their quotients) with appropriate bigradings (Lemma~\ref{lem:affine_RTT_bigradings}), and show that in the present setup,
the twisting of such algebras associated with the one-parameter $R_{q}(z)$ recovers the same-named algebras associated with the two-parameter
$R_{r,s}(z)$ (Proposition~\ref{prop:2_vs_1_parameter_affine_RTT}). We use these results to derive the two-parameter generalization of the isomorphism
\eqref{eq:RTT=Dr_1param} between the new Drinfeld and RTT presentations of affine quantum groups in classical types
(Theorems~\ref{thm:D=affine_RTT_Btype_2param},~\ref{thm:D=RTT_Ctype_2param},~\ref{thm:D=RTT_Dtype_2param}),
by first upgrading the former to the Cartan-doubled setup
(Corollaries~\ref{cor:affine_1param_double_Btype},~\ref{cor:D=RTT_Ctype_1param_double_cartan},~\ref{cor:D=RTT_Dtype_1param_double_cartan}).
The latter is based on the embedding of Proposition~\ref{prop:double_cartan_RTT_affine}, the proof of which is similar to the finite case
and requires a detour into triangular decompositions (utilizing the \emph{crossing symmetry} identities of Lemma~\ref{lem:affine_crossing_symmetries}, the proof of which is presented in the end of Appendix~\ref{ssec:app_C}).

\item[$\bullet$]
In Appendix~\ref{sec:app_A}, we present the $A$-type counterparts of all our main constructions for $BCD$-types.

\item[$\bullet$]
In Appendix~\ref{ssec:app_B}, we present the $BD$-type counterparts of the key factorization from~\cite{Z}, that is, we evaluate the
\emph{restricted $R$-matrix} of $U_{r,s}(\fg)$ solely in the FRT-realization (again utilizing the twisting).

\item[$\bullet$]
In Appendix~\ref{ssec:app_C}, we relate the normalized $R$-matrix of~\eqref{eq:f_prefactor} to the evaluation of the universal $R$-matrix for
$U_{q}(\what{\frak{g}})$, see Theorem~\ref{thm:universal_R_evaluation}, and use this to finally prove the crossing symmetry identities of
Lemma~\ref{lem:affine_crossing_symmetries}.

\end{itemize}

   %%%%%%%%%%%%%%%%%%%%%%%%%%%%%%%%%%%%%%%%%%%%%%%%%%%%%%%%%%%%%%%%%%%%%%%%%%%%%%%%%%%%%

\subsection{Acknowledgement}\label{ssec:acknowl}
\

This paper is a sequel to~\cite{MT1,MT2} (in a couple of spots, we refer to~\cite{MT0}, which is essentially~\cite{MT1} with
extra Appendices) and represents the last part of the REU project at Purdue University; we are grateful to Purdue University for support.
We are grateful to A.~Molev for a correspondence on~\cite{JLM1}; to G.~Garcia and F.~Gavarini for a correspondence on~\cite{GG,GGP}; to
H.~Zhang for a correspondence on~\cite{Z}; and finally to the referee for their useful suggestions that improved the exposition.
A.T.\ is grateful for the hospitality and wonderful working conditions to IHES (Bures-sur-Yvette, France) where the major part of this paper
was written in the Summer 2025; to MPIM Bonn where the final version of this paper was prepared in July~2026.

The work of A.T.\ was partially supported by an NSF Grant DMS-$2302661$.

   %%%%%%%%%%%%%%%%%%%%%%%%%%%%%%%%%%%%%%%%%%%%%%%%%%%%%%%%%%%%%%%%%%%%%%%%%%
   %%%%%%%%%%%%%%%%%%%%%%%%%%%%%%%%%%%%%%%%%%%%%%%%%%%%%%%%%%%%%%%%%%%%%%%%%%
   %%%%%%%%%%%%%%%%%%%%%%%%%%%%%%%%%%%%%%%%%%%%%%%%%%%%%%%%%%%%%%%%%%%%%%%%%%

\section{Notations and definitions}\label{sec:notation}

In this Section, we recall the notion of two-parameter quantum groups $U_{r,s}(\fg)$ for simple finite-dimensional
Lie algebras~$\fg$, the Hopf algebra structure and the Hopf pairing on those, as well as the standard construction
of intertwiners arising from the universal $R$-matrix.

   %%%%%%%%%%%%%%%%%%%%%%%%%%%%%%%%%%%%%%%%%%%%%%%%%%%%%%%%%%%%%%%%%%%%%%%%%%

\subsection{Two-parameter quantum groups}
\

Let $E$ be a Euclidean space with a positive-definite symmetric bilinear form $(\cdot ,\cdot)$. Let $\Phi \subset E$ be an irreducible
reduced root system with an ordered set of simple roots $\Pi = \{\alpha_{1},\ldots ,\alpha_{n}\}$, and let $\fg$ be the corresponding
complex simple Lie algebra. We set $\fn^{\pm} = \bigoplus_{\alpha \in \Phi^{+}}\fg_{\pm \alpha}$,
where $\fg_{\alpha}$ denotes the root space of $\fg$ corresponding to $\alpha \in \Phi$, and $\Phi^+$ denotes the set of positive
roots of $\Phi$. Let $C = (a_{ij})_{i,j = 1}^{n}$ be the Cartan matrix of $\fg$, with entries given explicitly
by $a_{ij} = \frac{2(\alpha_{i},\alpha_{j})}{(\alpha_{i},\alpha_{i})}$, and set $d_{i} = \frac{1}{2}(\alpha_{i},\alpha_{i})$,
where $(\cdot,\cdot)$ is normalized so that the short roots have square length $2$. The root and weight lattices of $\fg$ will be
denoted by $Q$ and $P$,~respectively: $\bigoplus_{i=1}^n \BZ \alpha_i=Q\subset P=\bigoplus_{i=1}^n \BZ \varpi_i$ with
$(\alpha_i,\varpi_j)=d_i\delta_{ij}$. Let $Q^{+}=\bigoplus_{i=1}^n \BZ_{\geq 0} \alpha_i$ be the positive cone of the root
lattice $Q$, and for any $\lambda=\sum_{i=1}^n c_i\alpha_i\in Q^+$, we define its height as $\hgt(\lambda)=\sum_{i=1}^n c_i$.
Having fixed above the order on the set of simple roots $\Pi$, we consider the (modified) Ringel bilinear form
$\langle \cdot ,\cdot \rangle$ on $Q$, such that (unless $\{i,j\}=\{n-1,n\}$ in type $D_n$) we have:
\begin{equation*}
  \langle \alpha_{i},\alpha_{j}\rangle =
  \begin{cases}
     d_{i}a_{ij} & \text{if}\ \ i < j \\
     d_{i} & \text{if}\ \ i = j \\
     0 & \text{if}\ \ i > j
  \end{cases} \,,
\end{equation*}
while in the remaining case of $D_n$-type system, we set
\begin{equation*}
  \langle \alpha_{n-1},\alpha_{n} \rangle = -1, \qquad \langle \alpha_{n},\alpha_{n-1}\rangle = 1.
\end{equation*}
We note that $(\mu,\nu) = \langle \mu,\nu \rangle + \langle \nu,\mu \rangle$ for any $\mu,\nu \in Q$.

Let $\bb{K}$ be an algebraically closed field, and fix two transcendental elements $r,s \in \bb{K}$ such that $rs^{-1}$ is not
a root of unity. We also assume that there exists an automorphism of $\bb{K}$ mapping $r \mapsto r^{-1}$ and $s \mapsto s^{-1}$. 
Then we define the following two-parameter analogues of $q$-integers, $q$-factorials, and $q$-binomial coefficients:
\begin{equation*}
  [m]_{r,s} = \frac{r^{m} - s^{m}}{r - s} = r^{m-1} + r^{m-2}s + \ldots + rs^{m-2} + s^{m-1}
  \qquad \mathrm{for\ all} \quad  m\in \BZ_{\geq 0},
\end{equation*}
\begin{equation*}
  [m]_{r,s}! = [m]_{r,s} [m-1]_{r,s} \dots [1]_{r,s}  \qquad \mathrm{for} \quad m>0,
  \qquad [0]_{r,s}!=1,
\end{equation*}
and
\begin{equation*}
  \qbinom{m}{k}_{r,s} = \frac{[m]_{r,s}!}{[m - k]_{r,s}![k]_{r,s}!}
  \qquad \mathrm{for\ all} \quad 0\leq k\leq m.
\end{equation*}
Finally, we also define
\begin{equation}\label{eq:rs_gamma}
\begin{split}
  & r_{\gamma} = r^{(\gamma,\gamma)/2},\ \qquad s_{\gamma} = s^{(\gamma,\gamma)/2}
    \qquad \mathrm{for\ all} \quad \gamma\in \Phi,\\
  & r_{i} = r_{\alpha_i} = r^{d_{i}},\qquad s_{i} = s_{\alpha_i} = s^{d_{i}}
    \qquad \mathrm{for\ all} \quad 1\leq i\leq n.
\end{split}
\end{equation}

We now recall the definition of the \textbf{two-parameter quantum group} of $\fg$:

\begin{definition}\label{def:general_2param}
The two-parameter quantum group $U_{r,s}(\fg)$ of a simple Lie algebra $\fg$ is the associative $\bb{K}$-algebra
generated by $\{e_{i},f_{i},\omega_{i}^{\pm 1},(\omega_{i}')^{\pm 1}\}_{i=1}^{n}$ with the following defining relations
(for all $1\leq i,j\leq n$):
\begin{equation}\label{eq:R1}
   [\omega_i,\omega_j]=[\omega_i,\omega'_j]=[\omega'_i,\omega'_j]=0, \qquad
   \omega_{i}^{\pm 1}\omega_{i}^{\mp 1} = 1 = (\omega_{i}')^{\pm 1}(\omega_{i}')^{\mp 1},
\end{equation}
\begin{equation}\label{eq:R2}
   \omega_{i}e_{j} = r^{\langle \alpha_{j},\alpha_{i}\rangle}s^{-\langle \alpha_{i},\alpha_{j}\rangle}e_{j}\omega_{i}, \qquad
   \omega_{i}f_{j} = r^{-\langle \alpha_{j},\alpha_{i}\rangle}s^{\langle \alpha_{i},\alpha_{j}\rangle}f_{j}\omega_{i},
\end{equation}
\begin{equation}\label{eq:R3}
   \omega_{i}'e_{j} = r^{-\langle \alpha_{i},\alpha_{j}\rangle}s^{\langle \alpha_{j},\alpha_{i}\rangle}e_{j}\omega_{i}', \qquad
   \omega_{i}'f_{j} = r^{\langle \alpha_{i},\alpha_{j}\rangle}s^{-\langle \alpha_{j},\alpha_{i}\rangle}f_{j}\omega_{i}',
\end{equation}
\begin{equation}\label{eq:R4}
    e_{i}f_{j} - f_{j}e_{i} = \delta_{ij}\frac{\omega_i-\omega'_i}{r_{i} - s_{i}},
\end{equation}
and quantum $(r,s)$-Serre relations
\begin{equation}\label{eq:R5}
\begin{split}
  & \sum_{k = 0}^{1 - a_{ij}} (-1)^k \qbinom{1 - a_{ij}}{k}_{r_{i},s_{i}} (r_{i}s_{i})^{\frac{1}{2}k(k-1)}
    (rs)^{k\langle \alpha_{j},\alpha_{i}\rangle}e_{i}^{1 - a_{ij} - k}e_{j}e_{i}^{k} = 0
    \qquad \mathrm{for}\ i\ne j, \\
  & \sum_{k = 0}^{1 - a_{ij}} (-1)^k \qbinom{1 - a_{ij}}{k}_{r_{i},s_{i}} (r_{i}s_{i})^{\frac{1}{2}k(k-1)}
    (rs)^{k\langle \alpha_{j},\alpha_{i}\rangle}f_{i}^{k}f_{j}f_{i}^{1-a_{ij} - k} = 0
    \qquad \mathrm{for}\  i\ne j.
\end{split}
\end{equation}
\end{definition}

We note that the algebra $U_{r,s}(\fg)$ admits a $Q$-grading, defined on the generators via:
\begin{equation*}
  \deg(e_{i})=\alpha_i,\quad \deg(f_i)=-\alpha_i, \quad
  \deg(\omega_{i})=0, \quad \deg(\omega_{i}')=0 \qquad \mathrm{for\ all} \quad 1\leq i\leq n.
\end{equation*}
For $\mu\in Q$, let $U_{r,s}(\fg)_{\mu}$ (or simply $(U_{r,s})_{\mu}$) denote the degree $\mu$ component of
$U_{r,s}(\fg)$ under this $Q$-grading. We shall also need several subalgebras of $U_{r,s}(\fg)$:
\begin{itemize}

\item
the ``positive'' subalgebra $U_{r,s}^{+}=U_{r,s}^{+}(\fg)$, generated by $\{e_{i}\}_{i=1}^{n}$,

\item
the ``negative'' subalgebra $U_{r,s}^{-}= U_{r,s}^{-}(\fg)$, generated by $\{f_{i}\}_{i=1}^{n}$,

\item
the ``Cartan'' subalgebra $U_{r,s}^{0}= U_{r,s}^{0}(\fg)$, generated by $\{\omega_{i}^{\pm 1},(\omega'_{i})^{\pm 1}\}_{i=1}^{n}$,

\item
the ``non-negative subalgebra'' $U_{r,s}^{\ge}= U_{r,s}^{\ge}(\fg)$, generated by $\{e_{i},\omega_{i}^{\pm 1}\}_{i=1}^{n}$,

\item
the ``non-positive subalgebra'' $U_{r,s}^{\le}= U_{r,s}^{\le}(\fg)$, generated by $\{f_{i},(\omega'_{i})^{\pm 1}\}_{i=1}^{n}$.

\end{itemize}
Evoking~\eqref{eq:R1}, for any $\mu=\sum_{i=1}^{n} c_{i}\alpha_{i}\in Q$, we define
$\omega_\mu,\omega'_\mu\in U_{r,s}^{0}(\fg)$ via:
\begin{equation*}
  \omega_{\mu} = \omega_{1}^{c_{1}}\omega_{2}^{c_{2}} \ldots \omega_{n}^{c_{n}}, \qquad
  \omega'_\mu = (\omega_{1}')^{c_{1}}(\omega_{2}')^{c_{2}} \ldots (\omega_{n}')^{c_{n}}.
\end{equation*}
The following result is proved in a standard way (see~\cite[Proposition 2.8]{MT2}):

\begin{prop}\label{prop:dim-count}
For all $\mu \in Q^{+}$, we have
  $\dim_{\bb{K}}(U_{r,s}^{+})_{\mu} = \dim_{\bb{K}}(U_{r,s}^{-})_{-\mu} = \dim_{\BC}U(\fn^{+})_{\mu} = \dim_{\BC}U(\fn^{-})_{-\mu}$.
\end{prop}

The algebra $U_{r,s}(\fg)$ has a Hopf algebra structure, where the coproduct $\Delta$,
counit $\epsilon$, and antipode $S$ are defined on the generators by the following formulas:
\begin{align*}
  &\Delta(\omega_{i}^{\pm 1}) = \omega_{i}^{\pm 1} \otimes \omega_{i}^{\pm 1}
  & &\epsilon(\omega_{i}^{\pm 1}) = 1
  & &S(\omega_{i}^{\pm 1}) = \omega_{i}^{\mp 1} \\
  &\Delta((\omega_{i}')^{\pm 1}) = (\omega_{i}')^{\pm 1} \otimes (\omega_{i}')^{\pm 1}
  & & \epsilon((\omega_{i}')^{\pm 1}) = 1
  & & S((\omega_{i}')^{\pm 1}) = (\omega_{i}')^{\mp 1} \\
  &\Delta(e_{i}) = e_{i} \otimes 1 + \omega_{i} \otimes e_{i}
  & &\epsilon(e_{i}) = 0
  & &S(e_{i}) = -\omega_{i}^{-1}e_{i} \\
  &\Delta(f_{i}) = 1 \otimes f_{i} + f_{i} \otimes \omega_{i}'
  & &\epsilon(f_{i}) = 0
  & &S(f_{i}) = -f_{i}(\omega_{i}')^{-1}
\end{align*}

We shall also need to work with the usual one-parameter quantum groups $U_{q}(\fg)$ (see~\cite[\S 4.3]{J} for a definition), and to avoid confusion
with $U_{q,q^{-1}}(\fg)$, we denote the generators of $U_{q}(\fg)$ by $\{E_{i},F_{i},K_{i}^{\pm 1}\}_{i=1}^{n}$. We note that there is a Hopf algebra
homomorphism
\begin{equation}\label{eq:Uq_quotient}
  \pi\colon U_{q,q^{-1}}(\fg) \longrightarrow U_{q}(\fg)\qquad  \text{such that}\qquad
  \pi(e_{i}) = E_{i},\quad \pi(f_{i}) = F_{i},\quad \pi(\omega_{i}) = K_{i},\quad \pi(\omega_{i}') = K_{i}^{-1}.
\end{equation}
Moreover, if $U_{q}^{+}$ (resp.\ $U_{q}^{-}$) is the subalgebra of $U_{q}(\fg)$ generated by $\{E_{i}\}_{i = 1}^{n}$ (resp.\ $\{F_{i}\}_{i = 1}^{n}$),
then we have:

\begin{cor}\label{cor:double_cartan_+-_iso}
For any transcendental $q \in \bb{K}$, we have algebra isomorphisms $U_{q,q^{-1}}^{+} \iso U_{q}^{+}$ and
$U_{q,q^{-1}}^{-} \iso U_{q}^{-}$ given by $e_{i} \mapsto E_{i}$ and $f_{i} \mapsto F_{i}$, respectively.
\end{cor}

\begin{proof}
It is clear that $\pi((U_{q,q^{-1}}^{\pm})_{\pm \mu}) = (U_{q}^{\pm})_{\pm \mu}$ for all $\mu \in Q^{+}$, where $\pi$ is the
Hopf algebra homomorphism of~\eqref{eq:Uq_quotient}, and we have
$\dim_{\bb{K}}(U_{q,q^{-1}}^{\pm})_{\pm \mu} = \dim_{\BC}U(\fn^{+})_{\mu} = \dim_{\bb{K}}(U_{q}^{\pm})_{\pm \mu} < \infty$
by Proposition~\ref{prop:dim-count} and~\cite[\S8.24]{J}. Thus, the restrictions of $\pi$ to $U_{q,q^{-1}}^{\pm}$ give
the corresponding isomorphisms.
\end{proof}

   %%%%%%%%%%%%%%%%%%%%%%%%%%%%%%%%%%%%%%%%%%%%%%%%%%%%%%%%%%%%%%%%%%%%%%%%%%

\subsection{Hopf pairing}
\

In this Subsection, we recall the Hopf pairing on $U_{r,s}(\fg)$, which allows one to realize $U_{r,s}(\fg)$
as a Drinfeld double of its Hopf subalgebras $U_{r,s}^{\le}(\fg)$, $U_{r,s}^{\ge}(\fg)$:

\begin{prop}\label{prop:pairing_2param}
There exists a unique non-degenerate bilinear pairing
\begin{equation}\label{eq:Hopf-parity}
  (\cdot,\cdot)=(\cdot,\cdot)_{r,s}\colon U_{r,s}^{\le}(\fg) \times U_{r,s}^{\ge}(\fg) \longrightarrow \bb{K}
\end{equation}
satisfying the following structural properties
\begin{equation*}
  (yy',x) = (y \otimes y',\Delta(x)), \qquad (y, xx') = (\Delta(y),x' \otimes x)
  \qquad \forall\, x,x' \in U_{r,s}^{\ge}(\fg),\ y,y' \in U_{r,s}^{\le}(\fg),
\end{equation*}
where $(x'\otimes x'',y'\otimes y'')=(x',y') (x'',y'')$, as well as being given on the generators by:
\begin{equation*}
  (f_{i},\omega_{j}) = 0, \qquad (\omega_{i}', e_{i}) = 0, \qquad
  (f_{i},e_{j}) = \delta_{ij}\frac{1}{s_{i} - r_{i}} \qquad \mathrm{for\ all} \quad 1\leq i,j\leq n ,
\end{equation*}
\begin{equation}\label{eq:generators-parity-2}
  (\omega_{\lambda}',\omega_{\mu}) = r^{\langle \lambda,\mu\rangle}s^{-\langle \mu,\lambda \rangle}
  \qquad \mathrm{for\ all} \quad \lambda,\mu\in Q.
\end{equation}
\end{prop}

\smallskip

We may formally extend~\eqref{eq:generators-parity-2} to the weight lattice $P$ as follows:
\begin{equation}\label{eq:pairing_extension}
(\omega_{\lambda}',\omega_{\mu}) = \prod_{i,j = 1}^{n}(\omega_{i}',\omega_{j})^{\lambda_{i}\mu_{j}}\qquad \text{where}\qquad \lambda = \sum_{i = 1}^{n}\lambda_{i}\alpha_{i} \in P,\ \mu = \sum_{i = 1}^{n}\mu_{i}\alpha_{i} \in P.
\end{equation}
This requires fixing a choice of $d$-th roots for $r$ and $s$ (which is possible because $\bb{K}$ is algebraically closed), where $d$ is the square of the exponent of $P/Q$.

The above pairing is clearly homogeneous of degree zero with respect to the above $Q$-grading:
\begin{equation}\label{eq:pairing-orthogonal}
  (y,x) = 0 \quad \mathrm{for} \quad
  x\in U_{r,s}^{\ge}(\fg)_{\mu},\ y\in U_{r,s}^{\le}(\fg)_{-\nu}
  \quad \mathrm{with} \quad \mu\ne \nu \in Q^+.
\end{equation}

Since in this paper we are often interested in the classical Lie algebras $\fg$, it will be helpful to have
more explicit formulas for the bilinear form, avoiding the use of the form $\langle \cdot,\cdot \rangle$ on $Q$.
To this end, let us first recall the explicit realization of the classical root systems as well as
specify the choice of simple roots for them:
\begin{itemize}

\item
\emph{$A_{n}$-type} (corresponding to $\fg\simeq \ssl_{n+1}$).

Let $\{\varepsilon_{i}\}_{i=1}^{n+1}$ be an orthonormal basis of $\bb{R}^{n+1}$. Then, we have
\begin{equation}
\begin{split}
  & \Phi_{A_{n}} = \{\varepsilon_{i} - \varepsilon_{j} \,|\, 1\leq i\ne j\leq n+1\} \subset \bb{R}^{n+1}, \\
  & \Pi_{A_{n}} = \{\alpha_{i} = \varepsilon_{i} - \varepsilon_{i + 1}\}_{i=1}^{n}.
\end{split}
\end{equation}

\item
\emph{$B_{n}$-type} (corresponding to $\fg\simeq \sso_{2n+1}$).

Let $\{\varepsilon_{i}\}_{i=1}^{n}$ be an orthogonal basis of $\bb{R}^{n}$ with $(\varepsilon_{i},\varepsilon_{i})=2$ for all $i$.
Then, we have
\begin{equation}
\begin{split}
  & \Phi_{B_{n}} = \{\pm \varepsilon_{i} \pm \varepsilon_{j} \,|\, 1 \le i < j \le n\} \cup
    \{\pm \varepsilon_{i} \,|\, 1\leq i\leq n\} \subset \bb{R}^{n}, \\
  & \Pi_{B_{n}} = \{\alpha_{i} = \varepsilon_{i} - \varepsilon_{i + 1}\}_{i=1}^{n-1} \cup \{\alpha_{n} = \varepsilon_{n}\}.
\end{split}
\end{equation}

\item
\emph{$C_{n}$-type} (corresponding to $\fg\simeq \ssp_{2n}$).

Let $\{\varepsilon_{i}\}_{i=1}^{n}$ be an orthonormal basis of $\bb{R}^{n}$. Then, we have
\begin{equation}
\begin{split}
  & \Phi_{C_{n}} = \{\pm \varepsilon_{i} \pm \varepsilon_{j} \,|\, 1 \le i < j \le n\}
    \cup \{\pm 2\varepsilon_{i} \,|\, 1\leq i\leq n\} \subset \bb{R}^{n}, \\
  & \Pi_{C_{n}} = \{\alpha_{i} = \varepsilon_{i} - \varepsilon_{i + 1}\}_{i=1}^{n-1} \cup \{\alpha_{n} = 2\varepsilon_{n}\}.
\end{split}
\end{equation}

\item
\emph{$D_{n}$-type} (corresponding to $\fg\simeq \sso_{2n}$).

Let $\{\varepsilon_{i}\}_{i=1}^{n}$ be an orthonormal basis of $\bb{R}^{n}$. Then, we have
\begin{equation}\label{eq:D-system}
\begin{split}
  & \Phi_{D_{n}} = \{\pm \varepsilon_{i} \pm \varepsilon_{j} \,|\, 1 \le i < j \le n\} \subset \bb{R}^{n}, \\
  & \Pi_{D_{n}} = \{\alpha_{i} = \varepsilon_{i} - \varepsilon_{i + 1}\}_{i=1}^{n-1}
    \cup \{\alpha_{n} = \varepsilon_{n-1} + \varepsilon_{n}\}.
\end{split}
\end{equation}

\end{itemize}

Then we have the following respective formulas for the pairing of Cartan elements (we note that while the second formula
follows from the first one in each case, it will be convenient to use both later on):
\begin{itemize}

\item
\emph{$A_{n}$-type}
\begin{align}\label{eq:A-pairing}
\begin{split}
  (\omega_{\lambda}',\omega_{i}) &= r^{( \varepsilon_{i},\lambda )}s^{( \varepsilon_{i +1},\lambda)}, \\
  (\omega_{i}',\omega_{\lambda}) &= r^{-(\varepsilon_{i + 1},\lambda )}s^{-( \varepsilon_{i},\lambda)}.
\end{split}
\end{align}

\item
\emph{$B_{n}$-type}
\begin{align}\label{eq:B-pairing}
\begin{split}
  (\omega_{\lambda}',\omega_{i}) &=
  \begin{cases}
    r^{( \varepsilon_{i},\lambda )}s^{( \varepsilon_{i + 1},\lambda )}
      & \mathrm{if}\ \ 1\leq i<n  \\
    r^{( \varepsilon_{n},\lambda ) }(rs)^{-\lambda_{n}}
      & \mathrm{if}\ \ i = n
  \end{cases} \,, \\
  (\omega_{i}',\omega_{\lambda}) &=
  \begin{cases}
    r^{-( \varepsilon_{i + 1},\lambda)}s^{-( \varepsilon_{i},\lambda)}
      & \mathrm{if}\ \ 1\leq i<n \\
    s^{-( \varepsilon_{n},\lambda ) }(rs)^{\lambda_{n}}
      & \mathrm{if}\ \ i = n
  \end{cases} \,.
\end{split}
\end{align}

\item
\emph{$C_{n}$-type}
\begin{align}\label{eq:C-pairing}
\begin{split}
  (\omega_{\lambda}',\omega_{i}) &=
  \begin{cases}
    r^{( \varepsilon_{i},\lambda) }s^{( \varepsilon_{i + 1},\lambda )}
      &  \mathrm{if}\ \ 1\leq i<n \\
    r^{2( \varepsilon_{n},\lambda ) }(rs)^{-2\lambda_{n}}
      & \mathrm{if}\ \ i=n
  \end{cases} \,, \\
  (\omega_{i}',\omega_{\lambda}) &=
  \begin{cases}
    r^{-( \varepsilon_{i + 1},\lambda )}s^{-( \varepsilon_{i},\lambda)}
      & \mathrm{if}\ \ 1\leq i<n \\
    s^{-2( \varepsilon_{n},\lambda )}(rs)^{2\lambda_{n}}
      & \mathrm{if}\ \ i = n
  \end{cases} \,.
\end{split}
\end{align}

\item
\emph{$D_{n}$-type}
\begin{align}\label{eq:D-pairing}
\begin{split}
  (\omega_{\lambda}',\omega_{i}) &=
  \begin{cases}
    r^{( \varepsilon_{i},\lambda )}s^{( \varepsilon_{i + 1},\lambda )}
      & \mathrm{if}\ \ 1\leq i<n \\
    r^{( \varepsilon_{n-1},\lambda)}s^{-(\varepsilon_{n},\lambda)}(rs)^{-2\lambda_{n - 1}}
      & \mathrm{if}\ \ i = n
  \end{cases} \,, \\
  (\omega_{i}',\omega_{\lambda}) &=
  \begin{cases}
    r^{-( \varepsilon_{i + 1},\lambda )}s^{-(\varepsilon_{i},\lambda )}
      & \mathrm{if}\ \ 1\leq i<n \\
    r^{( \varepsilon_{n},\lambda )}s^{-(\varepsilon_{n-1},\lambda )}(rs)^{2\lambda_{n-1}}
      & \mathrm{if}\ \ i = n
  \end{cases} \,.
\end{split}
\end{align}

\end{itemize}

   %%%%%%%%%%%%%%%%%%%%%%%%%%%%%%%%%%%%%%%%%%%%%%%%%%%%%%%%%%%%%%%%%%%%%%%%%%

\subsection{Representations and standard intertwiners}\label{ssec:reps and braidings}
\

For any $U_{r,s}(\fg)$-module $V$, a vector $v\in V$ is said to have weight $\lambda \in P$ if
\[
  \omega_{i}\cdot v = (\omega_{\lambda}',\omega_{i})v \qquad \text{and} \qquad
  \omega_{i}'\cdot v = (\omega_{i}',\omega_{\lambda})^{-1}v \quad \text{for all}\quad 1 \le i \le n.
\]
Let $V[\lambda]$ denote the subspace of all vectors of weight $\lambda$ in $V$. Similarly to the one-parameter case,
consider the category $\mathcal{O}$ consisting of $U_{r,s}(\fg)$-modules $V$ which have a weight space
decomposition $V=\bigoplus_{\lambda\in P} V[\lambda]$ with $\dim V[\lambda]<\infty$ for all $\lambda$ and such that the
support $\mathrm{supp}(V)=\{\lambda\in P \,|\, V[\lambda]\ne 0\}$ is contained in a finite union of sets $D(\mu)=\mu-Q^+$
for some $\mu\in P$. The morphisms in $\mathcal{O}$ are $U_{r,s}(\fg)$-module morphisms. This category contains highest
weight Verma modules $\{M_\lambda\}_{\lambda\in P}$ as well as any finite-dimensional $U_{r,s}(\fg)$-modules with a
weight space decomposition. Moreover, any simple finite-dimensional $U_{r,s}(\fg)$-module is
obtained from an irreducible quotient $L_\lambda$ of $M_\lambda\ (\lambda\in P^+)$ by tensoring it with a one-dimensional
$U_{r,s}(\fg)$-module. Using standard arguments, one can construct natural intertwiners between tensor products of modules in
$\mathcal{O}$ arising through the universal $R$-matrix, which we shall recall now.\footnote{For classical $\fg$ these results are
contained in~\cite[\S2-4]{BW2} and~\cite[\S2-3]{BGH2}, while for general $\fg$ they can be similarly derived (or alternatively they
can be deduced from their one-parameter counterparts through the results of Subsections 3.1--3.2 below).}

To do so, pick dual bases $\{\tilde{x}_{i}^{\mu}\}$, $\{\tilde{y}_{i}^{\mu}\}$ of $U^{+}_{r,s}(\fg)_{\mu}$, $U^{-}_{r,s}(\fg)_{-\mu}$
with respect to the Hopf pairing~\eqref{eq:Hopf-parity}, and set
\begin{equation}\label{eq:Theta}
  \Theta_{r,s} = 1 + \sum_{\mu > 0}\Theta_{\mu} \qquad \text{with} \qquad
  \Theta_{\mu} = \sum_{i} \tilde{y}_{i}^{\mu} \otimes \tilde{x}_{i}^{\mu}.
\end{equation}
Let $\tau\colon W_1 \otimes W_2 \to W_2 \otimes W_1$ be the flip map
\[
  \tau(w_1 \otimes w_2) = w_2 \otimes w_1 \quad \text{for all}\quad w_1 \in W_1,\ w_2 \in W_2.
\]
Finally, define $\tilde{f}_{r,s}\colon W_1 \otimes W_2 \to W_1 \otimes W_2$ via
\begin{equation}\label{eq:tildef-map}
  \tilde{f}_{r,s}(w_1 \otimes w_2) = (\omega_{\lambda_2}',\omega_{\lambda_1})^{-1} w_1 \otimes w_2
  \quad \text{for all}\quad w_1 \in W_1[\lambda_1],\ w_2 \in W_2[\lambda_2].
\end{equation}
The following is standard (cf.~\cite[Theorem 4.1]{MT1}):

\begin{prop}\label{prop:universal-R}
For any $U_{r,s}(\fg)$-modules $W_1$ and $W_2$ in the category $\mathcal{O}$, the map
\begin{equation}\label{eq:R-int}
  \hat{R}_{W_1,W_2}=\Theta_{r,s} \circ \wtd{f} \circ \tau\colon W_1 \otimes W_2 \iso W_2 \otimes W_1
\end{equation}
is an isomorphism of $U_{r,s}(\fg)$-modules.
\end{prop}

   %%%%%%%%%%%%%%%%%%%%%%%%%%%%%%%%%%%%%%%%%%%%%%%%%%%%%%%%%%%%%%%%%%%%%%%%%%
   %%%%%%%%%%%%%%%%%%%%%%%%%%%%%%%%%%%%%%%%%%%%%%%%%%%%%%%%%%%%%%%%%%%%%%%%%%
   %%%%%%%%%%%%%%%%%%%%%%%%%%%%%%%%%%%%%%%%%%%%%%%%%%%%%%%%%%%%%%%%%%%%%%%%%%

\section{Two-parameter $R$-matrices via twisting}\label{sec:twisted_R_matrices}

Let $G$ be an abelian group, and let $A$ be a bialgebra over a field $F$, with the coproduct $\Delta$ and counit~$\epsilon$.
Recall that $A$ is a \textbf{$G$-bigraded bialgebra} if it is $G \times G$-graded as an algebra, and satisfies the following
additional properties for all $\alpha,\beta \in G$:
\begin{equation}\label{eq:bigraded-bialgebra}
  \Delta(A_{\alpha,\beta}) \subset \sum_{\gamma \in G}A_{\alpha,\gamma} \otimes A_{-\gamma,\beta} \qquad \text{and}\qquad
  \epsilon(A_{\alpha,\beta}) = 0 \ \ \text{if} \ \ \alpha + \beta \neq 0.
\end{equation}
If $A$ is a Hopf algebra with the antipode $S$, then $A$ is a \textbf{$G$-bigraded Hopf algebra} if
it satisfies~\eqref{eq:bigraded-bialgebra} and
\begin{equation}\label{eq:bigraded-antipode}
  S(A_{\alpha,\beta}) \subset A_{\beta,\alpha}\qquad \text{for all}\ \alpha,\beta \in G.
\end{equation}

We call a function $\zeta\colon G \times G \to F^{*}$ a \textbf{skew bicharacter} if the following properties are satisfied:
\begin{equation}
\begin{split}
  & \zeta(\alpha + \beta,\gamma) = \zeta(\alpha,\gamma)\zeta(\beta,\gamma), \qquad \zeta(\alpha,\beta + \gamma) = \zeta(\alpha,\beta)\zeta(\alpha,\gamma),\\
  & \zeta(\alpha,\beta) = \zeta(\beta,\alpha)^{-1},\qquad \zeta(\alpha,\alpha) = 1  \qquad \forall\, \alpha,\beta,\gamma \in G.
\end{split}
\end{equation}
Given any $G \times G$-graded algebra $A$ and a skew bicharacter $\zeta\colon G \times G \to F^{*}$, we may define an algebra
$A_{\zeta}$ with multiplication given by:
\begin{equation}\label{eq:twisted_product_general}
  a \circ b = \zeta(\alpha,\alpha')\zeta(\beta,\beta')^{-1}ab\quad \text{for all}\quad a \in A_{\alpha,\beta},\ b \in A_{\alpha',\beta'}.
\end{equation}
If we assume additionally that $A$ is a $G$-bigraded bialgebra as defined above, then due to~\eqref{eq:bigraded-bialgebra}, the maps
$\Delta$ and $\epsilon$ make $A_{\zeta}$ into a bialgebra. Moreover, if $A$ is a $G$-bigraded Hopf algebra with antipode $S$, then
$A_{\zeta}$ is also a Hopf algebra with the same antipode $S$.

   %%%%%%%%%%%%%%%%%%%%%%%%%%%%%%%%%%%%%%%%%%%%%%%%%%%%%%%%%%%%%%%%%%%%%%%%%%

\subsection{Twisting of quantum groups and Hopf pairing}\label{ssec:quantum_group_twisting}
\

Let $q = r^{1/2}s^{-1/2} \in \bb{K}$. Consider the algebra
$U_{q,q^{-1}}(\fg)$. It is easy to verify that the assignment
\begin{equation}\label{eq:our-bigrading}
  \deg(e_{i}) = (\alpha_{i},0), \qquad \deg(f_{i}) = (0,-\alpha_{i}), \qquad
  \deg(\omega_{i}) = \deg(\omega_{i}') = (\alpha_{i},-\alpha_{i}),
\end{equation}
makes $U_{q,q^{-1}}(\fg)$ into a $Q$-bigraded Hopf algebra in the sense described above.
Following~\cite[\S 5]{HP}, let
\[
  p_{ij} = r^{\langle \alpha_{j},\alpha_{i}\rangle} s^{-\langle \alpha_{i},\alpha_{j}\rangle}q^{-d_{i}a_{ij}}
  \quad \text{for all} \quad 1 \le i,j \le n.
\]
Then we have $p_{ii} = 1$ and $p_{ij}p_{ji} = 1$ for all $i,j$. Thus, the map $\zeta\colon Q \times Q \to \bb{K}^{*}$
defined by $\zeta(\alpha_{i},\alpha_{j}) = p_{ij}^{1/2}$ is a skew bicharacter, and therefore we may define an algebra
$U_{q,\zeta}(\fg) = (U_{q,q^{-1}}(\fg))_{\zeta}$ by changing the multiplication via~\eqref{eq:twisted_product_general}.
Since $U_{q,q^{-1}}(\fg)$ is a $Q$-bigraded Hopf algebra, the algebra $U_{q,\zeta}(\fg)$ is also a Hopf algebra,
with the coproduct, counit, and antipode being the same as those of $U_{q,q^{-1}}(\fg)$. We also define subalgebras
$U^{\leq}_{q,\zeta} = (U^{\leq}_{q,q^{-1}})_{\zeta}$ and $U^{\geq}_{q,\zeta} = (U^{\geq}_{q,q^{-1}})_{\zeta}$
of $U_{q,\zeta}(\fg)$.

We may extend $\zeta$ to the weight lattice $P$ of $\fg$ by defining
\[
  \zeta(\lambda,\mu) = \prod_{1\leq i,j \leq n}\zeta(\alpha_{i},\alpha_{j})^{\lambda_{i}\mu_{j}}
\]
for any $\lambda = \sum_{i = 1}^{n}\lambda_{i}\alpha_{i}$, $\mu = \sum_{i = 1}^{n}\mu_{i}\alpha_{i}$ in $P$.
The definition of $p_{ij}$ implies the following useful formula for $\zeta$:
\begin{equation}\label{eq:zeta_formula}
  \zeta(\lambda,\mu) = \left( r^{\langle \mu,\lambda\rangle}s^{-\langle \lambda,\mu\rangle}q^{-(\lambda,\mu)} \right)^{1/2}
  = (\omega_{\mu}',\omega_{\lambda})^{1/2}q^{-\frac{1}{2}(\lambda,\mu)} \quad \text{for all}\quad \lambda,\mu \in P.
\end{equation}
Herein, we fix appropriate roots of $r$ and $s$, which are compatible with the choices made in~\eqref{eq:pairing_extension}.

The next result is essentially due to~\cite{HP}:

\begin{prop}\label{prop:twisted_algebra}
There is a unique isomorphism $\varphi\colon U_{r,s}(\fg) \iso U_{q,\zeta}(\fg)$ of Hopf algebras such that
\begin{equation}\label{eq:phi-map}
  \varphi(e_{i}) = e_{i},\quad \varphi(f_{i}) = (r_{i}s_{i})^{-1/2}f_{i},\quad
  \varphi(\omega_{i}) = \omega_{i},\quad \varphi(\omega_{i}') = \omega_{i}' \qquad \forall\, 1\leq i\leq n.
\end{equation}
\end{prop}

\begin{remark}
A similar construction also featured in~\cite{FL}, while its version (restricted to the ``positive half'')
played the key role in~\cite{CFLW} to establish direct connections between quantum groups and supergroups.
\end{remark}

\begin{remark}\label{rem:shuffle-twist}
Recall the algebra embedding $\Upsilon_{r,s}\colon U^+_{r,s}\hookrightarrow (\mathcal{F}_{r,s},\ast_{r,s})$ of~\cite[Proposition 4.6]{MT2}
(which coincides with that of~\cite[Proposition 4.4]{MT2}), where $\mathcal{F}_{r,s}$ is the two-parameter shuffle algebra of~\cite[\S4.1]{MT2}.
In the special case $r=q=s^{-1}$, we shall rather use a single index $q$. The algebra $\mathcal{F}_{r,s}$ is $Q$-bigraded by declaring degrees
of words as $\deg([i_1 \ldots i_k])=(\alpha_{i_1}+\ldots+\alpha_{i_k},0)$. Then, the vector space isomorphism $\varphi^{\mathcal{F}}$ defined
in the word basis via $[i_1\ldots i_k]\mapsto [i_1\ldots i_k]\cdot \prod_{1\leq b<a\leq k} \zeta(\alpha_{i_a},\alpha_{i_b})$ is actually
an algebra isomorphism
  $$\varphi^{\mathcal{F}}\colon (\mathcal{F}_{r,s},\ast_{r,s}) \iso (\mathcal{F}_{q},\ast_{q})_\zeta.$$
Moreover, it is compatible with $\varphi$ of Proposition~\ref{prop:twisted_algebra}:
  $\varphi^{\mathcal{F}}(\Upsilon_{r,s}(x))=\Upsilon_q(\varphi(x))$ for any $x\in U^+_{r,s}$.
\end{remark}

\begin{remark}\label{rem:multiparameter}
Although we shall mainly be working with two-parameter quantum groups because they have received more study in the literature, it is worth
noting that one can prove an analogue of Proposition~\ref{prop:twisted_algebra} for the multiparameter quantum groups $U_{\bf{q}}(\fg)$
of~\cite{HPR} with $\bf{q}=(q_{ij})_{i,j=1}^n$ subject to $q_{ij}q_{ji} = q_{ii}^{a_{ij}}$ for all $i,j$ (with an additional assumption
that $q_{ii}^{1/d_{i}} = q_{jj}^{1/d_{j}}$ for all $i,j$, cf.~\cite[\S2.9]{HPR}). Specifically, if $\zeta_{\bf{q}}\colon Q \times Q \to \bb{K}^{*}$
is a skew bicharacter defined by $\zeta_{\bf{q}}(\alpha_{i},\alpha_{j}) = (q_{ij}q^{-d_{i}a_{ij}})^{1/2}$ with $q = q_{ii}^{\frac{1}{2d_{i}}}$,
then there is an algebra isomorphism
  $$U_{\bf{q}}(\frak{g}) \iso U_{q,\zeta_{\bf{q}}}(\frak{g})\qquad \text{such that}\qquad
    e_{i} \mapsto e_{i},\quad f_{i} \mapsto q_{ii}^{1/2}f_{i},\quad \omega_{i} \mapsto \omega_{i},\quad \omega_{i}' \mapsto \omega_{i}'\qquad
    \forall\, 1 \le i \le n.$$
This allows for a maximum of ${{n}\choose{2}} + 1$ quantum parameters. Moreover, the restriction of this isomorphism to the positive
subalgebra $U^+_{\bf{q}}(\fg) \subset U_{\bf{q}}(\fg)$ embedded (cf.~\cite[\S5]{HPR}) into the corresponding multiparameter shuffle
algebra $(\mathcal{F}_{\bf{q}},\ast_{\bf{q}})$ can be further lifted to the isomorphism
$(\mathcal{F}_{\bf{q}},\ast_{\bf{q}}) \iso (\mathcal{F}_q,\ast_q)_{\zeta_{\bf{q}}}$, in analogy with Remark~\ref{rem:shuffle-twist}.
\end{remark}

Moreover, one can also realize the Hopf pairing of Proposition~\ref{prop:pairing_2param} as a twisting of its
one-parameter version. This relies on the following result (cf.~\cite[Proposition 38]{HP}):

\begin{prop}\label{prop:twisted_pairing}
Define $(\cdot,\cdot)_{q,\zeta}\colon U_{q,\zeta}^{\le} \times U_{q,\zeta}^{\ge} \to \bb{K}$ by
\[
  (y,x)_{q,\zeta} = \zeta(\alpha,\alpha')^{-1}\zeta(\beta,\beta')^{-1}(y,x)_{q,q^{-1}} \quad \text{for all} \quad
  y \in (U_{q,\zeta}^{\le})_{\alpha,\beta},\ x \in (U_{q,\zeta}^{\ge})_{\alpha',\beta'}.
\]
Then $(\cdot,\cdot)_{q,\zeta}$ is the unique bilinear form on $U_{q,\zeta}^{\le} \times U_{q,\zeta}^{\ge}$ such that
\begin{equation*}
\begin{split}
  & (f_{i},\omega_{j})_{q,\zeta} = 0, \qquad (\omega_{i}', e_{j})_{q,\zeta} = 0, \qquad
    (f_{i},e_{j})_{q,\zeta} = \delta_{ij}\frac{1}{q_{i}^{-1} - q_{i}} \qquad \mathrm{for\ all} \quad 1\leq i,j\leq n , \\
  & (\omega_{\lambda}',\omega_{\mu})_{q,\zeta} = q^{(\lambda,\mu)}
    \qquad \mathrm{for\ all} \quad \lambda,\mu\in Q,
\end{split}
\end{equation*}
and
\begin{equation*}
  (y\circ y',x)_{q,\zeta} = (y \otimes y',\Delta(x))_{q,\zeta}, \quad (y, x \circ x')_{q,\zeta} = (\Delta(y),x' \otimes x)_{q,\zeta}
  \qquad \forall\, x,x' \in U_{r,s}^{\ge}(\fg),\ y,y' \in U_{r,s}^{\le}(\fg).
\end{equation*}
\end{prop}

Now, let $(\cdot,\cdot)_{r,s}\colon U_{r,s}^{\le} \times U_{r,s}^{\ge} \to \bb{K}$ be the Hopf pairing of~\eqref{eq:Hopf-parity}.
It is easy to verify that the pairing $(\varphi(\cdot),\varphi(\cdot))_{q,\zeta}$ coincides with $(\cdot,\cdot)_{r,s}$ on generators,
and because the map $\varphi\colon U_{r,s}(\fg) \to U_{q,\zeta}(\fg)$ of Proposition~\ref{prop:twisted_algebra} induces Hopf algebra
isomorphisms $U^{\leq}_{r,s}\iso U^{\leq}_{q,\zeta}$ and $U^{\geq}_{r,s}\iso U^{\geq}_{q,\zeta}$, we obtain the following result:
\begin{equation}\label{eq:pairing-twist-matching}
  (\cdot,\cdot)_{r,s} = (\varphi(\cdot),\varphi(\cdot))_{q,\zeta}.
\end{equation}

   %%%%%%%%%%%%%%%%%%%%%%%%%%%%%%%%%%%%%%%%%%%%%%%%%%%%%%%%%%%%%%%%%%%%%%%%%%

\subsection{Twisting of representations and intertwiners}
\

We shall now discuss the effect of twisting on modules with a weight space decomposition and their standard intertwiners,
as discussed in Subsection~\ref{ssec:reps and braidings}. To this end, we start with the following result of~\cite{HP}:

\begin{prop}\label{prop:twisted_representation}
Let $V$ be a $U_{q,q^{-1}}(\fg)$-module with a weight space decomposition $V = \bigoplus_{\lambda \in P} V[\lambda]$.
Then $V$ has a $U_{q,\zeta}(\fg)$-module structure given by
\[
  x\cdot_{\zeta} v = \zeta(\alpha - \beta,\lambda)\zeta(\alpha,\beta)xv \quad \text{for all}\quad
  x \in {U_{q,q^{-1}}(\fg)}_{\alpha,\beta},\ v \in V[\lambda].
\]
\end{prop}

Combining this with Proposition~\ref{prop:twisted_algebra} and formula~\eqref{eq:zeta_formula}, we obtain the following corollary:

\begin{cor}\label{cor:2_parameter_rep}
Let $V$ be a $U_{q,q^{-1}}(\fg)$-module with a weight space decomposition $V = \bigoplus_{\lambda \in P}V[\lambda]$.
Then $V$ has a $U_{r,s}(\fg)$-module structure such that for all $v \in V[\lambda]$ and $1 \le i \le n$, we have
\[
  e_{i}\cdot v = \zeta(\alpha_{i},\lambda)e_{i}v,\quad
  f_{i}\cdot v = (r_{i}s_{i})^{-1/2}\zeta(\alpha_{i},\lambda)f_{i}v,\quad
  \omega_{i}\cdot v = (\omega_{\lambda}',\omega_{i})v,\quad
  \omega_{i}'\cdot v = (\omega_{i}',\omega_{\lambda})^{-1}v.
\]
In particular, the $\lambda$-weight space component of $V$ as a $U_{r,s}(\fg)$-module is the same as the $\lambda$-weight
space component of $V$ as a $U_{q,q^{-1}}(\fg)$-module, for all $\lambda \in P$.
\end{cor}

Given a $U_{q,q^{-1}}(\fg)$-module $V$ with a weight space decomposition $V = \bigoplus V[\mu]$, we shall use the notation $V_{\zeta}$
to denote the $U_{q,\zeta}(\fg)$-module obtained from Proposition~\ref{prop:twisted_representation}, and we use $V_{r,s}$ to denote the
$U_{r,s}(\fg)$-module obtained from Corollary~\ref{cor:2_parameter_rep}. Let $V,V'$ be $U_{q,q^{-1}}(\fg)$-modules with weight space
decompositions, and define $\xi_{V,V'}\colon V \otimes V' \to V \otimes V'$ by
\begin{equation}\label{eq:xi_formula}
  \xi_{V,V'}(v \otimes v') = \zeta(\lambda',\lambda)v \otimes v' \quad \text{for all} \quad v \in V[\lambda],\ v' \in V[\lambda'].
\end{equation}
It was shown in the proof of~\cite[Theorem 42]{HP} that $\xi_{V,V'}$ defines a $U_{q,\zeta}(\fg)$-module isomorphism from
$(V \otimes V')_{\zeta}$ to $V_{\zeta} \otimes V'_{\zeta}$. Since the $U_{r,s}(\fg)$-modules $(V \otimes V')_{r,s}$, $V_{r,s}$, and
$V'_{r,s}$ were obtained from $(V \otimes V')_{\zeta}$, $V_{\zeta}$, and $V'_{\zeta}$, respectively, via the isomorphism
$\varphi\colon U_{r,s}(\fg) \iso U_{q,\zeta}(\fg)$ of Proposition~\ref{prop:twisted_algebra}, we get:

\begin{prop}
(a) The formula~\eqref{eq:xi_formula} determines an isomorphism of $U_{r,s}(\fg)$-modules
$$\xi_{V,V'}\colon (V \otimes V')_{r,s} \iso V_{r,s} \otimes V'_{r,s}.$$

\noindent
(b) Any $U_{q,q^{-1}}(\fg)$-module isomorphism $\hat{R}_{V,V'}\colon V \otimes V' \to V' \otimes V$ gives rise
to a $U_{r,s}(\fg)$-module isomorphism
\begin{equation}\label{eq:R-matrix-twist}
  \xi_{V',V} \circ \hat{R}_{V,V'} \circ \xi_{V,V'}^{-1}\colon V_{r,s} \otimes V'_{r,s} \iso V'_{r,s} \otimes V_{r,s}.
\end{equation}
\end{prop}

We shall now show that $\hat{R}_{V_{r,s},V'_{r,s}}$ of Proposition~\ref{prop:universal-R} is just a twisted version of its one-parameter
counterpart. To this end, for each $\mu \in Q^{+}$, pick dual bases $\{y_{i}^{\mu}\}_{i = 1}^{n_{\mu}}$ and $\{x_{i}^{\mu}\}_{i = 1}^{n_{\mu}}$
of $(U_{q,q^{-1}}^{-})_{-\mu}$ and $(U_{q,q^{-1}}^{+})_{\mu}$. Following~\eqref{eq:Theta} and~\eqref{eq:tildef-map}, we define
$\Theta_{q} = \sum_{\mu \in Q^{+}} \sum_{i = 1}^{n_{\mu}} y_{i}^{\mu} \otimes x_{i}^{\mu}$, as well as a map
$\tilde{f}_{q}\colon V \otimes V' \to V \otimes V'$
\[
  \tilde{f}_{q}(v \otimes v') = q^{-(\lambda,\lambda')}v \otimes v'\quad \text{for all}\ v \in V[\lambda],\ v' \in V[\lambda']
\]
for any $U_{q,q^{-1}}(\fg)$-modules $V,V'$ in the category $\mathcal{O}$ (see Subsection~\ref{ssec:reps and braidings}).
According to Proposition~\ref{prop:universal-R},
\[
  \hat{R}_{V,V'} = \Theta_{q} \circ \tilde{f} \circ \tau \colon V \otimes V' \iso V' \otimes V
\]
is a $U_{q,q^{-1}}(\fg)$-module isomorphism. For any $v \in V[\lambda]$ and $v' \in V'[\lambda']$, we have:
\begin{align*}
  & \xi_{V',V} \circ \hat{R}_{V,V'} \circ \xi_{V,V'}^{-1}(v \otimes v')
  = \zeta(\lambda',\lambda)^{-1} \xi_{V',V} \circ \Theta_{q} (\tilde{f}(v' \otimes v)) \\
  &= \zeta(\lambda',\lambda)^{-1}q^{-(\lambda,\lambda')}\sum_{\mu \in Q^{+}}\sum_{i = 1}^{n_{\mu}}\xi_{V',V}(y_{i}^{\mu}v' \otimes x_{i}^{\mu}v)
   = \zeta(\lambda',\lambda)^{-1}q^{-(\lambda,\lambda')}\sum_{\mu \in Q^{+}}
     \sum_{i = 1}^{n_{\mu}} \zeta(\lambda + \mu,\lambda' - \mu)y_{i}^{\mu}v' \otimes x_{i}^{\mu}v \\
  &= \zeta(\lambda,\lambda')^{2}q^{-(\lambda,\lambda')}\sum_{\mu \in Q^{+}}
     \sum_{i = 1}^{n_{\mu}}\zeta(\mu,\lambda + \lambda')y_{i}^{\mu}v' \otimes x_{i}^{\mu}v
     \overset{\eqref{eq:zeta_formula}}{=} (\omega_{\lambda}',\omega_{\lambda'})^{-1}\sum_{\mu \in Q^{+}}
     \sum_{i = 1}^{n_{\mu}}\zeta(\mu,\lambda + \lambda')y_{i}^{\mu}v' \otimes x_{i}^{\mu}v.
\end{align*}

On the other hand, since $y_{i}^{\mu} \in U_{q,q^{-1}}(\fg)_{0,-\mu}$ and $x_{i}^{\mu} \in U_{q,q^{-1}}(\fg)_{\mu,0}$, it follows
from Proposition~\ref{prop:twisted_pairing} and the equality~\eqref{eq:pairing-twist-matching} that $\{\varphi^{-1}(y_{i}^{\mu})\}$
and $\{\varphi^{-1}(x_{i}^{\mu})\}$ are dual bases of $U^-_{r,s}(\fg)$ and $U^+_{r,s}(\fg)$ with respect to the Hopf
pairing~\eqref{eq:Hopf-parity}, where $\varphi$ is the map from Proposition~\ref{prop:twisted_algebra}. Using these bases
in~\eqref{eq:Theta} for $\Theta_{r,s}$, we get:
\[
  \hat{R}_{V_{r,s},V'_{r,s}}(v \otimes v')
  = (\omega_{\lambda}',\omega_{\lambda'})^{-1}\sum_{\mu \in Q^{+}}\sum_{i=1}^{n_{\mu}} y_{i}^{\mu}\cdot_{\zeta}v' \otimes x_{i}^{\mu} \cdot_{\zeta}v
  = (\omega_{\lambda}',\omega_{\lambda'})^{-1}\sum_{\mu \in Q^{+}}\sum_{i=1}^{n_{\mu}} \zeta(\mu,\lambda + \lambda')y_{i}^{\mu}v' \otimes x_{i}^{\mu} v
\]
for any $v \in V[\lambda]$ and $v' \in V'[\lambda']$. Thus we obtain:

\begin{prop}\label{prop:R-matrix-twist}
$\hat{R}_{V_{r,s},V'_{r,s}} = \xi_{V',V} \circ \hat{R}_{V,V'} \circ \xi^{-1}_{V,V'}$.
\end{prop}

   %%%%%%%%%%%%%%%%%%%%%%%%%%%%%%%%%%%%%%%%%%%%%%%%%%%%%%%%%%%%%%%%%%%%%%%%%%

\subsection{Twisting finite $R$-matrices in classical types}\label{ssec:B-type}
\

In this Subsection, we apply Proposition~\ref{prop:R-matrix-twist} to evaluate the intertwiner $\hat{R}_{W,W}$ for the first
fundamental $U_{r,s}(\fg)$-module $W$ when $\fg$ is classical, providing new proofs of~\cite[Theorems 4.4--4.6]{MT1}.
To simplify notations, we shall use $\hat{R}_q$ for $\hat{R}_{V,V}$ and $\hat{R}_{r,s}$ for $\hat{R}_{V_{r,s},V_{r,s}}$
where $V$ is the first fundamental $U_{q,q^{-1}}(\fg)$-module.

\noindent
$\bullet$ \textbf{Type $B_{n}$.}

Let $V = \bb{K}^{2n + 1}$, and let $v_{i}$ denote the $i$-th standard basis vector. We write $i' = 2n + 2 - i$ for all
$1 \le i \le 2n + 1$. The next result is the special case of~\cite[Proposition 3.2]{MT1} with $r = q$ and $s = q^{-1}$:

\begin{prop}\label{prop:1-parameter_rep_B}
The following defines a representation $\rho_{q}\colon U_{q,q^{-1}}(\frak{so}_{2n + 1}) \to \End(V)$:
\begin{align*}
  \rho_{q}(e_{i}) &= E_{i,i + 1} - E_{(i + 1)',i'}\quad \mathrm{for}\ \ 1\leq i\leq n \,, \\
  \rho_{q}(f_{i}) &=
  \begin{cases}
    E_{i + 1, i} - E_{i', (i + 1)'} & \mathrm{if}\ \  1\leq i<n \\
    (q + q^{-1})E_{n + 1,n} - (q + q^{-1})E_{n', n + 1} & \mathrm{if}\ \ i = n
  \end{cases}, \\
  \rho_{q}(\omega_{i}) &=
  \begin{cases}
    \displaystyle{q^{2}E_{ii} + q^{-2}E_{i + 1, i + 1} + q^{-2}E_{i'i'} + q^{2}E_{(i + 1)', (i + 1)'} +
                  \sum_{1\leq j\leq n}^{j \neq i,i + 1} (E_{jj}+E_{j'j'}) + E_{n+1,n+1}}
       & \mathrm{if}\ \ 1\leq i<n \\
    \displaystyle{q^{2}E_{nn} + E_{n + 1,n+1} + q^{-2}E_{n'n'} + \sum_{1\leq j\leq n-1}\left (E_{jj} + E_{j'j'}\right )} & \mathrm{if}\ \ i = n
  \end{cases}, \\
  \rho_{q}(\omega_{i}') &=
  \begin{cases}
    \displaystyle{q^{-2}E_{ii} + q^{2}E_{i + 1, i + 1} + q^{2}E_{i'i'} + q^{-2}E_{(i + 1)', (i + 1)'} +
                  \sum_{1\leq j\leq n}^{j \neq i,i + 1} (E_{jj} + E_{j'j'}) + E_{n+1,n+1}}
       & \mathrm{if}\ \ 1\leq i<n  \\
    \displaystyle{q^{-2}E_{nn} + E_{n + 1,n+1} + q^{2}E_{n'n'} + \sum_{1\leq j\leq n-1}\left (E_{jj} + E_{j'j'}\right )} & \mathrm{if}\ \ i = n
  \end{cases}.
\end{align*}
\end{prop}

We can now apply Corollary~\ref{cor:2_parameter_rep} to derive a two-parameter counterpart:

\begin{prop}\label{prop:2-parameter_rep_B}
(a) The following defines a representation $\tilde{\rho}_{r,s}\colon U_{r,s}(\frak{so}_{2n+1}) \to \End(V)$:
\begin{align*}
  \tilde{\rho}_{r,s}(e_{i}) &= \begin{cases}
  (rs)^{1/2}E_{i,i + 1} - (rs)^{-1/2}E_{(i + 1)',i'} & 1\leq i < n \\
  E_{n,n + 1} - E_{n + 1,n'} & i = n
  \end{cases}, \\
  \tilde{\rho}_{r,s}(f_{i}) &=
  \begin{cases}
    (rs)^{-1/2}E_{i + 1, i} - (rs)^{-3/2}E_{i', (i + 1)'} & \mathrm{if}\ \  1\leq i<n \\
    (r^{-1} + s^{-1})E_{n + 1,n} - (r^{-1} + s^{-1})E_{n', n + 1} & \mathrm{if}\ \ i = n
  \end{cases}, \\
  \tilde{\rho}_{r,s}(\omega_{i}) &=
  \begin{cases}
    \displaystyle{r^{2}E_{ii} + s^{2}E_{i + 1, i + 1} + r^{-2}E_{i'i'} + s^{-2}E_{(i + 1)', (i + 1)'} +
                  \sum_{1\leq j\leq n}^{j \neq i,i + 1} (E_{jj}+E_{j'j'}) + E_{n+1,n+1}}
       & \mathrm{if}\ \ 1\leq i<n \\
    \displaystyle{rs^{-1}E_{nn} + E_{n + 1,n+1} + r^{-1}sE_{n'n'} + \sum_{1\leq j\leq n-1}\left (r^{-1}s^{-1}E_{jj} + rsE_{j'j'}\right )} & \mathrm{if}\ \ i = n
  \end{cases}, \\
  \tilde{\rho}_{r,s}(\omega_{i}') &=
  \begin{cases}
    \displaystyle{s^{2}E_{ii} + r^{2}E_{i + 1, i + 1} + s^{-2}E_{i'i'} + r^{-2}E_{(i + 1)', (i + 1)'} +
                  \sum_{1\leq j\leq n}^{j \neq i,i + 1} (E_{jj} + E_{j'j'}) + E_{n+1,n+1}}
       & \mathrm{if}\ \ 1\leq i<n  \\
    \displaystyle{r^{-1}sE_{nn} + E_{n + 1,n+1} + rs^{-1}E_{n'n'} + \sum_{1\leq j\leq n-1}\left (r^{-1}s^{-1}E_{jj} + rsE_{j'j'}\right )} & \mathrm{if}\ \ i = n
  \end{cases}.
\end{align*}

\noindent
(b) Let $\psi\colon V \to V$ be the linear map defined by
\begin{equation}\label{eq:psi-B}
  \psi(v_{i}) = \psi_{i}v_{i} \qquad \text{for all}\qquad 1 \le i \le 2n+1,
\end{equation}
with $\psi_{i}$ given by
\begin{equation}\label{eq:psi_values_B}
  (\psi_{1},\ldots ,\psi_{2n+1})
  = \left( (rs)^{\frac{n-1}{2}},(rs)^{\frac{n - 2}{2}},\ldots ,1,1,1,\ldots , (rs)^{\frac{n-2}{2}},(rs)^{\frac{n-1}{2}} \right).
\end{equation}
Then, the representation $\rho_{r,s}\colon U_{r,s}(\frak{so}_{2n + 1}) \to \End(V)$ of~\cite[Proposition~3.2]{MT1} is related to
$\tilde{\rho}_{r,s}$ of (a) via
\[
  \tilde{\rho}_{r,s}(x) \circ \psi = \psi \circ \rho_{r,s}(x) \qquad \forall\, x \in U_{r,s}(\frak{so}_{2n + 1}).
\]
\end{prop}

\begin{proof}
Part (a) follows by applying Corollary~\ref{cor:2_parameter_rep} to the representation in Proposition~\ref{prop:1-parameter_rep_B},
and part (b) is easy to verify directly.
\end{proof}

Evoking Proposition~\ref{prop:R-matrix-twist}, we can now relate the two-parameter $B$-type $R$-matrix to its one-parameter counterpart.
Let $\hat{R}_{q}\colon V \otimes V \to V \otimes V$ be the one-parameter $R$-matrix, which is given explicitly by
\begin{equation}\label{eq:1_parameter_R_B}
\begin{split}
  \hat{R}_{q} =\,
  & q^{-2} \sum_{1\leq i\leq 2n+1}^{i\ne n+1} E_{ii} \otimes E_{ii} +
    E_{n + 1,n + 1} \otimes E_{n + 1,n+1} +
    q^{2} \sum_{1\leq i\leq 2n+1}^{i\ne n+1} E_{ii'} \otimes E_{i'i}  \\
  & + \sum_{1\leq i,j\leq 2n+1}^{j\ne i,i'} E_{ij} \otimes E_{ji} +
    (q^{2} - q^{-2}) \sum_{1\leq i\leq n}(q^{4(n-i)+2}-1) E_{i'i'} \otimes E_{ii} \\
  & + (q^{-2} - q^{2}) \sum_{i > j}^{j \neq i'} E_{ii} \otimes E_{jj} +
    (q^{2} - q^{-2}) \sum_{i < j}^{j \neq i'} t_{i}(q)t_{j}(q)^{-1}E_{i'j} \otimes E_{ij'},
\end{split}
\end{equation}
where
\begin{equation}
  t_{i}(q) =
  \begin{cases}
    q^{2(n - i) + 1} & \text{if}\ \ i < n + 1 \\
    q & \text{if}\ \ i = n + 1 \\
    q^{2(n +1 - i) + 1} & \text{if}\ \ i > n + 1
  \end{cases}.
\end{equation}

Then, combining Proposition~\ref{prop:R-matrix-twist} and Proposition~\ref{prop:2-parameter_rep_B}(b), the matrix $\hat{R}_{r,s}$ of~\cite[(4.7)]{MT1}
is given by $\hat{R}_{r,s} = (\psi \otimes \psi)^{-1} \circ \xi \circ \hat{R}_{q} \circ \xi^{-1} \circ (\psi \otimes \psi)$, with $\xi = \xi_{V,V}$
of~\eqref{eq:xi_formula} and $\psi$ of~\eqref{eq:psi-B}. We can also use this to obtain a new proof of~\cite[Theorem~4.4]{MT1}. To this end, we note that
\begin{equation}\label{eq:reproof-B}
\begin{split}
  (\psi \otimes \psi)^{-1} & \circ \xi \circ \hat{R}_{q} \circ \xi^{-1} \circ (\psi \otimes \psi) \\
  =\, & q^{-2} \sum_{1\leq i\leq 2n+1}^{i\ne n+1} E_{ii} \otimes E_{ii} + E_{n + 1,n + 1} \otimes E_{n + 1,n+1} +
    q^{2} \sum_{1\leq i\leq 2n+1}^{i\ne n+1} E_{ii'} \otimes E_{i'i}  \\
  & + \sum_{1\leq i,j\leq 2n+1}^{j\ne i,i'} \zeta(\varepsilon_{i},\varepsilon_{j})^{-1}\zeta(\varepsilon_{j},\varepsilon_{i})E_{ij} \otimes E_{ji} +
    (q^{2} - q^{-2}) \sum_{1\leq i\leq n}(q^{4(n-i)+2}-1) E_{i'i'} \otimes E_{ii} \\
  & + (q^{-2} - q^{2}) \sum_{i > j}^{j \neq i'} E_{ii} \otimes E_{jj} +
    (q^{2} - q^{-2}) \sum_{i < j}^{j \neq i'} t_{i}(q)t_{j}(q)^{-1}\psi_{i}^{-1}\psi_{i'}^{-1}\psi_{j}\psi_{j'}E_{i'j} \otimes E_{ij'},
\end{split}
\end{equation}
where we use the convention $\varepsilon_{i'} = -\varepsilon_{i}$ for $i \le n$ and $\varepsilon_{n +1} = 0$. Setting $q = r^{1/2}s^{-1/2}$,
all but the fourth and last terms immediately recover the corresponding terms in the two-parameter $R$-matrix of~\cite[(4.7)]{MT1}. For the
fourth term, evoking~\eqref{eq:zeta_formula}, we obtain:
\[
  \zeta(\varepsilon_{i},\varepsilon_{j})^{-1}\zeta(\varepsilon_{j},\varepsilon_{i}) =   \zeta(\varepsilon_{i},\varepsilon_{j})^{-2} =
  (\omega_{\varepsilon_{j}}',\omega_{\varepsilon_{i}})^{-1} r^{\frac{1}{2}(\varepsilon_{j},\varepsilon_{i})} s^{-\frac{1}{2}(\varepsilon_{i},\varepsilon_{j})} =
  a_{ij},
\]
where we used $(\varepsilon_{i},\varepsilon_{j}) = 0$ for $j \neq i,i'$ as well as~\cite[(5.33)]{MT1} in the last equality.
We also have
\[
  t_{i}(q)t_{j}(q)^{-1}\psi_{i}^{-1}\psi_{i'}^{-1}\psi_{j}\psi_{j'} =
  \begin{cases}
     q^{2(j - i)}(rs)^{i - j} = s^{2(i - j)} & \text{if}\ i,j < n + 1 \\
     q^{2(n - i)}(rs)^{i - n} = s^{2(i - n)} & \text{if}\ i < n + 1 = j \\
     q^{2(j - i) - 2}(rs)^{j + i - 2n - 2} = r^{2(j - n - 1) - 1}s^{2(i - n) - 1} & \text{if}\ i < n + 1 < j \\
     q^{-2(n +1 - j)}(rs)^{-n-2+j} = r^{2(j - n - 1) - 1}s^{-1} & \text{if}\ i = n + 1 < j \\
     q^{2(j - i)}(rs)^{j - i} = r^{2(j - i)} & \text{if}\ n + 1 < i < j
  \end{cases}.
\]
Thus, the last term in~\eqref{eq:reproof-B} matches the corresponding term of the two-parameter $R$-matrix of~\cite[(4.7)]{MT1}.

   %%%%%%%%%%%%%%%%%%%%%%%%%%%%%%%%%%%%%%%%%%%%%%%%%%%%%%%%%%%%%%%%%%%%%%%%%%

\medskip
\noindent
$\bullet$ \textbf{Type $C_{n}$.}

Let $V = \bb{K}^{2n}$, and let $v_{i}$ denote the $i$-th standard basis vector. We write $i' = 2n + 1 - i $ for all $1 \le i \le 2n$.

\begin{prop}\label{prop:1_parameter_rep_C}
The following defines a representation $\rho_{q}\colon U_{q,q^{-1}}(\frak{sp}_{2n})\to \End(V)$:
\begin{align*}
  \rho_{q}(e_{i}) &=
  \begin{cases}
    E_{i,i + 1} - E_{(i + 1)',i'} & \mathrm{if}\ \  1\leq i<n \\
    E_{nn'} & \mathrm{if}\ \  i=n
  \end{cases},  \\
  \rho_{q}(f_{i}) &=
  \begin{cases}
    E_{i+ 1,i} - E_{i',(i + 1)'} & \mathrm{if}\ \  1\leq i<n \\
    E_{n'n} & \mathrm{if}\ \  i=n \\
  \end{cases},  \\
  \rho_{q}(\omega_{i}) &=
  \begin{cases}
    \displaystyle{qE_{ii} + q^{-1}E_{i + 1, i+1} + q^{-1}E_{i'i'} + qE_{(i +1)',(i+1)'} + \sum_{1\leq j\leq n}^{j \neq i,i + 1} \left (E_{jj} + E_{j'j'}\right )}
        & \mathrm{if}\ \  1\leq i<n \\
    \displaystyle{q^{2}E_{nn} + q^{-2}E_{n'n'} + \sum_{1\leq j\leq n-1}\left (E_{jj} + E_{j'j'}\right )}
        & \mathrm{if}\ \ i = n
  \end{cases},  \\
  \rho_{q}(\omega'_{i}) &=
  \begin{cases}
    \displaystyle{q^{-1}E_{ii} + qE_{i + 1, i+1} + qE_{i'i'} + q^{-1}E_{(i +1)',(i+1)'} +
                  \sum_{1\leq j\leq n}^{j \neq i,i + 1} \left (E_{jj} + E_{j'j'}\right )}
       & \mathrm{if}\ \  1\leq i<n  \\
    \displaystyle{q^{-2}E_{nn} + q^{2}E_{n'n'} + \sum_{1\leq j\leq n-1}\left (E_{jj} + E_{j'j'}\right )}
      & \mathrm{if}\ \  i = n
  \end{cases}.
\end{align*}
\end{prop}

Similarly to the above treatment of $B$-type, we can use Corollary~\ref{cor:2_parameter_rep} to prove part (a) of the following result,
while part (b) is just a straightforward calculation:

\begin{prop}\label{prop:2_parameter_rep_C}
(a) The following defines a representation $\tilde{\rho}_{r,s}\colon U_{r,s}(\frak{sp}_{2n}) \to \End(V)$:
\begin{align*}
  \tilde{\rho}_{r,s}(e_{i}) &=
  \begin{cases}
    (rs)^{1/4}E_{i,i + 1} - (rs)^{-1/4}E_{(i + 1)',i'} & \mathrm{if}\ \  1\leq i<n \\
    E_{nn'} & \mathrm{if}\ \  i=n
  \end{cases},  \\
  \tilde{\rho}_{r,s}(f_{i}) &=
  \begin{cases}
    (rs)^{-1/4}E_{i+ 1,i} - (rs)^{-3/4}E_{i',(i + 1)'} & \mathrm{if}\ \  1\leq i<n \\
    (rs)^{-1}E_{n'n} & \mathrm{if}\ \  i=n \\
  \end{cases},  \\
  \tilde{\rho}_{r,s}(\omega_{i}) &=
  \begin{cases}
    \displaystyle{rE_{ii} + sE_{i + 1, i+1} + r^{-1}E_{i'i'} + s^{-1}E_{(i +1)',(i+1)'} + \sum_{1\leq j\leq n}^{j \neq i,i + 1} \left (E_{jj} + E_{j'j'}\right )}
        & \mathrm{if}\ \  1\leq i<n \\
    \displaystyle{rs^{-1}E_{nn} + r^{-1}sE_{n'n'} + \sum_{1\leq j\leq n-1}\left (r^{-1}s^{-1}E_{jj} + rsE_{j'j'}\right )}
        & \mathrm{if}\ \ i = n
  \end{cases},  \\
  \tilde{\rho}_{r,s}(\omega'_{i}) &=
  \begin{cases}
    \displaystyle{sE_{ii} + rE_{i + 1, i+1} + s^{-1}E_{i'i'} + r^{-1}E_{(i +1)',(i+1)'} +
                  \sum_{1\leq j\leq n}^{j \neq i,i + 1} \left (E_{jj} + E_{j'j'}\right )}
       & \mathrm{if}\ \  1\leq i<n \\
    \displaystyle{r^{-1}sE_{nn} + rs^{-1}E_{n'n'} + \sum_{1\leq j\leq n-1}\left (r^{-1}s^{-1}E_{jj} + rsE_{j'j'}\right )}
      & \mathrm{if}\ \  i = n
  \end{cases}.
\end{align*}

\noindent
(b) Let $\psi\colon V \to V$ be the linear map defined by
\begin{equation}\label{eq:psi-C}
  \psi(v_{i}) = \psi_{i}v_{i} \qquad \text{for all}\qquad 1 \le i \le 2n,
\end{equation}
with $\psi_{i}$ given by
\begin{equation}\label{eq:psi_values_C}
  (\psi_{1},\ldots ,\psi_{2n})
  = \left( (rs)^{\frac{n-1}{4}},(rs)^{\frac{n-2}{4}},\ldots ,1,1,\ldots ,(rs)^{\frac{n-2}{4}},(rs)^{\frac{n-1}{4}} \right).
\end{equation}
Then, the representation $\rho_{r,s}\colon U_{r,s}(\frak{sp}_{2n}) \to \End(V)$ of~\cite[Proposition~3.3]{MT1} is related to
$\tilde{\rho}_{r,s}$ of (a) via
\[
  \tilde{\rho}_{r,s}(x) \circ \psi = \psi \circ \rho_{r,s}(x) \qquad \forall\, x \in U_{r,s}(\frak{sp}_{2n}).
\]
\end{prop}

Now, let $\hat{R}_{q}\colon V \otimes V \to V \otimes V$ be the one-parameter $R$-matrix, which is given explicitly by
\begin{equation}\label{eq:1_parameter_R_C}
\begin{split}
  \hat{R}_{q} =\,
  & q^{-1}\sum_{1\leq i\leq 2n} E_{ii} \otimes E_{ii} + q\sum_{1\leq i\leq 2n} E_{ii'} \otimes E_{i'i} \\
  & + \sum_{1\leq i,j\leq 2n}^{j \ne i,i'} E_{ij} \otimes E_{ji}
    + (q^{-1} - q)\sum_{1\leq i\leq n} (q^{2(n+ 1 - i)} + 1)E_{i'i'} \otimes E_{ii} \\
  & + (q^{-1} - q)\sum_{i > j}^{j \neq i'} E_{ii} \otimes E_{jj}
    + (q - q^{-1})\sum_{i < j}^{j \neq i'} t_{i}(q)t_{j}(q)^{-1}E_{i'j} \otimes E_{ij'} \,,
\end{split}
\end{equation}
where
\begin{equation*}
  t_{i}(q) = \
  \begin{cases}
    q^{n + 1 - i} & \text{if}\ \ i \le n \\
    -q^{n - i} & \text{if}\ \ i > n
  \end{cases}.
\end{equation*}

Then as for type $B_{n}$, we have $\hat{R}_{r,s} = (\psi \otimes \psi)^{-1} \circ \xi \circ \hat{R}_{q} \circ \xi^{-1} \circ (\psi \otimes \psi)$,
where $\xi$ denotes the isomorphism $\xi_{V,V}$ of~\eqref{eq:xi_formula}, $\psi$ is the isomorphism of~\eqref{eq:psi-C}, and $\hat{R}_{r,s}$ denotes
the matrix $\hat{R}$ of~\cite[(4.9)]{MT1}. We can also use this to obtain a new proof of~\cite[Theorem~4.5]{MT1}. To this end, we note that
\begin{align*}
  (\psi \otimes \psi)^{-1} \circ \xi & \circ \hat{R}_{q} \circ \xi^{-1} \circ (\psi \otimes \psi) \\
  =\, & q^{-1}\sum_{1\leq i\leq 2n} E_{ii} \otimes E_{ii} +
    q\sum_{1\leq i\leq 2n} E_{ii'} \otimes E_{i'i} \\
  & + \sum_{1\leq i,j\leq 2n}^{j \ne i,i'} \zeta(\varepsilon_{i},\varepsilon_{j})^{-1}\zeta(\varepsilon_{j},\varepsilon_{i})E_{ij} \otimes E_{ji}
    + (q^{-1} - q)\sum_{1\leq i\leq n} (q^{2(n+ 1 - i)} + 1)E_{i'i'} \otimes E_{ii} \\
  & + (q^{-1} - q)\sum_{i > j}^{j \neq i'} E_{ii} \otimes E_{jj}
    + (q - q^{-1})\sum_{i < j}^{j \neq i'} t_{i}(q)t_{j}(q)^{-1}\psi_{i}^{-1}\psi_{i'}^{-1}\psi_{j}\psi_{j'}E_{i'j} \otimes E_{ij'},
\end{align*}
with the convention $\varepsilon_{i'} = -\varepsilon_{i}$ for $1 \le i \le n$. As in $B$-type,
$\zeta(\varepsilon_{i},\varepsilon_{j})^{-1}\zeta(\varepsilon_{j},\varepsilon_{i}) = a_{ij}$ for $j\ne i,i'$, and
\[
  t_{i}(q)t_{j}(q)^{-1}\psi_{i}^{-1}\psi_{i'}^{-1}\psi_{j}\psi_{j'} =
  \begin{cases}
    q^{j - i}(rs)^{\frac{1}{2}(i - j)} = s^{i - j} & \text{if}\ i < j \le n \\
    - q^{j - i +1}(rs)^{\frac{1}{2}(i + j - 2n - 1)} = -r^{j - n}s^{i - n - 1} & \text{if}\ i \le n < j \\
    q^{j - i}(rs)^{\frac{1}{2}(j - i)} = r^{j - i} & \text{if}\ n < i < j
  \end{cases}.
\]
Thus, the resulting matrix exactly matches the corresponding two-parameter $R$-matrix of~\cite[(4.9)]{MT1}.

   %%%%%%%%%%%%%%%%%%%%%%%%%%%%%%%%%%%%%%%%%%%%%%%%%%%%%%%%%%%%%%%%%%%%%%%%%%

\medskip
\noindent
$\bullet$ \textbf{Type $D_{n}$.}

Let $V = \bb{K}^{2n}$, and let $v_{i}$ denote the $i$-th standard basis vector. We write $i' = 2n + 1 - i $ for all $1 \le i \le 2n$.

\begin{prop}\label{prop:1_parameter_rep_D}
For $n \ge 3$, the following defines a representation $\rho_{q}\colon U_{q,q^{-1}}(\frak{so}_{2n}) \to \End(V)$:
\begin{align*}
  \rho_{q}(e_{i}) &=
  \begin{cases}
    E_{i,i+1} - E_{(i + 1)',i'} & \mathrm{if}\ \  1\leq i<n \\
    E_{n - 1, n'} - E_{n,(n-1)'} & \mathrm{if}\ \  i=n \\
  \end{cases},  \\
  \rho_{q}(f_{i}) &=
  \begin{cases}
    E_{i + 1,i} - E_{i', (i + 1)'} & \mathrm{if}\ \  1\leq i<n \\
    E_{n', n-1} - E_{(n-1)',n} & \mathrm{if}\ \  i=n \\
  \end{cases},  \\
  \rho_{q}(\omega_{i}) &=
  \begin{cases}
    \displaystyle{qE_{ii} + q^{-1}E_{i + 1,i + 1} + q^{-1}E_{i'i'} + qE_{(i + 1)', (i+1)'} +
                  \sum_{1\leq j\leq n}^{j \neq i,i + 1} \left (E_{jj} + E_{j'j'}\right )}
      & \mathrm{if}\ \  1\leq i<n \\
    \displaystyle{qE_{n-1,n-1} + qE_{nn} + q^{-1}E_{(n-1)',(n-1)'} + q^{-1}E_{n'n'} + \sum_{1\leq j\leq n-2}\left (E_{jj} + E_{j'j'}\right )}
      & \mathrm{if}\ \ i = n \\
  \end{cases}, \\
  \rho_{q}(\omega'_{i}) &=
  \begin{cases}
    \displaystyle{q^{-1}E_{ii} + qE_{i + 1,i + 1} + qE_{i'i'} + q^{-1}E_{(i + 1)',(i + 1)'} +
                  \sum_{1\leq j\leq n}^{j \neq i,i + 1} \left (E_{jj} + E_{j'j'}\right )}
      & \mathrm{if}\ \  1\leq i<n \\
    \displaystyle{q^{-1}E_{n-1,n-1} + q^{-1}E_{nn} + qE_{(n-1)',(n-1)'} + qE_{n'n'} + \sum_{1\leq j\leq n-2}\left (E_{jj} + E_{j'j'}\right )}
      & \mathrm{if}\ \ i = n \\
  \end{cases}.	
\end{align*}
\end{prop}

As in the previous cases, applying Corollary~\ref{cor:2_parameter_rep} to this result, we obtain:

\begin{prop}\label{prop:2_parameter_rep_D}
(a) The following defines a representation $\tilde{\rho}_{r,s}\colon U_{r,s}(\frak{so}_{2n}) \to \End(V)$:
\begin{align*}
  \tilde{\rho}_{r,s}(e_{i}) &=
  \begin{cases}
    (rs)^{1/4}E_{i,i+1} - (rs)^{-1/4}E_{(i + 1)',i'} & \mathrm{if}\ \  1\leq i<n \\
   (rs)^{-1/4}E_{n - 1, n'} - (rs)^{1/4}E_{n,(n-1)'} & \mathrm{if}\ \  i=n \\
  \end{cases},  \\
  \tilde{\rho}_{r,s}(f_{i}) &=
  \begin{cases}
    (rs)^{-1/4}E_{i + 1,i} - (rs)^{-3/4}E_{i', (i + 1)'} & \mathrm{if}\ \  1\leq i<n \\
    (rs)^{-3/4}E_{n', n-1} - (rs)^{-1/4}E_{(n-1)',n} & \mathrm{if}\ \  i=n \\
  \end{cases},  \\
  \tilde{\rho}_{r,s}(\omega_{i}) &=
  \begin{cases}
    \displaystyle{rE_{ii} + sE_{i + 1,i + 1} + r^{-1}E_{i'i'} + s^{-1}E_{(i + 1)', (i+1)'} +
                  \sum_{1\leq j\leq n}^{j \neq i,i + 1} \left (E_{jj} + E_{j'j'}\right )}
      & \mathrm{if}\ \  1\leq i<n  \\
    \displaystyle{s^{-1}E_{n-1,n-1} + rE_{nn} + sE_{(n-1)',(n-1)'} + r^{-1}E_{n'n'} + \sum_{1\leq j\leq n-2}\left (r^{-1}s^{-1}E_{jj} + rsE_{j'j'}\right )}
      & \mathrm{if}\ \ i = n \\
  \end{cases},  \\
  \tilde{\rho}_{r,s}(\omega'_{i}) &=
  \begin{cases}
    \displaystyle{sE_{ii} + rE_{i + 1,i + 1} + s^{-1}E_{i'i'} + r^{-1}E_{(i + 1)',(i + 1)'} +
                  \sum_{1\leq j\leq n}^{j \neq i,i + 1} \left (E_{jj} + E_{j'j'}\right )}
      & \mathrm{if}\ \  1\leq i<n  \\
    \displaystyle{r^{-1}E_{n-1,n-1} + sE_{nn} + rE_{(n-1)',(n-1)'} + s^{-1}E_{n'n'} + \sum_{1\leq j\leq n-2}\left (r^{-1}s^{-1}E_{jj} + rsE_{j'j'}\right )}
      & \mathrm{if}\ \ i = n \\
  \end{cases}.	
\end{align*}

\noindent
(b) Let $\psi\colon V \to V$ be the linear map defined by
\begin{equation}\label{eq:psi-D}
  \psi(v_{i}) = \psi_{i}v_{i} \qquad \text{for all}\qquad  1 \le i \le 2n,
\end{equation}
with $\psi_{i}$ given by
\begin{equation}\label{eq:psi_values_D}
  (\psi_{1},\ldots ,\psi_{2n}) =
  \left( (rs)^{\frac{n-1}{4}},(rs)^{\frac{n-2}{4}},\ldots , 1,(rs)^{-\frac{1}{2}},\ldots ,(rs)^{\frac{n-4}{4}},(rs)^{\frac{n-3}{4}} \right).
\end{equation}
Then, the representation $\rho_{r,s}\colon U_{r,s}(\frak{so}_{2n}) \to \End(V)$ of~\cite[Proposition~3.4]{MT1} is related to
$\tilde{\rho}_{r,s}$ of (a) via
\[
  \tilde{\rho}_{r,s}(x) \circ \psi = \psi \circ \rho_{r,s}(x) \qquad \forall\, x \in U_{r,s}(\frak{so}_{2n}).
\]
\end{prop}

Now, let $\hat{R}_{q}\colon V \otimes V \to V \otimes V$ be the one-parameter $R$-matrix, which is given explicitly by
\begin{equation}\label{eq:1_parameter_R_D}
\begin{split}
  \hat{R}_{q} =\
  & q^{-1}\sum_{1\leq i\leq 2n} E_{ii} \otimes E_{ii} + q\sum_{1\leq i\leq 2n}E_{ii'} \otimes E_{i'i} \\
  & + \sum_{1\leq i,j\leq 2n}^{j \neq i,i'} E_{ij} \otimes E_{ji}
    + (q^{-1} - q)\sum_{1\leq i\leq n} (1 - q^{2(n- i)}) E_{i'i'} \otimes E_{ii} \\
  & + (q^{-1} -q)\sum_{i > j}^{j \neq i'} E_{ii} \otimes E_{jj}
    + (q - q^{-1})\sum_{i < j}^{j \neq i'} t_{i}(q)t_{j}(q)^{-1}E_{i'j} \otimes E_{ij'} \,,
  \end{split}
\end{equation}
where
\begin{equation*}
  t_{i}(q) =
  \begin{cases}
    q^{n- i} & \text{if}\ \ i \le n \\
    q^{n + 1 - i} & \text{if}\ \ i > n
  \end{cases}.
\end{equation*}
Then as for types $B_n$ and $C_n$, we have
  $\hat{R}_{r,s} = (\psi \otimes \psi)^{-1} \circ \xi \circ \hat{R}_{q} \circ \xi^{-1} \circ (\psi \otimes \psi)$,
where $\xi$ denotes the isomorphism $\xi_{V,V}$ of~\eqref{eq:xi_formula}, $\psi$ is the isomorphism of~\eqref{eq:psi-D}, and
$\hat{R}_{r,s}$ denotes the matrix $\hat{R}$ of~\cite[(4.11)]{MT1}. We can also use this to obtain a new proof of~\cite[Theorem~4.6]{MT1}.
To this end, we note that
\begin{equation*}
\begin{split}
  (\psi \otimes \psi)^{-1} \circ \xi & \circ \hat{R}_{q} \circ \xi^{-1} \circ (\psi \otimes \psi) \\
  =\, & q^{-1}\sum_{1\leq i\leq 2n} E_{ii} \otimes E_{ii}
    + q\sum_{1\leq i\leq 2n}E_{ii'} \otimes E_{i'i} \\
  & + \sum_{1\leq i,j\leq 2n}^{j \neq i,i'} \zeta(\varepsilon_{i},\varepsilon_{j})^{-1}\zeta(\varepsilon_{j},\varepsilon_{i}) E_{ij} \otimes E_{ji}
    + (q^{-1} - q)\sum_{1\leq i\leq n} (1 - q^{2(n- i)}) E_{i'i'} \otimes E_{ii} \\
  & + (q^{-1} -q)\sum_{i > j}^{j \neq i'} E_{ii} \otimes E_{jj}
    + (q - q^{-1})\sum_{i < j}^{j \neq i'} t_{i}(q)t_{j}(q)^{-1}\psi_{i}^{-1}\psi_{i'}^{-1}\psi_{j}\psi_{j'}E_{i'j} \otimes E_{ij'} \,,
  \end{split}
\end{equation*}
with the convention $\varepsilon_{i'} = -\varepsilon_{i}\ \forall\, i$. Arguing as in types $B_n$ and $C_n$, we get
$\frac{\zeta(\varepsilon_{j},\varepsilon_{i})}{\zeta(\varepsilon_{i},\varepsilon_{j})} = a_{ij}$ for $j\ne i,i'$, and
\[
  t_{i}(q)t_{j}(q)^{-1}\psi_{i}^{-1}\psi_{i'}^{-1}\psi_{j}\psi_{j'} =
  \begin{cases}
    q^{j-i}(rs)^{\frac{1}{2}(i - j)} = s^{i - j} & \text{if}\ i < j \le n \\
    q^{j-i-1}(rs)^{\frac{1}{2}(i + j - 2n - 1)} = r^{j - n - 1}s^{i - n} & \text{if}\ i \le n < j \\
    q^{j-i}(rs)^{\frac{1}{2}(j - i)} = r^{j - i} & \text{if}\ n < i < j
  \end{cases}.
\]
Thus, the resulting matrix exactly matches the corresponding two-parameter $R$-matrix of~\cite[(4.11)]{MT1}.

   %%%%%%%%%%%%%%%%%%%%%%%%%%%%%%%%%%%%%%%%%%%%%%%%%%%%%%%%%%%%%%%%%%%%%%%%%%

\subsection{Twisting affine $R$-matrices in classical types}\label{ssec:affine_R_twisting}
\

As in Subsection~\ref{ssec:B-type}, the affine two-parameter $R$-matrices $\hat{R}_{r,s}(z)$ of~\cite[Theorems 6.11--6.13]{MT1}
are similarly related to the corresponding affine one-parameter $R$-matrices $\hat{R}_{q}(z)$ via
\begin{equation}\label{eq:affineR-twisting}
  \hat{R}_{r,s}(z) = (\psi \otimes \psi)^{-1} \circ \xi \circ \hat{R}_{q}(z) \circ \xi^{-1} \circ (\psi \otimes \psi).
\end{equation}
This is due to~\cite[\S 7]{MT1} where $\hat{R}_{r,s}(z)$ are expressed as the following linear combinations of
$\hat{R}_{r,s}$, $\hat{R}_{r,s}^{-1}$, $\mathrm{Id}$:
\begin{itemize}[leftmargin=0.7cm]

\item
\emph{Type $B_n$}
\begin{equation}\label{eq:Baxterization-B}
\begin{split}
  \hat{R}_{r,s}(z) =
  r^{-1}sz(z - 1)\hat{R}_{r,s}^{-1} + (1 - r^{-2n + 1}s^{2n - 1})(1 - r^{-2}s^{2})z\Id - r^{-2n}s^{2n}(z - 1)\hat{R}_{r,s} \\
  =q^{-2}z(z - 1)\hat{R}_{r,s}^{-1} + (1 - q^{-4n + 2})(1 - q^{-4})z\Id - q^{-4n}(z - 1)\hat{R}_{r,s}
\end{split}
\end{equation}

\item
\emph{Type $C_n$}
\begin{equation}\label{eq:Baxterization-C}
\begin{split}
  \hat{R}_{r,s}(z) =
  r^{-1/2}s^{1/2}z(z-1)\hat{R}_{r,s}^{-1} + (1 - r^{-n-1}s^{n + 1})(1 - r^{-1}s)z\Id - r^{-n - 3/2}s^{n+ 3/2}(z - 1)\hat{R}_{r,s} \\
  =q^{-1}z(z-1)\hat{R}_{r,s}^{-1} + (1 - q^{-2n - 2})(1 - q^{-2})z\Id - q^{-2n - 3}(z - 1)\hat{R}_{r,s}
\end{split}
\end{equation}

\item
\emph{Type $D_n$}
\begin{equation}\label{eq:Baxterization-D}
\begin{split}
  \hat{R}_{r,s}(z) =
  r^{-1/2}s^{1/2}z(z-1)\hat{R}_{r,s}^{-1} + (1 - r^{-n+1}s^{n-1})(1 - r^{-1}s)z\Id - r^{-n+1/2}s^{n-1/2}(z - 1)\hat{R}_{r,s}\\
  = q^{-1}z(z-1)\hat{R}_{r,s}^{-1} + (1 - q^{-2n + 2})(1 - q^{-2})z\Id - q^{-2n + 1}(z - 1)\hat{R}_{r,s}
\end{split}
\end{equation}

\end{itemize}
This provides a new proof of~\cite[Theorems 6.11--6.13]{MT1} and yields a simple relationship between $R_{q}(z)$ and $R_{r,s}(z)$,
thus completing the $BCD$-type case of~\cite[Appendix B]{MT0}.

   %%%%%%%%%%%%%%%%%%%%%%%%%%%%%%%%%%%%%%%%%%%%%%%%%%%%%%%%%%%%%%%%%%%%%%%%%%
   %%%%%%%%%%%%%%%%%%%%%%%%%%%%%%%%%%%%%%%%%%%%%%%%%%%%%%%%%%%%%%%%%%%%%%%%%%
   %%%%%%%%%%%%%%%%%%%%%%%%%%%%%%%%%%%%%%%%%%%%%%%%%%%%%%%%%%%%%%%%%%%%%%%%%%

\section{PBW bases}\label{sec:pbw}

In this Section, we use the twisting procedure of Subsection~\ref{ssec:quantum_group_twisting} to construct orthogonal PBW bases
for $U_{r,s}(\fg)$ from corresponding bases of $U_{q}(\fg)$, as well as to explicitly evaluate the nonzero pairing constants.
This establishes~\cite[Theorem 5.12]{MT1} in full generality, cf.~\cite[Theorems 7.1, 7.2]{MT2}.

   %%%%%%%%%%%%%%%%%%%%%%%%%%%%%%%%%%%%%%%%%%%%%%%%%%%%%%%%%%%%%%%%%%%%%%%%%%

\subsection{Orthogonal bases via twisting}
\

Let $\{E_{\gamma}\}_{\gamma \in \Phi^{+}}$ and $\{F_{\gamma}\}_{\gamma \in \Phi^{+}}$ be the \textbf{(quantum) root vectors} for $U_{q}(\fg)$.
Then $E_{\alpha_{i}} = E_{i}$ and $F_{\alpha_{i}} = F_{i}$ for all $\alpha_{i} \in \Pi$. To describe root vectors for $\gamma \in \Phi^{+} \setminus \Pi$,
let $\ell(\gamma) = \ell_{1}\ell_{2}$ be the costandard factorization for the dominant Lyndon word associated to $\gamma$, and let $\alpha = |\ell_{1}|$,
$\beta = |\ell_{2}|$, cf.~\eqref{eq:cost-factorization}. Following~\eqref{eq:root_vectors_intro}, we set:
\begin{equation}\label{eq:root-vectors-1param}
  E_{\gamma} = E_{\alpha}E_{\beta} - q^{(\alpha,\beta)}E_{\beta}E_{\alpha} \qquad \text{and} \qquad
  F_{\gamma} = F_{\beta}F_{\alpha} - q^{-(\alpha,\beta)}F_{\alpha}F_{\beta}.
\end{equation}
Let $(\cdot,\cdot)_{q}\colon U_{q}^{\le} \times U_{q}^{\ge} \to \bb{K}$ be the Hopf pairing on $U_{q}(\fg)$ (see~\cite[Proposition 6.12]{J}),
and $(\cdot,\cdot)_{q,q^{-1}}\colon U_{q,q^{-1}}^{\le} \times U_{q,q^{-1}}^{\ge} \to \bb{K}$ be the Hopf pairing on $U_{q,q^{-1}}(\fg)$
(which is the $r = q, s = q^{-1}$ case of~\eqref{eq:Hopf-parity}). Note that the bilinear form $(y,x)'_{q,q^{-1}} = (\pi(y),\pi(x))_{q}$,
for $\pi$ of~\eqref{eq:Uq_quotient}, coincides with $(\cdot,\cdot)_{q,q^{-1}}$ on generators, and since $\pi$ is a homomorphism of
Hopf algebras, this implies that in fact $(y,x)'_{q,q^{-1}} = (y,x)_{q,q^{-1}}$ for all $y \in U_{q,q^{-1}}^{\le}$, $x \in U_{q,q^{-1}}^{\ge}$.
Moreover, by Corollary~\ref{cor:double_cartan_+-_iso}, we have $\ker(\pi) \cap U_{q,q^{-1}}^{\pm} = 0$. Therefore, we obtain:

\begin{cor}
If $\{y_{i}^{\mu}\}_{i=1}^{n_\mu}$ and $\{x_{i}^{\mu}\}_{i=1}^{n_\mu}$ are dual bases of $(U_{q}^{-})_{-\mu}$ and $(U_{q}^{+})_{\mu}$ with respect
to $(\cdot,\cdot)_{q}$, then $\{\pi^{-1}(y_{i}^{\mu})\}_{i=1}^{n_\mu}$ and $\{\pi^{-1}(x_{i}^{\mu})\}_{i=1}^{n_\mu}$ are dual bases of
$(U_{q,q^{-1}}^{-})_{-\mu}$ and $(U_{q,q^{-1}}^{+})_{\mu}$ with respect to $(\cdot,\cdot)_{q,q^{-1}}$.
\end{cor}

Let $e_{\gamma;q} = \pi^{-1}(E_{\gamma})$ and $f_{\gamma;q} = \pi^{-1}(F_{\gamma})$, which are given by the same recursive
formulas~\eqref{eq:root-vectors-1param}. Given $\gamma \in \Phi^{+}\setminus \Pi$, pick $\alpha,\beta$ as above, and inductively define
(cf.~\cite[(7.1)]{MT2}):
\begin{equation}\label{eq:root-vectors-2param}
  e_{\gamma;r,s} = e_{\alpha;r,s}e_{\beta;r,s} - (\omega_{\beta}',\omega_{\alpha})e_{\beta;r,s}e_{\alpha;r,s} \qquad \text{and} \qquad
  f_{\gamma;r,s} = f_{\beta;r,s}f_{\alpha;r,s} - (\omega_{\alpha}',\omega_{\beta})^{-1}f_{\alpha;r,s}f_{\beta;r,s}.
\end{equation}

We can now relate the quantum root vectors of $U_{r,s}(\fg)$ and $U_{q,q^{-1}}(\fg)$:

\begin{prop}\label{prop:pbw-basis-twist}
For all $\gamma \in \Phi^{+}$, we have
\[
  e_{\gamma;q} = c_{\gamma}^{+}\varphi(e_{\gamma;r,s}) \qquad \text{and} \qquad f_{\gamma;q} = c_{\gamma}^{-}\varphi(f_{\gamma;r,s}),
\]
where the constants $c_{\gamma}^{\pm}$ are given recursively by $c_{\alpha_{i}}^{+} = 1,c_{\alpha_{i}}^{-} = (r_{i}s_{i})^{1/2}$
for $1\leq i\leq n$, and
\[
  c_{\gamma}^{+} = \zeta(\beta,\alpha)c_{\alpha}^{+}c_{\beta}^{+}, \qquad
  c_{\gamma}^{-} = \zeta(\beta,\alpha)c_{\alpha}^{-}c_{\beta}^{-} \qquad \forall\, \gamma \in \Phi^{+} \setminus \Pi,
\]
where $\ell(\gamma) = \ell(\alpha)\ell(\beta)$ is the costandard factorization of the dominant Lyndon word $\ell(\gamma)$.
\end{prop}

\begin{proof}
We proceed by induction on $\hgt(\gamma)$. If $\hgt(\gamma) = 1$, then the claim is immediate from the definition of
$\varphi$ in~\eqref{eq:phi-map}. Now suppose that $\hgt(\gamma) > 1$ and consider the costandard factorization
$\ell(\gamma) = \ell(\alpha)\ell(\beta)$. Then by definition, we have
\begin{multline*}
  \varphi(e_{\gamma;r,s})
  = \varphi(e_{\alpha;r,s}e_{\beta;r,s} - (\omega_{\beta}',\omega_{\alpha})e_{\beta;r,s}e_{\alpha;r,s})
  = \varphi(e_{\alpha;r,s})\circ \varphi(e_{\beta;r,s}) - (\omega_{\beta}',\omega_{\alpha})\varphi(e_{\beta;r,s}) \circ \varphi(e_{\alpha;r,s}) \\
  = \frac{1}{c_{\alpha}^{+}c_{\beta}^{+}}
     \left( \zeta(\alpha,\beta)e_{\alpha;q}e_{\beta;q} - (\omega_{\beta}',\omega_{\alpha})\zeta(\beta,\alpha)e_{\beta;q}e_{\alpha;q} \right).
\end{multline*}
As $\zeta(\alpha,\beta) = (\omega_{\beta}',\omega_{\alpha})^{1/2}q^{-\frac{1}{2}(\alpha,\beta)} = \zeta(\beta,\alpha)^{-1}$
by~\eqref{eq:zeta_formula}, we find that
\begin{align*}
  \frac{\zeta(\alpha,\beta)e_{\alpha;q}e_{\beta;q}
  - (\omega_{\beta}',\omega_{\alpha})\zeta(\beta,\alpha)e_{\beta;q}e_{\alpha;q}}{c_{\alpha}^{+}c_{\beta}^{+}}
  &= \frac{\zeta(\alpha,\beta)(e_{\alpha;q}e_{\beta;q} - q^{(\alpha,\beta)}e_{\beta;q}e_{\alpha;q})}{c_{\alpha}^{+}c_{\beta}^{+}}
   = \frac{\zeta(\alpha,\beta)}{c_{\alpha}^{+}c_{\beta}^{+}}e_{\gamma;q},
\end{align*}
and therefore
\[
  \varphi(e_{\gamma;r,s}) = \frac{\zeta(\alpha,\beta)}{c_{\alpha}^{+}c_{\beta}^{+}} e_{\gamma;q} =
  \frac{1}{c_{\alpha}^{+}c_{\beta}^{+}\zeta(\beta,\alpha)}e_{\gamma;q} = \frac{1}{c^+_\gamma} e_{\gamma;q}.
\]

Similarly,
\begin{multline*}
  \varphi(f_{\gamma;r,s})
  = \varphi(f_{\beta;r,s}f_{\alpha;r,s} - (\omega_{\alpha}',\omega_{\beta})^{-1}f_{\alpha;r,s}f_{\beta;r,s})
  = \varphi(f_{\beta;r,s}) \circ \varphi(f_{\alpha;r,s}) - (\omega_{\alpha}',\omega_{\beta})^{-1}\varphi( f_{\alpha;r,s}) \circ \varphi(f_{\beta;r,s}) \\
  = \frac{1}{c_{\alpha}^{-}c_{\beta}^{-}}
    \left( \zeta(\alpha,\beta)f_{\beta;q}f_{\alpha;q} - (\omega_{\alpha}',\omega_{\beta})^{-1}\zeta(\beta,\alpha)f_{\alpha;q}f_{\beta;q} \right)
  = \frac{\zeta(\alpha,\beta)}{c_{\alpha}^{-}c_{\beta}^{-}}
    \left( f_{\beta;q}f_{\alpha;q} - (\omega_{\alpha}',\omega_{\beta})^{-1}\zeta(\beta,\alpha)^{2}f_{\alpha;q}f_{\beta;q} \right).
\end{multline*}
As $\zeta(\beta,\alpha)^{2} = (\omega_{\alpha}',\omega_{\beta})q^{-(\alpha,\beta)}$ by~\eqref{eq:zeta_formula}, we find that
\[
  \varphi(f_{\gamma;r,s}) = \frac{\zeta(\alpha,\beta)}{c_{\alpha}^{-}c_{\beta}^{-}} f_{\gamma;q} =
  \frac{1}{c_{\alpha}^{-}c_{\beta}^{-}\zeta(\beta,\alpha)} f_{\gamma;q} = \frac{1}{c^-_\gamma} f_{\gamma;q}.
\]
This completes the proof.
\end{proof}

Recall the lexicographical order~\eqref{eq:lyndon_order} on $\Phi^+$. We are now ready to provide a new proof of~\cite[Theorem~7.1]{MT2},
cf.~\cite[Theorem~5.12(a,b)]{MT1}:

\begin{theorem}\label{thm:orthogonal PBW bases}
(1) The ordered products
\begin{equation}
  \left\{ \overset{\longleftarrow}{\underset{\gamma \in \Phi^{+}}{\prod}} e_{\gamma;r,s}^{m_{\gamma}}\, \Big|\, m_{\gamma} \ge 0 \right\}
  \quad \text{and}\quad
  \left\{ \overset{\longleftarrow}{\underset{\gamma \in \Phi^{+}}{\prod}} f_{\gamma;r,s}^{m_{\gamma}}\, \Big|\, m_{\gamma} \ge 0 \right\}
\end{equation}
are bases for $U_{r,s}^{+}(\fg)$ and $U_{r,s}^{-}(\fg)$, respectively. Here and below, the arrow $\leftarrow$ over the product signs refers
to the total order~\eqref{eq:lyndon_order} on $\Phi^+$, thus ordering the positive roots in decreasing order.

\medskip
\noindent
(2) The Hopf pairing is orthogonal with respect to these bases. More explicitly, we have:
\begin{equation}
    \left( \overset{\longleftarrow}{\underset{\gamma \in \Phi^{+}}{\prod}} f_{\gamma;r,s}^{n_{\gamma}},
           \overset{\longleftarrow}{\underset{\gamma \in \Phi^{+}}{\prod}} e_{\gamma;r,s}^{m_{\gamma}} \right)_{r,s} = \,
    \prod_{\gamma\in \Phi^+} \Big( \delta_{n_\gamma,m_\gamma} [m_{\gamma}]_{r_{\gamma},s_{\gamma}}!
        s_{\gamma}^{-\frac{1}{2}m_{\gamma}(m_{\gamma} - 1)} (f_{\gamma;r,s},e_{\gamma;r,s})^{m_{\gamma}}_{r,s} \Big).
\end{equation}
\end{theorem}

\begin{proof}
The ordering~\eqref{eq:lyndon_order} is known to be convex (see~\cite[Proposition 26]{L}), and therefore it coincides with an order obtained
from some reduced decomposition of the longest element of the Weyl group (see~\cite{P}). In particular, $E_\gamma$ and $F_\gamma$ coincide
with Lusztig's root vectors, up to nonzero constants. As such, part~(1) holds when $e_{\gamma;r,s},f_{\gamma;r,s}$ are replaced by
$e_{\gamma;q},f_{\gamma;q}$, respectively, and we have the following analogue of part~(2):
\[
  \left( \overset{\longleftarrow}{\underset{\gamma \in \Phi^{+}}{\prod}} f_{\gamma;q}^{n_{\gamma}},
           \overset{\longleftarrow}{\underset{\gamma \in \Phi^{+}}{\prod}} e_{\gamma;q}^{m_{\gamma}} \right)_{q,q^{-1}} = \,
  \prod_{\gamma\in \Phi^+} \Big( \delta_{n_\gamma,m_\gamma} [m_{\gamma}]_{q_{\gamma}}!
      q_{\gamma}^{\frac{1}{2}m_{\gamma}(m_{\gamma} - 1)} (f_{\gamma;q},e_{\gamma;q})^{m_{\gamma}}_{q,q^{-1}} \Big).
\]

Combining Propositions~\ref{prop:twisted_algebra} and~\ref{prop:pbw-basis-twist}, we thus obtain part (1) of above theorem. Moreover,
since $f_{\gamma;q} \in {U_{q,q^{-1}}(\fg)}_{0,-\gamma}$ and $e_{\gamma;q} \in {U_{q,q^{-1}}(\fg)}_{\gamma,0}$ for all $\gamma$, it
follows from the definition of $\circ$ in~\eqref{eq:twisted_product_general} that $f_{\gamma;q}^{n_{\gamma}} = f_{\gamma;q}^{\circ n_{\gamma}}$
and $e_{\gamma;q}^{n_{\gamma}} = e_{\gamma;q}^{\circ n_{\gamma}}$. Furthermore, for $\gamma_{1},\gamma_{2} \in \Phi^{+}$, we have
   $e_{\gamma_{1}}^{\circ n_{\gamma_{1}}} \circ e_{\gamma_{2}}^{\circ n_{\gamma_{2}}} =
    \zeta(n_{\gamma_{1}}\gamma_{1},n_{\gamma_{2}}\gamma_{2})e_{\gamma_{1}}^{n_{\gamma_{1}}}e_{\gamma_{2}}^{n_{\gamma_{2}}}$,
and
  $f_{\gamma_{1}}^{\circ n_{\gamma_{1}}} \circ f_{\gamma_{2}}^{\circ n_{\gamma_{2}}} =
   \zeta(n_{\gamma_{1}}\gamma_{1},n_{\gamma_{2}}\gamma_{2})^{-1}f_{\gamma_{1}}^{n_{\gamma_{1}}}f_{\gamma_{2}}^{n_{\gamma_{2}}}$.
Thus, it follows from Proposition~\ref{prop:twisted_pairing} and equation~\eqref{eq:pairing-twist-matching} that
\begin{align*}
  & \left( \overset{\longleftarrow}{\underset{\gamma \in \Phi^{+}}{\prod}} f_{\gamma;r,s}^{n_{\gamma}},
         \overset{\longleftarrow}{\underset{\gamma \in \Phi^{+}}{\prod}} e_{\gamma;r,s}^{m_{\gamma}} \right)_{r,s}
  =\, \left( \varphi \left( \overset{\longleftarrow}{\underset{\gamma \in \Phi^{+}}{\prod}} f_{\gamma;r,s}^{n_{\gamma}} \right),
            \varphi \left( \overset{\longleftarrow}{\underset{\gamma \in \Phi^{+}}{\prod}} e_{\gamma;r,s}^{m_{\gamma}} \right)  \right)_{q,\zeta} \\
  &= \left( \overset{\longleftarrow}{\underset{\gamma \in \Phi^{+}}{\prod}} (c_{\gamma}^{-})^{-n_{\gamma}} f_{\gamma;q}^{n_{\gamma}},
            \overset{\longleftarrow}{\underset{\gamma \in \Phi^{+}}{\prod}} (c_{\gamma}^{+})^{-m_{\gamma}}e_{\gamma;q}^{m_{\gamma}} \right)_{q,\zeta}
   =\, \left( \overset{\longleftarrow}{\underset{\gamma \in \Phi^{+}}{\prod}} (c_{\gamma}^{-})^{-n_{\gamma}}f_{\gamma;q}^{n_{\gamma}},
            \overset{\longleftarrow}{\underset{\gamma \in \Phi^{+}}{\prod}} (c_{\gamma}^{+})^{-m_{\gamma}}e_{\gamma;q}^{m_{\gamma}} \right)_{q,q^{-1}} \\
  &= \prod_{\gamma\in \Phi^+} \Big( \delta_{n_\gamma,m_\gamma} [m_{\gamma}]_{r_{\gamma},s_{\gamma}}!
        s_{\gamma}^{-\frac{1}{2}m_{\gamma}(m_{\gamma} - 1)} (f_{\gamma;r,s},e_{\gamma;r,s})^{m_{\gamma}}_{r,s} \Big).
\end{align*}
The last equality above follows from
\begin{equation}\label{eq:2_to_1_pairing_doubled}
  ((c_{\gamma}^{-})^{-1}f_{\gamma;q},(c_{\gamma}^{+})^{-1}e_{\gamma;q})_{q,q^{-1}} =
  (\varphi(f_{\gamma;r,s}),\varphi(e_{\gamma;r,s}))_{q,q^{-1}} = (f_{\gamma;r,s},e_{\gamma;r,s})_{r,s}
\end{equation}
and
\begin{align*}
  [m_{\gamma}]_{q_{\gamma}}!q_{\gamma}^{\frac{1}{2}m_{\gamma}(m_{\gamma} - 1)}
  &= (r_{\gamma}^{1/2}s_{\gamma}^{-1/2})^{\frac{1}{2}m_{\gamma}(m_{\gamma} - 1)} \prod_{1\leq k\leq m_{\gamma}}
  \left (\frac{r_{\gamma}^{k/2}s_{\gamma}^{-k/2} - r_{\gamma}^{-k/2}s_{\gamma}^{k/2}}{r_{\gamma}^{1/2}s_{\gamma}^{-1/2} - r_{\gamma}^{-1/2}s_{\gamma}^{1/2}}\right) \\
  &= (r_{\gamma}^{1/2}s_{\gamma}^{-1/2})^{\frac{1}{2}m_{\gamma}(m_{\gamma} - 1)}
     \prod_{1\leq k\leq m_{\gamma}}\left ((r_{\gamma}s_{\gamma})^{-(k-1)/2}\frac{r_{\gamma}^{k} - s_{\gamma}^{k}}{r_{\gamma} - s_{\gamma}}\right )
   = s_{\gamma}^{-\frac{1}{2}m_{\gamma}(m_{\gamma} - 1)}[m_{\gamma}]_{r_{\gamma},s_{\gamma}}!
\end{align*}
This completes the proof.
\end{proof}

   %%%%%%%%%%%%%%%%%%%%%%%%%%%%%%%%%%%%%%%%%%%%%%%%%%%%%%%%%%%%%%%%%%%%%%%%%%

\subsection{Evaluation of pairing constants}
\

We can also use Proposition~\ref{prop:pbw-basis-twist} along with~\cite[Theorem 4.2]{BKM} to derive a recursive formula for the pairings
$(f_{\gamma;r,s},e_{\gamma;r,s})_{r,s}$. To this end, we first recall that $F_{\gamma} = \pi(f_{\gamma;q}), E_{\gamma} = \pi(e_{\gamma;q})$,
so that \eqref{eq:2_to_1_pairing_doubled} implies
\begin{equation}\label{eq:2_to_1_pairing}
  \frac{1}{c_{\gamma}^{-}c_{\gamma}^{+}}(F_{\gamma},E_{\gamma})_{q} = (f_{\gamma;r,s},e_{\gamma;r,s})_{r,s}.
\end{equation}
We note that the results of~\cite{BKM} are valid for any convex ordering on $\Phi^{+}$, so that the results below are also true in this
more general context (where one takes $(\alpha,\beta)$ to be a minimal pair for $\gamma$ in the sense of~\cite[(5.4)]{MT1}). However,
for convenience, we will restrict our attention to Lyndon orderings in our exposition.

Let us first recall the relevant result of~\cite{BKM}, which shall then be restated in a slightly different form. The bilinear form
on $U_{q}^{+}$ considered in~\cite{BKM}, which we shall rather denote by $\{\cdot,\cdot\}$ to distinguish from our pairings, is defined
as the unique symmetric bilinear form on $U_{q}^{+}$ which satisfies
\[
  \{1,1\} = 1, \qquad \{E_{i},E_{j}\} = \frac{\delta_{ij}}{1 - q_{i}^{2}},\qquad \{xy,z\} = \{x \otimes y, \Delta_{q}(z)\}
  \qquad \text{for all}\qquad 1 \le i,j \le n,\ x,y,z \in U_{q}^{+},
\]
where $\Delta_{q}\colon U_{q}^{+} \to U_{q}^{+} \otimes U_{q}^{+}$ is a map which is uniquely determined by the following conditions:
\[
  \Delta_{q}(E_{i}) = E_{i} \otimes 1 + 1 \otimes E_{i} \qquad \text{for all}\qquad 1 \le i \le n,
\]
and $\Delta_{q}$ is an algebra homomorphism with respect to the product on $U_{q}^{+} \otimes U_{q}^{+}$ given by
$(x \otimes y)(x' \otimes y') = q^{-(\deg(y),\deg(x'))}xx' \otimes yy'$ for all homogeneous $x,x',y,y' \in U_{q}^{+}$.
Now, for each $\gamma \in \Phi^+$, we define elements $\bf{E}_{\gamma} \in U_{q}^{+}$ by $\bf{E}_{\alpha_{i}} = E_{i}$ for all $\alpha_{i} \in \Pi$, and
\[
  \bf{E}_{\gamma} = \bf{E}_{\alpha}\bf{E}_{\beta} - q^{-(\alpha,\beta)}\bf{E}_{\beta}\bf{E}_{\alpha} \qquad \forall\, \gamma \in \Phi^{+} \setminus \Pi,
\]
where $\ell(\gamma) = \ell(\alpha)\ell(\beta)$ is the costandard factorization of the dominant Lyndon word $\ell(\gamma)$.
For any $\alpha,\beta \in \Phi$, we also define
\[
  p_{\beta,\alpha} = \max \big\{ p \in \bb{Z} \,|\, \beta - p\alpha \in \Phi \big\}.
\]
Then, in the case of Lyndon orderings~\eqref{eq:lyndon_order}, Theorem 4.2 of~\cite{BKM} is equivalent to the following result
(here, we also use~\cite[Theorem 2.7, formula~(2.9)]{BKM} which are attributed to Lusztig):

\begin{theorem}\label{thm:pairing_BKM}
For all $\gamma \in \Phi^{+}$, we have
\[
  \{\bf{E}_{\gamma},\bf{E}_{\gamma}\} = \frac{\kappa_{\gamma}^{2}}{1 - q_{\gamma}^{2}},
\]
where the constants $\{\kappa_{\gamma}\}_{\gamma\in \Phi^+}$ are defined recursively by $\kappa_{\alpha_{i}} = 1$ for all $\alpha_{i} \in \Pi$,
and
\[
  \kappa_{\gamma} = [p_{\beta,\alpha} + 1]_{q}\kappa_{\alpha}\kappa_{\beta}\qquad \forall\, \gamma \in \Phi^{+} \setminus \Pi,
\]
where $\ell(\gamma) = \ell(\alpha)\ell(\beta)$ is the costandard factorization of the dominant Lyndon word.
\end{theorem}

Next, we need to describe the relationship between $\{\cdot,\cdot\}$ and the restriction
$(\cdot,\cdot)_{q}\colon U_{q}^{-} \times U_{q}^{+} \to \bb{K}$. To this end, we first recall that $U_{q}(\fg)$ has a \textbf{Cartan involution}
\begin{equation}\label{eq:cartan_involution}
  \omega\colon U_{q}(\fg) \to U_{q}(\fg) \qquad \text{such that}\qquad E_{i} \mapsto F_{i},\qquad F_{i} \mapsto E_{i},\qquad K_{i} \mapsto K_{i}^{-1},
\end{equation}
as well as a $\bb{Q}$-algebra automorphism (a \textbf{bar involution}) $x \mapsto \bar{x}$ such that
\[
  \bar{E}_{i} = E_{i},\qquad \bar{F}_{i} = F_{i},\qquad \bar{K}_{i} = K_{i}^{-1},\qquad \bar{q} = q^{-1}.
\]
We also introduce the notation
\[
  d_{\mu} = \sum_{1\leq i\leq n}c_{i} \frac{(\alpha_i,\alpha_i)}{2} \qquad \text{for all}\qquad \mu = \sum_{1\leq i\leq n}c_{i}\alpha_{i} \in Q.
\]

Finally, we recall from~\cite[\S6.14]{J} that for each $1 \le i \le n$ there exists a linear map $p_{i}'\colon U_{q}^{+}\to U_{q}^{+}$ (which is denoted by $r_{\alpha_{i}}'$ in~\cite[\S6.14]{J}) satisfying
\[
  p_{i}'(1) = 0,\qquad p_{i}'(E_{j}) = \delta_{ij},\qquad p_{i}'(xy) = q^{(\alpha_{i},\mu)}xp_{i}'(y) + p_{i}'(x)y
  ,\qquad \forall\, 1 \le i,j \le n,\ x \in (U_{q}^{+})_{\mu},\ y \in (U_{q}^{+})_{\nu}
\]
with any $\mu,\nu\in Q$. Likewise, by~\cite[\S1.2.13]{Lu}, there exists for each $1 \le i \le n$ a linear map $\partial_{i}'\colon U_{q}^{+} \to U_{q}^{+}$ (which is the map $\prescript{~}{i}r$ with $v = q^{-1}$ in~\cite[\S1.2.13]{Lu}) satisfying
\[
  \partial_{i}'(1) = 0,\qquad \partial_{i}'(E_{j}) = \delta_{ij},\qquad \partial_{i}'(xy) = \partial_{i}'(x)y + q^{-(\mu,\alpha_{i})}x\partial_{i}'(y)
  ,\qquad \forall\, 1\le i,j \le n,\ x \in (U_{q}^{+})_{\mu},\ y \in (U_{q}^{+})_{\nu}
\]
with any $\mu,\nu\in Q$. By~\cite[Lemma 6.14(b)]{J} and~\cite[\S1.2.13(a)]{Lu}, these maps satisfy
\begin{equation}\label{eq:pairing_properties}
  (F_{i}y,x)_{q} = \frac{1}{q_{i}^{-1} - q_{i}}(y,p_{i}'(x))_{q},\qquad \{E_{i}y,x\} = \frac{1}{1 - q_{i}^{2}}\{y,\partial_{i}'(x)\}.
\end{equation}
Furthermore, one can easily check by comparing~\cite[Lemma 6.14(a)]{J} with the properties of $\partial_{i}'$ in~\cite[\S1.2.13]{Lu} that
\begin{equation}\label{eq:derivation_comparison}
\ol{p_{i}'(x)} = \partial_{i}'(\bar{x})\qquad \text{for all}\qquad x \in U_{q}^{+}.
\end{equation}

We then have the following result:

\begin{lemma}\label{lem:pairing_comparison}
For all $y \in (U_{q}^{-})_{-\mu}$ and $x \in (U_{q}^{+})_{\mu}$, we have
\begin{equation}\label{eq:pairing_matching}
  (y,x)_{q} = (-1)^{\hgt(\mu)}q^{-d_{\mu}}\ol{\{\omega(\bar{y}),\bar{x}\}}.
\end{equation}
\end{lemma}

\begin{proof}
We proceed by induction on $\hgt(\mu)$. If $\mu = \alpha_{i} \in \Pi$, then it suffices to verify~\eqref{eq:pairing_matching} for $y=F_i,x=E_i$,
in which case both sides equal $\frac{1}{q_i^{-1}-q_i}$. For the induction step, it is enough to prove the claim for $x \in (U_{q}^{+})_{\mu}$ and
$y = F_{i}y'$ with $y' \in (U_{q}^{-})_{-\mu + \alpha_{i}}$.
By the induction hypothesis and equations~\eqref{eq:pairing_properties},~\eqref{eq:derivation_comparison}, we have
\begin{align*}
  (y,x)_{q}
  &= \frac{1}{q_{i}^{-1} - q_{i}}(y',p_{i}'(x))_{q}
   = \frac{1}{q_{i}^{-1} - q_{i}}(-1)^{\hgt(\mu) - 1}q^{-d_{\mu- \alpha_{i}}}\ol{\{\omega(\bar{y}'),\ol{p_{i}'(x)}\}} \\
  &= (-1)^{\hgt(\mu)}q^{-d_{\mu}}\ol{\frac{1}{1 - q_{i}^{2}}\{\omega(\bar{y}'),\partial_{i}'(\bar{x})\}}
   = (-1)^{\hgt(\mu)}q^{-d_{\mu}} \ol{\{\omega(\ol{F_{i}y'}),\bar{x}\}}
   =(-1)^{\hgt(\mu)}q^{-d_{\mu}} \ol{\{\omega(\ol{y}),\bar{x}\}},
\end{align*}
as claimed in~\eqref{eq:pairing_matching}.
\end{proof}

Combining the above result with Theorem~\ref{thm:pairing_BKM}, we obtain:

\begin{cor}\label{cor:1param_pairing_recursion}
For any $\gamma \in \Phi^{+} \setminus \Pi$, let $\ell(\gamma) = \ell(\alpha)\ell(\beta)$ be the costandard factorization of the
dominant Lyndon word $\ell(\gamma)$. Then
\[
  (F_{\gamma},E_{\gamma})_{q} =
  -q^{-(\alpha,\beta)}\frac{[p_{\beta,\alpha} + 1]_{q}^{2}(1 - q_{\alpha}^{-2})(1 - q_{\beta}^{-2})}{1 - q_{\gamma}^{-2}}
  (F_{\alpha},E_{\alpha})_{q}(F_{\beta},E_{\beta})_{q}.
\]
\end{cor}

\begin{proof}
We shall first prove by induction on $\hgt(\gamma)$ that $\omega(\bar{F}_{\gamma}) = c_{\gamma}\bf{E}_{\gamma}$, where the constants $c_{\gamma}$
are defined recursively by $c_{\alpha_{i}} = 1$ for all $\alpha_{i} \in \Pi$, and $c_{\gamma} = -q^{(\alpha,\beta)}c_{\alpha}c_{\beta}$ for
$\gamma \in \Phi^{+} \setminus \Pi$. The claim is obvious for $\gamma \in \Pi$, while for $\gamma \notin \Pi$ we have
\begin{align*}
  \omega(\bar{F}_{\gamma})
  &= \omega\left (\ol{F_{\beta}F_{\alpha} - q^{-(\alpha,\beta)}F_{\alpha}F_{\beta}}\right )
   = \omega\left (\bar{F}_{\beta}\bar{F}_{\alpha} - q^{(\alpha,\beta)}\bar{F}_{\alpha}\bar{F}_{\beta}\right ) \\
  &= -q^{(\alpha,\beta)}\left( \omega(\bar{F}_{\alpha})\omega(\bar{F}_{\beta}) - q^{-(\alpha,\beta)}\omega(\bar{F}_{\beta})\omega(\bar{F}_{\alpha}) \right)
   = -q^{(\alpha,\beta)} c_{\alpha} c_{\beta} (\bf{E}_{\alpha}\bf{E}_{\beta} - q^{-(\alpha,\beta)}\bf{E}_{\beta}\bf{E}_{\alpha}),
\end{align*}
so the claim follows from the definitions of $\bf{E}_{\gamma}$ and $c_{\gamma}$.
A similar argument shows that $\bar{E}_{\gamma} = \bf{E}_{\gamma}$ for all $\gamma$.
Thus, combining Lemma~\ref{lem:pairing_comparison} and Theorem~\ref{thm:pairing_BKM}, we obtain:
\begin{equation}\label{eq:pairing_ckappa}
  (F_{\gamma},E_{\gamma})_{q}
  = (-1)^{\hgt(\gamma)}q^{-d_{\gamma}}\ol{\{\omega(\bar{F_{\gamma}}),\bar{E_{\gamma}}\}}
  = (-1)^{\hgt(\gamma)}q^{-d_{\gamma}}\bar{c}_{\gamma}\ol{\{\bf{E}_{\gamma},\bf{E}_{\gamma}\}}
  = \frac{(-1)^{\hgt(\gamma)}q^{-d_{\gamma}}\bar{c}_{\gamma}\bar{\kappa}_{\gamma}^{2}}{1 - q_{\gamma}^{-2}}
\end{equation}
for all $\gamma \in \Phi^{+}$. Noting that $\ol{[p_{\beta,\alpha} + 1]_{q}} = [p_{\beta,\alpha} +1]_{q}$, $\hgt(\gamma) = \hgt(\alpha) + \hgt(\beta)$,
and $d_{\gamma} = d_{\alpha} + d_{\beta}$, we then get:
\begin{align*}
  (F_{\gamma},E_{\gamma})_{q}
  &= -q^{-(\alpha,\beta)}\frac{(-1)^{\hgt(\gamma)}q^{-d_{\gamma}}[p_{\beta,\alpha} + 1]_{q}^{2}
     \bar{c}_{\alpha}\bar{c}_{\beta}\bar{\kappa}_{\alpha}^{2}\bar{\kappa}_{\beta}^{2}}{1 - q_{\gamma}^{-2}} \\
  &= -q^{-(\alpha,\beta)}\frac{[p_{\beta,\alpha} + 1]_{q}^{2}(1 - q_{\alpha}^{-2})(1 - q_{\beta}^{-2})}{1 - q_{\gamma}^{-2}}
     \left( \frac{(-1)^{\hgt(\alpha)}q^{-d_{\alpha}}\bar{c}_{\alpha}\bar{\kappa}_{\alpha}^{2}}{1 - q_{\alpha}^{-2}} \right)
     \left( \frac{(-1)^{\hgt(\beta)}q^{-d_{\beta}}\bar{c}_{\beta}\bar{\kappa}_{\beta}^{2}}{1 - q_{\beta}^{-2}} \right) \\
  &= -q^{-(\alpha,\beta)}\frac{[p_{\beta,\alpha} + 1]_{q}^{2}(1 - q_{\alpha}^{-2})(1 - q_{\beta}^{-2})}{1 - q_{\gamma}^{-2}}
     (F_{\alpha},E_{\alpha})_{q}(F_{\beta},E_{\beta})_{q},
\end{align*}
where we used~\eqref{eq:pairing_ckappa} twice in the last equality. This completes the proof.
\end{proof}

Combining the above result with~\eqref{eq:2_to_1_pairing}, we finally obtain:

\begin{theorem}\label{thm:2param_pairing_recursion}
For any $\gamma \in \Phi^{+} \setminus \Pi$, let $\ell(\gamma) = \ell(\alpha)\ell(\beta)$ be the costandard factorization of the
dominant Lyndon word $\ell(\gamma)$. Then
\[
  (f_{\gamma;r,s},e_{\gamma;r,s})_{r,s} =
  \frac{r_{\gamma}(\omega_{\alpha}',\omega_{\beta})^{-1}(rs)^{-p_{\beta,\alpha}}[p_{\beta,\alpha} + 1]_{r,s}^{2}(r_{\alpha} - s_{\alpha})(r_{\beta} - s_{\beta})}
       {r_{\alpha}r_{\beta}(s_{\gamma} - r_{\gamma})}(f_{\alpha;r,s},e_{\alpha;r,s})_{r,s}(f_{\beta;r,s},e_{\beta;r,s})_{r,s}.
\]
\end{theorem}

\begin{proof}
First, we note that $[p_{\beta,\alpha} + 1]_{q}^{2}  = (rs)^{-p_{\beta,\alpha}}[p_{\beta,\alpha} + 1]_{r,s}^{2}$ and
$(1 - q_{\gamma}^{-2}) = (1 - r_{\gamma}^{-1}s_{\gamma}) = r_{\gamma}^{-1}(r_{\gamma} - s_{\gamma})$ for all $\gamma \in \Phi^{+}$.
Thus, by Corollary~\ref{cor:1param_pairing_recursion}, formula~\eqref{eq:2_to_1_pairing}, and the definitions of $c_{\gamma}^{\pm}$
in Proposition~\ref{prop:pbw-basis-twist}, we have:
\begin{align*}
  (f_{\gamma;r,s},e_{\gamma;r,s})_{r,s}
  &= q^{-(\alpha,\beta)}
     \frac{r_{\gamma}(rs)^{-p_{\beta,\alpha}}[p_{\beta,\alpha} + 1]_{r,s}^{2}(r_{\alpha} - s_{\alpha})(r_{\beta} - s_{\beta})c_{\alpha}^{+}c_{\alpha}^{-}c_{\beta}^{+}c_{\beta}^{-}}
          {r_{\alpha}r_{\beta}(s_{\gamma} - r_{\gamma})c_{\gamma}^{+}c_{\gamma}^{-}}
     (f_{\alpha;r,s},e_{\alpha;r,s})_{r,s}(f_{\beta;r,s},e_{\beta;r,s})_{r,s} \\
  &= \zeta(\beta,\alpha)^{-2}q^{-(\alpha,\beta)}
     \frac{r_{\gamma}(rs)^{-p_{\beta,\alpha}}[p_{\beta,\alpha} + 1]_{r,s}^{2}(r_{\alpha} - s_{\alpha})(r_{\beta} - s_{\beta})}
          {r_{\alpha}r_{\beta}(s_{\gamma} - r_{\gamma})}
     (f_{\alpha;r,s},e_{\alpha;r,s})_{r,s}(f_{\beta;r,s},e_{\beta;r,s})_{r,s} \\
  &\overset{\eqref{eq:zeta_formula}}{=}\frac{r_{\gamma}(\omega_{\alpha}',\omega_{\beta})^{-1}
   (rs)^{-p_{\beta,\alpha}}[p_{\beta,\alpha} + 1]_{r,s}^{2}(r_{\alpha} - s_{\alpha})(r_{\beta} - s_{\beta})}
         {r_{\alpha}r_{\beta}(s_{\gamma} - r_{\gamma})}
    (f_{\alpha;r,s},e_{\alpha;r,s})_{r,s}(f_{\beta;r,s},e_{\beta;r,s})_{r,s}.
\end{align*}
This completes the proof.
\end{proof}

This result finally establishes~\cite[Theorem 5.12(c)]{MT1}, of which~\cite[Theorem 7.2]{MT2} is a special case.

\begin{remark}
The above results immediately imply~\cite[Theorem 5.13, Remark 5.14]{MT1} thus providing a factorization of $\Theta_{r,s}$
from~\eqref{eq:Theta} into an ordered product of \emph{$(r,s)$-exponents}, labeled by positive roots of $\fg$.
This can also be applied in the super setup, thus providing a conceptual explanation of~\cite[formulas~(9),~(11)]{Z}.
\end{remark}

   %%%%%%%%%%%%%%%%%%%%%%%%%%%%%%%%%%%%%%%%%%%%%%%%%%%%%%%%%%%%%%%%%%%%%%%%%%
   %%%%%%%%%%%%%%%%%%%%%%%%%%%%%%%%%%%%%%%%%%%%%%%%%%%%%%%%%%%%%%%%%%%%%%%%%%
   %%%%%%%%%%%%%%%%%%%%%%%%%%%%%%%%%%%%%%%%%%%%%%%%%%%%%%%%%%%%%%%%%%%%%%%%%%

\section{Subalgebra interpretation}\label{sec:subalgebra-interpret}

In this Section, following~\cite[Appendix A]{MT0}, we show that $U_{q,\zeta}(\fg)$ may be realized as a subalgebra of a certain
extension of $U_{q,q^{-1}}(\fg)$. We subsequently realize $U_{q,q^{-1}}(\fg)$ as a subalgebra of an extension of~$U_{q}(\fg)$.

   %%%%%%%%%%%%%%%%%%%%%%%%%%%%%%%%%%%%%%%%%%%%%%%%%%%%%%%%%%%%%%%%%%%%%%%%%%

\subsection{An alternative approach to the twisting construction}
\

Consider the algebra $\wtd{U}_{q,q^{-1}}(\fg)$, which has the generators and relations
of $U_{q,q^{-1}}(\fg)$, along with additional generators $\{\omega_{i,\zeta}^{\pm 1}\}_{i = 1}^{n}$ that satisfy the following relations:
\begin{align*}
  &\omega_{i,\zeta}^{\pm 1}\omega_{i,\zeta}^{\mp 1} = 1, &
  &[\omega_{i,\zeta}^{\pm 1},\omega_{j}] = [\omega_{i,\zeta}^{\pm 1},\omega_{j}'] = [\omega_{i,\zeta}^{\pm 1},\omega_{j,\zeta}] = 0, \\
  &\omega_{i,\zeta}e_{j} = \zeta(\alpha_{i},\alpha_{j})e_{j}\omega_{i,\zeta}, & &\omega_{i,\zeta}f_{j} = \zeta(\alpha_{j},\alpha_{i})f_{j}\omega_{i,\zeta}.
\end{align*}
We define the following elements:
\begin{equation}\label{eq:Sev-twist}
  \wtd{e}_{i} = e_{i}\omega_{i,\zeta},\qquad \wtd{f}_{i} = f_{i}\omega_{i,\zeta},\qquad
  \wtd{\omega}_{i} = \omega_{i}\omega_{i,\zeta}^{2},\qquad \wtd{\omega}_{i}' = \omega_{i}'\omega_{i,\zeta}^{2}
  \qquad \forall\, 1\leq i \leq n.
\end{equation}
Then we have:

\begin{prop}\label{prop:subalgebra_interpretation_qgp}
(a) The elements~\eqref{eq:Sev-twist} satisfy the following relations:
\begin{align*}
  & \wtd{\omega}_{i}\wtd{e}_{j} = \zeta(\alpha_{i},\alpha_{j})^{2}q^{(\alpha_{i},\alpha_{j})}\wtd{e}_{j}\wtd{\omega}_{i},
  & & \wtd{\omega}_{i}'\wtd{e}_{j} = \zeta(\alpha_{i},\alpha_{j})^{2}q^{-(\alpha_{i},\alpha_{j})}\wtd{e}_{j}\wtd{\omega}_{i}', \\
  & \wtd{\omega}_{i}\wtd{f}_{j} = \zeta(\alpha_{j},\alpha_{i})^{2}q^{-(\alpha_{i},\alpha_{j})}\wtd{f}_{j}\wtd{\omega}_{i},
  & & \wtd{\omega}_{i}'\wtd{f}_{j} = \zeta(\alpha_{j},\alpha_{i})^{2}q^{(\alpha_{i},\alpha_{j})}\wtd{f}_{j}\wtd{\omega}_{i}',
\end{align*}
\[
  \wtd{e}_{i}\wtd{f}_{j} - \wtd{f}_{j}\wtd{e}_{i} =
  % \zeta(\alpha_{j},\alpha_{i})(E_{i}F_{j} - F_{j}E_{i})K_{i,\zeta}^{2} =
  \delta_{ij}\frac{\wtd{\omega}_{i} - \wtd{\omega}_{i}'}{q_{i} - q_{i}^{-1}},
\]
and
\begin{align*}
  &\sum_{k = 1}^{1 - a_{ij}} (-1)^k\qbinom{1 - a_{ij}}{k}_{q_{i}}\zeta(\alpha_{j},\alpha_{i})^{2k} \wtd{e}_{i}^{1 - a_{ij} - k}\wtd{e}_{j}\wtd{e}_{i}^{k}
   = 0 \quad \mathrm{for}\ \ i\ne j, \\
  &\sum_{k = 1}^{1 - a_{ij}} (-1)^k\qbinom{1 - a_{ij}}{k}_{q_{i}}\zeta(\alpha_{j},\alpha_{i})^{2k}\wtd{f}_{i}^{k}\wtd{f}_{j}\wtd{f}_{i}^{1 - a_{ij} - k}
   = 0 \quad \mathrm{for}\ \ i\ne j.
\end{align*}

\medskip
\noindent
(b) The assignment $e_{i} \mapsto \wtd{e}_{i}, f_{i} \mapsto \wtd{f}_{i}, \omega_{i} \mapsto \wtd{\omega}_{i}, \omega_{i}' \mapsto \wtd{\omega}_{i}'$
gives rise to an algebra isomorphism between $U_{q,\zeta}(\fg)$ and the subalgebra $\wtd{U}$ of $\wtd{U}_{q,q^{-1}}(\fg)$ generated by
$\{\wtd{e}_{i},\wtd{f}_{i},\wtd{\omega}^{\pm 1}_{i},(\wtd{\omega}_{i}')^{\pm 1}\}_{i = 1}^{n}$.
\end{prop}

In fact, the construction above is a special case of a more general one which we shall now describe. Recall
that an algebra $A$ over a field $F$ is a \textbf{left module-algebra} over a Hopf algebra $H$ if $A$ is a left $H$-module,
and the multiplication map $A \otimes A \to A$ and unit map $F \to A$ are both $H$-module homomorphisms, where $A \otimes A$
is made into an $H$-module via $\Delta$, and $F$ is made into an $H$-module via $\epsilon$. Explicitly, this means that:
\begin{equation*}
  h\cdot (ab) = \sum_{(h)}(h_{(1)}\cdot a)(h_{(2)}\cdot b) \qquad \text{and}\qquad
  h \cdot 1 = \epsilon(h)\qquad \text{for all}\qquad h \in H,\ a,b \in A.
\end{equation*}
Similarly, an algebra $A$ is called a \textbf{right module-algebra} over $H$ if $A$ is a right $H$-module such that the multiplication
and unit maps are both right $H$-module homomorphisms. Now, fix a group $G$, and let $H = F[G]$ be the group algebra of $G$ over $F$.
Let $A$ be an algebra which is also an $H$-$H$-bimodule, and such that the left $H$-module structure makes $A$ into
a left module-algebra over $H$, and the right $H$-module structure makes $A$ into a right module-algebra over $H$.

\begin{lemma}\label{lem:smash-product}
The following defines an algebra structure on $A \otimes H$:
\[
  \left( \sum_{g \in G}x_{g} \otimes g \right) \cdot \left( \sum_{h \in G}y_{h} \otimes h \right)
  = \sum_{g,h \in G}(x_{g}h)(gy_{h}) \otimes gh.
\]
\end{lemma}

\begin{proof}
Since $1 \cdot g = g \cdot 1 = \epsilon(g) = 1$ for all $g \in G$, the element $1 \otimes 1$ is the unit.
To check associativity, let $x,y,z \in A$ and $g,h,k \in G$. Then
\[
  \left( (x \otimes g) \cdot (y \otimes h) \right) \cdot (z \otimes k) =
  \left( (xh)(gy) \otimes gh \right) \cdot (z \otimes k) = \left( ((xh)(gy))k \right)((gh)z) \otimes ghk,
\]
while
\[
  (x \otimes g) \cdot \left( (y \otimes h) \cdot (z \otimes k) \right) =
  (x \otimes g) \cdot \left( (yk)(hz) \otimes hk \right) = (x(hk))(g((yk)(hz))) \otimes ghk.
\]
Then since $A$ is a left and right module-algebra over $H$ and $g,h,k$ are grouplike, we have $((xh)(gy))k = ((xh)k)((gy)k)$
and $g((yk)(hz)) = (g(yk))(g(hz))$. Since $A$ is also an $H$-$H$-bimodule, we therefore have
\[
  \left( ((xh)(gy))k \right)((gh)z) = (x(hk))(g(yk))(g(hz)) = (x(hk))(g((yk)(hz))),
\]
and hence
  $\left( (x \otimes g) \cdot (y \otimes h) \right) \cdot (z \otimes k) = (x \otimes g) \cdot \left( (y \otimes h)\cdot (z \otimes k) \right)$,
as desired.
\end{proof}

Now suppose that $G$ is an abelian group (written additively), and $A$ is a $G$-bigraded Hopf algebra over an algebraically closed field $F$.
We shall write the elements of $F[G]$ as $\sum_{g \in G}\alpha_{g}e^{g}$, where $e^{g}\cdot e^{h} = e^{g + h}$.
Let $\sigma\colon G \times G \to F^{*}$ be a skew bicharacter. Then we may define a left action of $F[G]$ on $A$ via
$e^{g}\cdot x = \sigma(g,h + k)^{1/2}x$ for $x \in A_{h,k}$. Similarly, we can define a right action of $F[G]$ on $A$
via $x\cdot e^{g} = \sigma(h + k,g)^{1/2}x$ for $x \in A_{h,k}$. It is clear that $A$ is an $F[G]$-$F[G]$-bimodule. Moreover, since
$1 \in A_{0,0}$, $A_{h,k}A_{h',k'} \subset A_{h+ h',k + k'}$ and $\sigma(g,-)^{1/2}$ is a group homomorphism for all $g \in G$, $A$ is
a left and right module-algebra over $F[G]$ with respect to these actions. Thus, we may define an algebra structure on $A \otimes F[G]$
as in Lemma~\ref{lem:smash-product}. Consider the subspace
\begin{equation}\label{eq:A'-subspace}
  A' = \bigoplus_{g,h \in G} A_{g,h} \otimes e^{g - h}.
\end{equation}
It is easy to see that $A'$ is a subalgebra of $A \otimes F[G]$. Moreover, for any $x \in A_{g,h}$ and $y \in A_{g',h'}$, we have:
\begin{multline*}
  (x \otimes e^{g-h}) \cdot (y \otimes e^{g'-h'}) = (xe^{g'-h'})(e^{g-h}y) \otimes e^{g + g' - h - h'} \\
  = \sigma(g + h,g' - h')^{1/2}\sigma(g - h,g' + h')^{1/2}xy \otimes e^{g + g' - h - h'}
  = \sigma(g,g')\sigma(h,h')^{-1}xy \otimes e^{g + g' - h - h'}.
\end{multline*}
Let $A_{\sigma}$ denote the algebra formed by twisting the multiplication of $A$ via the bicharacter $\sigma$,
see~\eqref{eq:twisted_product_general}. Then, the above discussion implies:

\begin{prop}\label{prop:subalgebra_interpretation_general}
The map $\varphi\colon A \to A'$ given by
\[
  \varphi(x) = x \otimes e^{g-h}\quad \text{for all}\ g,h \in G,\ x \in A_{g,h}
\]
defines an algebra isomorphism $\varphi\colon A_{\sigma}\iso A'$.
\end{prop}

We shall now use this result to prove Proposition~\ref{prop:subalgebra_interpretation_qgp}.

\begin{proof}[Proof of Proposition~\ref{prop:subalgebra_interpretation_qgp}]
Part (a) is based on elementary calculations, which we leave to the interested reader.
Let us now prove part (b). For $\mu = \sum_{i = 1}^{n}c_{i}\alpha_{i} \in Q$, we define
\[
  \omega_{\mu,\zeta} = \prod_{1\leq i \leq n}\omega_{i,\zeta}^{c_{i}}.
\]
Note that for any $x \in U_{q,q^{-1}}(\fg)_{\alpha,\beta}$ and $\mu \in Q$, we have
\begin{equation}\label{eq:Ktail-ad}
  \omega_{\mu,\zeta}x\omega_{\mu,\zeta}^{-1} = \zeta(\mu,\alpha + \beta)x.
\end{equation}
Now, we define a linear map $\psi\colon U_{q,q^{-1}}(\fg) \otimes \bb{K}[Q] \to \wtd{U}_{q,q^{-1}}(\fg)$ by
\begin{equation}\label{eq:f-map-aux}
  \psi(x \otimes e^{\mu}) = \zeta(\mu,\alpha + \beta)^{1/2}x\omega_{\mu,\zeta} \quad \text{for all}\quad
  x \in U_{q,q^{-1}}(\fg)_{\alpha,\beta},\ \alpha,\beta,\mu \in Q,
\end{equation}
where $U_{q,q^{-1}}(\fg)$ is identified with its image in $\wtd{U}_{q,q^{-1}}(\fg)$. We endow
$U_{q,q^{-1}}(\fg) \otimes \bb{K}[Q]$ with the aforementioned algebra structure with $\sigma = \zeta$,
see the paragraph preceding~\eqref{eq:A'-subspace}.
For any $x \in U_{q,q^{-1}}(\fg)_{\alpha,\beta}$, $y \in U_{q,q^{-1}}(\fg)_{\alpha',\beta'}$, and $\mu,\nu \in Q$, we have:
\begin{equation*}
\begin{split}
  \psi((x \otimes e^{\mu}) \cdot (y \otimes e^{\nu})) &= \psi((xe^{\nu})(e^{\mu}y) \otimes e^{\mu + \nu}) \\
  &= \zeta(\alpha + \beta,\nu)^{1/2}\zeta(\mu,\alpha' + \beta')^{1/2}\psi(xy \otimes e^{\mu + \nu}) \\
  &= \zeta(\alpha + \beta,\nu)^{1/2}\zeta(\mu,\alpha' + \beta')^{1/2}\zeta(\mu + \nu,\alpha + \alpha' + \beta + \beta')^{1/2}xy\omega_{\mu + \nu,\zeta} \\
  &= \zeta(\mu,\alpha' + \beta')\zeta(\mu,\alpha + \beta)^{1/2}\zeta(\nu,\alpha' + \beta')^{1/2}xy\omega_{\mu + \nu,\zeta},
\end{split}
\end{equation*}
while
\begin{equation*}
\begin{split}
  \psi(x \otimes e^{\mu})\psi(y \otimes e^{\nu})
  &= \zeta(\mu,\alpha + \beta)^{1/2}\zeta(\nu,\alpha' + \beta')^{1/2}x\omega_{\mu,\zeta}y\omega_{\nu,\zeta} \\
  &\overset{\eqref{eq:Ktail-ad}}{=}
   \zeta(\mu,\alpha' + \beta')\zeta(\mu,\alpha + \beta)^{1/2}\zeta(\nu,\alpha' + \beta')^{1/2}xy\omega_{\mu + \nu,\zeta}.
\end{split}
\end{equation*}
This shows that $\psi$ of~\eqref{eq:f-map-aux} is an algebra homomorphism. In fact, it is an algebra isomorphism, with
the inverse map $\wtd{U}_{q,q^{-1}}(\fg) \to U_{q,q^{-1}}(\fg) \otimes \bb{K}[Q]$ explicitly given on the generators by
\[
  e_{i} \mapsto e_{i} \otimes 1, \quad f_{i} \mapsto f_{i} \otimes 1, \quad \omega_{i} \mapsto \omega_{i} \otimes 1, \quad
  \omega_{i}'\mapsto \omega_{i}' \otimes 1, \quad \omega_{i,\zeta} \mapsto 1 \otimes e^{\alpha_{i}} \quad \forall\, 1 \le i \le n.
\]

Following~\eqref{eq:A'-subspace}, we consider the subalgebra
\[
  U' = \bigoplus_{\mu,\nu \in Q}U_{q,q^{-1}}(\fg)_{\mu,\nu} \otimes e^{\mu - \nu} \subset U_{q,q^{-1}}(\fg) \otimes \bb{K}[Q].
\]
Explicitly, $U'$ is generated by
  $\{e_{i} \otimes e^{\alpha_{i}}, f_{i} \otimes e^{\alpha_{i}}, \omega_{i}^{\pm 1} \otimes e^{\pm 2\alpha_{i}},
     (\omega_{i}')^{\pm 1} \otimes e^{\pm 2\alpha_{i}}\}_{i=1}^{n}$,
and we have:
\[
  \psi(e_{i} \otimes e^{\alpha_{i}}) = \wtd{e}_{i}, \quad \psi(f_{i} \otimes e^{\alpha_{i}}) = \wtd{f}_{i},\quad
  \psi(\omega_{i}^{\pm 1} \otimes e^{\pm 2\alpha_{i}}) = (\wtd{\omega}_{i})^{\pm 1},\quad
  \psi((\omega_{i}')^{\pm 1} \otimes e^{\pm 2\alpha_{i}}) = (\wtd{\omega}_{i}')^{\pm 1}.
\]
Thus, we have an algebra isomorphism:
\[
  \psi\colon U'\iso \wtd{U}.
\]

Now, by Proposition~\ref{prop:subalgebra_interpretation_general}, we have an algebra isomorphism $\varphi\colon U_{q,\zeta}(\fg) \iso U'$
defined by $\varphi(x) = x \otimes e^{\mu-\nu}$ for all $x \in U_{q,q^{-1}}(\fg)_{\mu,\nu}$, and therefore
$\psi \circ \varphi\colon U_{q,\zeta}(\fg) \iso \wtd{U}$ is an algebra isomorphism that maps
\[
  x \mapsto \zeta(\mu - \nu,\mu + \nu)^{1/2}x\omega_{\mu - \nu,\zeta} \qquad \text{for all}\qquad x \in {U_{q,q^{-1}}(\fg)}_{\mu,\nu}.
\]
It is easy to see that this is precisely the isomorphism of Proposition~\ref{prop:subalgebra_interpretation_qgp}(b), which completes the proof.
\end{proof}

   %%%%%%%%%%%%%%%%%%%%%%%%%%%%%%%%%%%%%%%%%%%%%%%%%%%%%%%%%%%%%%%%%%%%%%%%%%

\subsection{Cartan-doubled quantum group via the key embedding}
\

We can also realize $U_{q,q^{-1}}(\fg)$, which has twice as many Cartan elements as the usual one-parameter quantum group $U_{q}(\fg)$,
as a subalgebra of a certain extension of $U_{q}(\fg)$, as indicated by the following result:

\begin{prop}\label{prop:double_cartan_DJ}
Let $L_{\alpha_{1}},\ldots ,L_{\alpha_{n}}$ be algebraically independent invertible transcendental elements. Then there is a unique injective
algebra homomorphism $\eta \colon U_{q,q^{-1}}(\fg) \to U_{q}(\fg) \otimes \bb{K}[L_{\alpha_{1}}^{\pm 1},\ldots ,L_{\alpha_{n}}^{\pm 1}]$
such that
\begin{equation}\label{eq:double_Cartan_DJ}
  \eta\colon \quad e_{i} \mapsto E_{i} \otimes L_{\alpha_{i}} , \quad f_{i} \mapsto F_{i} \otimes L_{\alpha_{i}}, \quad
  \omega_{i} \mapsto K_{i} \otimes L_{\alpha_{i}}^{2}, \quad \omega_{i}' \mapsto K_{i}^{-1} \otimes L_{\alpha_{i}}^{2}.
\end{equation}
\end{prop}

Henceforth, we shall often use the notation
\begin{equation}\label{eq:L_mu}
  L_{\mu} = L_{\alpha_{1}}^{c_{1}}\ldots L_{\alpha_{n}}^{c_{n}}\qquad \text{for all}\qquad \mu = \sum_{1\leq i\leq n}c_{i}\alpha_{i} \in Q.
\end{equation}

\begin{proof}
It is easy to verify that~\eqref{eq:double_Cartan_DJ} gives rise to an algebra homomorphism. For injectivity,
first recall from Corollary~\ref{cor:double_cartan_+-_iso} that $U_{q,q^{-1}}^{\pm} \simeq U_{q}^{\pm}$ via
$e_{i} \mapsto E_{i}$ and $f_{i} \mapsto F_{i}$, respectively. For each $\mu \in Q^{+}$, choose bases
$\{\wtd{x}_{i}^{\mu}\}_{i = 1}^{n_{\mu}}$ and $\{\wtd{y}_{i}^{\mu}\}_{i = 1}^{n_{\mu}}$ for $(U_{q,q^{-1}}^{+})_{\mu}$ and
$(U_{q,q^{-1}}^{-})_{-\mu}$, respectively. Under the aforementioned isomorphisms, these map to bases $\{x_{i}^{\mu}\}_{i = 1}^{n_{\mu}}$
and $\{y_{i}^{\mu}\}_{i = 1}^{n_{\mu}}$ for the respective subspaces $(U_{q}^{+})_{\mu}$ and $(U_{q}^{-})_{-\mu}$, and we have
$\eta(\wtd{x}_{i}^{\mu}) = x_{i}^{\mu} \otimes L_{\mu}$ and $\eta(\wtd{y}_{i}^{\mu}) = y_{i}^{\mu} \otimes L_{\mu}$ for all
$\mu \in Q^{+}$ and $1 \le i \le n_{\mu}$. Now, recalling from~\cite[Corollary~2.6]{HP1} that $U_{q,q^{-1}}(\fg)$ has a triangular
decomposition, we find that
\[
  \Big\{ \wtd{y}_{i}^{\mu}\omega_{\lambda}\omega_{\lambda'}'\wtd{x}_{j}^{\nu} \,\Big|\,
         \lambda,\lambda' \in Q,\, \mu,\nu \in Q^{+},\, 1 \le i \le n_{\mu},\, 1 \le j \le n_{\nu} \Big\}
\]
is a basis for $U_{q,q^{-1}}(\fg)$, and its image under $\eta$ consists of all
\[
  \eta(\wtd{y}_{i}^{\mu}\omega_{\lambda}\omega_{\lambda'}'\wtd{x}_{j}^{\nu})
  = y_{i}^{\mu}K_{\lambda - \lambda'}x_{j}^{\nu} \otimes L_{\mu + \nu + 2(\lambda + \lambda')}
\]
with $\lambda ,\lambda' \in Q$, $\mu,\nu \in Q^{+}$, $1 \le i \le n_{\mu}$ and $1 \le j \le n_{\nu}$. It is easy to see that this set
is linearly independent in $U_{q}(\fg) \otimes \bb{K}[L_{\alpha_{1}}^{\pm 1},\ldots ,L_{\alpha_{n}}^{\pm 1}]$, as one can determine
$\lambda$ and $\lambda'$ from $\lambda - \lambda'$ and $\lambda + \lambda'$. Therefore, $\eta$ is injective.
\end{proof}

We shall often use this result to derive theorems for $U_{q,q^{-1}}(\fg)$ from known results about $U_{q}(\fg)$.

   %%%%%%%%%%%%%%%%%%%%%%%%%%%%%%%%%%%%%%%%%%%%%%%%%%%%%%%%%%%%%%%%%%%%%%%%%%
   %%%%%%%%%%%%%%%%%%%%%%%%%%%%%%%%%%%%%%%%%%%%%%%%%%%%%%%%%%%%%%%%%%%%%%%%%%
   %%%%%%%%%%%%%%%%%%%%%%%%%%%%%%%%%%%%%%%%%%%%%%%%%%%%%%%%%%%%%%%%%%%%%%%%%%

\section{The FRT construction under the $R$-matrix twist}\label{sec:FRT construction finite}

Let $V$ be an $N$-dimensional $\bb{K}$-vector space with a basis $\{v_{1},\ldots ,v_{N}\}$. Throughout this Section, we will often be working
with elements $A \in \End(V) \otimes \End(V)$, and we shall use the following notational convention to express them entry-wise:
\begin{equation}\label{eq:matrix_notation}
  A = \sum_{1 \le i,j,k,\ell \le N} A^{ij}_{k\ell}E_{ik} \otimes E_{j\ell}.
\end{equation}
In particular, we will usually be working with elements $R \in \End(V) \otimes \End(V)$ that satisfy
\begin{equation}\label{eq:YBE}
  \textbf{Yang-Baxter\ equation}\quad R_{12}R_{13}R_{23} = R_{23}R_{13}R_{12}.
\end{equation}
Recall that if $\tau\colon V \otimes V \to V \otimes V$ denotes the flip map, then $R$ satisfies~\eqref{eq:YBE}
if and only if  $\hat{R} = R \circ \tau$ satisfies
\begin{equation}\label{eq:braid_relations}
   \textbf{braid\ relation}\quad  \hat{R}_{12}\hat{R}_{23}\hat{R}_{12} = \hat{R}_{23}\hat{R}_{12}\hat{R}_{23}.
\end{equation}

In what follows, we shall need the following simple result:

\begin{lemma}\label{lem:YBE_twist}
Let $R$ be a solution of~\eqref{eq:YBE}, and let $D = \sum_{i,j = 1}^{N}d_{ij}E_{ii} \otimes E_{jj} \in \End(V) \otimes \End(V)$
be such that $D\hat{R}D^{-1}$ satisfies~\eqref{eq:braid_relations} and $d_{ij} = d_{ji}^{-1}$ for all $i,j$.
Then $DRD$ is also a solution of~\eqref{eq:YBE}.
\end{lemma}

\begin{proof}
Since $D\circ \tau = \sum_{i,j = 1}^{N}d_{ji}E_{ji} \otimes E_{ij} = \sum_{i,j = 1}^{N}d_{ij}^{-1}E_{ji} \otimes E_{ij} = \tau \circ D^{-1}$,
we have that $DRD \circ \tau = D\hat{R}D^{-1}$, which satisfies~\eqref{eq:braid_relations} by the assumption.
Thus, $DRD$ satisfies~\eqref{eq:YBE}, as claimed.
\end{proof}

   %%%%%%%%%%%%%%%%%%%%%%%%%%%%%%%%%%%%%%%%%%%%%%%%%%%%%%%%%%%%%%%%%%%%%%%%%%

\subsection{The bialgebra $A(R)$}\label{ssec:algebra_A(R)}
\

Given any solution $R\in \End(V) \otimes \End(V)$ to the Yang-Baxter equation~\eqref{eq:YBE}, one can form the $\bb{K}$-algebra $A(R)$,
which is generated by elements $\{t_{ij} \,|\, 1 \le i,j \le N\}$ subject to the following \textbf{RTT-relations}:
\begin{equation}\label{eq:RTT-rel}
  RT_{1}T_{2} = T_{2}T_{1}R,
\end{equation}
where $T_{1} = \sum_{i,j = 1}^{N}E_{ij} \otimes 1 \otimes t_{ij} \in \End(V)^{\otimes 2} \otimes A(R)$ and
$T_{2} = \sum_{i,j = 1}^{N} 1 \otimes E_{ij} \otimes t_{ij} \in \End(V)^{\otimes 2} \otimes A(R)$.
The algebra $A(R)$ has a bialgebra structure defined via the following formulas:
\begin{equation}\label{eq:A(R)_bialgebra}
  \Delta(t_{ij}) = \sum_{1\leq k\leq N} t_{ik} \otimes t_{kj} \quad \text{and}\quad
  \epsilon(t_{ij}) = \delta_{ij}\quad \text{for all}\quad 1 \le i,j \le N.
\end{equation}
Using the convention~\eqref{eq:matrix_notation}, the above relations~\eqref{eq:RTT-rel}
can be explicitly written as:
\begin{equation}\label{eq:FRT_relations}
  \sum_{1 \le k,\ell \le N} R_{k\ell}^{mp}t_{ki}t_{\ell j} \ = \sum_{1 \le k,\ell \le N}t_{p\ell}t_{mk}R_{ij}^{k\ell}
  \quad \text{for all} \quad 1 \le i,j,m,p \le N.
\end{equation}

Finally, we recall a part of~\cite[Theorem 10.7]{KS} that we shall need later:

\begin{theorem}\label{thm:A(R)_skew_pairings}
There are unique bilinear maps $\sigma,\bar{\sigma}\colon A(R) \times A(R) \to \bb{K}$ which satisfy the following structural
properties for all $a,b,c \in A(R)$:
\begin{equation}\label{eq:skew_pairing_counit}
  \sigma(a,1) = \bar{\sigma}(a,1) = \epsilon(a),\qquad \sigma(1,a) = \bar{\sigma}(1,a) = \epsilon(a),
\end{equation}
\begin{equation}\label{eq:skew_pairing_adjointness}
  \sigma(ab,c) = \sum_{(c)}\sigma(a,c_{(1)})\sigma(b,c_{(2)}),\qquad \sigma(a,bc) = \sum_{(a)}\sigma(a_{(1)},c)\sigma(a_{(2)},b),
\end{equation}
\begin{equation}\label{eq:inverse_skew_pairing_adjointness}
  \bar{\sigma}(ab,c) = \sum_{(c)}\bar{\sigma}(b,c_{(1)})\bar{\sigma}(a,c_{(2)}),\qquad
  \bar{\sigma}(a,bc) = \sum_{(a)}\bar{\sigma}(a_{(1)},b)\bar{\sigma}(a_{(2)},c),
\end{equation}
\begin{equation}\label{eq:convolution_inverse}
  \sum_{(a)(b)}\sigma(a_{(1)},b_{(1)})\bar{\sigma}(a_{(2)},b_{(2)})
  = \sum_{(a)(b)}\bar{\sigma}(a_{(1)},b_{(1)})\sigma(a_{(2)},b_{(2)}) = \epsilon(a)\epsilon(b),
\end{equation}
and are given on generators by
\begin{equation}\label{eq:skew_pairing_generators}
  \sigma(t_{ij},t_{k\ell}) = R^{ik}_{j\ell},\qquad \bar{\sigma}(t_{ij},t_{k\ell}) = (R^{-1})^{ik}_{j\ell}
  \qquad \text{for all}\qquad 1 \le i,j,k,\ell \le N.
\end{equation}
\end{theorem}

Let us now examine the relationship between $A(R)$ and $A(DRD)$ for $D$ as in Lemma~\ref{lem:YBE_twist}. Since
\[
  DRD \ = \sum_{1 \le i,j,k,\ell \le N}R_{ij}^{k\ell}d_{ij}d_{k\ell}E_{ki} \otimes E_{\ell j}
\]
and $d_{ij} = d_{ji}^{-1}$, the defining relations in $A(DRD)$ read as follows (cf.~\eqref{eq:FRT_relations}):
\begin{equation}\label{eq:FRT_relations_twisted}
  d_{mp}\sum_{1 \le k,\ell \le N}d_{k\ell}R_{k\ell}^{mp}t_{ki}t_{\ell j}
  = d_{ij}\sum_{1 \le k, \ell \le N}d_{k\ell}t_{p\ell}t_{m k}R_{ij}^{k\ell}
  \quad \text{for all} \quad 1 \le i,j,m,n \le N.
\end{equation}

We shall now specialize to $R = R_{q} = \hat{R}_{q} \circ \tau$, where $\hat{R}_{q}$ is one of the matrices from
equations~\eqref{eq:1_parameter_R_B},~\eqref{eq:1_parameter_R_C}, or~\eqref{eq:1_parameter_R_D}. For these matrices,
we have the following result, which is crucial to the rest of this Section:

\begin{prop}\label{prop:A(R)_grading}
If $\hat{R}_{q}$ is the matrix of~\eqref{eq:1_parameter_R_B} (resp.~\eqref{eq:1_parameter_R_C},~\eqref{eq:1_parameter_R_D}),
and $R_{q} = \hat{R}_{q} \circ \tau$, then $A(R_{q})$ is a $P$-bigraded bialgebra with respect to the assignment
$\deg(t_{ij}) = (-\varepsilon_{i},\varepsilon_{j})$, where $P$ is the weight lattice of type $B_{n}$ (resp.\ $C_{n}$, $D_{n}$),
and $\varepsilon_{i'} = -\varepsilon_{i}$ for $1 \le i \le n$.
\end{prop}

\begin{proof}
Note that
$(R_{q})_{ij}^{k\ell} = 0$ unless $(i,j) = (k,\ell)$, $(i,j) = (\ell,k)$, or $j = i'$ and $\ell = k'$. Therefore,
$\varepsilon_{i} + \varepsilon_{j}=\varepsilon_{k} + \varepsilon_{\ell}$ whenever $(R_{q})_{ij}^{k\ell} \ne 0$.
Hence,
\[
  \deg\left (\sum_{1 \le k,\ell \le N} (R_{q})_{k\ell}^{mp}t_{ki}t_{\ell j}\right ) =
  (-\varepsilon_{m} - \varepsilon_{p},\varepsilon_{i}+  \varepsilon_{j}), \quad
  \deg\left (\sum_{1 \le k ,\ell \le N}t_{p\ell}t_{mk}(R_{q})_{ij}^{k\ell}\right ) =
  \deg(-\varepsilon_{m} - \varepsilon_{p},\varepsilon_{i} + \varepsilon_{j})
\]
for all $i,j,m,p$. This proves that the relations~\eqref{eq:FRT_relations} are homogeneous, and hence $A(R_{q})$
is a $P \times P$-graded algebra with $\deg(t_{ij}) = (-\varepsilon_{i},\varepsilon_{j})$. Moreover, it is clear from
the definitions of $\Delta$ and $\epsilon$ given in~\eqref{eq:A(R)_bialgebra} that $A(R_{q})$ is actually a $P$-bigraded bialgebra.
\end{proof}

Now, set
\begin{equation}\label{eq:D_zeta}
  D \, = \sum_{1 \le i,j \le N}d_{ij}E_{ii} \otimes E_{jj} \, =
  \sum_{1 \le i,j \le N}\zeta(\varepsilon_{j},\varepsilon_{i})E_{ii} \otimes E_{jj},
\end{equation}
where $\zeta\colon P \times P \to \bb{K}$ is the skew bicharacter satisfying~\eqref{eq:zeta_formula}. By the results of
Subsection~\ref{ssec:B-type}, $D\hat{R}_{q}D^{-1}$ satisfies~\eqref{eq:braid_relations}, so by Lemma~\ref{lem:YBE_twist}, $DR_{q}D$
satisfies~\eqref{eq:YBE}. Therefore, we can define the bialgebra $A(DR_{q}D)$ as above. On the other hand, we may also define a bialgebra
$A(R_{q})_{\zeta}$ as in Section~\ref{sec:twisted_R_matrices}, which has multiplication given by~\eqref{eq:twisted_product_general},
while the coproduct and counit are given by~\eqref{eq:A(R)_bialgebra}. We then have:

\begin{prop}\label{prop:A(R)_twisted}
There is a unique bialgebra isomorphism $\Psi\colon A(DR_{q}D) \iso A(R_{q})_{\zeta}$ such that $\Psi(t_{ij}) = t_{ij}$.
\end{prop}

\begin{proof}
In $A(R_{q})_{\zeta}$, we have
  $t_{ij} \circ t_{k\ell} = \zeta(\varepsilon_{i},\varepsilon_{k})\zeta(\varepsilon_{j},\varepsilon_{\ell})^{-1}t_{ij}t_{k\ell}
   = d_{ki}d_{j\ell}t_{ij}t_{k\ell}$,
so that
\[
  \sum_{1 \le k,\ell \le N}d_{ij}^{-1}d_{k\ell}(R_{q})_{k\ell}^{mp}t_{ki} \circ t_{\ell j} \,
  = \sum_{1 \le k ,\ell \le N}(R_{q})_{k\ell}^{mp}t_{ki}t_{\ell j}
\]
and
\[
  \sum_{1 \le k,\ell \le N}d_{mp}^{-1}d_{k\ell}(R_{q})_{ij}^{k\ell}t_{p\ell} \circ t_{mk} \,
  = \sum_{1 \le k,\ell \le N}(R_{q})_{ij}^{k\ell}t_{p\ell}t_{mk}.
\]
By~\eqref{eq:FRT_relations_twisted}, this implies that there is a well-defined algebra homomorphism $\Psi\colon A(DR_{q}D) \to A(R_{q})_{\zeta}$
mapping $t_{ij}\mapsto t_{ij}$. Likewise, if $\bar{\zeta}\colon P \times P \to \bb{K}$ is defined by $\bar{\zeta}(\lambda,\mu) = \zeta(\lambda,\mu)^{-1}$,
then there is a well-defined homomorphism $\Psi'\colon A(R_{q}) \to A(DR_{q}D)_{\bar{\zeta}}$ mapping $t_{ij} \mapsto t_{ij}$. Because $\Psi'$
preserves degrees, it also defines an algebra homomorphism from $A(R_{q})_{\zeta}$ to $(A(DR_{q}D)_{\bar{\zeta}})_{\zeta}$. Finally, since the
identity map on the underlying vector spaces is an algebra isomorphism $(A(DR_{q}D)_{\bar{\zeta}})_{\zeta} \iso A(DR_{q}D)$, we have an algebra
homomorphism $\Psi'\colon A(R_{q})_{\zeta} \to A(DR_{q}D)$ mapping $t_{ij} \mapsto t_{ij}$. Clearly, $\Psi'$ is the inverse to $\Psi$,
so $\Psi$ is an algebra isomorphism.

Since the coproduct and antipode in $A(R_{q})_{\zeta}$ are the same as the coproduct and antipode in $A(R_{q})$,
it is clear that $\Psi$ is in fact a bialgebra isomorphism.
\end{proof}

Now, consider any matrix of the form
\begin{equation}\label{eq:S_matrix}
  S = \sum_{1 \le i,j \le N} s_{i}s_{j}E_{ii} \otimes E_{jj}\qquad \text{with}\qquad s_{1},\ldots ,s_{N} \in \bb{K} \setminus \{0\}.
\end{equation}
Note that $S = S' \otimes S'$ with $S' = \sum_{i = 1}^{N} s_{i}E_{ii} \in \End(V)$. Thus, if $\hat{R}$ satisfies the braid
relations~\eqref{eq:braid_relations}, then a simple computation shows that $S\hat{R}S^{-1}$ does as well. Moreover, we have
$S\tau = \tau S$, so that if $R$ is a solution to the Yang-Baxter equation~\eqref{eq:YBE}, then $SRS^{-1}$ is also a solution
to~\eqref{eq:YBE}. Thus, we can form the algebra $A(SRS^{-1})$ as above. This algebra is related to the original $A(R)$ as follows:

\begin{prop}\label{prop:diagonal_conjugation_A(R)}
There is a unique bialgebra isomorphism $\Phi_{S}\colon A(R) \iso A(SRS^{-1})$ such that $\Phi_{S}(t_{ij}) = s_{i}^{-1}s_{j}t_{ij}$.
\end{prop}

\begin{proof}
Since the relations in $A(SRS^{-1})$ take the form
\[
  \sum_{1 \le k,\ell \le N}R_{k\ell}^{mp}s_{m}s_{p}s_{k}^{-1}s_{\ell}^{-1}t_{ki}t_{\ell j}\, =
  \sum_{1 \le k,\ell \le N}R_{ij}^{k\ell}s_{k}s_{\ell}s_{i}^{-1}s_{j}^{-1}t_{p\ell}t_{mk}
  \quad \text{for all} \quad 1 \le i,j,m,p \le N,
\]
it follows that there is an algebra homomorphism $\Phi_{S}\colon A(R) \to A(SRS^{-1})$ which sends $t_{ij}$ to $s_{i}^{-1}s_{j}t_{ij}$
for all $i,j$. On the other hand, replacing $R$ by $SRS^{-1}$ and $S$ by $S^{-1}$ we also get an algebra homomorphism $A(SRS^{-1})\to A(R)$
sending $t_{ij} \mapsto s_{i}s_{j}^{-1}t_{ij}$. This is clearly the inverse to $\Phi_{S}$, so $\Phi_{S}$ is an algebra isomorphism.

Moreover, we have
\[
  (\Phi_{S} \otimes \Phi_{S})(\Delta(t_{ij})) = \sum_{1\leq k\leq N}  s_{i}^{-1}s_{k}t_{ik} \otimes s_{k}^{-1}s_{j}t_{kj}  =
  s_{i}^{-1}s_{j}\sum_{1\leq k\leq N}  t_{ik} \otimes t_{kj} = \Delta(\Phi_{S}(t_{ij}))
\]
and
\[
  \epsilon (\Phi_{S}(t_{ij})) = \delta_{ij}s_{i}^{-1}s_{j} = \delta_{ij} = \epsilon(t_{ij}),
\]
which proves that $\Phi_{S}$ is in fact a bialgebra isomorphism.
\end{proof}

Combining the above two results, we can now relate $A(R_{r,s})$ to $A(R_{q})$, where $R_{q} = \hat{R}_{q} \circ \tau$,
$R_{r,s} = \hat{R}_{r,s} \circ \tau$, and $\hat{R}_{q},\hat{R}_{r,s}$ are the one- and two-parameter $R$-matrices of type
$B_{n}$ (resp.\ $C_{n}$, $D_{n}$) from~\eqref{eq:1_parameter_R_B} and~\cite[(4.7)]{MT1}
(resp.~\eqref{eq:1_parameter_R_C} and~\cite[(4.9)]{MT1},~\eqref{eq:1_parameter_R_D} and~\cite[(4.11)]{MT1}).
Then we have the following result, where $\psi_{i}$ are given by~\eqref{eq:psi_values_B} (resp.~\eqref{eq:psi_values_C},~\eqref{eq:psi_values_D}):

\begin{cor}\label{cor:A(R)_2_vs_1_parameter}
There is a bialgebra isomorphism $\wtd{\Psi}\colon A(R_{r,s}) \iso A(R_{q})_{\zeta}$ given by $\wtd{\Psi}(t_{ij}) = \psi_{i}^{-1}\psi_{j}t_{ij}$.
\end{cor}

\begin{proof}
Recall that $\hat{R}_{r,s} = (\psi \otimes \psi)^{-1} \circ \xi \circ \hat{R}_{q} \circ \xi^{-1} \circ (\psi \otimes \psi)$,
where $\xi = \xi_{V,V}$ is the isomorphism of~\eqref{eq:xi_formula}. Note that the matrix of $\xi$ is precisely the matrix
$D$ of~\eqref{eq:D_zeta}, and $S = \sum_{1 \le i,j\le N}\psi_{i}\psi_{j}E_{ii} \otimes E_{jj}$ is the matrix of $\psi \otimes \psi$.
Then we have
\begin{equation}\label{eq:R-relation}
  R_{r,s} = \hat{R}_{r,s}\tau = S^{-1}D(R_{q} \tau)D^{-1}S \tau = S^{-1}(DR_{q}D \tau)S \tau = S^{-1}DR_{q}DS.
\end{equation}
By Proposition~\ref{prop:diagonal_conjugation_A(R)}, we have a bialgebra isomorphism $\Phi_{S}\colon A(R_{r,s}) \iso A(SR_{r,s}S^{-1}) = A(DR_{q}D)$
such that $\Phi_{S}(t_{ij}) = \psi_{i}^{-1}\psi_{j}t_{ij}$. By Proposition~\ref{prop:A(R)_twisted}, we have a bialgebra isomorphism
$\Psi\colon A(DR_{q}D) \iso A(R_{q})_{\zeta}$ with $\Psi(t_{ij}) = t_{ij}$. Thus,
$\wtd{\Psi} = \Psi \circ \Phi_{S}\colon A(R_{r,s}) \iso A(R_{q})_{\zeta}$ is a bialgebra isomorphism and
$\wtd{\Psi}(t_{ij}) = \psi_{i}^{-1}\psi_{j}t_{ij}$.
\end{proof}

   %%%%%%%%%%%%%%%%%%%%%%%%%%%%%%%%%%%%%%%%%%%%%%%%%%%%%%%%%%%%%%%%%%%%%%%%%%

\subsection{The Hopf algebra $U(R)$ and its variations}\label{ssec:algebra_U(R)}
\

For any solution $R\in \End(V) \otimes \End(V)$ to the Yang-Baxter equation~\eqref{eq:YBE}, one can construct a $\bb{K}$-algebra $U(R)$
which is generated by elements
  $\{\ell_{ji}^{+},\ell_{ij}^{-} \,|\, 1 \le i \le j \le N\}\cup \{(\ell_{ii}^{+})^{-1},(\ell_{ii}^{-})^{-1} \,|\, 1 \le i \le N\}$
subject to:
\begin{equation}\label{eq:RLL_relations}
  RL_{1}^{\pm}L_{2}^{\pm} = L_{2}^{\pm}L_{1}^{\pm}R,\qquad
  RL_{1}^{+}L_{2}^{-} = L_{2}^{-}L_{1}^{+}R,\qquad
  \ell_{ii}^{\pm}(\ell_{ii}^{\pm})^{-1} = (\ell_{ii}^{\pm})^{-1}\ell_{ii}^{\pm} = 1,
\end{equation}
where $L_{1}^{\pm} = \sum_{1 \le i,j \le N} E_{ij} \otimes 1 \otimes \ell_{ij}^{\pm} \in \End(V)^{\otimes 2} \otimes U(R)$ and
$L_{2}^{\pm} = \sum_{1 \le i,j \le N} 1 \otimes E_{ij} \otimes \ell_{ij}^{\pm} \in \End(V)^{\otimes 2} \otimes U(R)$, with
$\ell_{ij}^{+} = 0 = \ell_{ji}^{-}$ if $i < j$. There is a Hopf algebra structure on $U(R)$ given by
\begin{equation}\label{eq:U(R)_Hopf}
  \Delta(\ell_{ij}^{\pm}) = \sum_{1\leq k\leq N} \ell_{ik}^{\pm} \otimes \ell_{kj}^{\pm}, \qquad
  \epsilon(\ell_{ij}^{\pm}) = \delta_{ij},\qquad
  S\left (L^{\pm}\right ) = (L^{\pm})^{-1}\qquad \text{for all}\ \qquad 1 \le i,j \le N.
\end{equation}
The matrices $L^{\pm}$ are invertible because $L^{+}$ (resp.\ $L^{-}$) is lower (resp.\ upper) triangular, with invertible entries on the diagonal.
Henceforth, we shall denote the $(i,j)$-th entry of $(L^{\pm})^{-1}$ by $\wtd{\ell}_{ij}^{\pm}$ for all $1 \le i,j \le N$.

As in Subsection~\ref{ssec:algebra_A(R)}, we now specialize to the case $R = R_{q} = \hat{R}_{q} \circ \tau$, where $\hat{R}_{q}$ is one of the matrices
from equations~\eqref{eq:1_parameter_R_B},~\eqref{eq:1_parameter_R_C}, or~\eqref{eq:1_parameter_R_D}. Then $R_{q}$ is an $N \times N$-matrix satisfying
the Yang-Baxter equation~\eqref{eq:YBE}, where $N = 2n + 1$ in type $B_{n}$, and $N = 2n$ in types $C_{n}$ and $D_{n}$.

Below, we shall work with several quotients of $U(R_{q})$. First, we define the algebra $U'(R_{q}) = U(R_{q})/J$, where $J$ is the ideal of $U(R_{q})$
generated by the elements $\{(\ell_{ii}^{\pm})^{-1} - \ell_{ii}^{\mp}\}_{i=1}^N$. As the elements $\ell_{ii}^{\pm}$ are grouplike, it follows that
$U'(R_{q})$ is also a Hopf algebra with the coproduct, counit, and antipode given by~\eqref{eq:U(R)_Hopf}. Before introducing other quotient algebras
we need, let us set up some notation. For $BCD$-types, we define sequences $(\rho_{i})_{i=1}^{N}$ and corresponding matrices
$C_{q} = (c_{ij})_{i,j = 1}^{N}$ as follows (cf.~\cite[(1.11)]{JLM1},~\cite[(1.14)]{JLM2}):
\begin{itemize}

\item Type $B_{n}$:\footnote{The differences with~\cite[(1.11)]{JLM1} are due to some slightly differing conventions,
see Remark~\ref{rmk:JLM_comparison_B} for more details.}
\begin{equation}\label{eq:rho_B}
  (\rho_{1},\ldots ,\rho_{N}) =
  \left( n - \frac{1}{2},n - \frac{3}{2},\ldots ,\frac{1}{2},\frac{1}{2},-\frac{1}{2},-\frac{3}{2},\ldots ,-n + \frac{1}{2} \right),
\end{equation}
\begin{equation}\label{eq:Cq_B}
  c_{ij} = \delta_{ij}q^{2\rho_{i}}\qquad \text{for all}\qquad 1 \le i,j\le N.
\end{equation}

\item Type $C_{n}$:
\[
  (\rho_{1},\ldots ,\rho_{N}) = \left (n, n-1,\ldots ,1,-1,\ldots ,-n + 1,-n\right ),
\]
\begin{equation}\label{eq:Cq_C}
  c_{ij} = \delta_{ij}\sigma_{i}q^{\rho_{i}}\qquad \text{for all}\qquad 1 \le i,j \le N,
\end{equation}
where $\sigma_{i} = 1$ for $i \le n$, and $\sigma_{i} = -1$ if $i > n$.

\item Type $D_{n}$:
\[
  (\rho_{1},\ldots ,\rho_{N}) = \left( n-1,n-2,\ldots,1,0,0,-1,\ldots ,-n+2,-n+1 \right),
\]
\begin{equation}\label{eq:Cq_D}
  c_{ij} = \delta_{ij}q^{\rho_{i}}\qquad \text{for all}\qquad 1 \le i,j \le N.
\end{equation}

\end{itemize}
Now, let $\cal{I}_{q}$ (resp.\ $\cal{I}_{q}'$) be the ideal in $U(R_{q})$ (resp.\ $U'(R_{q})$) generated by the entries of the matrix
\begin{equation}\label{eq:LC_relations}
  (L^{\pm})C_{q}(L^{\pm})^{\mathsf{t}'}C_{q}^{-1} - I
\end{equation}
where $A^{\mathsf{t}'} = (a_{j'i'})_{i,j = 1}^{N}$ for any $N \times N$ matrix $A = (a_{ij})_{i,j = 1}^{N}$, cf.~\cite[(1.10)]{JLM1}.
Explicitly, one can check that $\cal{I}_{q}$ and $\cal{I}_{q}'$ are generated by the following elements:
\begin{equation}\label{eq:LC_explicit_1}
  \sum_{1\leq k\leq N}c_{kk}c_{jj}^{-1}\ell_{ik}^{\pm}\ell_{j'k'}^{\pm} -\delta_{ij}\qquad \text{for all}\qquad 1 \le i,j \le N.
\end{equation}

It turns out that the ideals $\cal{I}_{q}$ and $\cal{I}_{q}'$ are each generated by two elements:

\begin{prop}\label{prop:central_element}
In the algebras $U(R_{q})$ and $U'(R_{q})$, we have the following equalities:
\begin{equation}\label{eq:central_element}
  \sum_{1\leq k\leq N}c_{ii}c_{kk}^{-1}\ell_{k'i'}^{\pm}\ell_{ki}^{\pm}
  = \sum_{1\leq k\leq N}c_{jj}^{-1}c_{kk}\ell_{jk}^{\pm}\ell_{j'k'}^{\pm}\qquad \text{for all}\qquad 1 \le i,j \le N.
\end{equation}
If $z_{q}^{\pm}$ denote the common values above, then $z_{q}^{\pm} - 1$ generate the ideal $\cal{I}_{q} \subset U(R_{q})$
(resp.\ $\cal{I}_{q}' \subset U'(R_{q})$). Furthermore, $z_{q}^{\pm}$ are grouplike.
\end{prop}

\begin{proof}
We shall only consider type $B_{n}$, since the proof is essentially the same for types $C_{n}$ and $D_{n}$.
First, we recall that $\hat{R}_{q} = R_{q} \circ \tau$ satisfies the cubic equation
$(\hat{R}_{q} - q^{-2})(\hat{R}_{q} + q^{2})(\hat{R}_{q} - q^{4n}) = 0$ by \cite[Lemma 4.8]{MT1}. This implies that
we can write $\hat{R}_{q} = q^{-2}P_{1} - q^{2}P_{2} + q^{4n}P_{3}$, where $P_{1}$ (resp.\ $P_{2}$, $P_{3}$) is the
projection onto the $q^{-2}$-eigenspace (resp.\ $-q^{2}$-eigenspace, $q^{4n}$-eigenspace) of $V \otimes V$ for $\hat{R}_{q}$.
As $(q^{-2} + q^{2})(q^{-2} - q^{4n})(q^{4n} + q^{2}) \neq 0$, it follows from the Vandermonde determinant formula that
each $P_{i}$ may be written as a linear combination of $\mathrm{Id}$, $\hat{R}_{q}$, and $\hat{R}_{q}^{2}$. On the other hand,
using the fact that the $q^{4n}$-eigenspace of $R_{q}$ is spanned by the vector
\[
  w = \sum_{1\leq i\leq n+1}q^{2(i - 1)}v_{i} \otimes v_{i'} + \sum_{1\leq i\leq n}q^{4n - 2i}v_{i'} \otimes v_{i}
\]
(cf.~\cite[Lemma 4.8, formula (3.2)]{MT1}), one can check that the matrix
\begin{equation}\label{eq:K_matrix}
  K = \sum_{1 \le i,j\le N}q^{2(\rho_{i} - \rho_{j})}E_{i'j} \otimes E_{ij'}
  = \sum_{1 \le i,j \le N}c_{ii}c_{jj}^{-1}E_{i'j} \otimes E_{ij'}
\end{equation}
is a multiple of $P_{3}$, and hence $K$ is a quadratic polynomial in $\hat{R}_{q}$. Moreover, due to the fact that $\hat{R}_{q} = R_q \circ \tau$,
one can deduce from the relations~\eqref{eq:RLL_relations} that $\hat{R}_{q}L_{2}^{\pm}L_{1}^{\pm} = L_{2}^{\pm}L_{1}^{\pm} \hat{R}_{q}$.
This equality implies that $\hat{R}_{q}^{m}L_{2}^{\pm}L_{1}^{\pm} = L_{2}^{\pm}L_{1}^{\pm}\hat{R}_{q}^{m}$ for all $m \ge 0$, and hence
$KL_{2}^{\pm}L_{1}^{\pm} = L_{2}^{\pm}L_{1}^{\pm}K$. It follows from these equalities that
\begin{equation}\label{eq:KLL_relation_1}
  \sum_{1\leq k\leq N}c_{kk}\ell_{ik}^{\pm}\ell_{jk'}^{\pm} = 0\qquad \text{for all}\qquad i \neq j',
\end{equation}
and
\begin{equation}\label{eq:KLL_relation_2}
  \sum_{1\leq k\leq N}c_{jj}c_{kk}^{-1}\ell_{k'i'}^{\pm}\ell_{ki}^{\pm}
  = \sum_{1\leq k\leq N}c_{kk}c_{ii}^{-1}\ell_{jk}^{\pm}\ell_{j'k'}^{\pm}
  \qquad \text{for all}\qquad 1 \le i,j\le N.
\end{equation}
It is easy to see that the equality~\eqref{eq:KLL_relation_2} is equivalent to~\eqref{eq:central_element}.
Moreover, comparing~\eqref{eq:KLL_relation_1} to~\eqref{eq:LC_explicit_1}, we find that the off-diagonal entries of the matrix
$L^{\pm}C_{q}(L^{\pm})^{\mathsf{t}'}C_{q}^{-1} - I$ are all $0$ in $U(R_{q})$ (resp.\ $U'(R_q)$), while each diagonal entry of
$L^{\pm}C_{q}(L^{\pm})^{\mathsf{t}'}C_{q}^{-1} - I$ is equal to $z_{q}^{\pm} - 1$ due to~\eqref{eq:central_element}. Finally,
we show that $z_{q}^{\pm}$ are grouplike by directly applying $\Delta$ to the right-hand side of~\eqref{eq:central_element} and
using~\eqref{eq:KLL_relation_1}:
\begin{equation}\label{eq:z-coproduct}
\begin{split}
  \Delta&(z_{q}^{\pm})
  = \sum_{1\leq k\leq N}c_{jj}^{-1}c_{kk}\Delta(\ell_{jk}^{\pm})\Delta(\ell_{j'k'}^{\pm})
   \ = \sum_{1 \le k,p_{1},p_{2} \le N} c_{jj}^{-1}c_{kk}\ell_{jp_{1}}^{\pm}\ell_{j'p_{2}}^{\pm} \otimes \ell_{p_{1}k}^{\pm}\ell_{p_{2}k'}^{\pm} \\
  &= \sum_{1 \le p_{1},p_{2} \le N} c_{jj}^{-1}\ell_{jp_{1}}^{\pm}\ell_{j'p_{2}}^{\pm} \otimes
      \left (\sum_{1\leq k\leq N} c_{kk}\ell_{p_{1}k}^{\pm}\ell_{p_{2}k'}^{\pm}\right )
   \overset{\eqref{eq:KLL_relation_1}}{=} \sum_{1\leq p\leq N} c_{jj}^{-1}\ell_{jp}^{\pm}\ell_{j'p'}^{\pm} \otimes \left
    (\sum_{1\leq k\leq N} c_{kk}\ell_{pk}^{\pm}\ell_{p'k'}^{\pm}\right ) \\
  &\overset{\eqref{eq:central_element}}{=} \sum_{1\leq p\leq N} c_{jj}^{-1}c_{pp}\ell_{jp}^{\pm}\ell_{j'p'}^{\pm} \otimes z_{q}^{\pm}
   = z_{q}^{\pm} \otimes z_{q}^{\pm}.
\end{split}
\end{equation}
This completes the proof.
\end{proof}

We now introduce the quotients $\wtd{U}(R_{q}) = U(R_{q})/\cal{I}_{q}$ and $\wtd{U}'(R_{q}) = U'(R_{q})/\cal{I}_{q}'$. The above result shows
that $\cal{I}_{q}$ and $\cal{I}_{q}'$ are both Hopf ideals, so that $\wtd{U}(R_{q})$ and $\wtd{U}'(R_{q})$ are Hopf algebras. We also note that
\begin{equation}\label{eq:diagonal_reflection}
  \ell_{i'i'}^{\pm} \ell_{ii}^{\pm} = z^\pm_q = 1 \ \ \mathrm{in} \ \ \wtd{U}(R_{q})
\end{equation}
for any $i$, which can be easily deduced from~\eqref{eq:LC_explicit_1}, along with the fact that $\ell_{ij}^{+} = \ell_{ji}^{-} = 0$ whenever $i < j$.

Then we have the following analogue of Proposition~\ref{prop:A(R)_grading}:

\begin{prop}\label{prop:U(R)_grading}
If $\hat{R}_{q}$ is the matrix of~\eqref{eq:1_parameter_R_B} (resp.~\eqref{eq:1_parameter_R_C},~\eqref{eq:1_parameter_R_D}),
and $R_{q} = \hat{R}_{q} \circ \tau$, then $U(R_{q})$ and $\wtd{U}(R_{q})$ are both $P$-bigraded Hopf algebras with respect to the assignment
$\deg(\ell_{ij}^{\pm}) = (-\varepsilon_{i},\varepsilon_{j})$, where $P$ is the weight lattice of type $B_{n}$ (resp.\ $C_{n}$, $D_{n}$),
and $\varepsilon_{i'} = -\varepsilon_{i}$ for $1 \le i \le n$ (and $\varepsilon_{n+1}=0$ in type $B_n$).
\end{prop}

\begin{proof}
To prove that $U(R_{q})$ is a $P$-bigraded bialgebra, we can apply arguments similar to those used in the proof of Proposition~\ref{prop:A(R)_grading}.
To show that it is a $P$-bigraded Hopf algebra, it suffices to show that $\deg(\wtd{\ell}_{ij}^{\pm}) = (\varepsilon_{j},-\varepsilon_{i})$.
We only present the proof for $\wtd{\ell}_{ij}^{+}$, as the arguments for $\wtd{\ell}_{ij}^{-}$ are similar. To prove the claim,
we proceed by induction on $i - j$. The base case $i=j$ follows from $\wtd{\ell}_{ii}^{+} = (\ell_{ii}^{+})^{-1}$ for all $i$.
For the induction step, suppose that $i - j = m > 0$ and that $\deg(\wtd{\ell}_{kt}^{+}) = (\varepsilon_{t},-\varepsilon_{k})$
if $0 \le k - t < m$. Using $L^{+}(L^{+})^{-1} = I$, we get
\[
  \wtd{\ell}_{ij}^{+} = -(\ell_{ii}^{+})^{-1}\sum_{j\leq k\leq i-1}\ell_{ik}^{+}\wtd{\ell}_{kj}^{+}.
\]
By the induction hypothesis,
  $\deg((\ell_{ii}^{+})^{-1} \ell_{ik}^{+}\wtd{\ell}_{kj}^{+})
   = (\varepsilon_{i},-\varepsilon_{i}) + (-\varepsilon_{i},\varepsilon_{k}) + (\varepsilon_{j},-\varepsilon_{k})
   = (\varepsilon_{j},-\varepsilon_{i})$
for $j \le k < i$. Hence, $\wtd{\ell}_{ji}^{+}$ is homogeneous of the same bidegree.

To prove the claim for $\wtd{U}(R_{q})$, we need to prove additionally that $\cal{I}_{q}$ is homogeneous. Since $\varepsilon_{i'} = -\varepsilon_{i}$
for all $1 \le i \le n+1$, we have $\deg(\ell_{k'i'}^{\pm}\ell_{ki}^{\pm}) = (-\varepsilon_{k'}-\varepsilon_{k},\varepsilon_{i'}+\varepsilon_{i}) = (0,0)$
for all $1 \le i,k \le N$. Thus, the elements $z_{q}^{\pm}$ of~\eqref{eq:central_element} are homogeneous of degree $(0,0)$, which completes the proof as
$z_{q}^{\pm} - 1$ generate $\cal{I}_{q}$ by Proposition~\ref{prop:central_element}.
\end{proof}

The above result implies that we can form algebras $U(R_{q})_{\zeta}$ and $\wtd{U}(R_{q})_{\zeta}$ as in~\eqref{eq:twisted_product_general}.
On the other hand, if $D$ is the matrix of~\eqref{eq:D_zeta}, we may also consider algebras $U(DR_{q}D)$ and $\wtd{U}(DR_{q}D)$, where the latter
is the quotient of $U(DR_{q}D)$ by the ideal generated by the entries of the matrices~\eqref{eq:LC_relations}. We then get the following analogue
of Proposition~\ref{prop:A(R)_twisted}:

\begin{prop}\label{prop:U(R)_twisted}
There is a unique Hopf algebra isomorphism $\Psi\colon \wtd{U}(DR_{q}D) \iso \wtd{U}(R_{q})_{\zeta}$ mapping $\ell_{ij}^{\pm} \mapsto \ell_{ij}^{\pm}$.
\end{prop}

\begin{proof}
We first show that there is a Hopf algebra isomorphism $\Psi_{0}\colon U(DR_{q}D) \to U(R_{q})_{\zeta}$ mapping
$\ell_{ij}^{\pm} \mapsto \ell_{ij}^{\pm}$. The proof that $\Psi_{0}$ is a bialgebra isomorphism is analogous to the proof
of Proposition~\ref{prop:A(R)_twisted}. Then $\Psi_{0}$ is automatically an isomorphism of Hopf algebras. Indeed, a general argument
then shows that $\Psi_{0}^{-1}\circ S\circ \Psi_{0}$ is an antipode for $U(DR_{q}D)$, so the uniqueness of the antipode then implies
that $S \circ \Psi_{0} = \Psi_{0} \circ S$.

Next, since the quotient map $\pi\colon U(R_{q}) \to \wtd{U}(R_{q})$ also defines an algebra homomorphism
$\pi\colon U(R_{q})_{\zeta} \to \wtd{U}(R_{q})_{\zeta}$, it follows that $\Psi^{-1}_{0}$ induces a Hopf algebra isomorphism
$\wtd{U}(R_{q})_{\zeta} \iso U(DR_{q}D)/\Psi_{0}^{-1}(\cal{I}_{q})$ sending $\ell_{ij}^{\pm} \mapsto \ell_{ij}^{\pm}$.
Since $\zeta(\varepsilon_{i},\varepsilon_{i'})^{-1} = 1$, it follows that applying $\Psi_{0}^{-1}$ to the elements
in~\eqref{eq:LC_explicit_1}  multiplies them by $\zeta(\varepsilon_{i},\varepsilon_{j'})^{-1}$, and therefore $\Psi_{0}^{-1}(\cal{I}_{q})$
is also the ideal generated by the entries of the matrix~\eqref{eq:LC_relations}. Thus, $\Psi_{0}$ induces a Hopf algebra isomorphism
$\Psi\colon \wtd{U}(DR_{q}D) \iso \wtd{U}(R_{q})_{\zeta}$, as claimed.
\end{proof}

Next, let $s_{1},\ldots ,s_{N} \in \bb{K} \setminus \{0\}$ and set $S' = \sum_{i=1}^{N} s_{i}E_{ii}$, so that $S = S' \otimes S'$
is a matrix of the form~\eqref{eq:S_matrix}. We also set $\bar{S}' = \sum_{i=1}^{N} s_{i'}E_{ii}$. Then we have the following analogue
of Proposition~\ref{prop:diagonal_conjugation_A(R)}:

\begin{prop}\label{prop:diagonal_conjugation_U(R)}
(a) There is a unique Hopf algebra isomorphism $\Phi_S\colon U(R) \iso U(SRS^{-1})$ mapping $\ell_{ij}^{\pm}\mapsto s^{-1}_{i}s_{j}\ell_{ij}^{\pm}$.

\medskip
\noindent
(b) There is a unique Hopf algebra isomorphism $\wtd{\Phi}_{S}\colon \wtd{U}(R_{q}) \iso \wtd{U}(SR_{q}S^{-1})$ mapping
$\ell_{ij}^{\pm} \mapsto s_{i}^{-1}s_{j}\ell_{ij}^{\pm}$, where $\wtd{U}(SR_{q}S^{-1}) = U(SR_{q}S^{-1})/\cal{I}_{S}$
with the ideal $\cal{I}_{S}\subseteq U(SR_{q}S^{-1})$ generated by the entries of the matrix
\begin{equation}\label{eq:LCS_relations}
  (L^{\pm})(S'C_{q}\bar{S}')(L^{\pm})^{\mathsf{t}'}(S'C_{q}\bar{S}')^{-1} - I .
\end{equation}
\end{prop}

\begin{proof}
(a) An argument analogous to the proof of Proposition~\ref{prop:diagonal_conjugation_A(R)} shows that $\Phi_{S}$ is a bialgebra isomorphism.
Arguing as in the proof of Proposition~~\ref{prop:U(R)_twisted}, it is automatically a Hopf algebra isomorphism.

(b) The isomorphism $\Phi_{S}$ induces a Hopf algebra isomorphism $\wtd{\Phi}_{S}\colon \wtd{U}(R_{q}) \iso U(SR_{q}S^{-1})/\Phi_{S}(\cal{I}_{q})$,
so we just need to show that $\Phi_{S}(\cal{I}_{q}) = \cal{I}_{S}$. Applying $\Phi_{S}$ to~\eqref{eq:LC_explicit_1}, we find that
$\Phi_{S}(\cal{I}_{q})$ is generated by the elements
\[
  \sum_{1\leq k\leq N}c_{kk}c_{jj}^{-1}s_{i}^{-1}s_{k}s_{j'}^{-1}s_{k'}\ell_{ik}^{\pm}\ell_{j'k'}^{\pm} - \delta_{ij} =
  s_{i}^{-1}s_{j}\sum_{1\leq k\leq N}s_{k}s_{k'}c_{kk}s_{j}^{-1}s_{j'}^{-1}c_{jj}^{-1}\ell_{ik}^{\pm}\ell_{j'k'}^{\pm} - \delta_{ij}
\]
for all $1 \le i,j \le N$. It is easy to check that these elements are multiples of the corresponding $(i,j)$-th entries of
the matrix~\eqref{eq:LCS_relations}. This completes the proof.
\end{proof}

Combining the previous two results, we can relate $U(R_{q})$ and $\wtd{U}(R_{q})$ to their two-parameter analogues,
as in Corollary~\ref{cor:A(R)_2_vs_1_parameter}. Below, we shall use the notation introduced in the paragraph preceding
Corollary~\ref{cor:A(R)_2_vs_1_parameter}.

First, we recall that $R_{r,s} =  S^{-1}DR_{q}DS$, see~\eqref{eq:R-relation}, where $S = \sum_{i,j=1}^{N} \psi_{i}\psi_{j}E_{ii} \otimes E_{jj}$,
with $\psi_{i}$ given by either~\eqref{eq:psi_values_B},~\eqref{eq:psi_values_C}, or~\eqref{eq:psi_values_D} in types $B_{n}$, $C_{n}$,
or~$D_{n}$, respectively. Then we define matrices $C_{r,s} = (\wtd{c}_{ij})_{i,j = 1}^{N}$ as follows:

\begin{itemize}

\item Type $B_{n}$:
\begin{equation}\label{eq:Crs_B}
  \tilde{c}_{ij}
  = \delta_{ij}\psi_{i}^{-1}\psi_{j'}^{-1}r^{\rho_{i}}s^{-\rho_{i}}
  = \delta_{ij}
    \begin{cases}
      r^{\frac{1}{2}}s^{2(i - n) - \frac{1}{2}} & \text{if}\ 1 \le i \le n \\
      r^{\frac{1}{2}}s^{-\frac{1}{2}} & \text{if}\ i = n +1 \\
      r^{2(n + 2 - i) - \frac{1}{2}}s^{\frac{1}{2}} & \text{if}\ n + 2 \le i \le 2n + 1
    \end{cases}\,,
\end{equation}

\item Type $C_{n}$:
\begin{equation}\label{eq:Crs_C}
  \wtd{c}_{ij}
  = \delta_{ij}\psi_{i}^{-1}\psi_{j'}^{-1}\sigma_{i}r^{\frac{1}{2}\rho_{i}}s^{-\frac{1}{2}\rho_{i}}
  = \delta_{ij}\cdot
    \begin{cases}
      r^{\frac{1}{2}}s^{i-n-\frac{1}{2}} & \text{if}\ i \le n \\
      -r^{n - i + \frac{1}{2}}s^{\frac{1}{2}} & \text{if}\ i > n
    \end{cases}\,,
\end{equation}

\item Type $D_{n}$:
\begin{equation}\label{eq:Crs_D}
  \wtd{c}_{ij}
  = \delta_{ij}\psi_{i}^{-1}\psi_{j'}^{-1}r^{\frac{1}{2}\rho_{i}}s^{-\frac{1}{2}\rho_{i}}
  = \delta_{ij}\cdot
    \begin{cases}
      r^{\frac{1}{2}}s^{i-n+\frac{1}{2}} & \text{if}\ i \le n \\
      r^{n + \frac{3}{2} - i}s^{\frac{1}{2}} & \text{if}\ i > n
    \end{cases}\,.
\end{equation}
\end{itemize}
We now define $\wtd{U}(R_{r,s}) = U(R_{r,s})/\cal{I}_{r,s}$, where $\cal{I}_{r,s}$ is the ideal generated by the entries of the matrices
\begin{equation}\label{eq:LC_relations_2_parameter}
  (L^{\pm})C_{r,s}(L^{\pm})^{\mathsf{t}'}C_{r,s}^{-1} - I .
\end{equation}
Then, we get the following analogue of Corollary~\ref{cor:A(R)_2_vs_1_parameter}:

\begin{cor}\label{cor:U(R)_2_vs_1_parameter}
There is a unique Hopf algebra isomorphism $\wtd{\Psi}\colon \wtd{U}(R_{r,s}) \iso \wtd{U}(R_{q})_{\zeta}$ mapping
$\ell_{ij}^{\pm} \mapsto \psi^{-1}_{i}\psi_{j}\ell_{ij}^{\pm}$.
\end{cor}

\begin{proof}
By Proposition~\ref{prop:U(R)_twisted}, we get a Hopf algebra isomorphism $\Psi\colon \wtd{U}(DR_{q}D) \to \wtd{U}(R_{q})_{\zeta}$
mapping $\ell_{ij}^{\pm} \mapsto \ell_{ij}^{\pm}$. Moreover, by definition, we have $C_{r,s} = (S')^{-1}C_{q}(\bar{S}')^{-1}$ for
$S' = \sum_{i=1}^{N} \psi_{i}E_{ii}$ and $\bar{S}' = \sum_{i=1}^{N} \psi_{i'}E_{ii}$. Thus, it follows from
Proposition~\ref{prop:diagonal_conjugation_U(R)}(b) (with $C_{r,s}$ instead of $C_q$, see formula~\eqref{eq:LCS_relations}) that
there is a Hopf algebra isomorphism
$\wtd{\Phi}_{S^{-1}} = \wtd{\Phi}_{S}^{-1}\colon \wtd{U}(DR_{q}D) \iso \wtd{U}(R_{r,s})$ mapping
$\ell_{ij}^{\pm} \mapsto \psi_{i}\psi_{j}^{-1}\ell_{ij}^{\pm}$
(we can apply Proposition~\ref{prop:diagonal_conjugation_U(R)}(b) with $\wtd{U}(DR_{q}D)$ in place of $\wtd{U}(R_{q})$ as $\wtd{U}(DR_{q}D)$
is also formed by killing the elements~\eqref{eq:LC_explicit_1}). Thus,
$\wtd{\Psi} = \Psi \circ \wtd{\Phi}_{S}\colon \wtd{U}(R_{r,s}) \iso \wtd{U}(R_{q})_{\zeta}$ is a Hopf algebra isomorphism mapping
$\ell_{ij}^{\pm} \mapsto \psi_{i}^{-1}\psi_{j}\ell_{ij}^{\pm}$.
\end{proof}

\begin{remark}\label{rmk:square_root_zq}
In type $B_{n}$, the $i=n+1$ case of~\eqref{eq:diagonal_reflection} yields $(\ell_{n+1,n+1}^{\pm})^{2} = z_{q}^{\pm} = 1$ in
$\wtd{U}(R_{q}), \wtd{U}'(R_{q}), \wtd{U}(R_{r,s})$. However, to obtain isomorphisms between these algebras and the corresponding Drinfeld-Jimbo
quantum groups, one actually needs $\ell_{n+1,n+1}^{\pm} = 1$. As such, for the remainder of this section we shall impose these additional relations
in each of the algebras $\wtd{U}(R_{q})$, $\wtd{U}'(R_{q})$, $\wtd{U}(R_{r,s})$:
\begin{equation}\label{eq:rel_B_extra}
  \ell_{n+1,n+1}^{\pm} = 1.
\end{equation}
It is easily verified that Propositions~\ref{prop:U(R)_grading},~\ref{prop:U(R)_twisted},~\ref{prop:diagonal_conjugation_U(R)}(b),
and Corollary~\ref{cor:U(R)_2_vs_1_parameter} still hold with~\eqref{eq:rel_B_extra} imposed.
\end{remark}

   %%%%%%%%%%%%%%%%%%%%%%%%%%%%%%%%%%%%%%%%%%%%%%%%%%%%%%%%%%%%%%%%%%%%%%%%%%

\subsection{Triangular decomposition and the key embedding: finite FRT}
\

We shall now develop a couple of structural results for $\wtd{U}(R_{q})$ and $\wtd{U}'(R_{q})$, which will ultimately be needed to obtain
an analogue of Proposition~\ref{prop:double_cartan_DJ}. In what follows, we shall need the following elements:
\begin{equation}\label{eq:cal-EF-generators}
  \cal{E}_{i} = \ell_{i + 1,i}^{+}(\ell_{ii}^{+})^{-1}, \qquad
  \cal{F}_{i} = (\ell_{ii}^{-})^{-1}\ell_{i,i + 1}^{-} \qquad \forall\, 1\leq i\leq n.
\end{equation}
Let us also introduce the following subalgebras of $\wtd{U}(R_{q})$ and $\wtd{U}'(R_{q})$:
\begin{itemize}

\item $\wtd{U}_{\ge}(R_{q}) \subset \wtd{U}(R_{q})$ and $\wtd{U}'_{\ge}(R_{q}) \subset \wtd{U}'(R_{q})$,
both generated by all $\ell_{ij}^{+}$ with $i \ge j$,

\item $\wtd{U}_{\le}(R_{q}) \subset \wtd{U}(R_{q})$ and $\wtd{U}'_{\le}(R_{q}) \subset \wtd{U}'(R_{q})$,
both generated by all $\ell_{ij}^{-}$ with $i \le j$,

\item $\wtd{U}_{+}(R_{q}) \subset \wtd{U}(R_{q})$ and $\wtd{U}'_{+}(R_{q}) \subset \wtd{U}'(R_{q})$,
both generated by all $\cal{E}_{i}$ of~\eqref{eq:cal-EF-generators} with $1 \le i \le n$,

\item $\wtd{U}_{-}(R_{q}) \subset \wtd{U}(R_{q})$ and $\wtd{U}'_{-}(R_{q}) \subset \wtd{U}'(R_{q})$,
both generated by all $\cal{F}_{i}$ of~\eqref{eq:cal-EF-generators} with $1 \le i \le n$,

\item $\wtd{U}_{0}(R_{q}) \subset \wtd{U}(R_{q})$ and $\wtd{U}'_{0}(R_{q}) \subset \wtd{U}'(R_{q})$,
generated by all $(\ell_{ii}^{+})^{\pm 1},(\ell_{ii}^{-})^{\pm 1}$ with $1 \le i \le N$,

\item $\wtd{U}_{0,+}(R_{q}) \subset \wtd{U}(R_{q})$, generated by all $(\ell_{ii}^{+})^{\pm 1}$ with $ 1 \le i \le N$,

\item $\wtd{U}_{0,-}(R_{q}) \subset \wtd{U}(R_{q})$, generated by all $(\ell_{ii}^{-})^{\pm 1}$ with $ 1 \le i \le N$.

\end{itemize}
We first prove the following result:

\begin{lemma}\label{lem:cartan_subalgebras_RTT}
There is an isomorphism from the Laurent polynomial algebra $H = \bb{K}[x_{1}^{\pm 1},\ldots ,x_{n}^{\pm 1}]$ to $\wtd{U}_{0,\pm}(R_{q})$
defined by $x_{i} \mapsto \ell_{ii}^{\pm}$.
\end{lemma}

\begin{proof}
Let us first show that the subalgebras $\wtd{U}_{0,\pm}(R_{q})$ are commutative. Due to~\eqref{eq:diagonal_reflection}, it suffices to show
that $\ell_{ii}^{\pm}$ commutes with $\ell_{jj}^{\pm}$ whenever $i < j \ne i'$. To do so, we consider the $(m,p) = (i,j)$ case of the
defining relations (cf.~\eqref{eq:FRT_relations})
\begin{equation}\label{eq:RLL_relations_explicit}
  \sum_{1 \le k,t \le N}(R_{q})^{mp}_{kt}\ell_{ki}^{\pm}\ell_{tj}^{\pm} \ = \sum_{1 \le k,t \le N}(R_{q})_{ij}^{kt}\ell_{pt}^{\pm}\ell_{mk}^{\pm},
\end{equation}
and using the aforementioned fact that $(R_{q})^{ij}_{k\ell} = 0$ unless $\{i,j\} = \{k,\ell\}$, we obtain
\[
  (R_{q})_{ij}^{ij}\ell_{ii}^{\pm}\ell_{jj}^{\pm} + (R_{q})_{ji}^{ij}\ell_{ji}^{\pm}\ell_{ij}^{\pm} =
  (R_{q})_{ij}^{ij}\ell_{jj}^{\pm}\ell_{ii}^{\pm} + (R_{q})_{ij}^{ji}\ell_{ji}^{\pm}\ell_{ij}^{\pm},
\]
which reduces to the claimed equality $\ell_{ii}^{\pm}\ell_{jj}^{\pm} = \ell_{jj}^{\pm}\ell_{ii}^{\pm}$, due to $(R_{q})_{ij}^{ij} \neq  0$ and
$\ell_{ij}^{+} = 0 = \ell_{ji}^{-}$ whenever $i < j$.

Thus, because $(\ell_{ii}^{\pm})^{-1} = \ell_{i'i'}^{\pm} \in \wtd{U}_{0,\pm}$ by~\eqref{eq:diagonal_reflection}, there are algebra homomorphisms
$\phi_{\pm}\colon H \to \wtd{U}_{0,\pm}(R_{q})$ sending $x_{i} \mapsto \ell_{ii}^{\pm}$. On the other hand, we can define an algebra homomorphism
$\phi'\colon \wtd{U}(R_{q}) \to H$ by $\phi'(u) = \epsilon(u)x_{-\lambda}$ for all $u \in \wtd{U}(R_{q})_{\lambda,\lambda'}$, where
$x_{\lambda} = x_{1}^{c_{1}}\ldots x_{n}^{c_{n}}$ for any $\lambda = \sum_{i = 1}^{n}c_{i}\varepsilon_{i}$.
Then the restriction of $\phi'$ to $\wtd{U}_{0,\pm}$ is the inverse to the corresponding map $\phi_{\pm}$ (note that this uses
the relation~\eqref{eq:rel_B_extra} in type $B_{n}$), and hence $\phi_{\pm}$ are algebra isomorphisms.

This completes the proof.
\end{proof}

As a consequence, the elements $\ell_{\mu}^{\pm} = (\ell_{11}^{\pm})^{c_{1}}\ldots (\ell_{nn}^{\pm})^{c_{n}}$, ranging over all
$\mu = \sum_{i = 1}^{n}c_{i}\varepsilon_{i} \in \bb{Z}\varepsilon_{1} \oplus \cdots \oplus \bb{Z}\varepsilon_{n}$, form bases for
$\wtd{U}_{0,\pm}(R_{q})$. We also have the following result:

\begin{lemma}\label{lem:borel_spanning_sets}
The elements
  $\{\ell_{\mu}^{+}\cal{E}_{i_{1}}\ldots \cal{E}_{i_{k}} \,|\,
     k \in \bb{Z}_{\ge 0},\, 1 \le i_{p} \le n,\, \mu \in \bb{Z}\varepsilon_{1} \oplus \cdots \oplus \bb{Z}\varepsilon_{n}\}$
span $\wtd{U}_{\ge}(R_q)$, while the elements
  $\{\cal{F}_{i_{1}}\ldots \cal{F}_{i_{k}}\ell_{\mu}^{-} \,|\,
     k \in \bb{Z}_{\ge 0},\, 1 \le i_{p} \le n,\, \mu \in \bb{Z}\varepsilon_{1} \oplus \cdots \oplus \bb{Z}\varepsilon_{n}\}$
span $\wtd{U}_{\le}(R_{q})$.
\end{lemma}

\begin{proof}
First, we shall prove that
\begin{equation}\label{eq:RTT_q_commute}
  \ell_{ii}^{\pm}\ell_{kj}^{\pm} = v^{\pm (\delta_{ij} - \delta_{ik})}\ell_{kj}^{\pm}\ell_{ii}^{\pm}
  \qquad \text{for all}\qquad k \neq j,\qquad i \neq k',j',
\end{equation}
where $v = q^{-2}$ in type $B_{n}$ and $v = q^{-1}$ in types $C_{n}$ and $D_{n}$.
If $i = k'$ or $i = j'$ and $i \neq \frac{N + 1}{2}$, we may combine the equality above with~\eqref{eq:diagonal_reflection} to get
$\ell_{ii}^{\pm}\ell_{kj}^{\pm} = v^{\pm (\delta_{i'k} - \delta_{i'j})}\ell_{kj}^{\pm}\ell_{ii}^{\pm}$.
Finally, if $i = k'$ or $i = j'$ with $i = \frac{N + 1}{2}$, then we are in type $B_{n}$ and $\ell_{ii}^{\pm} = \ell_{n + 1,n+1}^{\pm} = 1$,
so that $\ell_{ii}^{\pm}\ell_{kj}^{\pm} = \ell_{kj}^{\pm}\ell_{ii}^{\pm}$. Thus, because $\deg(\ell_{kj}^{\pm}) = (-\varepsilon_{k},\varepsilon_{j})$,
$\deg(\ell_{ii}^{\pm}) = (-\varepsilon_{i},\varepsilon_{i})$ and $\varepsilon_{t'} = -\varepsilon_{t}$ for all $t$, it follows that
\begin{equation}\label{eq:RTT_q_commute_general}
  \ell_{\mu}^{+}u = q^{-(\mu,\lambda + \lambda')}u\ell_{\mu}^{+},\qquad \ell_{\mu}^{-}u'=q^{(\mu,\lambda + \lambda')}u'\ell_{\mu}^{-}
  \qquad \text{for all}\qquad u \in \wtd{U}_{\ge}(R_{q})_{\lambda,\lambda'},\ u' \in \wtd{U}_{\le}(R_{q})_{\lambda,\lambda'}.
  \footnote{Recall that, in type $B_{n}$, the form $(\cdot,\cdot)$ on $\bb{R}^{n}$ is scaled so that $(\varepsilon_{i},\varepsilon_{i}) = 2$,
  which is why $q^{-1}$ is used instead of $v$ here.}
\end{equation}
We shall only present the proof for the sign $+$, since the arguments are similar for the sign $-$. We may further assume that $k > j$, for otherwise
$\ell_{kj}^{+} = 0$. Consider the relations~\eqref{eq:RLL_relations_explicit} with sign $+$, $m = i$, and $p = k$. If $i = k \neq j$, these read
$(R_{q})_{ii}^{ii}\ell_{ii}^{+}\ell_{ij}^{+} = (R_{q})^{ij}_{ij}\ell_{ij}^{+}\ell_{ii}^{+} + (R_{q})^{ji}_{ij}\ell_{ii}^{+}\ell_{ij}^{+}$.
An inspection of the formulas for $R_{q}$ implies that $(R_{q})^{ii}_{ii} = v$, $(R_{q})^{ij}_{ij} = 1$ for $i\ne j$, and $(R_{q})^{ji}_{ij} = 0$
when $i > j$. Since we also have $\ell_{ij}^{+} = 0$ for $i < j$, we conclude that $\ell_{ii}^{+}\ell_{ij}^{+} = v^{-1}\ell_{ij}^{+}\ell_{ii}^{+}$,
as desired. Similarly, if $i = j \neq k$, we get $\ell_{ii}^{+}\ell_{ki}^{+} = v\ell_{ki}^{+}\ell_{ii}^{+}$. Finally, if $i\neq j$ and $i \neq k$,
then we have
\begin{equation}\label{eq:RTT_commutation_relations}
  (R_{q})^{ik}_{ik}\ell_{ii}^{+}\ell_{kj}^{+} + (R_{q})^{ik}_{ki}\ell_{ki}^{+}\ell_{ij}^{+} =
  (R_{q})^{ij}_{ij}\ell_{kj}^{+}\ell_{ii}^{+} + (R_{q})^{ji}_{ij}\ell_{ki}^{+}\ell_{ij}^{+}.
\end{equation}
If $i < j$, then $\ell_{ij}^{+} = 0$, and hence the expression above reduces to
$(R_{q})^{ik}_{ik}\ell_{ii}^{+}\ell_{kj}^{+} = (R_{q})^{ij}_{ij}\ell_{kj}^{+}\ell_{ii}^{+}$.
Since $(R_{q})^{ij}_{ij} = (R_{q})^{ik}_{ik} = 1$ in this case, we get the desired equality~\eqref{eq:RTT_q_commute}. If $i > k > j$,
then we have $\ell_{ki}^{+} = 0$, so that~\eqref{eq:RTT_commutation_relations} again yields~\eqref{eq:RTT_q_commute}. Finally, if $k > i > j$,
then~\eqref{eq:RTT_q_commute} follows from the fact that $(R_{q})^{ji}_{ij} = (R_{q})^{ik}_{ki} = 0$.

Due to equation~\eqref{eq:RTT_q_commute_general}, it is enough to prove that $\wtd{U}_{\ge}(R_{q})$ (resp.\ $\wtd{U}_{\le}(R_{q})$) is generated
as an algebra by all $\cal{E}_{i}$ and $\ell_{ii}^{+}$ with $1 \le i \le n$ (resp.\ all $\cal{F}_{i}$ and $\ell_{ii}^{-}$ with $1 \le i \le n$).
We shall sketch the proof for $\wtd{U}_{\ge}(R_{q})$. First, it is clear that one can obtain all $\ell_{i+1,i}^{+}$ with $1 \le i \le n$ from
$\cal{E}_{i}$ and $\ell_{ii}^{+}$. Next, using~\eqref{eq:LC_explicit_1} with $i + 1$ in place of $i$ and $i$ in place of $j$, we find that all
$\ell_{i',(i + 1)'}^{+}$ can be expressed via $\ell_{i+1,i}^{+}$ and $\ell_{\mu}^{+}$ for any $1 \le i \le n$. Finally, one obtains the remaining
$\ell_{ij}^{+}$ via an analogue of~\cite[Proposition 8.29]{KS}.
\end{proof}

Now, we can define a left action of $\wtd{U}_{0,-}(R_{q})$ on $\wtd{U}_{-}(R_{q})$ via
$\ell_{\mu}^{-}\cdot u = q^{(\mu,\lambda + \lambda')}u$ for all $u \in \wtd{U}_{-}(R_{q})_{\lambda,\lambda'}$. It is easy to see that this
makes $\wtd{U}_{-}(R_{q})$ into a left module-algebra over $\wtd{U}_{0,-}(R_{q})$, and thus we can form the \emph{left crossed product algebra}
$\wtd{U}_{-}(R_{q}) \rtimes \wtd{U}_{0,-}(R_{q})$, which is $\wtd{U}_{-}(R_{q}) \otimes \wtd{U}_{0,-}(R_{q})$ as a vector space, and multiplication
is given by (see~\cite[Proposition 10.18]{KS}):
\begin{equation}\label{eq:left_crossed_product}
  (u \otimes h)(u' \otimes h') = \sum_{(h)} u(h_{(1)}\cdot u') \otimes h_{(2)}h'
  \qquad \text{for all} \qquad u,u' \in \wtd{U}_{-}(R_{q}),\ h,h' \in \wtd{U}_{0,-}(R_{q}).
\end{equation}
Similarly, we can define a right action of $\wtd{U}_{0,+}(R_{q})$ on $\wtd{U}_{+}(R_{q})$ via
$u\cdot \ell_{\mu}^{+} = q^{(\mu,\lambda + \lambda')}u$ for all $u \in \wtd{U}_{+}(R_{q})_{\lambda,\lambda'}$, which makes $\wtd{U}_{+}(R_{q})$ into
a right module-algebra over $\wtd{U}_{0,+}(R_{q})$. Then we can form the \emph{right crossed product algebra}
$\wtd{U}_{0,+}(R_{q}) \ltimes \wtd{U}_{+}(R_{q})$, which is $\wtd{U}_{0,+}(R_{q}) \otimes \wtd{U}_{+}(R_{q})$ as a vector space, with multiplication
given by (cf.~\cite[\S10.2.1]{KS}):
\begin{equation}\label{eq:right_crossed_product}
  (h \otimes u)(h' \otimes u') = \sum_{(h')} hh_{(1)}' \otimes (u\cdot h_{(2)}')u'
  \qquad \text{for all} \qquad u,u' \in \wtd{U}_{+}(R_{q}),\ h,h' \in \wtd{U}_{0,+}(R_{q}).
\end{equation}

Then we have the following result:

\begin{prop}\label{prop:RTT_crossed_products}
The assignments
\begin{align*}
    &\psi_{+}(\cal{E}_{i}) = 1 \otimes \cal{E}_{i}, & & \psi_{+}(\ell_{ii}^{+}) = \ell_{ii}^{+} \otimes 1 & & 1 \le i \le n, \\
    &\psi_{-}(\cal{F}_{i}) = \cal{F}_{i} \otimes 1, & & \psi_{-}(\ell_{ii}^{-}) = 1 \otimes \ell_{ii}^{-} & & 1 \le i \le n,
\end{align*}
give rise to algebra isomorphisms
$\psi_{+}\colon \wtd{U}_{\ge}(R_{q}) \iso \wtd{U}_{0,+}(R_{q}) \ltimes \wtd{U}_{+}(R_{q})$,
$\psi_{-}\colon \wtd{U}_{\le}(R_{q}) \iso \wtd{U}_{-}(R_{q}) \rtimes \wtd{U}_{0,-}(R_{q})$.
\end{prop}

\begin{proof}
First, we note that we have a linear map $\psi_{-}'\colon \wtd{U}_{-}(R_{q}) \rtimes \wtd{U}_{0,-}(R_{q}) \to \wtd{U}_{\le}(R_{q})$
given by $u \otimes h \mapsto uh$, which is an algebra homomorphism due to~\eqref{eq:RTT_q_commute_general}. Likewise, the linear map
$\psi_{+}'\colon \wtd{U}_{0,+}(R_{q}) \ltimes\wtd{U}_{+}(R_{q}) \to \wtd{U}_{\ge}(R_{q})$ given by $h \otimes u \mapsto hu$ is
an algebra homomorphism. On the other hand, note that $\wtd{U}_{\ge}(R_{q})$ and $\wtd{U}_{\le}(R_{q})$ both inherit bigradings from
$\wtd{U}(R_{q})$, as each algebra is generated by homogeneous elements relative to the bigrading. Moreover, since
$\deg(\cal{E}_{i}) = (\varepsilon_{i} - \varepsilon_{i+1},0)$ and $\deg(\cal{F}_{i}) = (0,\varepsilon_{i+1} - \varepsilon_{i})$ for all
$1\leq i\leq n$, it follows from Lemma~\ref{lem:borel_spanning_sets} that $(\ell_{\mu}^{+})^{-1}u \in \wtd{U}_{+}(R_{q})$ and
$u'\ell_{\lambda}^{-} \in \wtd{U}_{-}(R_{q})$ for any $u \in \wtd{U}_{\ge}(R_{q})_{\lambda,\mu}, u' \in \wtd{U}_{\le}(R_{q})_{\lambda,\mu}$.
Thus, we can define linear maps
$\psi_{+}\colon \wtd{U}_{\ge}(R_{q}) \to \wtd{U}_{0,+}(R_q) \ltimes\wtd{U}_{+}(R_{q})$,
$\psi_{-}\colon \wtd{U}_{\le}(R_{q}) \to \wtd{U}_{-}(R_{q}) \rtimes \wtd{U}_{0,-}(R_{q})$ via
\[
  \psi_{+}(u) =\ell_{\mu}^{+} \otimes (\ell_{\mu}^{+})^{-1}u,\qquad \psi_{-}(u') =  u'\ell_{\lambda}^{-} \otimes (\ell_{\lambda}^{-})^{-1}
  \qquad \text{for all}\qquad u \in \wtd{U}_{\ge}(R_{q})_{\lambda,\mu},\ u' \in \wtd{U}_{\le}(R_{q})_{\lambda,\mu}.
\]
It is easy to check using~\eqref{eq:RTT_q_commute_general} that $\psi_{\pm}$ are algebra homomorphisms which are inverse to the above homomorphisms
$\psi_{\pm}'$. Thus, $\psi_{\pm}$ are algebra isomorphisms, and they map $\cal{E}_{i}$, $\cal{F}_{i}$, $\ell_{ii}^{\pm}$ as specified in the statement.
\end{proof}

We shall next realize $\wtd{U}(R_{q})$ as a Drinfeld double of $\wtd{U}_{\le}(R_{q})$ and $\wtd{U}_{\ge}(R_{q})$. To this end, we first introduce
algebras $A_{\pm}(R_{q}) = A(R_{q})/\cal{I}_{\pm}$, where $\cal{I}_{+}$ (resp.\ $\cal{I}_{-}$) is the ideal generated by all $t_{ij}$ with $i < j$
(resp.\ $i > j$), along with the entries of the matrix (cf.~\eqref{eq:LC_relations})
\begin{equation}\label{eq:TC_relations}
  TC_{q}T^{\mathsf{t}'}C_{q}^{-1} - I.
\end{equation}
In type $B_{n}$, we impose the additional relation $t_{n+1,n+1} = 1$, in accordance with Remark~\ref{rmk:square_root_zq}. We shall denote
the images of the elements $t_{ij}$ in $A_{\pm}(R_{q})$ by $t_{ij}^{\pm}$. The same arguments as those used for the ideal
$\cal{I}_{q} \subset U(R_{q})$ show that $\cal{I}_{\pm}$ are also coideals, so that $A_{\pm}(R_{q})$ are both bialgebras. In fact, since
the matrices $T^{\pm} = (t_{ij}^{\pm})_{i, j = 1}^{N}$ are invertible, the bialgebras $A_{\pm}(R_{q})$ are actually Hopf algebras, with
the antipode $S(T^{\pm}) = (T^{\pm})^{-1}$.

To prove the next proposition, we first need to record a few identities among $R_{q}$ and $C_{q}$, which involve the partial transpositions
on $\End(V) \otimes \End(V)$ defined via $(A \otimes B)^{\mathsf{t}_{1}'} = A^{\mathsf{t}'} \otimes B$ and
$(A \otimes B)^{\mathsf{t}_{2}'} = A \otimes B^{\mathsf{t}'}$ for any $A,B \in \End(V)$. We also write
$A_{1} = A \otimes I$ and $A_{2} = I \otimes A$ for any $A \in \End(V)$.

\begin{lemma}\label{lem:crossing_symmetries}
The following \textbf{crossing symmetry} identities hold:
\begin{equation}\label{eq:crossing_symmetries_1}
  R_{q}(C_{q})_{1}R_{q}^{\mathsf{t}_{1}'}(C_{q})_{1}^{-1} = I,\qquad
  (C_{q}^{\mathsf{t}'})_{1}^{-1}R_{q}^{-1}(C_{q}^{\mathsf{t}'})_{1}(R_{q}^{-1})^{\mathsf{t}_{1}'} = I,
\end{equation}
\begin{equation}\label{eq:crossing_symmetries_2}
  (C_{q}^{\mathsf{t}'})^{-1}_{2}R_{q}(C_{q}^{\mathsf{t}'})_{2}R_{q}^{\mathsf{t}_{2}'} = I,
  \qquad R_{q}^{-1}(C_{q})_{2}(R_{q}^{-1})^{\mathsf{t}_{2}'}(C_{q}^{-1})_{2} = I.
\end{equation}
\end{lemma}

\begin{proof}
These relations follow by specializing $z=0$ in their affine counterparts from Lemma~\ref{lem:affine_crossing_symmetries} and using the
formulas~\eqref{eq:Baxterization-B}--\eqref{eq:Baxterization-D} as well as $f(0)=1$ for $f(z)$ satisfying~\eqref{eq:f_qdifference}.
\end{proof}

\begin{prop}\label{prop:skew_pairing_descent}
There are unique bilinear maps $\sigma,\bar{\sigma}\colon A_{+}(R_{q}) \times A_{-}(R_{q}) \to \bb{K}$
satisfying~\eqref{eq:skew_pairing_counit}--\eqref{eq:convolution_inverse} and
\[
  \sigma(t_{ij}^{+}, t_{k\ell}^{-}) = (R_{q})_{j\ell}^{ik}\qquad \text{and}\qquad
  \bar{\sigma}(t_{ij}^{+} , t_{k\ell}^{-}) = (R_{q}^{-1})_{j\ell}^{ik}
  \qquad \text{for all}\qquad 1 \le i,j,k,\ell \le N.
\]
\end{prop}

\begin{proof}
Due to Theorem~\ref{thm:A(R)_skew_pairings}, we just need to show that $\sigma(\cal{I}_{+},A(R_{q})) = \sigma( A(R_{q}), \cal{I}_{-}) = 0$ and
$\bar{\sigma}(\cal{I}_{+}, A(R_{q})) = \bar{\sigma}(A(R_{q}), \cal{I}_{-}) = 0$. First, let $\cal{I}'$ be the ideal generated by the entries of
the matrix~\eqref{eq:TC_relations}, and the additional element $t_{n+1,n+1} - 1$ if we are in type $B_{n}$. Also, let $\cal{I}_{+}''$
(resp.\ $\cal{I}_{-}''$) be the ideal generated by all $t_{ij}$ with $i < j$ (resp.\ $i > j$). Then $\cal{I_{\pm}} = \cal{I}' + \cal{I}_{\pm}''$.
Since the proof of Proposition~\ref{prop:central_element} also goes through for $A(R_{q})$, we find that the ideal generated by the entries
of~\eqref{eq:TC_relations} is in fact generated by a single element $z_{q} -1$, where $z_{q}$ is a grouplike element given by
\[
  z_{q} = \sum_{1\leq k\leq N}c_{ii}c_{kk}^{-1}t_{k'i'}t_{ki}
  = \sum_{1\leq k\leq N}c_{jj}^{-1}c_{kk}t_{jk}t_{j'k'} \qquad \text{for all}\qquad 1 \le i,j \le N.
\]
Note that $z_{q} = t_{n+1,n+1}^{2}$ if we are in type $B_{n}$. Thus, in types $C_{n}$ and $D_{n}$, the ideal $\cal{I}'$ is generated by $z_{q} - 1$,
while in type $B_{n}$, the ideal $\cal{I}'$ is generated by $t_{n+1,n+1} - 1$. In particular,
  $\sigma(z_{q},1) = \bar{\sigma}(z_{q},1) = \epsilon(z_{q}) = 1$.
Next, let us prove that $\sigma(z_{q},t_{i\ell}) = \delta_{i\ell} = \bar{\sigma}(z_{q},t_{i\ell})$ for all $i,\ell$.
First, using~\eqref{eq:skew_pairing_adjointness}, we have
\[
  \sigma(z_{q},t_{i\ell}) = \sum_{1\leq k\leq N}c_{jj}^{-1}c_{kk} \sum_{1\leq p\leq N}\sigma(t_{jk},t_{ip})\sigma(t_{j'k'},t_{p\ell}) \, =
  \sum_{1 \le k,p \le N}c_{jj}^{-1}c_{kk}(R_{q})^{ji}_{kp}(R_{q})^{j'p}_{k'\ell} = \delta_{i\ell},
\]
where the last equality follows from the first identity in~\eqref{eq:crossing_symmetries_1}.
Similarly, using~\eqref{eq:inverse_skew_pairing_adjointness}, we have
\[
  \bar{\sigma}(z_{q},t_{i\ell}) = \sum_{1\leq k\leq N}c_{jj}^{-1}c_{kk} \sum_{1\leq p\leq N}\bar{\sigma}(t_{j'k'},t_{ip})\bar{\sigma}(t_{jk},t_{p\ell}) \,
  = \sum_{1 \le k,p \le N}c_{jj}^{-1}c_{kk}(R_{q}^{-1})^{j'i}_{k'p}(R_{q}^{-1})^{jp}_{k\ell} = \delta_{i\ell},
\]
where the last equality follows from the second identity in~\eqref{eq:crossing_symmetries_1}. Thus, because the $t_{i\ell}$ generate $A(R_{q})$, we may
then use~\eqref{eq:skew_pairing_adjointness},~\eqref{eq:inverse_skew_pairing_adjointness}, the property $\Delta(z_q)=z_q\otimes z_q$, and the definition
of $\epsilon$ to conclude that $\sigma(z_{q},x) = \bar{\sigma}(z_{q},x) = \epsilon(x)$ for all $x \in A(R_{q})$. Evoking~\eqref{eq:skew_pairing_counit},
this implies that $\sigma(z_{q} - 1,x) = \bar{\sigma}(z_{q} - 1,x) = 0$ for all $x \in A(R_{q})$. Since $(R_{q})^{n+1,k}_{n + 1,j} = \delta_{kj}$
in type $B_{n}$, a similar argument shows that $\sigma(t_{n+1,n+1} - 1,x) = \bar{\sigma}(t_{n+1,n+1} - 1,x) = 0$ for all $x \in A(R_{q})$.
Using~\eqref{eq:skew_pairing_adjointness} and~\eqref{eq:inverse_skew_pairing_adjointness} again, we get
$\sigma(\cal{I}',A(R_{q})) = \bar{\sigma}(\cal{I}',A(R_{q})) = 0$. A similar argument, this time using~\eqref{eq:crossing_symmetries_2}, shows that
$\sigma(A(R_{q}),\cal{I}') = \bar{\sigma}(A(R_{q}),\cal{I}') = 0$.

We now need to prove that $\sigma(\cal{I}_{+}'',A(R_{q})) = \bar{\sigma}(\cal{I}_{+}'',A(R_{q})) = 0$ and
$\sigma(A(R_{q}),\cal{I}_{-}'') = \bar{\sigma}(A(R_{q}),\cal{I}_{-}'') = 0$. Note first that $\sigma(t_{ij},1) = \bar{\sigma}(t_{ij},1) = 0$
and $\sigma(1,t_{ij}) = \bar{\sigma}(1,t_{ij}) = 0$ for all $i \neq j$. Thus, as above, it is enough to prove that for any $k,\ell$, we have
$\sigma(t_{ij},t_{k\ell}) = \bar{\sigma}(t_{ij},t_{k\ell}) = 0$ when $i < j$, and $\sigma(t_{k\ell},t_{ij}) = \bar{\sigma}(t_{k\ell},t_{ij}) = 0$
when $i > j$. We shall only prove the first set of equalities, leaving the other two to the reader.

By~\eqref{eq:skew_pairing_generators}, we have $\sigma(t_{ij},t_{k\ell}) =(R_{q})^{ik}_{j\ell}$. Suppose first that $i \neq k'$, so that
$(R_{q})^{ik}_{j\ell} = 0$ unless $\{i,k\} = \{j,\ell\}$. Then we are reduced to consideration of $\sigma(t_{ij},t_{ji}) = (R_{q})^{ij}_{ji}$,
which vanishes for $i < j$, as desired. If $i = k'$, then $\sigma(t_{ij},t_{k\ell}) = 0$ unless $j = \ell'$, so we need only to consider
$\sigma(t_{ij},t_{i'j'}) = (R_{q})^{ii'}_{jj'}$. In this case, a direct inspection of the formulas for $R_{q}$ shows that $(R_{q})^{ii'}_{jj'} = 0$
unless $i \ge j$. Therefore, we again have the desired equality $\sigma(t_{ij},t_{i'j'}) = 0$ for $i<j$.

For $\bar{\sigma}$, we proceed as above, using explicit formulas for $R_{q}^{-1}$ that can be easily obtained from~\cite[(7.7)--(7.9)]{MT1}. Indeed,
as with $R_{q}$, we have $(R_{q}^{-1})^{ij}_{k\ell} = 0$ unless $\{i,j\} = \{k,\ell\}$ or $i = j'$ and $k = \ell'$. Furthermore,  we also have
$(R_{q}^{-1})^{ij}_{ji} = 0$ if $i < j \ne i'$ and $(R_{q}^{-1})^{ii'}_{jj'} = 0$ if $i < j$.
Thus, $\bar{\sigma}(t_{ij},t_{k\ell}) = 0$ whenever $i < j$.

This completes the proof.
\end{proof}

As a consequence of the above result, we can form the Drinfeld double $\cal{D}(A_{+}(R_{q}),A_{-}(R_{q});\sigma)$ in the sense
of~\cite[Definition 8.4]{KS}. Following the procedure of~\cite[Example 10.16]{KS}, we then obtain:

\begin{prop}\label{prop:Drinfeld_double_RTT}
There is a Hopf algebra isomorphism $\psi\colon \wtd{U}(R_{q}) \iso \cal{D}(A_{+}(R_{q}),A_{-}(R_{q});\sigma)$ which sends
$\ell_{ij}^{+} \mapsto 1 \otimes t_{ij}^{+}$ and $\ell_{ij}^{-} \mapsto t_{ij}^{-} \otimes 1$.
\end{prop}

Combining this with Proposition~\ref{prop:RTT_crossed_products}, we finally obtain:

\begin{theorem}\label{thm:quadrilangular_decomposition}
The multiplication map $\wtd{U}_{-}(R_{q}) \otimes \wtd{U}_{0,-}(R_{q}) \otimes \wtd{U}_{0,+}(R_{q}) \otimes \wtd{U}_{+}(R_{q}) \to \wtd{U}(R_{q})$
is an isomorphism of vector spaces.
\end{theorem}

\begin{cor}\label{cor:triangular_decomposition_RTT}
The multiplication map $\mu\colon \wtd{U}_{-}(R_{q}) \otimes \wtd{U}_{0}(R_{q}) \otimes \wtd{U}_{+}(R_{q}) \to \wtd{U}(R_{q})$
is an isomorphism of vector spaces.
\end{cor}

\begin{proof}
Due to Theorem~\ref{thm:quadrilangular_decomposition}, we need to show that the multiplication map
$\wtd{U}_{0,-}(R_{q}) \otimes \wtd{U}_{0,+}(R_{q}) \to \wtd{U}_{0}(R_{q})$ is an isomorphism of vector spaces. The multiplication map
$\bb{K} \otimes \wtd{U}_{0,-}(R_{q}) \otimes \wtd{U}_{0,+}(R_{q}) \otimes \bb{K} \to \wtd{U}(R_{q})$ is injective by
Theorem~\ref{thm:quadrilangular_decomposition}, and it clearly maps into $\wtd{U}_{0}(R_{q})$. To complete the proof, it remains to show
that the image is all of $\wtd{U}_{0}(R_{q})$, which amounts to establishing the commutativity of $\ell_{ii}^{+}$ with $\ell_{jj}^{-}$ for
any $i,j$. Since $\ell_{ii}^{\pm} = (\ell_{i'i'}^{\pm})^{-1}$ by~\eqref{eq:diagonal_reflection}, we may further assume that $j \ne i'$.
As in the proof of Lemma~\ref{lem:cartan_subalgebras_RTT}, we consider the $(m,p) = (i,j)$ case of the defining relations
\begin{equation}\label{eq:RLL_cross_relations}
  \sum_{1 \le k,t \le N}(R_{q})^{mp}_{kt}\ell_{ki}^{+}\ell_{tj}^{-} \, = \sum_{1 \le k,t \le N}(R_{q})^{kt}_{ij}\ell_{pt}^{-}\ell_{mk}^{+}.
\end{equation}
If $i = j$, then we get $(R_{q})^{ii}_{ii}\ell_{ii}^{+}\ell_{ii}^{-} = (R_{q})^{ii}_{ii}\ell_{ii}^{-}\ell_{ii}^{+}$, so that
$\ell_{ii}^{+}\ell_{ii}^{-} = \ell_{ii}^{-}\ell_{ii}^{+}$ as $(R_{q})^{ii}_{ii} \neq 0$. If $i \neq j$, then
\[
  (R_{q})_{ij}^{ij}\ell_{ii}^{+}\ell_{jj}^{-} + (R_{q})_{ji}^{ij}\ell_{ji}^{+}\ell_{ij}^{-}
  = (R_{q})_{ij}^{ij}\ell_{jj}^{-}\ell_{ii}^{+} + (R_{q})_{ij}^{ji}\ell_{ji}^{-}\ell_{ij}^{+},
\]
so that $\ell_{ii}^{+}\ell_{jj}^{-} = \ell_{jj}^{-}\ell_{ii}^{+}$, since $(R_{q})_{ji}^{ij} = 0 = \ell_{ij}^{+}$ if $i < j$ and
$(R_{q})_{ij}^{ji} = 0 = \ell_{ji}^{+}$ if $i > j$.
\end{proof}

\begin{cor}\label{cor:double_cartan_subalgebra_RTT}
The assignment $x_{i} \mapsto \ell_{ii}^{+}, y_{i} \mapsto \ell_{ii}^{-}\ (1\leq i\leq n)$ gives rise to an algebra isomorphism
from the Laurent polynomial algebra $H' = \bb{K}[x_{1}^{\pm 1},\ldots ,x_{n}^{\pm 1},y_{1}^{\pm 1},\ldots ,y_{n}^{\pm 1}]$ to $\wtd{U}_{0}(R_{q})$.
\end{cor}

\begin{proof}
Since $\wtd{U}_{0}(R_{q})$ is commutative by above and all $\ell_{ii}^{\pm}$ are invertible, we have an algebra homomorphism
$f\colon H' \to \wtd{U}_{0}(R_{q})$ sending $x_{i} \mapsto \ell_{ii}^{+}$ and $y_{i} \mapsto \ell_{ii}^{-}$. On the other hand,
by Theorem~\ref{thm:quadrilangular_decomposition} and the argument in Corollary~\ref{cor:triangular_decomposition_RTT}, the multiplication
map takes $\bb{K} \otimes \wtd{U}_{0,-}(R_{q}) \otimes \wtd{U}_{0,+}(R_{q}) \otimes \bb{K}$ isomorphically onto $\wtd{U}_{0}(R_{q})$,
so using Lemma~\ref{lem:cartan_subalgebras_RTT}, it follows that $f$ takes a basis of $H'$ to a basis of $\wtd{U}_{0}(R_{q})$.
Hence $f$ is an algebra isomorphism, as desired.
\end{proof}

\begin{cor}\label{cor:positive_negative_RTT}
There are algebra isomorphisms $\wtd{U}_{\pm}(R_{q}) \to \wtd{U}'_{\pm}(R_{q})$ given by $\cal{E}_{i} \mapsto \cal{E}_{i}$ and
$\cal{F}_{i} \mapsto \cal{F}_{i}$.
\end{cor}

\begin{proof}
It suffices to show that the ideal $J$ of $\wtd{U}(R_{q})$ generated by $\{\ell_{ii}^{+} - (\ell_{ii}^{-})^{-1}\}_{i=1}^N$ is equal to
$J' = \mu(\wtd{U}_{-}(R_{q}) \otimes J_{0} \otimes \wtd{U}_{+}(R_{q}))$, where $J_{0}$ is the ideal of $\wtd{U}_{0}(R_{q})$ generated by all
$\{\ell_{ii}^{+} - (\ell_{ii}^{-})^{-1}\}_{i=1}^N$. Indeed, we have $\wtd{U}_{\pm}'(R_{q}) = \wtd{U}_{\pm}(R_{q})/(J \cap \wtd{U}_{\pm}(R_{q}))$,
but if $J = \mu(\wtd{U}_{-}(R_{q}) \otimes J_{0} \otimes \wtd{U}_{+}(R_{q}))$, then Corollary~\ref{cor:triangular_decomposition_RTT} implies that
$J \cap \wtd{U}_{+}(R_{q}) = \mu((\wtd{U}_{-}(R_{q}) \cap J_{0} \cap \wtd{U}_{+}(R_{q})) \cap (\bb{K} \otimes \bb{K} \otimes \wtd{U}_{+}(R_{q}))) = 0$,
and a similar equality holds for $\wtd{U}_{-}(R_{q})$.

Since we obviously have $\{\ell_{ii}^{+} - (\ell_{ii}^{-})^{-1}\}_{i = 1}^{N}\subset J' \subset J$, it is enough to prove that $J'$ is an ideal
of $\wtd{U}(R_{q})$. Note that for any $u \in \wtd{U}(R_{q})$ and $w \in \wtd{U}_{-}(R_{q})$, we have
$uw \in \wtd{U}_{-}(R_{q})\wtd{U}_{0}(R_{q})\wtd{U}_{+}(R_{q})$ by the triangular decomposition of Corollary~\ref{cor:triangular_decomposition_RTT}.
Thus, to show that $uJ' \subset J'$ for any $u \in \wtd{U}(R_{q})$, it is enough to prove that $u_{+}J' \subset J'$ for any
$u_{+} \in \wtd{U}_{+}(R_{q})$. For this, we need the following formula:
\begin{equation}\label{eq:RTT_q_commute_cross}
  \ell_{\mu}^{-}u = q^{(\mu,\lambda + \lambda')}u\ell_{\mu}^{-}\qquad \text{for all}\qquad u \in \wtd{U}_{\ge}(R_{q})_{\lambda,\lambda'},
\end{equation}
 As in the first part of the proof of Lemma~\ref{lem:borel_spanning_sets},
it is enough to show that
\[
  \ell_{ii}^{-}\ell_{kj}^{+} = v^{\delta_{ik} - \delta_{ij}}\ell_{kj}^{+}\ell_{ii}^{-}\qquad \text{for all}\qquad k \neq j,\ i \neq k',j',
\]
where $v = q^{-2}$ in type $B_{n}$ and $v = q^{-1}$ in types $C_{n}$ and $D_{n}$. This is done by following the same procedure as in the proof
of Lemma~\ref{lem:borel_spanning_sets}, but instead using the following particular case of relations~\eqref{eq:RLL_cross_relations}:
\[
  \sum_{1 \le p,t \le N}(R_{q})^{ki}_{pt}\ell_{pj}^{+}\ell_{ti}^{-} = \sum_{1 \le p,t \le N}(R_{q})^{pt}_{ji}\ell_{it}^{-}\ell_{kp}^{+}.
\]
We leave further details to the reader. Thus, equalities~\eqref{eq:RTT_q_commute_general} and~\eqref{eq:RTT_q_commute_cross} imply that
$u_{+}(\ell_{ii}^{+} - (\ell_{ii}^{-})^{-1}) = q^{(\varepsilon_{i},\lambda + \lambda')}(\ell_{ii}^{+} - (\ell_{ii}^{-})^{-1})u_{+}$
for all $1\leq i\leq N$ and $u_{+} \in \wtd{U}_{+}(R_{q})_{\lambda,\lambda'}$, so that
$u_{+}(\ell_{ii}^{+} - (\ell_{ii}^{-})^{-1}) \in J'$, as desired.

A similar argument shows that $J'u \subset J'$ for all $u \in \wtd{U}(R_{q})$. This completes the proof.
\end{proof}

As an immediate application of Corollary~\ref{cor:positive_negative_RTT} and its proof, we obtain:

\begin{prop}\label{prop:triangular_decomposition_RTT'}
The multiplication map $\wtd{U}_{-}'(R_{q}) \otimes \wtd{U}_{0}'(R_{q}) \otimes \wtd{U}_{+}'(R_{q}) \to \wtd{U}'(R_{q})$ is an isomorphism
of vector spaces. Moreover, $\wtd{U}_{0}'(R_{q}) = \wtd{U}_{0}(R_{q})/(J \cap \wtd{U}_{0}(R_{q})) = \wtd{U}_{0}(R_{q})/J_{0}$.
\end{prop}

We also note that $\wtd{U}(R_{q})$ and $\wtd{U}'(R_{q})$ are both $Q$-graded algebras via the assignment
$\deg(\ell_{ij}^{\pm}) = \varepsilon_{j} - \varepsilon_{i}$. For the former, this follows because the $P \times P$-grading on $\wtd{U}(R_{q})$
from Proposition~\ref{prop:U(R)_grading} induces a $P$-grading where
$\wtd{U}(R_{q})_{\mu} = \bigoplus_{\mu_{1} + \mu_{2} = \mu}\wtd{U}(R_{q})_{\mu_{1},\mu_{2}}$, and it is easy to see from the definition of the
$P \times P$-grading that $\wtd{U}(R_{q})_{\mu_{1},\mu_{2}} = 0$ unless $\mu_{1} + \mu_{2} \in Q$. Then $\wtd{U}'(R_{q})$ inherits this grading
because the ideal $J$ is homogeneous.

We are finally ready to prove an analogue of Proposition~\ref{prop:double_cartan_DJ} for $\wtd{U}(R_{q})$:

\begin{prop}\label{prop:double_cartan_RTT_finite}
Let $L_{\varepsilon_{i}}$ be algebraically independent invertible transcendental elements. There is a unique injective algebra homomorphism
$\eta_{R}\colon \wtd{U}(R_{q}) \to \wtd{U}'(R_{q}) \otimes \bb{K}[L_{\varepsilon_{1}}^{\pm 1},\ldots ,L_{\varepsilon_{n}}^{\pm 1}]$ such that
\begin{equation}\label{eq:double_cartan_RTT_finite}
  \eta_{R}(\ell_{ij}^{\pm}) = \ell_{ij}^{\pm} \otimes L_{-\varepsilon_{i}-\varepsilon_{j}}\qquad \text{for all}\qquad 1 \le i,j \le N,
\end{equation}
where we write $L_{\mu} = L_{\varepsilon_{1}}^{c_{1}}L_{\varepsilon_{2}}^{c_{2}}\ldots L_{\varepsilon_{n}}^{c_{n}}$ for all
$\mu = \sum_{i = 1}^{n}c_{i}\varepsilon_{i}\in P$.
\end{prop}

\begin{proof}
Define $\eta_{R}\colon \wtd{U}(R_{q}) \to \wtd{U}'(R_{q}) \otimes \bb{K}[L_{\varepsilon_{1}}^{\pm 1},\ldots ,L_{\varepsilon_{n}}^{\pm 1}]$
by $\eta_{R}(u) = \bar{u} \otimes L_{\lambda - \mu}$ for all $u \in \wtd{U}(R_{q})_{\lambda,\mu}$, where $\bar{u}$ denotes the image of $u$
under the quotient map $\wtd{U}(R_{q}) \twoheadrightarrow \wtd{U}'(R_{q})$. Then, by definition, $\eta_{R}$ is an algebra homomorphism,
and we clearly have~\eqref{eq:double_cartan_RTT_finite}.

For the injectivity of $\eta_R$, let $\{b_{i}^{\mu,+}\}_{i \in \frak{I}_{\mu}^{+}}$ be a basis for $\wtd{U}_{+}(R_{q})_{\mu,0}$ and
$\{b_{i}^{\mu,-}\}_{i \in \frak{I}_{\mu}^{-}}$ be a basis for $\wtd{U}_{-}(R_{q})_{0,\mu}$. Note that $\wtd{U}_{+}(R_{q})_{\mu,\nu} = 0$
if $\nu \neq 0$ and $\wtd{U}_{-}(R_{q})_{\mu,\nu} = 0$ if $\mu \neq 0$, so by Corollaries~\ref{cor:triangular_decomposition_RTT}
and~\ref{cor:double_cartan_subalgebra_RTT}, the set
\begin{equation}\label{eq:PBW_basis_RTT_double_cartan}
  \big\{b_{i}^{\mu,-}\ell_{\lambda}^{-}\ell_{\lambda'}^{+}b_{j}^{\nu,+} \,\big|\,
        \mu,\nu,\lambda,\lambda' \in Q,\, i \in \frak{I}_{\mu}^{-},\, j \in \frak{I}_{\nu}^{+}\big\}
\end{equation}
is a basis for $\wtd{U}(R_{q})$. By Corollary~\ref{cor:positive_negative_RTT}, the sets
$\{b_{i}^{\mu,\pm}\}_{i \in \frak{I}^\pm_{\mu}}$ also form bases for $\wtd{U}_{\pm}'(R_{q})_{\mu}$, where we identify $b_{i}^{\mu,\pm}$
with their images in $\wtd{U}_{\pm}'(R_{q})$. Thus, the set
$\{b_{i}^{\mu,-}\ell_{\lambda}^{+}b_{j}^{\nu,+} \,|\, \mu,\nu,\lambda \in Q,\, i \in \frak{I}_{\mu}^{-},\, j \in \frak{I}_{\nu}^{+}\}$
is a basis for $\wtd{U}'(R_{q})$, due to Proposition~\ref{prop:triangular_decomposition_RTT'}. Since
\[
  \eta_{R}(b_{i}^{\mu,-}\ell_{\lambda}^{-}\ell_{\lambda'}^{+}b_{j}^{\nu,+}) =
  b_{i}^{\mu,-}\ell_{\lambda' - \lambda}^{+}b_{j}^{\nu,+} \otimes L_{\nu - \mu - 2(\lambda + \lambda')},
\]
it follows that $\eta_{R}$ sends the basis~\eqref{eq:PBW_basis_RTT_double_cartan} to a set of linearly independent elements,
because one can determine $\lambda$ and $\lambda'$ from $\lambda' - \lambda$ and $\lambda + \lambda'$.
\end{proof}

   %%%%%%%%%%%%%%%%%%%%%%%%%%%%%%%%%%%%%%%%%%%%%%%%%%%%%%%%%%%%%%%%%%%%%%%%%%

\subsection{Isomorphism between RTT and Drinfeld-Jimbo realizations in classical types}\label{ssec:finite_RTT_BCD}
\

In this Subsection, we construct isomorphisms between Drinfeld-Jimbo and RTT-type two-parameter quantum groups in types $B_n,C_n,D_n$
by combining the above discussion with isomorphisms of~\cite{JLM1,JLM2}.

\noindent
$\bullet$ \textbf{Type $B_n$.}

Let $R_{q} = \hat{R}_{q} \circ \tau$, where $\hat{R}_{q}$ is the one-parameter $B_{n}$-type $R$-matrix of~\eqref{eq:1_parameter_R_B}.
Similarly, let $R_{r,s} = \hat{R}_{r,s} \circ \tau$, where $\hat{R}_{r,s}$ is the two-parameter $B_{n}$-type $R$-matrix of~\cite[(4.7)]{MT1}.
In this Subsection, we recall the relationship between $\wtd{U}'(R_{q})$ and $U_{q}(\frak{so}_{2n+1})$, and upgrade it to the one
between $\wtd{U}(R_{r,s})$ and $U_{r,s}(\frak{so}_{2n+1})$, which comes naturally from the twisting procedure developed above.
This corrects~\cite[Theorem 3.17]{HXZ1}.

For the one-parameter case, we mainly follow~\cite{JLM1}, though our exposition is adapted to fit our present conventions, which are
slightly different to theirs (see Remark~\ref{rmk:JLM_comparison_B} below). First, we recall the following result (which is
essentially~\cite[Main Theorem]{JLM1} with some minor adjustments made to ensure that the formulas are compatible with the bigradings):

\begin{theorem}\label{thm:DJ=RTT_Btype_1param}
There is a unique Hopf algebra isomorphism $\theta_{q}'\colon U_{q}(\frak{so}_{2n+1}) \iso \wtd{U}'(R_{q})$ such that
\begin{align*}
  &\theta_{q}'(K_{i}) = \ell_{i+1,i+1}^{+}(\ell_{ii}^{+})^{-1}  & & ~ & & \text{for}\quad 1 \le i \le n,\\
  &\theta_{q}'(E_{i}) = \frac{1}{q^{2} - q^{-2}}\ell_{i+1,i}^{+}(\ell_{ii}^{+})^{-1}, & &
   \theta_{q}'(F_{i}) = \frac{1}{q^{-2} - q^{2}}(\ell_{ii}^{-})^{-1}\ell_{i,i+1}^{-}  & &
    \text{for} \quad 1 \le i < n, \\
  &\theta_{q}'(E_{n}) = \frac{q^{-1/2}}{[2]_{q}^{1/2}(q - q^{-1})}\ell_{n+1,n}^{+}(\ell_{nn}^{+})^{-1}, & &
   \theta_{q}'(F_{n}) = \frac{q^{1/2}}{[2]_{q}^{1/2}(q^{-1} - q)}(\ell_{nn}^{-})^{-1}\ell_{n,n+1}^{-}.
\end{align*}
\end{theorem}

\begin{proof}
If we compose the restriction of the algebra isomorphism of~\cite[Main Theorem]{JLM1} with the Cartan involution~\eqref{eq:cartan_involution}, then
we get a (unique) algebra isomorphism $\theta_{q}'\colon U_{q}(\frak{so}_{2n+1}) \iso \wtd{U}'(R_{q})$ defined on generators by the above formulas.
We also emphasize that $\theta'_{q}$ is compatible with the $P\times P$-grading of both algebras (which is the reason for applying the Cartan
involution above). It is easy to verify on generators that
\begin{equation}\label{eq:Hopf_compatibility}
  (\theta_{q}' \otimes \theta_{q}') \circ \Delta = \Delta \circ \theta_{q}' \qquad \mathrm{and} \qquad
  \epsilon = \epsilon\circ \theta_{q}',
\end{equation}
so it follows that $\theta_{q}'\colon U_{q}(\frak{so}_{2n+1}) \to \wtd{U}'(R_{q})$ is in fact a Hopf algebra isomorphism.
\end{proof}

Using Propositions~\ref{prop:double_cartan_DJ} and~\ref{prop:double_cartan_RTT_finite}, this can be upgraded to a Cartan-doubled version:

\begin{cor}\label{cor:DJ=RTT_Btype_1param_double_cartan}
There is a unique Hopf algebra isomorphism $\theta_{q}\colon U_{q,q^{-1}}(\frak{so}_{2n + 1}) \iso \wtd{U}(R_{q})$ such that
\begin{align*}
  &\theta_{q}(\omega_{i}) = \ell_{i+1,i+1}^{+}(\ell_{ii}^{+})^{-1},
    & & \theta_{q}(\omega_{i}') = (\ell_{ii}^{-})^{-1}\ell_{i+1,i+1}^{-}  & & \text{for}\quad 1 \le i \le n,\\
  &\theta_{q}(e_{i}) = \frac{1}{q^{2} - q^{-2}}\ell_{i+1,i}^{+}(\ell_{ii}^{+})^{-1}, & &
   \theta_{q}(f_{i}) = \frac{1}{q^{-2} - q^{2}}(\ell_{ii}^{-})^{-1}\ell_{i,i+1}^{-} & &
    \text{for} \quad 1 \le i < n, \\
  &\theta_{q}(e_{n}) = \frac{q^{-1/2}}{[2]_{q}^{1/2}(q - q^{-1})}\ell_{n+1,n}^{+}(\ell_{nn}^{+})^{-1}, & &
   \theta_{q}(f_{n}) = \frac{q^{1/2}}{[2]_{q}^{1/2}(q^{-1} - q)}(\ell_{nn}^{-})^{-1}\ell_{n,n+1}^{-}.
\end{align*}
\end{cor}

\begin{proof}
First, we define a map
\begin{equation}\label{eq:g-map}
  g\colon
  \bb{K}[L_{\alpha_{1}}^{\pm 1},\ldots ,L_{\alpha_{n}}^{\pm 1}] \longrightarrow \bb{K}[L_{\varepsilon_{1}}^{\pm 1},\ldots ,L_{\varepsilon_{n}}^{\pm 1}]
  \quad \mathrm{via} \quad L_{\alpha_{i}}\mapsto L_{\alpha_{i}} = L_{\varepsilon_{i} - \varepsilon_{i+1}}.
\end{equation}
We note that for $i = n$, the second equality is due to our convention $\varepsilon_{n+1} = 0$. By Theorem~\ref{thm:DJ=RTT_Btype_1param}
we have an injective algebra homomorphism
  $$\theta_{q}' \otimes g\colon U_{q}(\frak{so}_{2n+1}) \otimes \bb{K}[L_{\alpha_{1}}^{\pm 1},\ldots ,L_{\alpha_{n}}^{\pm 1}]
    \longrightarrow \wtd{U}'(R_{q}) \otimes \bb{K}[L_{\varepsilon_{1}}^{\pm 1},\ldots ,L_{\varepsilon_{n}}^{\pm 1}].$$
Then, if $\eta\colon U_{q,q^{-1}}(\frak{so}_{2n+1}) \to U_{q}(\frak{so}_{2n+1}) \otimes \bb{K}[L_{\alpha_{1}}^{\pm 1},\ldots ,L_{\alpha_{n}}^{\pm 1}]$
is the algebra embedding of Proposition~\ref{prop:double_cartan_DJ} and
$\eta_R\colon \wtd{U}(R_{q}) \to \wtd{U}'(R_{q}) \otimes \bb{K}[L_{\varepsilon_{1}}^{\pm 1},\ldots ,L_{\varepsilon_{n}}^{\pm 1}]$
is the algebra embedding of Proposition~\ref{prop:double_cartan_RTT_finite}, we can define
$\theta_{q}\colon  U_{q,q^{-1}}(\frak{so}_{2n+1}) \to \wtd{U}(R_{q})$ by $\theta_{q} = \eta_{R}^{-1}(\theta_{q}' \otimes g)\eta$, provided that
$(\theta_{q}' \otimes g)\eta$ maps $U_{q,q^{-1}}(\frak{so}_{2n + 1})$ into $\eta_{R}(\wtd{U}(R_{q}))$. It suffices to verify the latter on generators.
For $e_{i}$ with $i < n$, we have
\begin{align*}
  (\theta_{q}' \otimes g)\eta(e_{i})
  &= (\theta_{q}' \otimes g)(E_i \otimes L_{\alpha_{i}})
   = \frac{1}{q^{2} - q^{-2}}\ell_{i+1,i}^{+}(\ell_{ii}^{+})^{-1}\otimes L_{\varepsilon_{i} - \varepsilon_{i+1}} \\
  &= \frac{1}{q^{2} - q^{-2}}(\ell_{i+1,i}^{+} \otimes L_{-\varepsilon_{i+1} - \varepsilon_{i}})(\ell_{ii}^{+} \otimes L_{-2\varepsilon_{i}})^{-1}
   = \frac{1}{q^{2} - q^{-2}}\eta_{R}(\ell_{i+1,i}^{+})\eta_{R}(\ell_{ii}^{+})^{-1},
\end{align*}
which shows that $(\theta_{q}' \otimes g)\eta (e_{i}) \in \eta_{R}(\wtd{U}(R_{q}))$ and that
$\theta_{q}(e_{i}) = \frac{1}{q^{2} - q^{-2}}\ell_{i+1,i}^{+}(\ell_{ii}^{+})^{-1}$, as desired.
The verifications for $e_{n}$ and $f_{i}$ with $1 \le i \le n$ are similar. For $\omega_{i}$, we have
\begin{align*}
  (\theta_{q}' \otimes g)\eta(\omega_{i})
  &= (\theta_{q}' \otimes g)(K_i \otimes L_{\alpha_{i}}^{2}) = \ell_{i+1,i+1}^{+}(\ell_{ii}^{+})^{-1} \otimes L_{2\varepsilon_{i} - 2\varepsilon_{i+1}} \\
  &= (\ell_{i+1,i+1}^{+} \otimes L_{-2\varepsilon_{i+1}}) (\ell_{ii}^{+} \otimes L_{-2\varepsilon_{i}})^{-1}
   = \eta_{R}(\ell_{i+1,i+1}^{+}) \eta_{R}(\ell_{ii}^{+})^{-1},
\end{align*}
so that $\theta_{q}(\omega_{i}) = \ell_{i+1,i+1}^{+} (\ell_{ii}^{+})^{-1}$. Finally, for $\omega_{i}'$, we have
\begin{align*}
  (\theta_{q}' \otimes g)\eta(\omega_{i}')
  &= (\theta_{q}' \otimes g)(K_i^{-1} \otimes L_{\alpha_{i}}^{2})
   = (\ell_{ii}^{-})^{-1} \ell_{i+1,i+1}^{-} \otimes L_{2\varepsilon_{i} - 2\varepsilon_{i+1}} \\
  &= (\ell_{ii}^{-} \otimes L_{-2\varepsilon_{i}})^{-1} (\ell_{i+1,i+1}^{-} \otimes L_{-2\varepsilon_{i+1}})
   = \eta_R(\ell_{ii}^{-})^{-1} \eta_{R}(\ell_{i+1,i+1}^{-}),
\end{align*}
so that $\theta_{q}(\omega_{i}') = (\ell_{ii}^{-})^{-1} \ell_{i+1,i+1}^{-}$. This shows that $\theta_{q} = \eta_{R}^{-1}(\theta_{q}'\otimes g)\eta$
is well-defined, and it is injective because each of $\eta,\eta_R,\theta_{q}'$ is injective. For surjectivity, it suffices to show that all
$\ell_{ij}^{\pm}$ lie in the image of $\theta_{q}$. Since we clearly have all $\cal{E}_{i}$, $\cal{F}_{i}$, and $\{\ell_{ii}^{\pm}\}_{i=1}^n$
in the image of $\theta_{q}$ (due to the extra relation~\eqref{eq:rel_B_extra} in Remark~\ref{rmk:square_root_zq}),
the result follows from Lemma~\ref{lem:borel_spanning_sets}.

Therefore, $\theta_q$ is an algebra isomorphism.  Since one can easily check on generators that~\eqref{eq:Hopf_compatibility} holds,
$\theta_{q}$ is actually a Hopf algebra isomorphism, which completes the proof.
\end{proof}

Now, combining Corollary~\ref{cor:U(R)_2_vs_1_parameter} with Proposition~\ref{prop:twisted_algebra}, we can immediately derive the two-parameter generalization of
Theorem~\ref{thm:DJ=RTT_Btype_1param} and Corollary~\ref{cor:DJ=RTT_Btype_1param_double_cartan}:

\begin{theorem}\label{thm:DJ=RTT_Btype_2_param}
There is a unique Hopf algebra isomorphism $\theta_{r,s}\colon U_{r,s}(\frak{so}_{2n+1}) \iso \wtd{U}(R_{r,s})$ such that
\begin{align*}
  &\theta_{r,s}(\omega_{i}) = \ell_{i+1,i+1}^{+} (\ell_{ii}^{+})^{-1},   & & \theta_{r,s}(\omega_{i}') = (\ell_{ii}^{-})^{-1} \ell_{i+1,i+1}^{-} & &
    \text{for} \quad 1 \le i \le n, \\
  &\theta_{r,s}(e_{i}) = \frac{1}{r^{2} - s^{2}}\ell_{i+1,i}^{+}(\ell_{ii}^{+})^{-1}, & &
    \theta_{r,s}(f_{i}) = \frac{1}{s^{2} - r^{2}}(\ell_{ii}^{-})^{-1}\ell_{i,i+1}^{-} & &
     \text{for} \quad 1 \le i < n, \\
  &\theta_{r,s}(e_{n}) = \frac{r^{1/2}s}{[2]_{r,s}^{1/2}(r - s)}\ell_{n+1,n}^{+}(\ell_{nn}^{+})^{-1}, & &
   \theta_{r,s}(f_{n}) = \frac{r^{1/2}}{[2]_{r,s}^{1/2}(s - r)}(\ell_{nn}^{-})^{-1}\ell_{n,n+1}^{-}.
\end{align*}
\end{theorem}

\begin{proof}
Note that the map $\theta_{q}$ of Corollary~\ref{cor:DJ=RTT_Btype_1param_double_cartan} preserves $P$-bidegrees, so it induces a Hopf algebra isomorphism
$\theta_{q}\colon U_{q,\zeta}(\frak{so}_{2n+1}) \iso \wtd{U}(R_{q})_{\zeta}$. Then we define $\theta_{r,s} = \wtd{\Psi}^{-1} \circ \theta_{q}\circ \varphi$,
where $\wtd{\Psi}$ and $\varphi$ are the Hopf algebra isomorphisms of Corollary~\ref{cor:U(R)_2_vs_1_parameter} and Proposition~\ref{prop:twisted_algebra},
respectively. By definition, $\theta_{r,s}$ is a Hopf algebra isomorphism, so it only remains to show that it maps the generators of
$U_{r,s}(\frak{so}_{2n+1})$ to the elements indicated above. The latter is a straightforward computation, so we shall only carry it out for $e_{i}$
and $f_{i}$ with $i < n$, leaving the rest to the interested reader. For $e_{i}$ with $i<n$, we have:
\begin{multline*}
  \theta_{r,s}(e_{i}) = \wtd{\Psi}^{-1}(\theta_{q}(\varphi(e_{i})))
  = \wtd{\Psi}^{-1}\left (\frac{1}{q^{2} - q^{-2}} \ell_{i+1,i}^{+}(\ell_{ii}^{+})^{-1} \right ) \\
  = \frac{rs}{r^{2} - s^{2}}\zeta(\varepsilon_{i+1},\varepsilon_{i})\psi_{i+1}\psi_{i}^{-1}\ell_{i+1,i}^{+}(\ell_{ii}^{+})^{-1}
  = \frac{1}{r^{2} - s^{2}}\ell_{i+1,i}^{+}(\ell_{ii}^{+})^{-1},
\end{multline*}
where in the last equality we used $\psi_{i+1}\psi_{i}^{-1} = (rs)^{-1/2}$ and $\zeta(\varepsilon_{i+1},\varepsilon_{i}) = (rs)^{-1/2}$
(the former follows from~\eqref{eq:psi_values_B}, while the latter is due to~\cite[(5.33)]{MT1} and~\eqref{eq:zeta_formula}).
Similarly, for $f_{i}$ with $i<n$, we get:
\begin{multline*}
  \theta_{r,s}(f_{i}) = \wtd{\Psi}^{-1}(\theta_{q}(\varphi(f_{i})))
  = \wtd{\Psi}^{-1}\left ((rs)^{-1}\frac{1}{q^{-2} - q^{2}}(\ell_{ii}^{-})^{-1}\ell_{i,i+1}^{-}  \right ) \\
  = \frac{1}{s^{2} - r^{2}}\zeta(\varepsilon_{i+1},\varepsilon_{i})\psi_{i}\psi_{i+1}^{-1}(\ell_{ii}^{-})^{-1}\ell_{i,i+1}^{-}
  = \frac{1}{s^{2} - r^{2}}(\ell_{ii}^{-})^{-1}\ell_{i,i+1}^{-}.
\end{multline*}
This completes the proof.
\end{proof}

\begin{remark}\label{rmk:JLM_comparison_B}
The formulas above look slightly different to those of~\cite{JLM1} for two reasons. First, they use $q_{i} = q$ for $i < n$ and $q_{n} = q^{1/2}$,
while we use $q_{i} = q^{2}$ for $i < n$ and $q_{n} = q$, which affects the matrix $C_{q}$ and the formulas of Theorem~\ref{thm:DJ=RTT_Btype_1param}.
Second, the formula for their $R$-matrix $R$ is slightly different from our $R_{q}$, which affects the value of $\rho_{n+1}$ in~\eqref{eq:rho_B}.
However, their $R$-matrix is related to ours via $R_{q^{1/2}} = SRS^{-1}$ with
$S = \sum_{i,j=1}^{N} q^{\frac{1}{4}(\delta_{i,n+1} + \delta_{j,n+1})}E_{ii} \otimes E_{jj}$, so we can pass from their isomorphism to ours by
replacing $q$ with $q^{2}$ and using Proposition~\ref{prop:diagonal_conjugation_U(R)}(b).
We also note that~\cite{JLM1} seem to forget to impose~\eqref{eq:rel_B_extra} in type~$B_n$.
\end{remark}

   %%%%%%%%%%%%%%%%%%%%%%%%%%%%%%%%%%%%%%%%%%%%%%%%%%%%%%%%%%%%%%%%%%%%%%%%%%

\noindent
$\bullet$
\textbf{Type $C_n$.}

Let $R_{q} = \hat{R}_{q} \circ \tau$, where $\hat{R}_{q}$ is the one-parameter $C_{n}$-type $R$-matrix of~\eqref{eq:1_parameter_R_C}. Similarly,
let $R_{r,s} = \hat{R}_{r,s} \circ \tau$, where $\hat{R}_{r,s}$ is the two-parameter $C_{n}$-type $R$-matrix of~\cite[(4.9)]{MT1}. In this Subsection,
we recall the relationship between $\wtd{U}'(R_{q})$ and $U_{q}(\frak{sp}_{2n})$, and using the methods of the previous Subsection, upgrade it
to the one between $\wtd{U}(R_{r,s})$ and $U_{r,s}(\frak{sp}_{2n})$, which comes naturally from the twisting procedure developed above.
This corrects~\cite[Theorem 9]{HJZ}.

For type $C_{n}$, we actually need to work with an extension of $U_{r,s}(\frak{sp}_{2n})$, which is obtained by adjoining generators
$\omega_{n}^{1/2},(\omega_{n}')^{1/2}$, and imposing the obvious additional relations. We also need a similarly defined extension of
$U_{q}(\frak{sp}_{2n})$. Throughout this Subsection, we shall always assume that $U_{r,s}(\frak{sp}_{2n})$ contains these additional
elements (including the case $r = q$, $s = q^{-1}$), and we make a similar assumption for $U_{q}(\frak{sp}_{2n})$.
We then have the following result (see~\cite[Main Theorem]{JLM2}):

\begin{theorem}\label{thm:DJ=RTT_Ctype_1param}
There is a unique Hopf algebra isomorphism $\theta_{q}'\colon U_{q}(\frak{sp}_{2n}) \iso \wtd{U}'(R_{q})$ such that
\begin{align*}
  &\theta_{q}'(K_{i}) = \ell_{i+1,i+1}^{+} (\ell_{ii}^{+})^{-1}, & & \theta_{q}'(K_{n}^{1/2}) = (\ell_{nn}^{+})^{-1} & &
    \text{for} \quad 1 \le i < n, \\
  &\theta_{q}'(E_{i}) = \frac{1}{q - q^{-1}}\ell_{i+1,i}^{+}(\ell_{ii}^{+})^{-1}, & &
   \theta_{q}'(F_{i}) = \frac{1}{q^{-1} - q}(\ell_{ii}^{-})^{-1}\ell_{i,i+1}^{-}  & &
    \text{for} \quad 1 \le i < n, \\
  &\theta_{q}'(E_{n}) = \frac{1}{q^{2} - q^{-2}}\ell_{n+1,n}^{+}(\ell_{nn}^{+})^{-1},
    & & \theta_{q}'(F_{n}) = \frac{1}{q^{-2} - q^{2}}(\ell_{nn}^{-})^{-1}\ell_{n,n+1}^{-}.
\end{align*}
\end{theorem}

\begin{proof}
First, we note that $\ell_{n+1,n+1}^{\pm}\ell_{nn}^{\pm} = \ell_{nn}^{\pm}\ell_{n+1,n+1}^{\pm} = 1$ in $\wtd{U}'(R_{q})$
by~\eqref{eq:diagonal_reflection}, and thus, by composing the restriction of the isomorphism of~\cite[Main Theorem]{JLM2}
with the Cartan involution~\eqref{eq:cartan_involution}, we get the algebra isomorphism $\theta_{q}'$. It is easy to
verify~\eqref{eq:Hopf_compatibility} on generators, so that $\theta_{q}'$ is in fact a Hopf algebra isomorphism.
\end{proof}

Similarly to Corollary~\ref{cor:DJ=RTT_Btype_1param_double_cartan}, we can upgrade this result to a Cartan-doubled version
by applying Propositions~\ref{prop:double_cartan_DJ} and~\ref{prop:double_cartan_RTT_finite}, where we extend $\eta$ via
\begin{align*}
  \eta\colon \quad
  \omega_{n}^{1/2}\mapsto K_{n}^{1/2} \otimes L_{\alpha_{n}},\quad
 (\omega_{n}')^{1/2} \mapsto K_{n}^{-1/2} \otimes L_{\alpha_{n}}.
\end{align*}

\begin{cor}\label{thm:DJ=RTT_Ctype_1param_double_cartan}
There is a unique Hopf algebra isomorphism $\theta_{q}\colon U_{q,q^{-1}}(\frak{sp}_{2n}) \iso \wtd{U}(R_{q})$ such that
\begin{align*}
  &\theta_{q}(\omega_{i}) = \ell_{i+1,i+1}^{+} (\ell_{ii}^{+})^{-1}, & & \theta_{q}(\omega_{i}') = (\ell_{ii}^{-})^{-1} \ell_{i+1,i+1}^{-} & &
    \text{for} \quad 1 \le i < n, \\
 &\theta_{q}(\omega_{n}^{1/2}) = (\ell_{nn}^{+})^{-1}, & & \theta_{q}((\omega_{n}')^{1/2}) = (\ell_{nn}^{-})^{-1}, \\
  &\theta_{q}(e_{i}) = \frac{1}{q - q^{-1}}\ell_{i+1,i}^{+}(\ell_{ii}^{+})^{-1}, & &
   \theta_{q}(f_{i}) = \frac{1}{q^{-1} - q}(\ell_{ii}^{-})^{-1}\ell_{i,i+1}^{-} & &
    \text{for} \quad 1 \le i < n, \\
  &\theta_{q}(e_{n}) = \frac{1}{q^{2} - q^{-2}}\ell_{n+1,n}^{+}(\ell_{nn}^{+})^{-1}, &
  & \theta_{q}(f_{n}) = \frac{1}{q^{-2} - q^{2}}(\ell_{nn}^{-})^{-1}\ell_{n,n+1}^{-}.
\end{align*}
\end{cor}

Combining this result with Corollary~\ref{cor:U(R)_2_vs_1_parameter}, we obtain the two-parameter generalization of Theorem~\ref{thm:DJ=RTT_Ctype_1param}:

\begin{theorem}\label{thm:DJ=RTT_Ctype_2param}
There is a unique Hopf algebra isomorphism $\theta_{r,s}\colon U_{r,s}(\frak{sp}_{2n}) \iso \wtd{U}(R_{r,s})$ such that
\begin{align*}
  &\theta_{r,s}(\omega_{i}) = \ell_{i+1,i+1}^{+} (\ell_{ii}^{+})^{-1}, & & \theta_{r,s}(\omega_{i}') = (\ell_{ii}^{-})^{-1} \ell_{i+1,i+1}^{-} & &
    \text{for} \quad 1 \le i < n, \\
  &\theta_{r,s}(\omega_{n}^{1/2}) = (\ell_{nn}^{+})^{-1}, & & \theta_{r,s}((\omega_{n}')^{1/2}) = (\ell_{nn}^{-})^{-1}, \\
  &\theta_{r,s}(e_{i}) = \frac{1}{r - s}\ell_{i+1,i}^{+}(\ell_{ii}^{+})^{-1}, & &
   \theta_{r,s}(f_{i}) = \frac{1}{s - r}(\ell_{ii}^{-})^{-1}\ell_{i,i+1}^{-} & &
    \text{for} \quad 1 \le i < n, \\
  &\theta_{r,s}(e_{n}) = \frac{rs}{r^{2} - s^{2}}\ell_{n+1,n}^{+}(\ell_{nn}^{+})^{-1}, &
  & \theta_{r,s}(f_{n}) = \frac{1}{s^{2} - r^{2}}(\ell_{nn}^{-})^{-1}\ell_{n,n+1}^{-}.
\end{align*}
\end{theorem}

\begin{proof}
Note that the map $\theta_{q}$ of Corollary~\ref{thm:DJ=RTT_Ctype_1param_double_cartan} preserves $P$-bidegrees, so it induces a Hopf algebra isomorphism
$\theta_{q}\colon U_{q,\zeta}(\frak{sp}_{2n}) \iso \wtd{U}(R_{q})_{\zeta}$. We define $\theta_{r,s} = \wtd{\Psi}^{-1} \circ \theta_{q} \circ \varphi$
with $\wtd{\Psi}$ of Corollary~\ref{cor:U(R)_2_vs_1_parameter} and $\varphi$ of Proposition~\ref{prop:twisted_algebra}. Then $\theta_{r,s}$ is a Hopf
algebra isomorphism, and the $\theta_{r,s}$-images of the generators are computed exactly as in the proof of Theorem~\ref{thm:DJ=RTT_Btype_2_param}.
\end{proof}

   %%%%%%%%%%%%%%%%%%%%%%%%%%%%%%%%%%%%%%%%%%%%%%%%%%%%%%%%%%%%%%%%%%%%%%%%%%

\noindent
$\bullet$
\textbf{Type $D_n$.}

Let $R_{q} = \hat{R}_{q} \circ \tau$, where $\hat{R}_{q}$ is the one-parameter $D_{n}$-type $R$-matrix of~\eqref{eq:1_parameter_R_D}, and let
$R_{r,s} = \hat{R}_{r,s} \circ \tau$, where $\hat{R}_{r,s}$ is the two-parameter $D_{n}$-type $R$-matrix of~\cite[(4.11)]{MT1}. Following
the above procedure applied in types $B_n$ and $C_n$, we use the relationship between $\wtd{U}'(R_{q})$ and $U_{q}(\frak{so}_{2n})$ to obtain
the corresponding relationship between $\wtd{U}(R_{r,s})$ and $U_{r,s}(\frak{so}_{2n})$, which corrects~\cite[Theorem 4.3]{HXZ2}.

As in type $C_{n}$, we shall need to work with an extended version of $U_{r,s}(\frak{so}_{2n})$, which is obtained by adjoining additional
generators $(\omega_{n-1}\omega_{n})^{1/2},(\omega_{n-1}^{-1}\omega_{n})^{1/2},(\omega_{n-1}'\omega_{n}')^{1/2},((\omega_{n-1}')^{-1}\omega_{n}')^{1/2}$,
and imposing the obvious additional relations. We shall also make a similar extension to $U_{q}(\frak{so}_{2n})$. For the remainder of this Subsection,
we will assume that $U_{r,s}(\frak{so}_{2n})$ contains these additional elements (including the case $r=q, s=q^{-1}$), and we make
similar assumptions for $U_{q}(\frak{so}_{2n})$. We then have the following result (see~\cite[Main Theorem]{JLM1}):

\begin{theorem}\label{thm:DJ=RTT_Dtype_1_param}
There is a unique Hopf algebra isomorphism $\theta_{q}'\colon U_{q}(\frak{so}_{2n}) \iso \wtd{U}'(R_{q})$ such that
\begin{align*}
  &\theta_{q}'(K_{i}) = \ell_{i+1,i+1}^{+} (\ell_{ii}^{+})^{-1} & & ~ & &
    \text{for} \quad 1 \le i < n, \\
  &\theta_{q}'\left ((K_{n-1}K_{n})^{1/2}\right ) = (\ell_{n-1,n-1}^{+})^{-1},
    & & \theta_{q}'\left ((K_{n-1}^{-1} K_{n})^{1/2}\right ) = (\ell_{nn}^{+})^{-1}, \\
  &\theta_{q}'(E_{i}) = \frac{1}{q - q^{-1}}\ell_{i+1,i}^{+}(\ell_{ii}^{+})^{-1}, & &
   \theta_{q}'(F_{i}) = \frac{1}{q^{-1} - q}(\ell_{ii}^{-})^{-1}\ell_{i,i+1}^{-}  & &
    \text{for} \quad 1 \le i < n, \\
  &\theta_{q}'(E_{n}) = \frac{1}{q - q^{-1}}\ell_{n+1,n-1}^{+}(\ell_{n-1,n-1}^{+})^{-1},
    & & \theta_{q}'(F_{n}) = \frac{1}{q^{-1} - q}(\ell_{n-1,n-1}^{-})^{-1}\ell_{n-1,n+1}^{-}.
\end{align*}
\end{theorem}

\begin{proof}
We first note that $\ell_{n+1,n+1}^{\pm}\ell_{nn}^{\pm} = \ell_{nn}^{\pm}\ell_{n+1,n+1}^{\pm} = 1$ in $\wtd{U}'(R_{q})$ by~\eqref{eq:diagonal_reflection},
hence, composing the restriction of the isomorphism from~\cite[Main Theorem]{JLM1} with the Cartan involution~\eqref{eq:cartan_involution} implies
that the above assignment gives rise to an algebra isomorphism $\theta_{q}'$. It is easy to verify~\eqref{eq:Hopf_compatibility} on generators,
so that $\theta_{q}'$ is actually a Hopf algebra isomorphism.
\end{proof}

As for types $B_{n}$ and $C_{n}$, we can upgrade this result to a Cartan-doubled version by applying
Propositions~\ref{prop:double_cartan_DJ} and~\ref{prop:double_cartan_RTT_finite}, where we extend $\eta$ via
\begin{align*}
  \eta\colon
  & (\omega_{n-1}\omega_{n})^{1/2} \mapsto (K_{n-1}K_{n})^{1/2} \otimes L_{\alpha_{n-1}+\alpha_{n}},
    & & (\omega_{n-1}^{-1}\omega_{n})^{1/2} \mapsto (K_{n-1}^{-1}K_{n})^{1/2}\otimes L_{-\alpha_{n-1}+\alpha_{n}}, \\
  & (\omega_{n-1}'\omega_{n}')^{1/2} \mapsto (K_{n-1}K_{n})^{-1/2} \otimes L_{\alpha_{n-1}+\alpha_{n}},
   & & ((\omega_{n-1}')^{-1}\omega_{n}')^{1/2} \mapsto (K_{n-1}^{-1}K_{n})^{-1/2} \otimes L_{-\alpha_{n-1}+\alpha_{n}}.
\end{align*}

\begin{cor}\label{thm:DJ=RTT_Dtype_1_param_double_cartan}
There is a unique Hopf algebra isomorphism $\theta_{q}\colon U_{q,q^{-1}}(\frak{so}_{2n}) \iso \wtd{U}(R_{q})$ such that
\begin{align*}
  &\theta_{q}(\omega_{i}) = \ell_{i+1,i+1}^{+} (\ell_{ii}^{+})^{-1}, & & \theta_{q}(\omega_{i}') = (\ell_{ii}^{-})^{-1} \ell_{i+1,i+1}^{-} & &
    \text{for} \quad 1 \le i < n, \\
  &\theta_{q}\left ((\omega_{n-1}\omega_{n})^{1/2}\right ) = (\ell_{n-1,n-1}^{+})^{-1},
    & &  \theta_{q}\left ((\omega_{n-1}'\omega_{n})^{1/2}\right ) = (\ell_{n-1,n-1}^{-})^{-1},\\
  &\theta_{q}\left ((\omega_{n-1}^{-1}\omega_{n})^{1/2}\right ) = (\ell_{nn}^{+})^{-1},
    & & \theta_{q}\left (((\omega_{n-1}')^{-1}\omega_{n}')^{1/2}\right ) = (\ell_{nn}^{-})^{-1}, \\
  &\theta_{q}(e_{i}) = \frac{1}{q - q^{-1}}\ell_{i+1,i}^{+}(\ell_{ii}^{+})^{-1}, & &
   \theta_{q}(f_{i}) = \frac{1}{q^{-1} - q}(\ell_{ii}^{-})^{-1}\ell_{i,i+1}^{-} & &
    \text{for} \quad 1 \le i < n, \\
  &\theta_{q}(e_{n}) = \frac{1}{q - q^{-1}}\ell_{n+1,n-1}^{+}(\ell_{n-1,n-1}^{+})^{-1},
    & & \theta_{q}(f_{n}) = \frac{1}{q^{-1} - q}(\ell_{n-1,n-1}^{-})^{-1}\ell_{n-1,n+1}^{-}.
\end{align*}
\end{cor}

Similarly to Theorems~\ref{thm:DJ=RTT_Btype_2_param} and~\ref{thm:DJ=RTT_Ctype_2param}, we then obtain a two-parameter analogue of
Theorem~\ref{thm:DJ=RTT_Dtype_1_param} and Corollary~\ref{thm:DJ=RTT_Dtype_1_param_double_cartan}:

\begin{theorem}\label{thm:DJ=RTT_Dtype_2param}
There is a unique Hopf algebra isomorphism $\theta_{r,s}\colon U_{r,s}(\frak{so}_{2n}) \iso \wtd{U}(R_{r,s})$ such that
\begin{align*}
  &\theta_{r,s}(\omega_{i}) = \ell_{i+1,i+1}^{+} (\ell_{ii}^{+})^{-1}, & & \theta_{r,s}(\omega_{i}') = (\ell_{ii}^{-})^{-1} \ell_{i+1,i+1}^{-} & &
    \text{for} \quad 1 \le i < n, \\
  &\theta_{r,s}\left ((\omega_{n-1}\omega_{n})^{1/2}\right ) = (\ell_{n-1,n-1}^{+})^{-1},
    & & \theta_{r,s}\left ((\omega_{n-1}'\omega_{n}')^{1/2}\right ) = (\ell_{n-1,n-1}^{-})^{-1}, \\
  &\theta_{r,s}\left ((\omega_{n-1}^{-1}\omega_{n})^{1/2}\right ) = (\ell_{nn}^{+})^{-1},
    & & \theta_{r,s}\left (\left ((\omega_{n-1}')^{-1}\omega_{n}'\right )^{1/2}\right ) = (\ell_{nn}^{-})^{-1}, \\
  &\theta_{r,s}(e_{i}) = \frac{1}{r - s}\ell_{i+1,i}^{+}(\ell_{ii}^{+})^{-1}, & &
   \theta_{r,s}(f_{i}) = \frac{1}{s - r}(\ell_{ii}^{-})^{-1}\ell_{i,i+1}^{-} & &
    \text{for} \quad 1 \le i < n, \\
  &\theta_{r,s}(e_{n}) = \frac{1}{r - s}\ell_{n+1,n-1}^{+}(\ell_{n-1,n-1}^{+})^{-1},
    & & \theta_{r,s}(f_{n}) = \frac{rs}{s - r}(\ell_{n-1,n-1}^{-})^{-1}\ell_{n-1,n+1}^{-}.
\end{align*}
\end{theorem}

\begin{remark}\label{rmk:JLM_comparison}
Comparing our Theorems~\ref{thm:DJ=RTT_Btype_2_param},~\ref{thm:DJ=RTT_Ctype_2param},~\ref{thm:DJ=RTT_Dtype_2param} to those of~\cite{HJZ,HXZ1,HXZ2},
first let us note that the latter references had errors in not imposing the extra relations~\eqref{eq:LC_relations_2_parameter}
and~\eqref{eq:rel_B_extra}, and second that their formulas look slightly different from our $\theta_{r,s}^{-1}$ as they consider $U(R)$
with $R=\tau\circ R_{r,s}\circ \tau$ instead.
\end{remark}

   %%%%%%%%%%%%%%%%%%%%%%%%%%%%%%%%%%%%%%%%%%%%%%%%%%%%%%%%%%%%%%%%%%%%%%%%%%
   %%%%%%%%%%%%%%%%%%%%%%%%%%%%%%%%%%%%%%%%%%%%%%%%%%%%%%%%%%%%%%%%%%%%%%%%%%
   %%%%%%%%%%%%%%%%%%%%%%%%%%%%%%%%%%%%%%%%%%%%%%%%%%%%%%%%%%%%%%%%%%%%%%%%%%

\section{Two-parameter quantum affine groups}\label{sec:loop-realization}

In this Section, we recall the notion of two-parameter quantum affine groups. We realize them as bicharacter twists of
the one-parameter ones, which allows to establish the isomorphism between their two realizations.

   %%%%%%%%%%%%%%%%%%%%%%%%%%%%%%%%%%%%%%%%%%%%%%%%%%%%%%%%%%%%%%%%%%%%%%%%%%

\subsection{The new Drinfeld realization}\label{ssec:loop-realization}
\

We first recall the definition of the new Drinfeld realization $U_{r,s}^{D}(\what{\fg})$ in the two-parameter case.
Below, we denote the additional simple root of the affine root system corresponding to $\Phi$ by $\alpha_{0}$, and we set $\delta = \alpha_{0} + \theta$, where $\theta \in \Phi^{+}$ is the highest root.
Moreover, in the formulas below, we extend the Ringel form to $\what{Q} = \bb{Z}\alpha_{0} \oplus Q$ via
\[
  \langle c\alpha_{0} + \mu,c'\alpha_{0} + \mu'\rangle = \langle -c\theta + \mu,-c'\theta + \mu'\rangle
  \quad \text{for all} \quad \mu,\mu' \in Q,\ c,c' \in \bb{Z}.
\]

\begin{definition}\label{def:drinfeld_2_parameter}
The \textbf{two-parameter quantum group of $\what{\fg}$ in the new Drinfeld presentation} is the associative $\bb{K}$-algebra
$U_{r,s}^{D}(\what{\fg})$ with generators
  $\{\omega_{i}^{\pm 1},(\omega_{i}')^{\pm 1},x_{i,m}^{\pm},a_{i,\ell} \,|\, 1 \le i \le n,\, m \in \bb{Z},\, \ell \in \bb{Z} \setminus \{0\}\}
   \cup \{\gamma^{\pm 1/2},(\gamma')^{\pm 1/2}\}$,
subject to the following relations (for all $1 \le i, j \le n$, $m,m' \in \bb{Z}$, $\ell,\ell' \in \bb{Z} \setminus \{0\}$):
\begin{equation}\label{eq:D1}
  [\omega_{i},\omega_{j}] = [\omega_{i},\omega_{j}'] = [\omega_{i}',\omega_{j}'] = 0,\qquad
  \omega_{i}^{\pm 1}\omega_{i}^{\mp 1} = 1 = (\omega_{i}')^{\pm 1}(\omega_{i}')^{\mp 1},
\end{equation}
\begin{equation}\label{eq:D2}
  \gamma^{1/2}\ \text{and}\ (\gamma')^{1/2}\ \text{are central}, \qquad
  \gamma^{\pm 1/2}\gamma^{\mp 1/2} = 1 = (\gamma')^{\pm 1/2}(\gamma')^{\mp 1/2},
\end{equation}
\begin{equation}\label{eq:D3}
  [\omega_{i},a_{j,\ell}] = [\omega_{i}',a_{j,\ell}] = 0,
\end{equation}
\begin{equation}\label{eq:D4}
  \omega_{i}x_{j,m}^{\pm } = r^{\pm \langle \alpha_{j},\alpha_{i}\rangle}s^{\mp \langle \alpha_{i},\alpha_{j}\rangle}x_{j,m}^{\pm }\omega_{i},\qquad
  \omega_{i}'x_{j,m}^{\pm} = r^{\mp \langle \alpha_{i},\alpha_{j}\rangle}s^{\pm \langle \alpha_{j},\alpha_{i}\rangle}x_{j,m}^{\pm}\omega_{i}',
\end{equation}
\begin{equation}\label{eq:D5}
  [a_{i,\ell},a_{j,\ell'}] = \delta_{\ell + \ell',0}(r_{i}s_{i})^{-\frac{\ell a_{ij}}{2}}\frac{[\ell a_{ij}]_{r_{i},s_{i}}}{|\ell|}
    (\gamma\gamma')^{\frac{|\ell|}{2}}\frac{\gamma^{|\ell|} - (\gamma')^{|\ell|}}{r_{j} - s_{j}},
\end{equation}
\begin{equation}\label{eq:D6}
  [a_{i,\ell},x_{j,m}^{\pm}] =
  \begin{cases}
    \pm \frac{1}{\ell}(r_{i}s_{i})^{-\frac{\ell a_{ij}}{2}}[\ell a_{ij}]_{r_{i},s_{i}}(\gamma')^{\pm \ell/2}(\gamma\gamma')^{\ell/2}x_{j,m + \ell}^{\pm}
      & \text{if}\ \ell > 0 \\
    \pm \frac{1}{\ell}(r_{i}s_{i})^{-\frac{\ell a_{ij}}{2}}[\ell a_{ij}]_{r_{i},s_{i}}\gamma^{\pm \ell/2}(\gamma\gamma')^{-\ell /2}x_{j,m + \ell}^{\pm}
      & \text{if}\ \ell < 0
  \end{cases},
\end{equation}
\begin{equation}\label{eq:D7}
\begin{split}
  &x_{i,m+1}^{\pm}x_{j,m'}^{\pm} - r^{\pm \langle \alpha_{j},\alpha_{i}\rangle}s^{\mp \langle \alpha_{i},\alpha_{j}\rangle}x_{j,m'}^{\pm}x_{i,m + 1}^{\pm} \\
  &\qquad = (rs)^{\pm \frac{1}{2}(\langle \alpha_{j},\alpha_{i}\rangle - \langle \alpha_{i},\alpha_{j}\rangle)}
     \left( r^{\pm \langle \alpha_{i},\alpha_{j}\rangle}s^{\mp \langle \alpha_{j},\alpha_{i}\rangle}x_{i,m}^{\pm}x_{j,m' + 1}^{\pm}
            - x_{j,m' + 1}^{\pm}x_{i,m}^{\pm} \right),
\end{split}
\end{equation}
\begin{equation}\label{eq:D8}
  [x_{i,m}^{+},x_{j,m'}^{-}] =
  \delta_{ij}\frac{(\gamma')^{-m}\gamma^{-\frac{m + m'}{2}}\omega_{i,m+m'} - \gamma^{m'}(\gamma')^{\frac{m + m'}{2}}\omega_{i,m+m'}'}{r_{i} - s_{i}},
\end{equation}
and, finally, for any $i \neq j$, $m \in \bb{Z}$, and sequence $(d_{1},\ldots ,d_{N})$ of integers (where $N = 1 - a_{ij}$):
\begin{equation}\label{eq:D9+}
  \sum_{\pi \in S_{N}}\sum_{k = 0}^{N} (-1)^{k}\qbinom{N}{k}_{r_{i},s_{i}} (r_{i}s_{i})^{\frac{1}{2}k(k-1)} (rs)^{k\langle \alpha_{j},\alpha_{i}\rangle}
  x_{i,d_{\pi(1)}}^{+}\ldots x_{i,d_{\pi(N-k)}}^{+}x_{j,m}^{+}x_{i,d_{\pi(N-k + 1)}}^{+}\ldots x_{i,d_{\pi(N)}}^{+} = 0,
\end{equation}
and
\begin{equation}\label{eq:D9-}
  \sum_{\pi \in S_{N}}\sum_{k = 0}^{N} (-1)^{k}\qbinom{N}{k}_{r_{i},s_{i}}(r_{i}s_{i})^{\frac{1}{2}k(k-1)}(rs)^{k\langle \alpha_{j},\alpha_{i}\rangle}
  x_{i,d_{\pi(1)}}^{-}\ldots x_{i,d_{\pi(k)}}^{-}x_{j,m}^{-}x_{i,d_{\pi(k+1)}}^{-}\ldots x_{i,d_{\pi(N)}}^{-} = 0.
\end{equation}
The elements $\omega_{i,m}$, $\omega_{i,-m}'$ appearing in~\eqref{eq:D8} are defined via $\omega_{i,m} = 0 = \omega_{i,-m}'$ for $m < 0$ and
\begin{align*}
  \omega_{i}(z) &= \sum_{m \geq 0} \omega_{i,m}z^{-m} = \omega_{i}\exp\left ((r_{i} - s_{i})\sum_{\ell > 0} a_{i,\ell}z^{-\ell}\right ), \\
  \omega_{i}'(z) &= \sum_{m \geq 0} \omega_{i,-m}'z^{m} = \omega_{i}'\exp\left (-(r_{i} - s_{i})\sum_{\ell > 0} a_{i,-\ell}z^{\ell}\right ).
\end{align*}
\end{definition}

\begin{remark}
(a) We note that~\eqref{eq:D5} corrects a typo in~\cite[(D2)]{HZ} but is in fact equivalent to~\cite[(3.5)]{HZ}
(as well as is compatible with~\cite[Definition 5.1]{HZ}).

\medskip
\noindent
(b) We note that~(\ref{eq:D9+},~\ref{eq:D9-}) correct~\cite[(D$9_2$, D$9_3$)]{HZ}, while~\cite[(D$9_1$)]{HZ} is a special case
of~(\ref{eq:D9+},~\ref{eq:D9-}).
\end{remark}

There are a few important subalgebras of $U_{r,s}^{D}(\what{\fg})$ that we will need later in this Section. We define
$U_{r,s}^{D,\pm} = U_{r,s}^{D,\pm}(\what{\fg})$ as the subalgebras of $U_{r,s}^{D}(\what{\fg})$ generated by
$\{x_{i,m}^{\pm} \,|\, 1 \le i \le n,\, m \in \bb{Z}\}$, and we define $U_{r,s}^{D,0} = U_{r,s}^{D,0}(\what{\fg})$ as the subalgebra generated by
  $\{\gamma^{\pm 1/2},(\gamma')^{\pm 1/2},\omega_{i}^{\pm 1},(\omega_{i}')^{\pm 1},a_{i,\ell} \,|\, 1 \le i \le n,\, \ell \in \bb{Z} \setminus \{0\}\}$.

The algebra $U_{r,s}^{D}(\what{\fg})$ is naturally $Q \times Q$-graded via
\begin{equation}\label{eq:Drinfeld-grading}
\begin{split}
  &\deg(\gamma^{1/2}) = \deg((\gamma')^{1/2}) = (0,0), \\
  &\deg(a_{i,\ell}) = (0,0)\quad \text{for all}\ 1 \le i \le n,\ \ell \in \bb{Z} \setminus \{0\}, \\
  &\deg(\omega_{i}) = \deg(\omega_{i}') = (\alpha_{i},-\alpha_{i})\quad \text{for all}\ 1 \le i \le n, \\
  &\deg(x_{i,m}^{+}) = (\alpha_{i},0),\ \deg(x_{i,m}^{-}) = (0,-\alpha_{i})\quad \text{for all}\ 1 \le i \le n,\ m \in \bb{Z}.
\end{split}
\end{equation}
We may thus define an algebra $U_{q,\zeta}^{D}(\what{\fg}) = (U_{q,q^{-1}}^{D}(\what{\fg}))_{\zeta}$ by changing the multiplication
via~\eqref{eq:twisted_product_general}.

Our first key result is:

\begin{theorem}\label{thm:twisted_algebra_loop}
Let $\zeta\colon Q \times Q \to \bb{K}$ be the bicharacter satisfying~\eqref{eq:zeta_formula}, and let $q = r^{1/2}s^{-1/2}$.
Then there is an algebra isomorphism $\varphi_{D}\colon U_{r,s}^{D}(\what{\fg}) \iso U_{q,\zeta}^{D}(\what{\fg})$ given by
\begin{equation*}
\begin{split}
  &\varphi_{D}(a_{i,\ell}) = (r_{i}s_{i})^{-1/2}a_{i,\ell}, \qquad \varphi_{D}(x_{i,m}^{+}) = x_{i,m}^{+},
    \qquad \varphi_{D}(x_{i,m}^{-}) = (r_{i}s_{i})^{-1/2}x_{i,m}^{-}, \\
  &\varphi_{D}(\omega_{i}) = \omega_{i}, \quad \varphi_{D}(\omega_{i}') = \omega_{i}',
    \quad \varphi_{D}(\gamma^{1/2}) = \gamma^{1/2}, \quad \varphi_{D}((\gamma')^{1/2}) = (\gamma')^{1/2},
\end{split}
\end{equation*}
for all $1 \le i \le n$, $\ell \in \bb{Z} \setminus \{0\}$, $m \in \bb{Z}$.
\end{theorem}

\begin{proof}
This result amounts to verifying that the relations of Definition~\ref{def:drinfeld_2_parameter} hold for the $\varphi_{D}$-images
of the same-named generators of $U_{r,s}^{D}(\what{\fg})$. Relations~\eqref{eq:D1}--\eqref{eq:D3} are immediate from the definition
of $\zeta$ and the grading~\eqref{eq:Drinfeld-grading}. For relation~\eqref{eq:D4}, we use the skew-symmetry of $\zeta$ and
equation~\eqref{eq:zeta_formula} to obtain
\begin{align*}
  \varphi_{D}(\omega_{i}) \circ \varphi_{D}(x_{j,m}^{+}) &= \omega_{i} \circ x_{j,m}^{+} =
  \zeta(\alpha_{i},\alpha_{j})\omega_{i}x_{j,m}^{+} = \zeta(\alpha_{i},\alpha_{j})q^{(\alpha_{i},\alpha_{j})}x_{j,m}^{+}\omega_{i} \\
  &= \zeta(\alpha_{i},\alpha_{j})^{2}q^{(\alpha_{i},\alpha_{j})}x_{j,m}^{+} \circ \omega_{i} = (\omega_{j}',\omega_{i})x_{j,m}^{+} \circ \omega_{i} =
   r^{\langle \alpha_{j},\alpha_{i}\rangle} s^{-\langle \alpha_{i},\alpha_{j}\rangle}\varphi_{D}(x_{j,m}^{+}) \circ \varphi_{D}(\omega_{i})
\end{align*}
as desired. Similarly,
\begin{align*}
  \varphi_{D}(\omega_{i}) \circ \varphi_{D}(x_{j,m}^{-}) &= (r_js_j)^{-1/2} \omega_{i} \circ x_{j,m}^{-}
  = (r_{j}s_{j})^{-1/2}\zeta(\alpha_{i},\alpha_{j})^{-1}\omega_{i}x_{j,m}^{-} \\
  &= (r_{j}s_{j})^{-1/2}\zeta(\alpha_{i},\alpha_{j})^{-1}q^{-(\alpha_{i},\alpha_{j})}x_{j,m}^{-}\omega_{i}
   = (r_{j}s_{j})^{-1/2}\zeta(\alpha_{i},\alpha_{j})^{-2}q^{-(\alpha_{i},\alpha_{j})}x_{j,m}^{-} \circ \omega_{i} \\
  &= r^{-\langle \alpha_{j},\alpha_{i}\rangle} s^{\langle \alpha_{i},\alpha_{j}\rangle} \varphi_{D}(x_{j,m}^{-}) \circ \varphi_{D}(\omega_{i}).
\end{align*}
Analogous computations show that the second relations in~\eqref{eq:D4} hold for $\varphi_D(x_{j,m}^{\pm})$ and $\varphi_D(\omega_{i}')$.

For~\eqref{eq:D5}, we get
\begin{equation*}
\begin{split}
  [\varphi_{D}(a_{i,\ell}),\varphi_{D}(a_{j,\ell'})] &= (r_{i}s_{i})^{-1/2}(r_{j}s_{j})^{-1/2}[a_{i,\ell},a_{j,\ell'}] \\
  &= (r_{i}s_{i})^{-1/2}(r_{j}s_{j})^{-1/2}\delta_{\ell+\ell',0}
     \frac{[\ell a_{ij}]_{q_{i}}}{|\ell|}(\gamma\gamma')^{\frac{|\ell|}{2}}\frac{\gamma^{|\ell|} - (\gamma')^{|\ell|}}{q_{j} - q_{j}^{-1}}.
\end{split}
\end{equation*}
Then since
\[
  [\ell a_{ij}]_{q_{i}} =
  \frac{r_{i}^{\frac{1}{2}\ell a_{ij}}s_{i}^{-\frac{1}{2}\ell a_{ij}} - r_{i}^{-\frac{1}{2}\ell a_{ij}}s_{i}^{\frac{1}{2}\ell a_{ij}}}
       {r_{i}^{1/2}s_{i}^{-1/2} - r_{i}^{-1/2}s_{i}^{1/2}} =
  (r_{i}s_{i})^{\frac{1}{2}(1 - \ell a_{ij})}
  \frac{r_{i}^{\ell a_{ij}} - s_{i}^{\ell a_{ij}}}{r_{i} - s_{i}} = (r_{i}s_{i})^{\frac{1}{2}(1 - \ell a_{ij})}[\ell a_{ij}]_{r_{i},s_{i}}
\]
and $\frac{1}{q_{j} - q_{j}^{-1}} = (r_{j}s_{j})^{1/2}\frac{1}{r_{j} - s_{j}}$, we find that~\eqref{eq:D5} holds
for $\varphi_{D}(a_{i,\ell})$ and $\varphi_{D}(a_{j,\ell'})$.

Next, for $\ell > 0$, we have
\begin{multline*}
  [\varphi_{D}(a_{i,\ell}), \varphi_{D}(x_{j,m}^{+})]
  = (r_{i}s_{i})^{-1/2}[a_{i,\ell},x_{j,m}^{+}]
  = \frac{1}{\ell}(r_{i}s_{i})^{-1/2}[\ell a_{ij}]_{q_{i}}(\gamma')^{\frac{\ell}{2}}(\gamma\gamma')^{\frac{\ell}{2}}x_{j,m + \ell}^{+} \\
  = \frac{1}{\ell}(r_{i}s_{i})^{-\frac{1}{2}\ell a_{ij}}[\ell a_{ij}]_{r_{i},s_{i}}(\gamma')^{\frac{\ell}{2}}(\gamma\gamma')^{\frac{\ell}{2}}x_{j,m + \ell}^{+}
  = \frac{1}{\ell}(r_{i}s_{i})^{-\frac{1}{2}\ell a_{ij}}[\ell a_{ij}]_{r_{i},s_{i}}(\gamma')^{\frac{\ell}{2}}
    (\gamma\gamma')^{\frac{\ell}{2}} \varphi_D(x_{j,m + \ell}^{+}),
\end{multline*}
so relation~\eqref{eq:D6} holds for $x^+_{j,m}$ when $\ell > 0$. The proof for $x^{-}_{j,m}$ as well as for $\ell < 0$ proceeds similarly.

To prove~\eqref{eq:D7}, we shall need the following simple identity:
\begin{equation}\label{eq:aux-7.7}
  (rs)^{\pm \frac{1}{2}(\langle \alpha_{j},\alpha_{i}\rangle - \langle \alpha_{i},\alpha_{j}\rangle)} =
  (\omega_{j}',\omega_{i})^{\pm 1/2}(\omega_{i}',\omega_{j})^{ \mp 1/2} \overset{\eqref{eq:zeta_formula}}{=}
  \zeta(\alpha_{i},\alpha_{j})^{\pm 1}\zeta(\alpha_{j},\alpha_{i})^{\mp 1} = \zeta(\alpha_{i},\alpha_{j})^{\pm 2}.
\end{equation}
Now, we have:
\begin{align*}
  &x_{i,m + 1}^{\pm} \circ x_{j,m'}^{\pm} - (\omega_{j}',\omega_{i})^{\pm 1}x_{j,m'}^{\pm} \circ x_{i,m + 1}^{\pm}
  = \zeta(\alpha_{i},\alpha_{j})^{\pm 1}x_{i,m + 1}^{\pm}x_{j,m'}^{\pm} - (\omega_{j}',\omega_{i})^{\pm 1}
    \zeta(\alpha_{i},\alpha_{j})^{\mp 1}x_{j,m'}^{\pm}x_{i,m+1}^{\pm} \\
%  &= \zeta(\alpha_{i},\alpha_{j})^{\pm 1}x_{i,m+1}^{\pm}x_{j,m'}^{\pm} - (\omega_{j}',\omega_{i})^{ \pm 1/2}q^{\pm \frac{1}{2}(\alpha_{i},\alpha_{j})}x_{j,m'}^{\pm}x_{i,m+1}^{\pm} \\
  &\overset{\eqref{eq:zeta_formula}}{=} \zeta(\alpha_{i},\alpha_{j})^{\pm 1}
    \left (x_{i,m+1}^{\pm}x_{j,m'}^{\pm} - q^{\pm (\alpha_{i},\alpha_{j})}x_{j,m'}^{\pm}x_{i,m+1}^{\pm}\right)
  = \zeta(\alpha_{i},\alpha_{j})^{\pm 1}\left(q^{\pm(\alpha_{i},\alpha_{j})}x_{i,m}^{\pm}x_{j,m'+1}^{\pm} - x_{j,m'+1}^{\pm}x_{i,m}^{\pm}\right ) \\
  &= \zeta(\alpha_{i},\alpha_{j})^{\pm 1}\left (q^{\pm (\alpha_{i},\alpha_{j})}\zeta(\alpha_{i},\alpha_{j})^{\mp 1}x_{i,m}^{\pm} \circ x_{j,m' + 1}^{\pm} -
     \zeta(\alpha_{i},\alpha_{j})^{\pm 1}x_{j,m'+1}^{\pm} \circ x_{i,m}^{\pm}\right ) \\
  &= \zeta(\alpha_{i},\alpha_{j})^{\pm 2}
     \left (q^{\pm (\alpha_{i},\alpha_{j})}\zeta(\alpha_{j},\alpha_{i})^{\pm 2}x_{i,m}^{\pm} \circ x_{j,m' + 1}^{\pm} - x_{j,m'+1}^{\pm} \circ x_{i,m}^{\pm}\right ) \\
  &\overset{\eqref{eq:aux-7.7}}{=} (rs)^{\pm \frac{1}{2}(\langle \alpha_{j},\alpha_{i}\rangle - \langle \alpha_{i},\alpha_{j}\rangle)}
     \left( (\omega_{i}',\omega_{j})^{\pm 1}x_{i,m}^{\pm} \circ x_{j,m'+1}^{\pm} - x_{j,m'+1}^{\pm}\circ x_{i,m}^{\pm} \right),
\end{align*}
so~\eqref{eq:D7} holds among the elements $\varphi_{D}(x_{i,m}^{\pm})$.

To prove~\eqref{eq:D8}, we first note that
\begin{align*}
  \exp \left( \pm (q_{i} - q_{i}^{-1}) \sum_{\ell = 1}^{\infty} a_{i,\pm \ell}u^{\mp\ell} \right)
  &= \exp \left( \pm \frac{r_{i} - s_{i}}{(r_{i}s_{i})^{1/2}}\sum_{\ell = 1}^{\infty} a_{i,\pm \ell}u^{\mp \ell} \right)
  &= \exp \left( \pm(r_{i} - s_{i})\sum_{\ell = 1}^{\infty}\varphi_{D}(a_{i,\pm \ell})u^{\mp \ell} \right),
\end{align*}
so that $\varphi_D(\omega_{i,m})=\omega_{i,m}$ and $\varphi_D(\omega'_{i,m})=\omega'_{i,m}$ for all $i,m$. Thus,
\begin{align*}
  [\varphi_{D}(x_{i,m}^{+}), \varphi_{D}(x_{j,m'}^{-})]
  &= (r_js_j)^{-1/2} (x_{i,m}^{+} \circ x_{j,m'}^{-} - x_{j,m'}^{-} \circ x_{i,m}^{+}) = (r_{j}s_{j})^{-1/2}[x_{i,m}^{+},x_{j,m'}^{-}] \\
  &= \delta_{ij}(r_{j}s_{j})^{-1/2}
     \frac{(\gamma')^{-m}\gamma^{-\frac{m + m'}{2}}\varphi_{D}(\omega_{i,m + m'}) - \gamma^{m'}(\gamma')^{\frac{m + m'}{2}}
     \varphi_{D}(\omega_{i,m+m'}')}{q_{i} - q_{i}^{-1}} \\
  &= \varphi_{D} \left(\delta_{ij}
     \frac{(\gamma')^{-m}\gamma^{-\frac{m + m'}{2}} \omega_{i,m + m'}- \gamma^{m'}(\gamma')^{\frac{m + m'}{2}} \omega_{i,m+m'}'}{r_{i} - s_{i}} \right),
\end{align*}
as desired.

Finally, let us verify that $\varphi_D$ is compatible with~\eqref{eq:D9+} and~\eqref{eq:D9-}.
To do so, we first note that for any integers $m',m_{1},\ldots ,m_{k + \ell}$, we have
\begin{multline*}
  x_{i,m_{1}}^{\pm} \circ \ldots \circ x_{i,m_{k}}^{\pm} \circ x_{j,m'}^{\pm} \circ x_{i,m_{k+1}}^{\pm} \circ \ldots \circ x_{i,m_{k + \ell}}^{\pm}
  = (x_{i,m_{1}}^{\pm}x_{i,m_{2}}^{\pm}\ldots x_{i,m_{k}}^{\pm}) \circ x_{j,m'}^{\pm} \circ (x_{i,m_{k+1}}^{\pm}\ldots x_{i,m_{k + \ell}}^{\pm}) \\
  = \zeta(\alpha_{i},\alpha_{j})^{\pm (k - \ell)}x_{i,m_{1}}^{\pm}\ldots x_{i,m_{k}}^{\pm}x_{j,m'}^{\pm}x_{i,m_{k+1}}^{\pm}\ldots x_{i,m_{k +\ell}}^{\pm}.
\end{multline*}
Evoking~\eqref{eq:zeta_formula} and the identity
\[
  \qbinom{N}{k}_{r_{i},s_{i}} =\, (r_{i}s_{i})^{\frac{k(N - k)}{2}}\qbinom{N}{k}_{q_{i}},
\]
we get for $N = 1 - a_{ij}$:
\begin{align*}
  & \qbinom{N}{k}_{r_{i},s_{i}} (r_{i}s_{i})^{\frac{1}{2}k(k - 1)} (rs)^{k\langle \alpha_{j},\alpha_{i}\rangle} \zeta(\alpha_{i},\alpha_{j})^{N - 2k} \\
  &= \qbinom{N}{k}_{q_{i}}(r_{i}s_{i})^{\frac{k(N - k)}{2}} (r_{i}s_{i})^{\frac{1}{2}k(k-1)} (rs)^{k\langle \alpha_{j},\alpha_{i}\rangle}
     (\omega_{j}',\omega_{i})^{\frac{N}{2}-k} q^{-\frac{1}{2}(N-2k)(\alpha_{i},\alpha_{j})} \\
  &= (\omega_{j}',\omega_{i})^{\frac{N}{2}} q^{-\frac{1}{2}N(\alpha_{i},\alpha_{j})}
     \left( \qbinom{N}{k}_{q_{i}}(r_{i}s_{i})^{-\frac{1}{2}ka_{ij}}(rs)^{k\langle \alpha_{j},\alpha_{i}\rangle} r^{-k\langle \alpha_{j},\alpha_{i}\rangle}
            s^{k \langle \alpha_{i},\alpha_{j}\rangle}r_{i}^{\frac{1}{2}ka_{ij}}s_{i}^{-\frac{1}{2}ka_{ij}} \right) \\
  &= (\omega_{j}',\omega_{i})^{\frac{N}{2}}q^{-\frac{1}{2}N(\alpha_{i},\alpha_{j})}\qbinom{N}{k}_{q_{i}},
\end{align*}
where we used $d_{i}a_{ij} = \langle \alpha_{i},\alpha_{j}\rangle + \langle \alpha_{j},\alpha_{i}\rangle$ in the last equality.
Thus, the above identities imply:
\begin{align*}
  \varphi_{D}& \left( \sum_{\pi \in S_{N}}\sum_{k = 0}^{N} (-1)^{k}\qbinom{N}{k}_{r_{i},s_{i}} (r_{i}s_{i})^{\frac{1}{2}k(k-1)}
  (rs)^{k\langle \alpha_{j},\alpha_{i}\rangle}
    x_{i,d_{\pi(1)}}^{+}\ldots x_{i,d_{\pi(N-k)}}^{+}x_{j,m}^{+}x_{i,d_{\pi(N-k + 1)}}^{+}\ldots x_{i,d_{\pi(N)}}^{+} \right) \\
  &= (\omega_{j}',\omega_{i})^{\frac{N}{2}} q^{-\frac{1}{2}N(\alpha_{i},\alpha_{j})} \sum_{\pi \in S_{N}}\sum_{k = 0}^{N} (-1)^{k}\qbinom{N}{k}_{q_{i}}
    x_{i,d_{\pi(1)}}^{+}\ldots x_{i,d_{\pi(N-k)}}^{+}x_{j,m}^{+}x_{i,d_{\pi(N-k + 1)}}^{+}\ldots x_{i,d_{\pi(N)}}^{+} = 0,
\end{align*}
and similarly
\begin{align*}
  \varphi_{D} & \left( \sum_{\pi \in S_{N}}\sum_{k = 0}^{N} (-1)^{k}\qbinom{N}{k}_{r_{i},s_{i}} (r_{i}s_{i})^{\frac{1}{2}k(k-1)}
  (rs)^{k\langle \alpha_{j},\alpha_{i}\rangle}
    x_{i,d_{\pi(1)}}^{-}\ldots x_{i,d_{\pi(k)}}^{-}x_{j,m}^{-}x_{i,d_{\pi(k+1)}}^{-}\ldots x_{i,d_{\pi(N)}}^{-} \right) \\
  &= (\omega_{j}',\omega_{i})^{\frac{N}{2}} q^{-\frac{1}{2}N(\alpha_{i},\alpha_{j})} \sum_{\pi \in S_{N}}\sum_{k = 0}^{N} (-1)^{k}\qbinom{N}{k}_{q_{i}}
    x_{i,d_{\pi(1)}}^{-}\ldots x_{i,d_{\pi(k)}}^{-}x_{j,m}^{-}x_{i,d_{\pi(k+1)}}^{-}\ldots x_{i,d_{\pi(N)}}^{-} = 0.
\end{align*}

This completes the proof.
\end{proof}

\begin{remark}\label{rem:finite-dimensional-2parameter}
As an immediate application, one can derive the main results of~\cite{JZ} from their standard one-parameter counterparts.
To this end, we note that if $V$ is a simple finite-dimensional $U_{q,q^{-1}}^{D}(\what{\fg})$-module, then $\omega_i\omega'_i$
act by some constants $\nu_i$. Therefore, twisting $V$ by an algebra automorphism of $U_{q,q^{-1}}^{D}(\what{\fg})$ sending
$x^\pm_{i,m}\mapsto \nu_i^{-1/4}\cdot x^\pm_{i,m}, \omega_i\mapsto \nu_i^{-1/2}\cdot \omega_i, \omega'_i\mapsto \nu_i^{-1/2}\cdot \omega'_i$
and $a_{i,\ell}\mapsto a_{i,\ell}, \gamma^{1/2}\mapsto \gamma^{1/2}, (\gamma')^{1/2}\mapsto (\gamma')^{1/2}$, we get a new
$U_{q,q^{-1}}^{D}(\what{\fg})$-module structure on $V$ that factors through the quotient
$U^D_{q,q^{-1}}(\what{\fg}) \twoheadrightarrow U^D_{q}(\what{\fg})$. The latter are classified
(up to further twisting by some algebra automorphisms, see~\cite[Proposition~12.2.3]{CP}) by Drinfeld polynomials, that is
$\Omega^{\pm}_i(z^{-1})v=q_i^{\deg(P_i)}\left(\frac{P_i(q_i^{-2}z)}{P_i(z)}\right)^{\pm}\cdot v$ (with $\Omega_{i}(z)$ as in~\eqref{eq:Phi_current})
for unique monic polynomials $P_i(z)$, where $v\in V$ is the pseudo-highest weight vector and $(\cdots)^{\pm}$ is a series expansion in $z^{\pm 1}$,
see~\cite[Theorem~12.2.6]{CP}. According to Theorem~\ref{thm:twisted_algebra_loop}, we have $\varphi_D(\omega_i(z^{-1}))=\omega_i(z^{-1})$ and
$\varphi_{D}(\omega_{i}'(z^{-1})) = \omega_{i}'(z^{-1})$, while evoking Proposition~\ref{prop:twisted_representation}, we see that
$(U_{q,q^{-1}}^{D}(\what{\fg}))_\zeta$-action $\cdot_\zeta$ on $V$ satisfies
$\omega_i(z^{-1})\cdot_{\zeta} v=\zeta(\alpha_i,\lambda)^2 q_i^{\deg(P_i)}\left(\frac{P_i(q_i^{-2}z)}{P_i(z)}\right)^{+}\cdot v$ and
$\omega_{i}'(z^{-1}) \cdot_{\zeta} v = \zeta(\alpha_{i},\lambda)^{2}q_{i}^{\deg(P_{i})}\left (\frac{P_{i}(q_{i}^{-2}z)}{P_{i}(z)}\right)^{-}\cdot v$,
with $\lambda=\sum_{i=1}^{n} \deg(P_i) \varpi_i$. Thus, up to above automorphisms, $\omega_i(z^{-1})$ and $\omega_{i}'(z^{-1})$ act via
$\left(\frac{P_i(s_i r_i^{-1} z)}{P_i(z)}\right)^{+}$ and $\left(\frac{P_i(s_i r_i^{-1} z)}{P_i(z)}\right)^{-}$, respectively, on the pseudo-highest
weight vector $v$. This result essentially recovers~\cite[Theorem~3.2]{JZ} (we note that an extra factor $(rs)^{\deg(P_i)}$ for $\epsilon=-$
in~\emph{loc.cit} is due to a slightly different presentation in~\cite[Definition~2.4]{JZ}). Likewise,~\cite[Proposition 3.6]{JZ} immediately
follows from the corresponding well-known result for $U_q^D(\what{\fg})$.
\end{remark}

\begin{remark}\label{rem:vertex-2parameter}
As another application of Theorem~\ref{thm:twisted_algebra_loop}, one can derive the construction of level-one vertex representations (\cite[\S5]{HZ})
of $U_{r,s}^{D}(\what{\fg})$ for simply-laced $\fg$ from the corresponding construction (\cite{FJ}) for $U_{q}^{D}(\what{\fg})$.
Besides a slightly different presentation in~\cite[Definition~3.4]{HZ}, this recovers~\cite[Theorem~5.2]{HZ}, up to further twisting by
an algebra automorphism $\varrho_{\lambda}$ (with $\lambda=(rs)^{1/2}$) of $U_{q,q^{-1}}^{D}(\what{\fg})$ defined on the generators by
\[
  \varrho_\lambda\colon
  \gamma\mapsto \lambda\gamma,\quad \gamma'\mapsto \lambda\gamma',\quad a_{i,\ell}\mapsto \lambda^{|\ell|}a_{i,\ell},\quad
  \omega_{i}\mapsto \omega_{i},\quad \omega'_{i}\mapsto \omega'_{i},\quad x^{+}_{i,m} \mapsto \lambda^{-m/2}x^+_{i,m},\quad
  x^{-}_{i,m} \mapsto \lambda^{m/2}x^-_{i,m}.
\]
In particular, after that twist central elements $\gamma,\gamma'$ act respectively by scalars $q\lambda=r, q^{-1}\lambda=s$ as in~\emph{loc.cit}.
\end{remark}

   %%%%%%%%%%%%%%%%%%%%%%%%%%%%%%%%%%%%%%%%%%%%%%%%%%%%%%%%%%%%%%%%%%%%%%%%%%

\subsection{Triangular decomposition and the key embedding: new Drinfeld realization}
\

We shall now prove an analogue of Propositions~\ref{prop:double_cartan_DJ} and~\ref{prop:double_cartan_RTT_finite} for the new Drinfeld realization,
which we will need to derive an algebra isomorphism $\Psi_{r,s}\colon U_{r,s}(\what{\fg}) \iso U_{r,s}^{D}(\what{\fg})$ from the corresponding
one-parameter isomorphism $\Psi_{q}\colon U_{q}(\what{\fg}) \iso U_{q}^{D}(\what{\fg})$ of~\cite{D}. To avoid confusion, we denote the generators
of $U_{q}^{D}(\what{\fg})$ by $\gamma^{\pm 1/2}$, $K_{i}^{\pm 1}$, $X_{i,m}^{\pm}$, and $H_{i,\ell}$, which correspond, respectively, to the generators
$\gamma^{\pm 1/2}$, $\omega_{i}^{\pm 1}$, $x_{i,m}^{\pm}$, and $a_{i,\ell}$ in $U_{q,q^{-1}}^{D}(\what{\fg})$. Note that $U_{q}^{D}(\what{\fg})$ is
the quotient of $U_{q,q^{-1}}^{D}(\what{\fg})$ by the ideal generated by $\omega_{i} - (\omega_{i}')^{-1}$ and $\gamma^{1/2} - (\gamma')^{-1/2}$.

We define subalgebras $U_{q}^{D,\pm} = U_{q}^{D,\pm}(\what{\fg})$ and $U_{q}^{D,0} = U_{q}^{D,0}(\what{\fg})$ of $U_{q}^{D}(\what{\fg})$ similarly
to the corresponding subalgebras $U_{r,s}^{D,\pm}$ and $U_{r,s}^{D,0}$ of $U_{r,s}^{D}(\what{\fg})$. We start with a couple of preliminary results,
the first of which is:

\begin{lemma}\label{lem:Q_hat_bigrading}
The algebra $U_{q,q^{-1}}^{D}(\what{\fg})$ is $\frac{1}{2}\what{Q} \times \frac{1}{2}\what{Q}$-graded via the assignment
\begin{equation*}
  \deg(\gamma^{1/2}) = \deg((\gamma')^{1/2}) = \left(-\frac{1}{2}\delta,\frac{1}{2}\delta\right),\qquad
  \deg(\omega_{i}) = \deg(\omega_{i}') = (\alpha_{i},-\alpha_{i}),
\end{equation*}
\begin{equation*}
  \deg(x_{i,m}^{+}) = (\alpha_{i} + m\delta,0),\qquad
  \deg(x_{i,m}^{-}) = (0,-\alpha_{i} + m\delta),
\end{equation*}
\begin{equation*}
  \deg(a_{i,\ell}) =
  \begin{cases}
    \left (-\frac{\ell}{2}\delta, \frac{3\ell}{2}\delta\right ) & \text{if}\ \ell > 0 \\
    \left (\frac{3\ell}{2}\delta,-\frac{\ell}{2}\delta\right) & \text{if}\ \ell < 0
  \end{cases}\,,
\end{equation*}
for all $1 \le i \le n$, $m \in \bb{Z}$, $\ell \in \bb{Z}\setminus \{0\}$.
\end{lemma}

\begin{proof}
One just needs to check that relations~\eqref{eq:D1}--\eqref{eq:D9-} are homogeneous with respect to the above grading. The only ones that
are not immediately clear are~\eqref{eq:D5},~\eqref{eq:D6}, and~\eqref{eq:D8}. The relation~\eqref{eq:D5} is obviously homogeneous if
$\ell' \neq -\ell$, and if $\ell' = -\ell$, then the left-hand side of~\eqref{eq:D5} has degree
  $\left (-\frac{|\ell|}{2}\delta,\frac{3|\ell|}{2}\delta\right ) + \left (-\frac{3|\ell|}{2}\delta,\frac{|\ell|}{2}\delta\right )
   = (-2|\ell|\delta,2|\ell|\delta)$,
which is easily seen to be the degree of the right-hand side of~\eqref{eq:D5}.

For~\eqref{eq:D6}, we have
\[
  \deg([a_{i,\ell},x_{j,m}^{+}]) =
  \begin{cases}
    (\alpha_{j} + (m - \frac{\ell}{2})\delta,\frac{3\ell}{2}\delta) & \text{if}\ \ell > 0 \\
    (\alpha_{j} + (m + \frac{3\ell}{2})\delta,-\frac{\ell}{2}\delta) & \text{if}\ \ell < 0
  \end{cases}\,.
\]
Since
\[
  \deg\left ((\gamma')^{\ell/2}(\gamma\gamma')^{\ell/2}x_{j,m + \ell}^{+}\right )
  = \left (-\tfrac{\ell}{2}\delta,\tfrac{\ell}{2}\delta\right ) + \left (-\ell \delta,\ell \delta\right ) + \left (\alpha_{j} + (m + \ell)\delta,0\right )
  = \left (\alpha_{j} + \left (m - \tfrac{\ell}{2}\right )\delta,\tfrac{3\ell}{2}\delta\right ),
\]
\[
  \deg\left (\gamma^{\ell/2}(\gamma\gamma')^{-\ell/2}x_{j,m+\ell}^{+}\right )
  = \left (-\tfrac{\ell}{2}\delta,\tfrac{\ell}{2}\delta\right ) + \left (\ell \delta,-\ell\delta \right ) + \left (\alpha_{j} + (m + \ell)\delta,0\right )
  = \left (\alpha_{j} + \left (m + \tfrac{3\ell}{2}\delta \right )\delta , -\tfrac{\ell}{2}\delta\right ),
\]
this shows that the relation~\eqref{eq:D6} for $x_{j,m}^{+}$ is homogeneous. The computation for $x_{j,m}^{-}$ is similar.

For~\eqref{eq:D8}, we may assume that $i = j$, as otherwise the relation is obviously homogeneous. Then the left-hand side has degree
$(\alpha_{i} + m\delta,-\alpha_{i} + m'\delta)$. To compute the degree of the right-hand side, we first note that, for all $k \ge 0$,
we have
\begin{align*}
  &\deg(\omega_{i,k}) = (\alpha_{i},-\alpha_{i}) + \left (-\tfrac{1}{2}k\delta,\tfrac{3}{2}k\delta\right )
    = \left (\alpha_{i} - \tfrac{1}{2}k\delta,-\alpha_{i} + \tfrac{3}{2}k\delta\right ), \\
  &\deg(\omega_{i,-k}') = (\alpha_{i},-\alpha_{i}) - \left (\tfrac{3}{2}k\delta,-\tfrac{1}{2}k\delta\right )
    = \left (\alpha_{i} - \tfrac{3}{2}k\delta,-\alpha_{i} + \tfrac{1}{2}k\delta\right ).
\end{align*}
If $m + m' > 0$, then the degree of the right-hand side of~\eqref{eq:D8} is
\begin{align*}
  \deg \left((\gamma')^{-m}\gamma^{-\frac{m + m'}{2}}\omega_{i,m+m'}\right)
  &= m(\delta,-\delta) + \tfrac{m + m'}{2}\left (\delta,-\delta\right ) +
     \left (\alpha_{i} - \tfrac{m + m'}{2}\delta,-\alpha_{i} + \tfrac{3(m + m')}{2}\delta\right ) \\
  &= (\alpha_{i} + m\delta, -\alpha_{i} + m'\delta),
\end{align*}
and if $m + m' < 0$, then the degree of the right-hand side of~\eqref{eq:D8} is
\begin{align*}
  \deg \left((\gamma)^{m'}(\gamma')^{\frac{m + m'}{2}}\omega_{i,m +m'}'\right)
  &= m'(-\delta,\delta) + \tfrac{m + m'}{2}(-\delta,\delta) + \left (\alpha_{i} + \tfrac{3(m + m')}{2}\delta,-\alpha_{i} - \tfrac{m + m'}{2}\delta\right )
    \\
  &= (\alpha_{i} + m\delta, -\alpha_{i} + m'\delta).
\end{align*}
If $m + m' = 0$, then the right-hand side of~\eqref{eq:D8} is a multiple of $(\gamma')^{-m}\omega_{i} - \gamma^{m'}\omega_{i}'$, which is
homogeneous of degree $(\alpha_{i} + m\delta, -\alpha_{i} - m\delta) = (\alpha_{i} + m\delta,-\alpha_{i} +m'\delta)$. Thus, relation~\eqref{eq:D8}
is also homogeneous.

This completes the proof.
\end{proof}

The next result we need is a triangular decomposition for $U_{q,q^{-1}}^{D}(\what{\fg})$:

\begin{lemma}\label{lem:triangular_decomposition_D}
(a) The multiplication map $U_{q,q^{-1}}^{D,-} \otimes U_{q,q^{-1}}^{D,0} \otimes U_{q,q^{-1}}^{D,+} \to U_{q,q^{-1}}^{D}(\what{\fg})$
is an isomorphism of vector spaces.

\medskip
\noindent
(b) The subalgebras $U_{q,q^{-1}}^{D,\pm}$ are isomorphic to the algebras generated by $\{x_{i,m}^{\pm}\}_{1 \le i \le n}^{m \in \bb{Z}}$
with defining relations~\eqref{eq:D7} and~\eqref{eq:D9+} for sign $+$ (resp.\ \eqref{eq:D9-} for sign $-$).

\medskip
\noindent
(c) The subalgebra $U_{q,q^{-1}}^{D,0}$ is isomorphic to the algebra generated by
$\{\gamma^{\pm 1/2},(\gamma')^{\pm 1/2},\omega_{i}^{\pm 1},(\omega_{i}')^{\pm 1},a_{i,\ell}\}_{1 \le i \le n}^{\ell \in \bb{Z} \setminus \{0\}}$,
with relations~\eqref{eq:D1},~\eqref{eq:D2},~\eqref{eq:D3}, and~\eqref{eq:D5}.
\end{lemma}

\begin{proof}
These results may be obtained by following the proof of~\cite[Theorem 3.2]{H}. We leave details to the reader.
\end{proof}

The last preliminary result we need amounts to a more refined triangular decomposition of $U_{q,q^{-1}}^{D}(\what{\fg})$.
To cleanly state the result, we introduce some additional notation. For any set $S = \{(i_{1},\ell_{1},N_{1}),\ldots ,(i_{k},\ell_{k},N_{k})\}$
with $1 \le i_{p} \le n$, $\ell_{p} \in \bb{Z}_{> 0}$, and $N_{p} \in \bb{Z}_{> 0}$, we define
\begin{equation}\label{eq:A_S_def}
  A_{S}^{\pm} = a_{i_{1},\pm \ell_{1}}^{N_{1}}\ldots a_{i_{k},\pm \ell_{k}}^{N_{k}} \in U_{q,q^{-1}}^{D,0}(\what{\fg})\qquad \text{and}\qquad
  H_{S}^{\pm} = H_{i_{1},\pm\ell_{1}}^{N_{1}}\ldots H_{i_{k},\pm\ell_{k}}^{N_{k}}\in U_{q}^{D,0}(\what{\fg}).
\end{equation}

\begin{lemma}\label{lem:cartan_basis_D}
The subalgebra $U_{q,q^{-1}}^{D,0}$ has a basis consisting of all $A_{S}^{-}\gamma^{m/2}\omega_{\mu}\omega_{\mu'}'(\gamma')^{m'/2}A_{S'}^{+}$
with $m,m' \in \bb{Z}$, $\mu,\mu' \in Q$, and sets $S,S'$ as above.
\end{lemma}

\begin{proof}
We shall only sketch the proof, since the arguments are standard.
The fact that the specified elements span $U_{q,q^{-1}}^{D,0}$ is obvious.
To prove their linear independence, we follow the method of~\cite[Proposition 4.16]{J}.

First, let $\cal{F} = \bb{K}[x_{i,\ell},y_{i}^{\pm 1},z_{i}^{\pm 1} \,|\, 1 \le i \le n,\, \ell \in \bb{Z}_{> 0}]$. Then, using
Lemma~\ref{lem:triangular_decomposition_D}(c), one can easily check that, for any $c,c' \in \bb{K}$, the following formulas define
a $U_{q,q^{-1}}^{D,0}$-module $\cal{F}^{-}_{c,c'}$ with underlying vector space $\cal{F}$:
\[
  \gamma^{1/2}\cdot f = cf,\qquad (\gamma')^{1/2}\cdot f = c'f,
\]
\[
  \omega_{i}^{\pm 1}\cdot f = y_{i}^{\pm 1}f,\qquad (\omega_{i}')^{\pm 1}\cdot f = z_{i}^{\pm 1}f,
\]
\[
  a_{i,\ell}\cdot f =
  \begin{cases}
    x_{i,-\ell}f & \text{if}\ \ell < 0 \\
    \sum_{j = 1}^{n}\frac{[\ell a_{ij}]_{q_{i}}}{\ell}(cc')^{\ell/2}\frac{c^{\ell} - (c')^{\ell}}{q_{j} - q_{j}^{-1}}\frac{\partial f}{\partial x_{j,\ell}}
      & \text{if}\ \ell > 0
  \end{cases}\,.
\]
Next, we note that there is an algebra automorphism of $U_{q,q^{-1}}^{D,0}$ given by $\omega_{i} \mapsto \omega_{i}$, $\omega_{i}' \mapsto \omega_{i}'$,
$\gamma^{1/2} \mapsto (\gamma')^{1/2}$, $(\gamma')^{1/2} \mapsto \gamma^{1/2}$, and $a_{i,\ell} \mapsto a_{i,-\ell}$. Using this, we may define another
$U_{q,q^{-1}}^{D,0}$-module $\cal{F}^{+}_{c,c'}$, obtained from $\cal{F}_{c,c'}^{-}$ by first composing with the aforementioned automorphism. Note that
$1 \in \cal{F}_{c,c'}^{-}$ is annihilated by all $a_{i,\ell}$ with $\ell > 0$, while $1 \in \cal{F}_{c,c'}^{+}$ is annihilated by all $a_{i,\ell}$ with
$\ell < 0$. Finally, we also note that there is an algebra homomorphism $\Delta\colon U_{q,q^{-1}}^{D,0} \to U_{q,q^{-1}}^{D,0} \otimes U_{q,q^{-1}}^{D,0}$
such that $\Delta(\omega_{i}) = \omega_{i} \otimes \omega_{i}$, $\Delta(\omega_{i}') = \omega_{i}' \otimes \omega_{i}'$,
$\Delta(\gamma^{1/2}) = \gamma^{1/2} \otimes \gamma^{1/2}$, $\Delta((\gamma')^{1/2}) = (\gamma')^{1/2} \otimes (\gamma')^{1/2}$, and
\[
  \Delta(a_{i,\ell}) =
  \begin{cases}
    a_{i,\ell} \otimes \gamma^{|\ell|/2}(\gamma\gamma')^{|\ell|/2} + (\gamma')^{|\ell|/2}(\gamma\gamma')^{|\ell|/2} \otimes a_{i,\ell} & \text{if}\ \ell > 0 \\
    a_{i,\ell} \otimes \gamma^{|\ell|/2} + (\gamma')^{|\ell|/2} \otimes a_{i,\ell} & \text{if}\ \ell < 0
  \end{cases}\,.
\]
Thus, the tensor product $\cal{F}^{-}_{c,c'} \otimes \cal{F}^{+}_{c,c'}$, is also a $U_{q,q^{-1}}^{D,0}$-module through $\Delta$. Then by taking
$c,c' \in \bb{K}$ to be algebraically independent and considering the action of $U_{q,q^{-1}}^{D,0}$ on $1 \otimes 1$, one can easily show by
mimicking the proof of~\cite[Proposition 4.16]{J} that the elements $A_{S}^{-} \gamma^{m/2}\omega_{\mu}\omega_{\mu'}'(\gamma')^{m'/2}A_{S'}^{+}$
are linearly independent.
\end{proof}

Similarly, $U_{q}^{D,0}$ has a basis consisting of all $H_{S}^{-}\gamma^{m/2}K_{\mu}H_{S'}^{+}$, and the corresponding proof is completely analogous
to the one above. We are now ready to prove an analogue of Propositions~\ref{prop:double_cartan_DJ} and~\ref{prop:double_cartan_RTT_finite}:

\begin{prop}\label{prop:double_cartan_D}
Let $L_{\alpha_{0}},\ldots ,L_{\alpha_{n}}$ be algebraically independent invertible transcendental elements. Then there is an injective algebra homomorphism
  $\eta_{D}\colon U_{q,q^{-1}}^{D}(\what{\fg}) \to U_{q}^{D}(\what{\fg}) \otimes \bb{K}[L_{\alpha_{0}}^{\pm 1},\ldots ,L_{\alpha_{n}}^{\pm 1}]$
such that
\begin{align*}
  &\eta_{D}(\omega_{i}) = K_{i} \otimes L_{\alpha_{i}}^{2},
    & & \eta_{D}(\omega_{i}')  = K_{i}^{-1} \otimes L_{\alpha_{i}}^{2},
    & & \eta_{D}(\gamma^{1/2}) = \gamma^{1/2} \otimes L_{\delta}^{-1}, \\
  &\eta_{D}((\gamma')^{1/2}) = \gamma^{-1/2} \otimes L_{\delta}^{-1},
    & &\eta_{D}(a_{i,\ell}) = H_{i,\ell} \otimes L_{\delta}^{-2|\ell|},
    & & \eta_{D}(x_{i,m}^{\pm}) = X_{i,m}^{\pm} \otimes L_{\alpha_{i} \pm m\delta},
\end{align*}
for all $1 \le i \le n$, $\ell \in \bb{Z} \setminus \{0\}$, and $m \in \bb{Z}$.
\end{prop}

Here, in analogy with~\eqref{eq:L_mu}, we use the notation $L_{\mu} = L_{\alpha_{0}}^{c_{0}}\ldots L_{\alpha_{n}}^{c_{n}}$ for any
$\mu = \sum_{i = 0}^{n}c_{i}\alpha_{i} \in \what{Q}$.

\begin{proof}
Note that the elements $\gamma^{1/2} - (\gamma')^{-1/2}$ and $\omega_{i} - (\omega_{i}')^{-1}$ are homogeneous relative to the $\what{Q}$-grading
given by $U_{q,q^{-1}}^{D}(\what{\fg})_{\mu} = \bigoplus_{\mu_{1} + \mu_{2} = \mu}U_{q,q^{-1}}^{D}(\what{\fg})_{\mu_{1},\mu_{2}}$ for all
$\mu \in \what{Q}$. We define
  $\eta_{D}\colon U_{q,q^{-1}}^{D}(\what{\fg}) \to U_{q}^{D}(\what{\fg}) \otimes \bb{K}[L_{\alpha_{0}}^{\pm 1},\ldots ,L_{\alpha_{n}}^{\pm 1}]$
by $\eta_{D}(u) = \bar{u} \otimes L_{\mu - \nu}$ for all $u \in U_{q,q^{-1}}(\fg)_{\mu,\nu}$, where $\bar{u}$ denotes the image of $u$ under the
quotient map $U_{q,q^{-1}}^{D}(\what{\fg}) \to U_{q}^{D}(\what{\fg})$. Then, by construction, $\eta_{D}$ is an algebra homomorphism, and it is
easy to see from the definition of the $\frac{1}{2}\what{Q}$-bigrading of Lemma~\ref{lem:Q_hat_bigrading} that $\eta_{D}$ maps the generators
of $U_{q,q^{-1}}^{D}(\what{\fg})$ as specified above.

For injectivity, we first note that, due to Lemma~\ref{lem:triangular_decomposition_D}(b) and its analogue for $U_{q}^{D}(\what{\fg})$,
the restriction of the quotient map $U_{q,q^{-1}}^{D}(\what{\fg}) \to U_{q}^{D}(\what{\fg})$ to $U_{q,q^{-1}}^{D,\pm}$ defines an isomorphism
$U_{q,q^{-1}}^{D,\pm} \iso U_{q}^{D,\pm}$, and it clearly also preserves the $\what{Q}$-grading on both algebras. Thus, if for all $\mu \in \what{Q}$
we choose bases $\{b_{j}^{\mu,+}\}_{j \in J^{+}_{\mu}}$ and $\{b_{j}^{\mu,-}\}_{j \in J^{-}_{\mu}}$ for $(U_{q,q^{-1}}^{D,+})_{\mu,0}$ and
$(U_{q,q^{-1}}^{D,-})_{0,\mu}$, then $\{\bar{b}_{j}^{\mu, \pm}\}_{j \in J^{\pm}_{\mu}}$ are bases for $(U_{q}^{D,\pm})_{\mu}$. Then by
Lemmas~\ref{lem:triangular_decomposition_D} and~\ref{lem:cartan_basis_D}, we have a basis of $U_{q,q^{-1}}^{D}(\what{\fg})$ consisting of
all elements $b_{j}^{\mu,-}A_{S}^{-}\gamma^{m/2}\omega_{\lambda}\omega_{\lambda'}'(\gamma')^{m'/2}A_{S'}^{+}b_{k}^{\nu,+}$ with $m,m' \in \bb{Z}$,
$\mu,\nu \in \what{Q}$, $j\in J^-_{\mu}, k\in J^+_\nu$, $\lambda,\lambda' \in Q$,
and $S,S'$ sets of the form $\{(i_{1},\ell_{1},N_{1}),\ldots ,(i_{t},\ell_{t},N_{t})\}$ with
$t \ge 0$, $1 \le i_{p} \le n$, $\ell_{p} \in \bb{Z}_{> 0}$ and $N_{p} \in \bb{Z}_{> 0}$, where
$A_{S}^{\pm}$ are defined in the first part of~\eqref{eq:A_S_def}. Likewise, we have a basis of $U_{q}^{D}(\what{\fg})$ consisting of all
elements of the form $\bar{b}_{j}^{\mu,-}H_{S}^{-}\gamma^{m/2}K_{\lambda}H_{S'}^{+}\bar{b}_{k}^{\nu,+}$ with $\mu,\nu,j,k,S,S',m,\lambda$
as above and $H_{S}^{\pm}$ introduced in the second part of~\eqref{eq:A_S_def}.

Now, for any set of the form $S = \{(i_{1},\ell_{1},N_{1}),\ldots ,(i_{t},\ell_{t},N_{t})\}$, set
$d_{S} = \sum_{r = 1}^{t} \ell_{r}N_{r}$. Then we have
\begin{multline}\label{eq:etaD_explicit}
  \eta_{D}(b_{j}^{\mu,-}A_{S}^{-}\gamma^{m/2}\omega_{\lambda}\omega_{\lambda'}'(\gamma')^{m'/2}A_{S'}^{+}b_{k}^{\nu,+}) = \\
  \bar{b}_{j}^{\mu, -}H_{S}^{-}\gamma^{(m -m')/2}K_{\lambda - \lambda'}H_{S'}^{+}\bar{b}_{k}^{\nu,+} \otimes
  L_{\nu - \mu + 2\lambda + 2\lambda' - (2d_{S} + 2d_{S'} + m + m')\delta}.
\end{multline}
Note that because $\lambda,\lambda' \in Q$, one can determine $\lambda,\lambda',m,m'$ from $\lambda - \lambda'$, $m - m'$,
and $2(\lambda + \lambda') - (m + m')\delta$.

Hence, it follows from~\eqref{eq:etaD_explicit} that $\eta_{D}$-images of the above basis of $U_{q,q^{-1}}^{D}(\what{\fg})$ are linearly independent.
\end{proof}

   %%%%%%%%%%%%%%%%%%%%%%%%%%%%%%%%%%%%%%%%%%%%%%%%%%%%%%%%%%%%%%%%%%%%%%%%%%

\subsection{Isomorphism between Drinfeld-Jimbo and new Drinfeld realizations}
\

We shall now use the results above to establish a two-parameter analogue of the isomorphism between the Drinfeld-Jimbo and new Drinfeld realizations.
First, we need to set up some notation, cf.~\cite{Ji}. For any $y_{1},\ldots ,y_{n} \in U_{r,s}^{D}(\what{\fg})$ and
$v_{1},\ldots ,v_{n-1} \in \bb{K}$, we define $[y_{1},\ldots ,y_{n}]_{(v_{1},\ldots ,v_{n-1})}$ and
$[y_{1},\ldots ,y_{n}]_{(v_{1},\ldots ,v_{n-1})}'$ via
\[
  [y_{1},y_{2}]_{v_{1}} = y_{1}y_{2} - v_{1}y_{2}y_{1} = [y_{1},y_{2}]_{v_{1}}',
\]
and for $n>2$:
\begin{equation}\label{eq:q_brackets}
\begin{split}
  &[y_{1},\ldots ,y_{n}]_{(v_{1},\ldots ,v_{n-1})} = [y_{1},[y_{2},\ldots ,y_{n}]_{(v_{1},\ldots ,v_{n-2})}]_{v_{n-1}}, \\
  &[y_{1},\ldots ,y_{n}]_{(v_{1},\ldots ,v_{n-1})}' = [[y_{1},\ldots ,y_{n-1}]_{(v_{1},\ldots ,v_{n-2})}',y_{n}]_{v_{n-1}}'.
\end{split}
\end{equation}
We note that the two bracketings are related by the following formula:
\begin{equation}\label{eq:bracket_to_bracket'}
  [y_{1},\ldots ,y_{n}]_{(v_{1},\ldots ,v_{n-1})} = (-1)^{n-1}v_{1}v_{2}\ldots v_{n-1} \cdot [y_{n},\ldots ,y_{1}]_{(v_{1}^{-1},\ldots ,v_{n-1}^{-1})}'.
\end{equation}

We also define $[y_{1},\ldots ,y_{n}]_{(v_{1},\ldots ,v_{n-1})}^{\circ}$ and $[y_{1},\ldots ,y_{n}]_{(v_{1},\ldots ,v_{n-1})}'^{\circ}$
by the same formulas, but using the multiplication $\circ$ in $U_{q,\zeta}^{D}(\what{\fg})$. Then, we have the following important technical result:

\begin{lemma}\label{lem:twisted_bracket}
Let $\lambda_{i},\mu_{i} \in Q$ and $y_{i} \in {U_{q,\zeta}^{D}(\what{\fg})}_{\lambda_{i},\mu_{i}}$ for $1 \le i \le n$.
Then, for any $v_{1},\ldots ,v_{n-1} \in \bb{K}$, we have:
\begin{equation}\label{eq:twisted_bracket}
  [y_{1},\ldots ,y_{n}]_{(v_{1},\ldots ,v_{n-1})} =
  \prod_{1 \le i < j \le n} \Big(\zeta(\lambda_{i},\lambda_{j})^{-1}\zeta(\mu_{i},\mu_{j}) \Big) \cdot
  [y_{1},\ldots ,y_{n}]_{(\wtd{v}_{1},\ldots ,\wtd{v}_{n-1})}^{\circ},
\end{equation}
where
\[
  \wtd{v}_{i} = \zeta(\lambda_{n-i},\lambda_{n-i+1} + \ldots + \lambda_{n})^{2}\zeta(\mu_{n-i},\mu_{n-i + 1} + \ldots + \mu_{n})^{-2}v_{i}
  \qquad \text{for all}\qquad 1 \le i \le n.
\]
\end{lemma}

\begin{proof}
We proceed by induction on $n$. For $n = 2$, we have
\begin{align*}
  [y_{1},y_{2}]_{v_{1}} & = y_{1}y_{2} - v_{1}y_{2}y_{1}
  = \zeta(\lambda_{1},\lambda_{2})^{-1}\zeta(\mu_{1},\mu_{2})y_{1} \circ y_{2}
    - v_{1}\zeta(\lambda_{2},\lambda_{1})^{-1}\zeta(\mu_{2},\mu_{1})y_{2} \circ y_{1} \\
  &= \zeta(\lambda_{1},\lambda_{2})^{-1}\zeta(\mu_{1},\mu_{2})\left (y_{1} \circ y_{2} - v_{1}
     \zeta(\lambda_{1},\lambda_{2})^{2}\zeta(\mu_{1},\mu_{2})^{-2}y_{2} \circ y_{1}\right )
  = \zeta(\lambda_{1},\lambda_{2})^{-1}\zeta(\mu_{1},\mu_{2})[y_{1},y_{2}]_{\wtd{v}_{1}}^{\circ},
\end{align*}
which proves the base case. Now, suppose that $n > 2$ and~\eqref{eq:twisted_bracket} holds for $m< n$.
Then by~\eqref{eq:q_brackets} and the induction hypothesis, we have:
\begin{align*}
  [y_{1},\ldots ,y_{n}]_{(v_{1},\ldots ,v_{n-1})} &= [y_{1},[y_{2},\ldots ,y_{n}]_{(v_{1},\ldots ,v_{n-2})}]_{v_{n-1}} \\
  &= \prod_{2 \le i < j \le n} \Big( \zeta(\lambda_{i},\lambda_{j})^{-1}\zeta(\mu_{i},\mu_{j}) \Big) \cdot
  [y_{1},[y_{2},\ldots ,y_{n}]_{(\wtd{v}_{1},\ldots ,\wtd{v}_{n-2})}^{\circ}]_{v_{n-1}}.
\end{align*}
Because $[y_{2},\ldots ,y_{n}]_{(\wtd{v}_{1},\ldots ,\wtd{v}_{n-2})}^{\circ}$ has bidegree
$(\lambda_{2} + \ldots + \lambda_{n},\mu_{2} + \ldots + \mu_{n})$, we can apply the $n = 2$ case to get
\[
  [y_{1},[y_{2},\ldots ,y_{n}]_{(\wtd{v}_{1},\ldots ,\wtd{v}_{n-2})}^{\circ}]_{v_{n-1}} =
  \zeta(\lambda_{1},\lambda_{2} + \ldots + \lambda_{n})^{-1}\zeta(\mu_{1},\mu_{2} + \ldots + \mu_{n})
  [y_{1},[y_{2},\ldots ,y_{n}]_{(\wtd{v}_{1},\ldots ,\wtd{v}_{n-2})}^{\circ}]_{\wtd{v}_{n-1}}^{\circ},
\]
and combining this with the expression above completes the proof.
\end{proof}

Finally, we need to extend $U_{q,q^{-1}}(\what{\fg})$ by adjoining central elements $\gamma^{1/2},(\gamma')^{1/2}$ that satisfy
$(\gamma^{1/2})^{2} = \omega_{\delta}$ and $((\gamma')^{1/2})^{2} = \omega_{\delta}'$, where $\delta = \alpha_{0} + \theta$.
By abuse of notation, we denote the extended algebra by $U_{q,q^{-1}}(\what{\fg})$ again.
To state the next theorem, let $i_{1},\ldots ,i_{h-1} \in \{1,2,\ldots ,n\}$ be any sequence such that
$\alpha_{i_{1}} + \ldots + \alpha_{i_{h-1}} = \theta$, and $\alpha_{i_{1}} + \ldots + \alpha_{i_{k}} \in \Phi^{+}$ for all
$1 \le k \le h-1$. Given this data, define $\sigma_{k} = (\alpha_{i_{1}} + \ldots + \alpha_{i_{k}},\alpha_{i_{k+1}})$
for all $1 \le k \le h - 2$, and let $\sigma = \sigma_{1} + \ldots + \sigma_{h-2}$ (in particular, $\sigma_{k} = -1$ for all $k$
when $\fg$ is simply-laced).

\begin{theorem}\label{thm:DJ=D_1_parameter}
The following defines an algebra isomorphism $\Psi_{q}\colon U_{q,q^{-1}}(\what{\fg}) \iso U_{q,q^{-1}}^{D}(\what{\fg})$:
\begin{equation*}
  \Psi_{q}(\omega_{i}) = \omega_{i},\qquad \Psi_{q}(\omega_{i}') = \omega_{i}',\qquad \Psi_{q}(e_{i}) = x^+_{i,0},\qquad
  \Psi_{q}(f_{i}) = x^-_{i,0} \qquad \forall\ 1 \le i \le n,
\end{equation*}
\begin{equation*}
  \Psi_{q}(\gamma^{1/2}) = \gamma^{1/2},\qquad \Psi_{q}((\gamma')^{1/2}) = (\gamma')^{1/2},\qquad
  \Psi_{q}(\omega_{0}) = (\gamma')^{-1}\omega_{\theta}^{-1},\qquad \Psi_{q}(\omega_{0}') = \gamma^{-1}(\omega_{\theta}')^{-1},
\end{equation*}
\begin{equation*}
  \Psi_{q}(e_{0}) =
  [x_{i_{h-1},0}^{-},\ldots , x_{i_{2},0}^{-}, x_{i_{1},1}^{-}]_{(q^{\sigma_{1}},\ldots ,q^{\sigma_{h-2}})}(\gamma')^{-1}\omega_{\theta}^{-1},
\end{equation*}
\begin{equation*}
  \Psi_{q}(f_{0}) = a_{q}(-q)^{-\sigma}\gamma^{-1}(\omega_{\theta}')^{-1}
  [x_{i_{h-1},0}^{+},\ldots ,x_{i_{2},0}^{+},x_{i_{1},-1}^{+}]_{(q^{\sigma_{1}},\ldots ,q^{\sigma_{h-2}})},
\end{equation*}
for some nonzero constant $a_{q}$.
\end{theorem}

\begin{proof}
We shall deduce this from the corresponding one-parameter algebra isomorphism
\begin{equation}\label{eq:Dr_iso}
  \bar{\Psi}_{q}\colon U_{q}(\what{\fg}) \iso U_{q}^{D}(\what{\fg})
\end{equation}
of~\cite{D} (see~\cite[Theorem 2.2]{Ji}), in conjunction with Proposition~\ref{prop:double_cartan_D}.

If we show that $(\bar{\Psi}_{q} \otimes \mathrm{Id}) \circ \eta$ maps $U_{q,q^{-1}}(\what{\fg})$ into $\eta_{D}(U_{q,q^{-1}}^{D}(\what{\fg}))$,
for $\eta_{D}$ of Proposition~\ref{prop:double_cartan_D} and $\eta$ of Proposition~\ref{prop:double_cartan_DJ} (which holds for any symmetrizable
Kac-Moody algebra $\fg$), then we obtain an algebra homomorphism
  $\Psi_{q} = \eta_{D}^{-1} \circ (\bar{\Psi}_{q} \otimes \mathrm{Id}) \circ \eta\colon U_{q,q^{-1}}(\what{\fg}) \to U_{q,q^{-1}}^{D}(\what{\fg})$.
It suffices to check the former on the generators of $U_{q,q^{-1}}(\what{\fg})$. This is clear for $\gamma^{1/2}$, $(\gamma')^{1/2}$, and
$\omega_{i},\omega_{i}',e_{i},f_{i}$ with $1 \le i \le n$. It is also easy to verify that
$\Psi_{q} = \eta_{D}^{-1} \circ (\bar{\Psi}_{q} \otimes \mathrm{Id}) \circ \eta$ maps these generators to the elements specified above.
For $\omega_{0}$, we have:
\[
  (\bar{\Psi}_{q}\otimes \mathrm{Id})\eta(\omega_{0})
  = \bar{\Psi}_{q}(K_{0}) \otimes L_{\alpha_{0}}^{2}
  = \gamma K_{\theta}^{-1}  \otimes L_{\alpha_{0}}^{2}
  = (\gamma^{-1} \otimes L_{\delta}^{-2})^{-1}(K_{\theta} \otimes L_{\theta}^{2})^{-1}
  = \eta_{D}(\gamma')^{-1}\eta_{D}(\omega_{\theta})^{-1},
\]
which shows that $(\bar{\Psi}_{q} \otimes \mathrm{Id})\eta(\omega_{0}) \in \eta_{D}(U_{q,q^{-1}}^{D}(\what{\fg}))$ and moreover
$\Psi_{q}(\omega_{0}) = (\gamma')^{-1}\omega_{\theta}^{-1}$. Similarly, one can show that $\omega_{0}'$ is mapped into
$\eta_{D}(U_{q,q^{-1}}^{D}(\what{\fg}))$ and that $\Psi_{q}(\omega_{0}') = \gamma^{-1}(\omega_{\theta}')^{-1}$.

For $e_{0}$, we have:
\begin{align*}
  (\bar{\Psi}_{q} \otimes \mathrm{Id})\eta(e_{0}) &= \bar{\Psi}_{q}(E_{0}) \otimes L_{\alpha_{0}}
  = [X_{i_{h-1},0}^{-},\ldots ,X_{i_{2},0}^{-},X_{i_{1},1}^{-}]_{(q^{\sigma_{1}},\ldots ,q^{\sigma_{h - 2}})}\gamma K_{\theta}^{-1} \otimes L_{\alpha_{0}} \\
  &= ([X_{i_{h-1},0}^{-},\ldots ,X_{i_{2},0}^{-},X_{i_{1},1}^{-}]_{(q^{\sigma_{1}},\ldots ,q^{\sigma_{h - 2}})} \otimes
     L_{\theta - \delta})(\gamma^{-1} \otimes L_{\delta}^{-2})^{-1}(K_{\theta} \otimes L_{\theta}^{2})^{-1} \\
  &= [X_{i_{h-1},0}^{-} \otimes L_{\alpha_{i_{h-1}}},\ldots ,X_{i_{2},0}^{-} \otimes L_{\alpha_{i_{2}}},X_{i_{1},1}^{-} \otimes
     L_{\alpha_{i_{1}} - \delta}]_{(q^{\sigma_{1}},\ldots ,q^{\sigma_{h-2}})}(\gamma^{-1} \otimes L_{\delta}^{-2})^{-1}(K_{\theta} \otimes
     L_{\theta}^{2})^{-1} \\
  &= \eta_{D}([x_{i_{h-1},0}^{-},\ldots ,x_{i_{2},0}^{-},x_{i_{1},1}^{-}]_{(q^{\sigma_{1}},\ldots ,q^{\sigma_{h-2}})}(\gamma')^{-1}\omega_{\theta}^{-1}),
\end{align*}
and hence $\Psi_{q}$ maps $e_{0}$ to the element of $U_{q,q^{-1}}^{D}(\what{\fg})$ specified above. The verification for $f_{0}$ is similar.

This shows that we have an injective algebra homomorphism $\Psi_{q}$ as claimed in the theorem. For surjectivity, one just needs
to show that the images of the generators of $U_{q,q^{-1}}(\what{\fg})$ under $\Psi_{q}$ generate all $U_{q,q^{-1}}^{D}(\what{\fg})$.
As $U_{q,q^{-1}}^{D}(\what{\fg})$ is generated by
  $\{(\omega_{j})^{\pm 1},(\omega_{j}')^{\pm 1},x_{j,0}^{+},x_{j,0}^{-} \,|\, 1 \le j \le n\}
   \cup \{\gamma^{\pm 1/2},(\gamma')^{\pm 1/2},a_{i_{1},\pm 1}\}$,
it suffices to express $a_{i_{1},\pm 1}$ in terms of
  $\{(\omega_{j})^{\pm 1},(\omega_{j}')^{\pm 1},x_{j,0}^{+},x_{j,0}^{-} \,|\, 1 \le j \le n\}
   \cup \{\gamma^{\pm 1/2},(\gamma')^{\pm 1/2}\} \cup \{\Psi_q(e_0),\Psi_q(f_0)\}$.
This can be done (cf.~\cite[page 7]{Ji}) by iteratively applying the $q$-bracket identities
\begin{align*}
  [a,[b,c]_{u}]_{v} &= [[a,b]_{x},c]_{\frac{uv}{x}} + x[b,[a,c]_{\frac{v}{x}}]_{\frac{u}{x}}, \\
  [[a,b]_{u},c]_{v} &= [a,[b,c]_{x}]_{\frac{uv}{x}} + x[[a,c]_{\frac{v}{x}},b]_{\frac{u}{x}}.
\end{align*}
We leave details to the interested reader.
\end{proof}

To obtain the two-parameter analogue of this result, we first observe that we can extend the $Q \times Q$-grading
\eqref{eq:our-bigrading} from $U_{q,q^{-1}}(\fg)$ to $U_{q,q^{-1}}(\what{\fg})$ via
\[
  \deg(e_{0}) = (-\theta,0),\qquad \deg(f_{0}) = (0,\theta),\qquad \deg(\omega_{0}) = \deg(\omega_{0}') = (-\theta,\theta).
\]
Then it is easy to show that we have the following analogue of Proposition~\ref{prop:twisted_algebra}:

\begin{prop}\label{prop:twisted_DJ_affine}
There is a unique Hopf algebra isomorphism $\what{\varphi}\colon U_{r,s}(\what{\fg}) \iso U_{q,\zeta}(\what{\fg})$ such that
\[
  \what{\varphi}(e_{i}) = e_{i},\qquad \what{\varphi}(f_{i}) = (r_{i}s_{i})^{-1/2}f_{i},\qquad
  \what{\varphi}(\omega_{i}) = \omega_{i},\qquad \what{\varphi}(\omega_{i}') = \omega_{i}'\qquad \text{for all}\qquad 0 \le i \le n,
\]
where $r_{0} = r^{(\theta,\theta)/2}$ and $s_{0} = s^{(\theta,\theta)/2}$.
\end{prop}

Combining this and Theorem~\ref{thm:twisted_algebra_loop}, we obtain the following two-parameter analogue of Theorem~\ref{thm:DJ=D_1_parameter}:

\begin{theorem}\label{thm:DJ=D_2_parameter}
The following defines an algebra isomorphism $\Psi_{r,s}\colon U_{r,s}(\what{\fg}) \iso U_{r,s}^{D}(\what{\fg})$:
\begin{equation*}
  \Psi_{r,s}(\omega_{i}) = \omega_{i},\qquad \Psi_{r,s}(\omega_{i}') = \omega_{i}',\qquad
  \Psi_{r,s}(e_{i}) = x^+_{i,0},\qquad \Psi_{r,s}(f_{i}) = x^-_{i,0} \qquad \forall\ 1 \le i \le n,
\end{equation*}
\begin{equation*}
  \Psi_{r,s}(\gamma^{1/2}) = \gamma^{1/2},\qquad \Psi_{r,s}((\gamma')^{1/2}) = (\gamma')^{1/2},\qquad
  \Psi_{r,s}(\omega_{0}) = (\gamma')^{-1}\omega_{\theta}^{-1},\qquad \Psi_{r,s}(\omega_{0}') = \gamma^{-1}(\omega_{\theta}')^{-1},
\end{equation*}
\begin{equation*}
  \Psi_{r,s}(e_{0})
  = [x_{i_{h-1},0}^{-},\ldots , x_{i_{2},0}^{-}, x_{i_{1},1}^{-}]_{(t_{1},\ldots ,t_{h-2})}(\gamma')^{-1}\omega_{\theta}^{-1},
\end{equation*}
\begin{equation*}
  \Psi_{r,s}(f_{0})
  = a_{r,s}(-r)^{-\sigma}\gamma^{-1}(\omega_{\theta}')^{-1}[x_{i_{h-1},0}^{+},\ldots ,x_{i_{2},0}^{+},x_{i_{1},-1}^{+}]_{(t_{1}',\ldots ,t_{h-2}')},
\end{equation*}
where
\[
  t_{k} = (\omega_{i_{k+1}}',\omega_{\alpha_{i_{1}} + \ldots + \alpha_{i_{k}}})\qquad \text{and}\qquad
  t_{k}' = (\omega_{\alpha_{i_{1}} + \ldots + \alpha_{i_{k}}}',\omega_{i_{k+1}})\qquad \text{for all}\qquad 1 \le k \le h - 2,
\]
and $a_{r,s}$ is a nonzero constant (which is equal to $a_{q}$ of Theorem~\ref{thm:DJ=D_1_parameter} with $q = r^{1/2}s^{-1/2}$).
\end{theorem}

\begin{proof}
Consider the composition
  $\bar{\Psi}_{r,s} = \varphi_{D}^{-1} \circ \Psi_{q} \circ \what{\varphi}\colon U_{r,s}(\what{\fg}) \iso U_{r,s}^{D}(\what{\fg})$
with the algebra isomorphisms $\varphi_{D}$ of Theorem~\ref{thm:twisted_algebra_loop}, $\Psi_{q}$ of Theorem~\ref{thm:DJ=D_1_parameter},
and $\what{\varphi}$ of Proposition~\ref{prop:twisted_DJ_affine}. It is easy to see that $\bar{\Psi}_{r,s}$ agrees with $\Psi_{r,s}$
on the generators $\{\omega_{i},\omega_{i}',e_{i},f_{i}\}_{i = 1}^{n} \cup \{\gamma^{1/2},(\gamma')^{1/2},\omega_{0},\omega_{0}'\}$.
For $e_{0}$, using Theorem~\ref{thm:DJ=D_1_parameter} and Lemma~\ref{lem:twisted_bracket} we get:
\begin{align*}
  \bar{\Psi}_{r,s}(e_{0})
  &= \varphi_{D}^{-1}\left ( [x_{i_{h-1},0}^{-},\ldots , x_{i_{2},0}^{-}, x_{i_{1},1}^{-}]_{(q^{\sigma_{1}},\ldots ,q^{\sigma_{h-2}})}
     (\gamma')^{-1}\omega_{\theta}^{-1}\right ) \\
  &= \prod_{1 \le j < k \le h-1}\zeta(\alpha_{i_{k}},\alpha_{i_{j}}) \cdot
    \varphi_{D}^{-1}\left ([x_{i_{h-1},0}^{-},\ldots ,x_{i_{2},0}^{-},x_{i_{1},1}^{-}]_{(t_{1},\ldots ,t_{h-2})}^{\circ}\circ
    (\gamma')^{-1} \circ \omega_{\theta}^{-1}\right ) \\
  &= \prod_{1 \le j < k \le h-1}\zeta(\alpha_{i_{k}},\alpha_{i_{j}}) \cdot
    [\varphi_{D}^{-1}(x_{i_{h-1},0}^{-}),\ldots ,\varphi_{D}^{-1}(x_{i_{2},0}^{-}),\varphi_{D}^{-1}(x_{i_{1},1}^{-})]_{(t_{1},\ldots ,t_{h-2})}
    (\gamma')^{-1}\omega_{\theta}^{-1} \\
  &= \prod_{1 \le j < k \le h-1}\zeta(\alpha_{i_{k}},\alpha_{i_{j}}) \cdot
    (rs)^{\frac{1}{2}N(\theta)}[x_{i_{h-1},0}^{-},\ldots ,x_{i_{2},0}^{-},x_{i_{1},1}^{-}]_{(t_{1},\ldots ,t_{h-2})}(\gamma')^{-1}\omega_{\theta}^{-1},
\end{align*}
where $N(\theta) = \sum_{k = 1}^{h - 1}\frac{(\alpha_{i_{k}},\alpha_{i_{k}})}{2} = \frac{(\theta,\theta)}{2} - \sigma$ and
\[
  t_{k} = \zeta(\alpha_{i_{k+1}},\alpha_{i_{1}} + \ldots + \alpha_{i_{k}})^{-2}q^{\sigma_{k}} =
  \zeta(\alpha_{i_{1}} + \ldots + \alpha_{i_{k}},\alpha_{i_{k + 1}})^{2}q^{\sigma_{k}} =
  (\omega_{i_{k+1}}',\omega_{\alpha_{i_{1}} + \ldots + \alpha_{i_{k}}})
\]
for all $1 \le k \le h-2$. Similarly, for $f_{0}$, we have
(evoking that $a_{r,s}$ coincides with $a_{(r/s)^{1/2}}$ from Theorem~\ref{thm:DJ=D_1_parameter}):
\begin{align*}
  &\bar{\Psi}_{r,s}(f_{0}) =
    a_{r,s}(-q)^{-\sigma}(r_{0}s_{0})^{-\frac{1}{2}}\varphi_{D}^{-1}
    \left( \gamma^{-1}(\omega_{\theta}')^{-1}[x_{i_{h-1},0}^{+},\ldots ,x_{i_{2},0}^{+},x_{i_{1},-1}^{+}]_{(q^{\sigma_{1}},\ldots ,q^{\sigma_{h-2}})}
    \right) \\
  &= \prod_{1 \le j < k \le h-1} \zeta(\alpha_{i_{k}},\alpha_{i_{j}})^{-1} \cdot a_{r,s}(-q)^{-\sigma}(r_{0}s_{0})^{-\frac{1}{2}}
     \varphi_{D}^{-1}\left( \gamma^{-1} \circ (\omega_{\theta}')^{-1} \circ
                            [x_{i_{h-1},0}^{+},\ldots ,x_{i_{2},0}^{+},x_{i_{1},-1}^{+}]_{(t_{1}',\ldots ,t_{h-2}')}^{\circ} \right) \\
  &= \prod_{1 \le j < k \le h - 1}\zeta(\alpha_{i_{k}},\alpha_{i_{j}})^{-1} \cdot a_{r,s}(-q)^{-\sigma}(r_{0}s_{0})^{-\frac{1}{2}}
     \gamma^{-1}(\omega_{\theta}')^{-1}[x_{i_{h-1},0}^{+},\ldots ,x_{i_{2},0}^{+},x_{i_{1},-1}^{+}]_{(t_{1}',\ldots ,t_{h-2}')},
\end{align*}
where
\[
  t_{k}' = \zeta(\alpha_{i_{k+1}},\alpha_{i_{1}} + \ldots + \alpha_{i_{k}})^{2}q^{\sigma_{k}} =
  (\omega_{\alpha_{i_{1}} + \ldots + \alpha_{i_{k}}}',\omega_{i_{k +1}})
\]
for all $1 \le k \le h - 2$. Note that
\[
  (-q)^{-\sigma}(r_{0}s_{0})^{-\frac{1}{2}} = (-1)^{-\sigma}r^{-\frac{1}{2}\sigma}s^{\frac{1}{2}\sigma}(rs)^{-\frac{(\theta,\theta)}{4}}
  = (-r)^{-\sigma}(rs)^{-\frac{1}{2}(\frac{(\theta,\theta)}{2} - \sigma)} = (-r)^{-\sigma}(rs)^{-\frac{1}{2}N(\theta)}.
\]
Now, set
\[
  c \ = \prod_{1 \le j < k \le h - 1}\zeta(\alpha_{i_{k}},\alpha_{i_{j}}) \cdot (rs)^{\frac{1}{2}N(\theta)}
\]
and consider the automorphism $\varrho$ of $U_{r,s}(\what{\fg})$ which fixes all the generators of $U_{r,s}(\what{\fg})$ other than $e_{0}$ and $f_{0}$,
and satisfies $\varrho(e_{0}) = c^{-1}e_{0}$, $\varrho(f_{0}) = cf_{0}$. Then the above computations show that $\Psi_{r,s} = \bar{\Psi}_{r,s} \circ \varrho$
is an algebra isomorphism that maps the generators of $U_{r,s}(\what{\fg})$ as specified above.
\end{proof}

\begin{remark}
For simply-laced $\fg$, the above theorem exactly recovers~\cite[Theorem 3.9]{HZ}. Indeed, the images of all generators
except $f_{0}$ match on the nose. For $f_{0}$, we note that
\[
  [x_{i_{h-1},0}^{+},\ldots ,x_{i_{2},0}^{+},x_{i_{1},-1}^{+}]_{(t_{1}',\ldots ,t_{h-2}')} =
  (-1)^{h-2}t_{1}'\ldots t_{h-2}'[x_{i_{1},-1}^{+},\ldots ,x_{i_{h-1},0}^{+}]'_{((t_{1}')^{-1},\ldots ,(t_{h-2}')^{-1})}
\]
by~\eqref{eq:bracket_to_bracket'}, as well as
\[
  t_{1}'\ldots t_{h-2}' = r^{\sum_{1 \le j < k \le h - 1}
    \langle \alpha_{i_{j}},\alpha_{i_{k}}\rangle}s^{-\sum_{1 \le j < k \le h - 1}\langle \alpha_{i_{k}},\alpha_{i_{j}}\rangle}
  = r^{\sigma}(rs)^{-\sum_{1 \le j < k \le h - 1}\langle \alpha_{i_{k}},\alpha_{i_{j}}\rangle},
\]
so that $\Psi_{r,s}(f_{0})$ can be written as
\begin{equation}\label{eq:f0_image_modified}
  \Psi_{r,s}(f_{0}) =
  a_{r,s}'(-1)^{h-2-\sigma}\gamma^{-1}(\omega_{\theta}')^{-1} [x_{i_{1},-1}^{+},\ldots ,x_{i_{h-1},0}^{+}]'_{((t_{1}')^{-1},\ldots ,(t_{h-2}')^{-1})},
\end{equation}
where
\[
  a_{r,s}' = a_{r,s}(rs)^{-\sum_{1 \le j < k \le h - 1}\langle \alpha_{i_{k}},\alpha_{i_{j}}\rangle}.
\]
The reader may now check that the right-hand side of~\eqref{eq:f0_image_modified} is equal to the image of $f_{0}$ in~\cite[Theorem~3.9]{HZ}.
\end{remark}

   %%%%%%%%%%%%%%%%%%%%%%%%%%%%%%%%%%%%%%%%%%%%%%%%%%%%%%%%%%%%%%%%%%%%%%%%%%
   %%%%%%%%%%%%%%%%%%%%%%%%%%%%%%%%%%%%%%%%%%%%%%%%%%%%%%%%%%%%%%%%%%%%%%%%%%
   %%%%%%%%%%%%%%%%%%%%%%%%%%%%%%%%%%%%%%%%%%%%%%%%%%%%%%%%%%%%%%%%%%%%%%%%%%

\section{Affine RTT}\label{sec:affine_RTT}

In this Section, we develop the affine analogue of Section~\ref{sec:FRT construction finite}. To this end, we start by evoking the FRT formalism
for solutions of the Yang-Baxter equation with a spectral parameter, and develop the twisting procedure for them. We finally apply this in the context
of $BCD$-type $R$-matrices with a spectral parameter to derive the isomorphism between this presentation and the new Drinfeld realizations from
Section~\ref{sec:loop-realization} for the two-parameter setup through twisting the one-parameter counterpart of~\cite{JLM1,JLM2}.

   %%%%%%%%%%%%%%%%%%%%%%%%%%%%%%%%%%%%%%%%%%%%%%%%%%%%%%%%%%%%%%%%%%%%%%%%%%

\subsection{The Hopf algebra $U(R(z))$ and its variations}\label{ssec:affine_FRT}
\

In this Section, we shall often be working with elements $A(z) \in \End(V)^{\otimes 2}[[z]]$, where $V$ is an $N$-dimensional $\bb{K}$-vector space.
Regarding these elements as $N\times N$ matrices with coefficients in $\bb{K}[[z]]$, we adopt the following convention for denoting the entries
of $A(z)$ and $A(z)^{-1}$ (assuming the latter exists):
\begin{equation}
\begin{split}
  A(z) &= \sum_{1 \le i,j \le N}A^{ij}_{k\ell}(z)E_{ik} \otimes E_{j\ell}
  \qquad \text{with}\qquad A^{ij}_{k\ell}(z) = \sum_{m \geq 0} A^{ij}_{k\ell}[-m]z^{m}, \\
  A(z)^{-1} &= \sum_{1 \le i,j \le N}(A^{-1})^{ij}_{k\ell}(z)E_{ik} \otimes E_{j\ell}
  \qquad \text{with}\qquad (A^{-1})^{ij}_{k\ell}(z) = \sum_{m \geq 0} (A^{-1})^{ij}_{k\ell}[-m]z^{m}.
\end{split}
\end{equation}
In particular, we shall often be working with $R(z) \in \End(V)^{\otimes 2}[[z]]$ that satisfy
the \textbf{Yang-Baxter equation with a spectral parameter}:
\begin{equation}\label{eq:YBE_affine}
  R_{12}(z)R_{13}(zw)R_{23}(w) = R_{23}(w)R_{13}(zw)R_{12}(z).
\end{equation}
Associated to any such $R(z)$, we define $U(R(z))$ as the associative $\bb{K}$-algebra generated by
\[
  \big\{\gamma^{\pm 1/2},(\gamma')^{\pm 1/2}\big\} \cup
  \big\{\ell_{ij}^{\pm}[\mp m] \,\big|\, 1 \le i,j \le N,\ m \in \bb{Z}_{\ge 0}\big\} \cup
  \big\{(\ell_{ii}^{\pm}[0])^{-1} \,\big|\, 1\le i \le N\big\} ,
\]
where $\gamma^{\pm 1/2}$ and $(\gamma')^{\pm 1/2}$ are central, and the remaining defining relations are:
\begin{equation}\label{eq:zero_mode_relations}
  \ell_{ij}^{+}[0] = \ell_{ji}^{-}[0] = 0,\qquad
  \ell_{ii}^{\pm}[0](\ell_{ii}^{\pm}[0])^{-1} = (\ell_{ii}^{\pm}[0])^{-1}\ell_{ii}^{\pm}[0] = 1 \qquad
  \text{for all}\qquad 1 \le i < j \le N,
\end{equation}
and
\begin{equation}\label{eq:Affine_RLL}
\begin{split}
   R(z/w)L^{\pm}_{1}(z)L_{2}^{\pm}(w) &= L_{2}^{\pm}(w)L_{1}^{\pm}(z)R(z/w), \\
   R(z\gamma(\gamma\gamma')^{1/2}/w)L_{1}^{+}(z)L_{2}^{-}(w) &= L_{2}^{-}(w)L_{1}^{+}(z)R(z\gamma'(\gamma\gamma')^{1/2}/w),
\end{split}
\end{equation}
where
\begin{equation*}
\begin{split}
  & L_{1}^{\pm}(w) = \sum_{1 \le i,j \le N}E_{ij} \otimes 1 \otimes \ell_{ij}^{\pm}(w) \in \End(V)^{\otimes 2}\otimes U(R(z))[[w^{\pm 1}]],\\
  & L_{2}^{\pm}(w) = \sum_{1 \le i,j \le N} 1 \otimes E_{ij} \otimes \ell_{ij}^{\pm}(w) \in \End(V)^{\otimes 2} \otimes U(R(z))[[w^{\pm 1}]],
\end{split}
\end{equation*}
with
\[
  \ell_{ij}^{\pm}(w) = \sum_{m \geq 0} \ell_{ij}^{\pm}[\mp m]w^{\pm m}.
\]
One can check directly that there is a Hopf algebra structure on $U(R(z))$ such that
\begin{equation*}
\begin{split}
  & \Delta(\gamma^{\pm 1/2})=\gamma^{\pm 1/2}\otimes \gamma^{\pm 1/2}, \quad
    S(\gamma^{\pm 1/2})=\gamma^{\mp 1/2}, \quad
    \epsilon(\gamma^{\pm 1/2})=1, \\
  & \Delta((\gamma')^{\pm 1/2})=(\gamma')^{\pm 1/2}\otimes (\gamma')^{\pm 1/2}, \quad
    S((\gamma')^{\pm 1/2})=(\gamma')^{\mp 1/2}, \quad \epsilon((\gamma')^{\pm 1/2})=1,
\end{split}
\end{equation*}
and for the remaining generators, the coproduct, counit, and antipode are given by the following formulas:
\begin{equation*}
  \Delta(\ell_{ij}^{+}(z)) = \sum_{1\leq k\leq N}\ell_{ik}^{+}(z(1 \otimes \gamma^{3/2})) \otimes \ell_{kj}^{+}(z(\gamma^{1/2} \otimes 1)), \quad
  S(L^{+}(z)) = (L^{+}(z\gamma^{-2}))^{-1}, \quad \epsilon(\ell_{ij}^{+}(z)) = \delta_{ij},
\end{equation*}
\begin{equation*}
  \Delta(\ell_{ij}^{-}(z)) = \sum_{1\leq k\leq N}\ell_{ik}^{-}(z(1 \otimes (\gamma')^{-1/2})) \otimes \ell_{kj}^{-}(z((\gamma')^{-3/2} \otimes 1)), \quad
  S(L^{-}(z)) = (L^{-}(z(\gamma')^{2}))^{-1}, \quad \epsilon(\ell_{ij}^{-}(z)) = \delta_{ij}.
\end{equation*}
To see why $L^{\pm}(z)^{-1}$ (and hence also $L^{+}(z\gamma^{-2})^{-1}$ and $L^{-}(z(\gamma')^{2})^{-1}$) exist, note that we can regard
$L^{\pm}(z)^{-1}$ as series with coefficients being $N\times N$ matrices with entries in $U(R(z))$. Since the constant term
$(\ell_{ij}^{\pm}[0])_{i,j = 1}^{N}$ is invertible due to the relations~\eqref{eq:zero_mode_relations}, the entire series
$L^{\pm}(z)$ must also be invertible.

\begin{remark}\label{rem:central-powers}
We should emphasize right away that the specific choice of powers of $\gamma,\gamma'$ featuring in the second line of~\eqref{eq:Affine_RLL}
follows from its one-parameter single-Cartan version of~\cite[(3.12)]{DF} (when $\gamma'=\gamma^{-1}$) via a combination of
Proposition~\ref{prop:2_vs_1_parameter_affine_RTT} (reducing to the case $r=q=s^{-1}$) and Proposition~\ref{prop:double_cartan_RTT_affine}
(which is essentially requesting that the defining relations are homogeneous with respect to the algebra bigrading of
Lemma~\ref{lem:affine_RTT_bigradings}(a)). On the other hand, our choice of powers of $\gamma,\gamma'$ in the above formulas for
$\Delta(\ell^\pm_{ij}(z))$ slightly differs from that of~\cite[(3.13)]{DF} and is made to ensure that $U(\wtd{R}_{q}(z))$ is a
$\what{P}$-bigraded Hopf algebra with respect to the $\what{P}$-bigrading of Lemma~\ref{lem:affine_RTT_bigradings}(a) (which in turn
is chosen to be compatible with Lemma~\ref{lem:Q_hat_bigrading}, see Remark~\ref{rem:rtt=Dr-hatP-compatibility}).
\end{remark}

For the rest of this Section, we will mainly work with $R(z) = \wtd{R}_{q}(z)$, where $\wtd{R}_{q}(z)$ is a certain normalization of
$R_{q}(z) = \hat{R}_{q}(z) \circ \tau$, with $\hat{R}_{q}(z)$ being one of the matrices of~\cite[(6.11)--(6.13)]{MT1} with $r = q$ and
$s = q^{-1}$. Recall that $R_{q}(z)$ is an $N^{2} \times N^{2}$-matrix (with $N = 2n +1$ in type $B_{n}$, and $N = 2n$ in types $C_{n}$
and $D_{n}$) that satisfies~\eqref{eq:YBE_affine}. Let also $\xi = q^{-4n + 2}$ and $v = q^{-2}$ in type $B_{n}$,  $\xi = q^{-2-2n}$ and
$v = q^{-1}$ in type $C_{n}$, and $\xi = q^{2-2n}$ and $v = q^{-1}$ in type $D_{n}$. Then we define
\begin{equation}\label{eq:f_prefactor}
  \wtd{R}_{q}(z) = f(z)R_{q}(z),
\end{equation}
where $f(z)$ is the formal power series in $z$ uniquely determined by the following conditions:
\begin{equation}\label{eq:f_qdifference}
  f(0) = 1 ,\qquad
  f(z)f(z\xi) = \frac{1}{(1 - zv^{2})(1 - zv^{-2})(1 - z\xi)(1 - z\xi^{-1})}.
\end{equation}

\begin{remark}\label{rem:normalization-factor}
The factor $f(z)$ essentially comes from the ``imaginary'' part of the universal $R$-matrix (the factor $\cal{R}^{0}(z)$ from~\eqref{eq:R0_formula}), see Appendix~\ref{ssec:app_C} for more details. We also note that the algebras $U(\wtd{R}_{q}(z))$ and $U(R_{q}(z))$ are closely related via a quantum Heisenberg algebra as in~\cite[Proposition 3.2]{JLM2}.
\end{remark}

\begin{remark}
If we regard $f(z)$ as an element in $\bb{K}((q^{-1}))[[z]]$, then it is given explicitly by the following infinite product formula:
\begin{equation}\label{eq:f_prefactor_formula}
  f(z) = \prod_{k \geq 0}
  \frac{(1 - z\xi^{2k})(1 - zv^{2}\xi^{2k + 1})(1 - zv^{-2}\xi^{2k + 1})(1 - z\xi^{2k + 2})}
       {(1 - z\xi^{2k - 1})(1 - z\xi^{2k + 1})(1 - zv^{-2}\xi^{2k})(1 - zv^{2}\xi^{2k})}.
\end{equation}
\end{remark}

In this Section, we shall also consider several quotients of $U(\wtd{R}_{q}(z))$. First, let $U'(\wtd{R}_{q}(z))$ be the quotient obtained
by imposing the additional relations $(\gamma')^{1/2} = \gamma^{-1/2}$ and $(\ell_{ii}^{\pm}[0])^{-1} = \ell_{ii}^{\mp}[0]$ for all
$1 \le i \le N$. Furthermore, we define $\wtd{U}(\wtd{R}_{q}(z))$, $\wtd{U}'(\wtd{R}_{q}(z))$ to be the algebras formed from
$U(\wtd{R}_{q}(z))$, $U'(\wtd{R}_{q}(z))$, respectively, by further imposing (cf.~\eqref{eq:LC_relations})
\begin{equation}\label{eq:LC_affine}
  L^{\pm}(z)C_{q}L^{\pm}(z\xi)^{\mathsf{t}'}C_{q}^{-1} = I.
\end{equation}
Here, as before, $A^{\mathsf{t}'} = (a_{j'i'})_{i,j = 1}^{N}$ for any $N \times N$-matrix $A = (a_{ij})_{i,j = 1}^{N}$, and $C_{q}$ is one of
the matrices defined in Section~\ref{sec:FRT construction finite} (see~\eqref{eq:Cq_B} for type $B_{n}$,~\eqref{eq:Cq_C} for type $C_{n}$,
and~\eqref{eq:Cq_D} for type $D_{n}$). In type $B_{n}$, we also impose the additional relation (cf.~\eqref{eq:rel_B_extra})
\begin{equation}\label{eq:B_extra_relation_affine}
  \ell_{n + 1,n+1}^{\pm}[0] = 1
\end{equation}
in both $\wtd{U}(\wtd{R}_{q}(z))$ and $\wtd{U}'(\wtd{R}_{q}(z))$. More explicitly,~\eqref{eq:LC_affine} is equivalent to the following set of relations:
\begin{equation}\label{eq:LC_explicit_series}
  \sum_{1\leq k\leq N}c_{kk}c_{jj}^{-1}\ell_{ik}^{\pm}(z)\ell_{j'k'}^{\pm}(z\xi) = \delta_{ij}
  \qquad \mathrm{for\ all} \qquad 1\leq i,j\leq N.
\end{equation}

We also have the following analogue of Proposition~\ref{prop:central_element}:

\begin{prop}\label{prop:central_element_affine}
In the algebras $U(\wtd{R}_{q}(z))$ and $U'(\wtd{R}_{q}(z))$, we have the following equalities:
\begin{equation}\label{eq:central_element_affine}
  \sum_{1\leq k\leq N}c_{ii}c_{kk}^{-1}\ell_{k'i'}^{\pm}(z\xi)\ell_{ki}^{\pm}(z) =
  \sum_{1\leq k\leq N}c_{jj}^{-1}c_{kk}\ell_{jk}^{\pm}(z)\ell_{j'k'}^{\pm}(z\xi)
  \qquad \text{for all}\qquad 1 \le i,j \le N.
\end{equation}
If $\frak{z}^{\pm}_q(z)$ denotes the common series above, then the relations~\eqref{eq:LC_affine} are equivalent to the relations
$\frak{z}^{\pm}_q(z) = 1$. Furthermore, we have
\begin{equation}\label{eq:zq_coproduct}
\begin{split}
  &\Delta(\frak{z}_{q}^{+}(z)) = \frak{z}_{q}^{+}(z(1 \otimes \gamma^{3/2})) \otimes \frak{z}_{q}^{+}(z(\gamma^{1/2} \otimes 1)), \\
  &\Delta(\frak{z}_{q}^{-}(z)) = \frak{z}_{q}^{-}(z(1 \otimes (\gamma')^{-1/2})) \otimes \frak{z}_{q}^{-}(z((\gamma')^{-3/2} \otimes 1)),
\end{split}
\end{equation}
\begin{equation}\label{eq:zq_counit_antipode}
  \epsilon(\frak{z}_{q}^{\pm}(z)) = 1,\qquad S(\frak{z}_{q}^{+}(z)) = \frak{z}_{q}^{+}(z\gamma^{-2})^{-1},\qquad
  S(\frak{z}_{q}^{-}(z)) = \frak{z}_{q}^{-}(z(\gamma')^{2})^{-1}.
\end{equation}
\end{prop}

This result implies that $\wtd{U}(\wtd{R}_{q}(z))$ and $\wtd{U}'(\wtd{R}_{q}(z))$ are in fact Hopf algebras.

\begin{proof}
We shall only consider type $B_{n}$, since the arguments are essentially the same for types $C_{n}$ and $D_{n}$. Note that it makes sense
to set $z = w\xi$ in the first equation of~\eqref{eq:Affine_RLL}. Then as~\cite[(1.5)]{JLM1} holds for $\wtd{R}_{q}(z)$
(with $\rho_{\imath}$ of~\eqref{eq:rho_B} instead of $\bar{\imath}$ and $q^{2}$ instead of $q$ from \emph{loc.cit.}), we obtain
\begin{equation}\label{eq:Q_relation}
  QL_{1}^{\pm}(w\xi)L_{2}^{\pm}(w) = L_{2}^{\pm}(w)L_{1}^{\pm}(w\xi)Q,
\end{equation}
where
\[
  Q \,= \sum_{1 \le i,j \le N}q^{2(\rho_{i} - \rho_{j})}E_{i'j'} \otimes E_{ij}.
\]
Since $Q = K \circ \tau$ with $K$ of~\eqref{eq:K_matrix}, it follows from~\eqref{eq:Q_relation} that
\[
  KL_{2}^{\pm}(z\xi)L_{1}^{\pm}(z) = L_{2}^{\pm}(z)L_{1}^{\pm}(z\xi)K,
\]
which yields
\begin{equation}\label{eq:off_diagonal_LC}
  \sum_{1\leq k\leq N}c_{kk}\ell_{ik}^{\pm}(z)\ell_{j'k'}^{\pm}(z\xi) = 0 \qquad \text{for all}\qquad i \neq j,
\end{equation}
and
\begin{equation}\label{eq:diagonal_LC}
  \sum_{1\leq k\leq N} c_{jj}c_{kk}^{-1}\ell_{k'i'}^{\pm}(z\xi)\ell_{ki}^{\pm}(z) =
  \sum_{1\leq k\leq N} c_{kk}c_{ii}^{-1}\ell_{jk}^{\pm}(z)\ell_{j'k'}^{\pm}(z\xi)
  \qquad \text{for all}\qquad 1 \le i,j \le N,
\end{equation}
cf.~\eqref{eq:KLL_relation_1}--\eqref{eq:KLL_relation_2}. Clearly,~\eqref{eq:diagonal_LC} is equivalent to the equalities~\eqref{eq:central_element_affine}.
Moreover,~\eqref{eq:off_diagonal_LC} implies that all off-diagonal entries of $L^{\pm}(z)C_{q}L^{\pm}(z\xi)^{\mathsf{t}'}C_{q}^{-1} - I$ are zero, while
its diagonal entries are all equal to $\frak{z}^{\pm}_q(z)-1$ due to~\eqref{eq:diagonal_LC}. Thus,~\eqref{eq:LC_affine} is equivalent to the relations
$\frak{z}^{\pm}_q(z) = 1$.

Next, applying $\epsilon$ to~\eqref{eq:central_element_affine}, we immediately get $\epsilon(\frak{z}_{q}^{\pm}(z)) = 1$.
To derive~\eqref{eq:zq_coproduct}, one applies $\Delta$ to the right-hand side of~\eqref{eq:central_element_affine} and
utilizes~\eqref{eq:off_diagonal_LC}, in analogy with~\eqref{eq:z-coproduct} in the finite case (note that~\eqref{eq:off_diagonal_LC} implies
that $\sum_{k = 1}^{N}c_{kk}1 \otimes \ell_{ik}^{\pm}(z(a \otimes 1))\ell_{j'k'}^{\pm}(z\xi(a \otimes 1)) = 0$ for $i\ne j$ and any invertible
$a$ in $U(\wtd{R}_{q}(z))$ or $U'(\wtd{R}_{q}(z))$). Finally, to prove the formulas for the antipode, we shall write
\[
  \frak{z}_{q}^{\pm}(z) = \sum_{m \ge 0}\frak{z}_{q}^{\pm}[\mp m]z^{\pm m} \qquad \text{and}\qquad
  S(\frak{z}_{q}^{\pm}(z)) = \sum_{m \ge 0}\wtd{\frak{z}}_{q}^{\pm}[\mp m]z^{\pm m},
\]
and use the equalities
  $\mu(S \otimes 1)\Delta(\frak{z}_{q}^{\pm}[\mp m]) =
   \mu(1 \otimes S)\Delta(\frak{z}_{q}^{\pm}[\mp m]) =
   \epsilon(\frak{z}_{q}^{\pm}[\mp m]) = \delta_{m,0}$
(where $\mu$ denotes the multiplication map) to get
\[
  \gamma^{3m/2}\sum_{d_1,d_2\geq 0}^{d_{1} + d_{2} = m} \wtd{\frak{z}}_{q}^{+}[-d_{1}]\gamma^{-2d_{2}}\frak{z}_{q}^{+}[-d_{2}] =
  \gamma^{m/2}\sum_{d_1,d_2\geq 0}^{d_{1} + d_{2} = m} \frak{z}_{q}^{+}[-d_{1}]\wtd{\frak{z}}_{q}^{+}[-d_{2}]\gamma^{-2d_{1}} =
  \delta_{m,0}
\]
as well as
\[
  (\gamma')^{m/2}\sum_{d_1,d_2\geq 0}^{d_{1} + d_{2} = m} \wtd{\frak{z}}_{q}^{-}[d_{1}](\gamma')^{-2d_{2}}\frak{z}_{q}^{-}[d_{2}] =
  (\gamma')^{3m/2}\sum_{d_1,d_2\geq 0}^{d_{1} + d_{2} = m} \frak{z}_{q}^{-}[d_{1}]\wtd{\frak{z}}_{q}^{-}[d_{2}](\gamma')^{-2d_{1}} =
  \delta_{m,0}
\]
which imply that $S(\frak{z}_{q}^{+}(z)) = \frak{z}_{q}^{+}(z\gamma^{-2})^{-1}$ and $S(\frak{z}_{q}^{-}(z)) = \frak{z}_{q}^{-}(z(\gamma')^{2})^{-1}$,
as claimed in~\eqref{eq:zq_counit_antipode}.
\end{proof}

Following the ideas of Section~\ref{sec:FRT construction finite}, let us now develop results that will yield two-parameter upgrades
of~\cite[Main Theorem]{JLM1} and~\cite[Main Theorem]{JLM2}, establishing isomorphisms $U_{q}^{D}(\what{\fg}) \iso \wtd{U}'(\wtd{R}_{q}(z))$
in $BCD$-types. To this end, we first define $\what{P}$-bigradings on $U(\wtd{R}_{q}(z))$ and $\wtd{U}(\wtd{R}_{q}(z))$, where
$\what{P} = P \oplus \bb{Z} \frac{\delta}{2}$.

\begin{lemma}\label{lem:affine_RTT_bigradings}
(a) The following assignment makes $U(\wtd{R}_{q}(z))$ and $\wtd{U}(\wtd{R}_{q}(z))$ into $\what{P}$-bigraded Hopf algebras:
\[
  \deg(\gamma^{1/2}) = \deg((\gamma')^{1/2}) = \left (-\frac{1}{2}\delta,\frac{1}{2}\delta\right ),
\]
\[
  \deg(\ell_{ij}^{+}[-m]) = \left (-\varepsilon_{i} - \frac{m}{2}\delta,\varepsilon_{j} + \frac{3m}{2}\delta\right ),\qquad
  \deg(\ell_{ij}^{-}[m]) = \left (-\varepsilon_{i} - \frac{3m}{2}\delta,\varepsilon_{j} + \frac{m}{2}\delta\right ),
\]
for all $1 \le i,j \le N$, $m \in \bb{Z}_{\ge 0}$.

\medskip
\noindent
(b) The following assignment makes $U(\wtd{R}_{q}(z))$ and $\wtd{U}(\wtd{R}_{q}(z))$ into $P$-bigraded Hopf algebras:
\[
  \deg(\gamma^{1/2}) = \deg((\gamma')^{1/2}) = (0,0),
\]
\[
  \deg(\ell_{ij}^{\pm}[\mp m]) = (-\varepsilon_{i},\varepsilon_{j})\qquad \text{for all}\qquad 1 \le i,j \le N,\ m \in \bb{Z}_{\ge 0}.
\]
\end{lemma}

\begin{proof}
Because $P = \what{P}/\bb{Z}\frac{\delta}{2}$, part (a) implies part (b), so we just need to prove (a). As in the finite case,
we have $(\wtd{R}_{q})_{ij}^{kt}(z) = 0$ unless $(i,j) = (k,t)$, $(i,j) = (t,k)$, or $i = j',k = t'$. Therefore
\begin{equation}\label{eq:weight_preserve_affine}
   (\wtd{R}_{q})_{ij}^{kt}[-d_1] \ne 0 \Longrightarrow \varepsilon_k+\varepsilon_t=\varepsilon_i+\varepsilon_j.
\end{equation}
Now, one can easily check that the matrix relation $\wtd{R}_{q}(z/w)L_{1}^{+}(z)L_{2}^{+}(w) = L_{2}^{+}(w)L_{1}^{+}(z)\wtd{R}_{q}(z/w)$
is equivalent to the following set of relations (for all $a,b \in \bb{Z}$ and $1 \le i,j,m,p \le N$):
\[
  \sum_{1 \le k,t \le N} \sum_{\substack{d_{1} + d_{2} = a \\ d_{3}-d_{1} = b}}^{d_{1},d_{2},d_{3} \ge 0}
    (\wtd{R}_{q})_{kt}^{mp}[-d_{1}]\ell_{ki}^{+}[-d_{2}]\ell_{tj}^{+}[- d_{3}] =
  \sum_{1 \le k,t \le N} \sum_{\substack{d_{1} + d_{2} = a \\ d_{3} - d_{1} = b}}^{d_{1},d_{2},d_{3} \ge 0}
    (\wtd{R}_{q})_{ij}^{kt}[-d_{1}]\ell_{pt}^{+}[-d_{3}]\ell_{mk}^{+}[-d_{2}].
\]
If $(\wtd{R}_{q})_{kt}^{mp}[-d_{1}]\ne 0$, then $\varepsilon_k+\varepsilon_t=\varepsilon_m+\varepsilon_p$ by~\eqref{eq:weight_preserve_affine}, and so
\begin{equation*}
  \deg(\ell_{ki}^{+}[-d_{2}]\ell_{tj}^{+}[-d_{3}])
  = (-\varepsilon_{k}-\varepsilon_{t},\varepsilon_{i} + \varepsilon_{j}) + (d_{2} + d_{3})\left (-\tfrac{1}{2}\delta,\tfrac{3}{2}\delta\right ) \\
  = (-\varepsilon_{m}-\varepsilon_{p},\varepsilon_{i} + \varepsilon_{j}) + (a + b)\left (-\tfrac{1}{2}\delta,\tfrac{3}{2}\delta\right )
\end{equation*}
Similarly, $(\wtd{R}_{q})_{ij}^{kt}[-d_{1}]\ne 0$ implies $\varepsilon_k+\varepsilon_t=\varepsilon_i+\varepsilon_j$ by~\eqref{eq:weight_preserve_affine},
so that
\begin{align*}
  \deg(\ell_{pt}^{+}[-d_{3}]\ell_{mk}^{+}[-d_{2}])
  = (-\varepsilon_{p}-\varepsilon_{m},\varepsilon_{k} + \varepsilon_{t}) + (a + b)\left (-\tfrac{1}{2}\delta,\tfrac{3}{2}\delta\right )
  = (-\varepsilon_{p}-\varepsilon_{m},\varepsilon_{i} + \varepsilon_{j}) + (a + b)\left (-\tfrac{1}{2}\delta,\tfrac{3}{2}\delta\right )
\end{align*}
for all $d_{1},d_{2},d_{3}$ with $d_{1} + d_{2} = a, d_{3} - d_{1} = b$. Thus, the relation
$\wtd{R}_{q}(z/w)L_{1}^{+}(z)L_{2}^{+}(w) = L_{2}^{+}(w)L_{1}^{+}(z)\wtd{R}_{q}(z/w)$ is homogeneous.
A similar argument works for the corresponding relation with $-$.

Likewise, the second relation of~\eqref{eq:Affine_RLL} can be explicitly written as (for all $a,b \in \bb{Z}_{\geq 0}$ and $1 \le i,j,m,p \le N$)
\begin{multline*}
  \sum_{ k,t =1}^{N}\sum_{\substack{d_{1} + d_{2} = a \\ d_{1} + d_{3} = -b}}^{d_{1},d_{2},d_{3} \ge 0}
    (\wtd{R}_{q})_{kt}^{mp}[-d_{1}]\gamma^{d_{1}}(\gamma\gamma')^{d_{1}/2}\ell_{ki}^{+}[-d_{2}]\ell_{tj}^{-}[d_{3}] \\
  = \sum_{k,t = 1}^{N}\sum_{\substack{d_{1} + d_{2} = a \\ d_{1} + d_{3} = -b}}^{d_{1},d_{2},d_{3} \ge 0}
    (\wtd{R}_{q})_{ij}^{kt}[-d_{1}](\gamma')^{d_{1}}(\gamma\gamma')^{d_{1}/2}\ell_{pt}^{-}[d_{3}]\ell_{mk}^{+}[-d_{2}].
\end{multline*}
Then
\begin{align*}
  \deg\left(\gamma^{d_{1}}(\gamma\gamma')^{d_{1}/2}\ell_{ki}^{+}[-d_{2}]\ell_{tj}^{-}[d_{3}]\right)
  &= (-2d_{1}\delta,2d_{1}\delta) + \left (-\varepsilon_{k} -\tfrac{1}{2}d_{2}\delta, \varepsilon_{i} + \tfrac{3}{2}d_{2}\delta\right ) +
    \left (-\varepsilon_{t} - \tfrac{3}{2}d_{3}\delta,\varepsilon_{j} + \tfrac{1}{2}d_{3}\delta\right ) \\
  &= (-\varepsilon_{k} - \varepsilon_{t}, \varepsilon_{i} + \varepsilon_{j}) +
     \left (-\left (-b + \tfrac{1}{2}(a - b)\right )\delta, \left (a + \tfrac{1}{2}(a - b)\right )\delta \right )
\end{align*}
and
\begin{align*}
  \deg\left((\gamma')^{d_{1}}(\gamma\gamma')^{d_{1}/2}\ell_{pt}^{-}[d_{3}]\ell_{mk}^{+}[-d_{2}]\right)
  = (-\varepsilon_{p} - \varepsilon_{m},\varepsilon_{t} + \varepsilon_{k}) +
    \left ( -\left (-b + \tfrac{1}{2}(a - b)\right )\delta, \left (a + \tfrac{1}{2}(a - b)\right )\delta \right )
\end{align*}
for all $d_{1},d_{2},d_{3}$ with $d_{1} + d_{2} = a$ and $d_{1} + d_{3} = -b$. Thus, due to~\eqref{eq:weight_preserve_affine},
this relation is also homogeneous.

To prove the result for $\wtd{U}(\wtd{R}_{q}(z))$, we note that $L^{\pm}(z)C_{q}L^{\pm}(z\xi)^{\mathsf{t}'}C_{q}^{-1} = I$
is equivalent to the following set of relations:
\begin{equation}\label{eq:LC_explicit_affine_1}
  \sum_{k = 1}^{N}c_{kk}c_{jj}^{-1}\sum_{m,p \ge 0}^{m + p = t} \xi^{\pm p}\ell_{ik}^{\pm}[\mp m]\ell_{j'k'}^{\pm}[\mp p] = 0
  \qquad \text{for all}\qquad t > 0,\ 1 \le i,j \le N,
\end{equation}
and
\begin{equation}\label{eq:LC_explicit_affine_2}
  \sum_{1\leq k\leq N} c_{kk}c_{jj}^{-1}\ell_{ik}^{\pm}[0]\ell_{j'k'}^{\pm}[0] = \delta_{ij}
  \qquad \text{for all}\qquad 1 \le i,j \le N.
\end{equation}
It is easy to verify that these expressions are homogeneous, and since $\deg(\ell_{n+1,n+1}^{\pm}[0]) = (0,0)$ in type $B_{n}$,
the relation~\eqref{eq:B_extra_relation_affine} is also homogeneous. Therefore, $\wtd{U}(\wtd{R}_{q}(z))$ inherits the
$\what{P} \times \what{P}$-grading from $U(\wtd{R}_{q}(z))$.

We shall now show that this bigrading makes $U(\wtd{R}_{q}(z))$ and $\wtd{U}(\wtd{R}_{q}(z))$ into $\what{P}$-bigraded Hopf algebras,
in the sense of~\eqref{eq:bigraded-bialgebra} and~\eqref{eq:bigraded-antipode}. The former of this is checked directly on the generators.
To show that the antipode satisfies~\eqref{eq:bigraded-antipode}, we first set up the following notation:
\[
  L^{\pm}(z) = \sum_{m \ge 0}L^{\pm}[\mp m]z^{\pm m} \qquad \text{where}\qquad L^{\pm}[\mp m] = (\ell_{ij}^{\pm}[\mp m])_{i,j = 1}^{N},
\]
and
\[
  L^{\pm}(z)^{-1} = \sum_{m \ge 0}\wtd{L}^{\pm}[\mp m]z^{\pm m}\qquad \text{where}\qquad
  \wtd{L}^{\pm}[\mp m] = (\wtd{\ell}_{ij}^{\pm}[\mp m])_{i,j = 1}^{N},
\]
for some $\wtd{\ell}_{ij}^{\pm}[\mp m]$ in $U(\wtd{R}_{q}(z))$ or $\wtd{U}(\wtd{R}_{q}(z))$. We shall show that
$\deg(\wtd{\ell}_{ij}^{+}[- m]) = (\varepsilon_{j} - \frac{m}{2}\delta,-\varepsilon_{i} + \frac{3m}{2}\delta)$ and
$\deg(\wtd{\ell}_{ij}^{-}[m]) = (\varepsilon_{j} - \frac{3m}{2}\delta,-\varepsilon_{i} + \frac{m}{2}\delta)$.
Since the imaginary part of the $\what{P}\times \what{P}$-degree of each entry of $L^{+}[-m]$ (resp.\ $L^{-}[m]$) is
$(-\frac{m}{2}\delta ,\frac{3m}{2}\delta)$ (resp.\ $(-\frac{3m}{2}\delta,\frac{m}{2}\delta)$), it is easy to see that the same is true for each
entry of $\wtd{L}^{\pm}[\mp m]$. Thus, it remains to check that $\wtd{\ell}_{ij}^{\pm}[\mp m]$ has degree $(\varepsilon_{j},-\varepsilon_{i})$
with respect to the $P \times P$-grading of part (b). For this, the same argument as in the proof of Proposition~\ref{prop:U(R)_grading} shows
that $\deg(\wtd{\ell}_{ij}^{\pm}[0]) = (\varepsilon_{j},-\varepsilon_{i})$, and using this one can check that the $(i,j)$-th entry of
$L^{\pm}[0]^{-1}L^{\pm}[\mp d]$ has degree $(0,\varepsilon_{j} -\varepsilon_{i})$ with respect to the $P \times P$-grading. Hence, all coefficients
of the $(i,j)$-th entry of $(L^{\pm}[0]^{-1}L^{\pm}(z))^{-1}$ have degree $(0,\varepsilon_{j} - \varepsilon_{i})$ as well. Combining this with
the equality
\[
  L^{\pm}(z)^{-1} = (L^{\pm}[0]^{-1}L^{\pm}(z))^{-1}L^{\pm}[0]^{-1},
\]
one then finds that $\deg(\wtd{\ell}_{ij}^{\pm}[\mp m]) = (\varepsilon_{j}, -\varepsilon_{i})$ with respect to the $P\times P$-grading of part (b),
as desired. Since $S(\ell_{ij}^{+}[- m]) = \wtd{\ell}_{ij}^{+}[-m]\gamma^{-2m}$ and $S(\ell_{ij}^{-}[m]) = \wtd{\ell}_{ij}^{-}[m](\gamma')^{-2m}$,
we get $\deg(S(\ell_{ij}^{+}[-m])) = (\varepsilon_{j} + \frac{3m}{2}\delta,-\varepsilon_{i} - \frac{m}{2}\delta)$ and
$\deg(S(\ell_{ij}^{-}[m])) = (\varepsilon_{j} + \frac{m}{2}\delta,-\varepsilon_{i} -\frac{3m}{2}\delta)$, as required for~\eqref{eq:bigraded-antipode}.
This completes the proof.
\end{proof}

Due to Lemma~\ref{lem:affine_RTT_bigradings}(b), we can define the twisted algebra $\wtd{U}(\wtd{R}_{q}(z))_{\zeta}$ for $\zeta$
satisfying~\eqref{eq:zeta_formula}, which has multiplication given by~\eqref{eq:twisted_product_general}. Moreover, if $D$ is the matrix
of~\eqref{eq:D_zeta}, then $D\wtd{R}_{q}(z)D$ satisfies~\eqref{eq:YBE_affine} by the results of Subsection~\ref{ssec:affine_R_twisting},
Lemma~\ref{lem:YBE_twist}, and the fact that $D$ is independent of $z$. Thus, we can form the algebras $U(D\wtd{R}_{q}(z)D)$ and
$\wtd{U}(D\wtd{R}_{q}(z)D)$, where the latter is obtained from the former by imposing the additional relations~\eqref{eq:LC_affine},
as well as~\eqref{eq:B_extra_relation_affine} in type $B_{n}$. Then, we have the following analogue of Proposition~\ref{prop:U(R)_twisted}:

\begin{prop}\label{prop:U(R)_twisted_affine}
There is a unique Hopf algebra isomorphism $\what{\Psi}\colon \wtd{U}(D\wtd{R}_{q}(z)D) \iso \wtd{U}(\wtd{R}_{q}(z))_{\zeta}$
mapping $\gamma^{1/2} \mapsto \gamma^{1/2}$, $(\gamma')^{1/2} \mapsto (\gamma')^{1/2}$, and
$\ell_{ij}^{\pm}[\mp m] \mapsto \ell_{ij}^{\pm}[\mp m]$ for all $1 \le i,j \le N$, $m \in \bb{Z}_{\ge 0}$.
\end{prop}

\begin{proof}
Applying the same arguments as those used in the proof of Proposition~\ref{prop:A(R)_twisted}, we find that there is an algebra isomorphism
$\what{\Psi}\colon U(D\wtd{R}_{q}(z)D) \iso U(\wtd{R}_{q}(z))_{\zeta}$ sending $\ell_{ij}^{\pm}[\mp m] \mapsto \ell_{ij}^{\pm}[\mp m]$
as well as $\gamma^{1/2}\mapsto \gamma^{1/2}, (\gamma')^{1/2}\mapsto (\gamma')^{1/2}$. To see
that this map descends to an isomorphism $\what{\Psi}\colon \wtd{U}(D\wtd{R}_{q}(z)D) \iso \wtd{U}(\wtd{R}_{q}(z))_{\zeta}$, we observe that
applying $\what{\Psi}$ to the left-hand sides of~\eqref{eq:LC_explicit_affine_1} and~\eqref{eq:LC_explicit_affine_2} multiplies the expressions
by $\zeta(\varepsilon_{i},\varepsilon_{j'})$, and since $\zeta(\varepsilon_{i},\varepsilon_{i'}) = 1$, this implies that $\what{\Psi}$ maps
the ideal in $U(D\wtd{R}_{q}(z)D)$ generated by the relations~\eqref{eq:LC_affine} onto the corresponding ideal in $U(\wtd{R}_{q}(z))$.
Additionally, in type $B_{n}$, we also have $\what{\Psi}(\ell_{n+1,n+1}^{\pm}[0] - 1) = \ell_{n+1,n+1}^{\pm}[0] - 1$. This completes the proof.
\end{proof}

Moreover, for any $s_{1},\ldots ,s_{N} \in \bb{K}\setminus \{0\}$, let $S' = \sum_{i=1}^{N} s_{i}E_{ii}$ and
$\bar{S}' = \sum_{i=1}^{N} s_{i'}E_{ii}$. Then the matrix $S = S' \otimes S'$ is of the form~\eqref{eq:S_matrix},
and we have the following analogue of Proposition~\ref{prop:diagonal_conjugation_U(R)}(b):

\begin{prop}\label{prop:diagonal_conjugation_affine}
There is a unique Hopf algebra isomorphism $\what{\Phi}_{S}\colon \wtd{U}(\wtd{R}_{q}(z)) \iso \wtd{U}(S\wtd{R}_{q}(z)S^{-1})$
mapping $\gamma^{1/2} \mapsto \gamma^{1/2}$, $(\gamma')^{1/2} \mapsto (\gamma')^{1/2}$,
$\ell_{ij}^{\pm}[\mp m] \mapsto s_{i}^{-1}s_{j}\ell_{ij}^{\pm}[\mp m]$, where $\wtd{U}(S\wtd{R}_{q}(z)S^{-1}) = U(S\wtd{R}_{q}(z)S^{-1})/\cal{I}_{S}$,
with the ideal $\cal{I}_{S}$ of $U(S\wtd{R}_{q}(z)S^{-1})$ being generated by the relations
\[
  L^{\pm}(z)(S'C_{q}\bar{S}')L^{\pm}(z\xi)^{\mathsf{t}'}(S'C_{q}\bar{S}')^{-1} = I.
\]
\end{prop}

\begin{proof}
The proof is completely analogous to those of Propositions~\ref{prop:diagonal_conjugation_A(R)} and~\ref{prop:diagonal_conjugation_U(R)}.
\end{proof}

We may now relate $\wtd{U}(\wtd{R}_{q}(z))$ to its two-parameter counterpart. For this, let $R_{r,s}(z) = \hat{R}_{r,s}(z) \circ \tau$,
where $\hat{R}_{r,s}(z)$ is the two-parameter $B_{n}$-type (resp.\ $C_{n}$-type or $D_{n}$-type) affine $R$-matrix of~\cite[(6.11)]{MT1}
(resp.~\cite[(6.12)]{MT1} or~\cite[(6.13)]{MT1}). We then define
\[
  \wtd{R}_{r,s}(z) = f(z)R_{r,s}(z),
\]
where $f(z)$ is the formal power series satisfying~\eqref{eq:f_qdifference} with $v = r^{-1}s$ in type $B_{n}$ and $v = r^{-1/2}s^{1/2}$
in types $C_{n}$ and $D_{n}$. Then it follows from Subsection~\ref{ssec:affine_R_twisting}, cf.~\eqref{eq:R-relation}, that
$\wtd{R}_{r,s}(z) = S^{-1}D\wtd{R}_{q}(z)DS$, where $q = r^{1/2}s^{-1/2}$, $D$ is the matrix of~\eqref{eq:D_zeta}, and
$S = \sum_{i,j=1}^{N} \psi_{i}\psi_{j}E_{ii} \otimes E_{jj}$ with $\psi_{i}$ given by~\eqref{eq:psi_values_B} in type $B_{n}$,~\eqref{eq:psi_values_C}
in type $C_{n}$, and~\eqref{eq:psi_values_D} in type $D_{n}$. Now, we define $\wtd{U}(\wtd{R}_{r,s}(z)) = U(\wtd{R}_{r,s}(z))/\cal{I}_{r,s}$,
where $\cal{I}_{r,s}$ is the ideal generated by the relations
\begin{equation}\label{eq:LC_affine_2_param}
  L^{\pm}(z)C_{r,s}L^{\pm}(z\xi)^{\mathsf{t}'}C_{r,s}^{-1} = I,
\end{equation}
with the matrix $C_{r,s}$ of~\eqref{eq:Crs_B} in type $B_{n}$,~\eqref{eq:Crs_C} in type $C_{n}$, and~\eqref{eq:Crs_D} in type $D_{n}$. As above,
for type $B_{n}$ we also impose the relations $\ell_{n+1,n+1}^{\pm}[0] = 1$ in $\wtd{U}(\wtd{R}_{r,s}(z))$. Then, combining
Propositions~\ref{prop:U(R)_twisted_affine} and~\ref{prop:diagonal_conjugation_affine}, we obtain the following affine analogue
of Corollary~\ref{cor:U(R)_2_vs_1_parameter}:

\begin{prop}\label{prop:2_vs_1_parameter_affine_RTT}
There is a unique Hopf algebra isomorphism $\what{\Psi}' \colon \wtd{U}(\wtd{R}_{r,s}(z)) \iso \wtd{U}(\wtd{R}_{q}(z))_{\zeta}$
mapping $\gamma^{1/2} \mapsto \gamma^{1/2}$, $(\gamma')^{1/2} \mapsto (\gamma')^{1/2}$, and
$\ell_{ij}^{\pm}[\mp m] \mapsto \psi_{i}^{-1}\psi_{j}\ell_{ij}^{\pm}[\mp m]$ for all $1 \le i,j \le N$, $m \in \bb{Z}_{\ge 0}$.
\end{prop}

   %%%%%%%%%%%%%%%%%%%%%%%%%%%%%%%%%%%%%%%%%%%%%%%%%%%%%%%%%%%%%%%%%%%%%%%%%%

\subsection{Triangular decomposition and the key embedding: affine FRT}
\

Next, we shall prove triangular decompositions for the algebras $\wtd{U}(\wtd{R}_{q}(z))$ and $\wtd{U}'(\wtd{R}_{q}(z))$, in analogy with those
for $\wtd{U}(R_{q})$ and $\wtd{U}'(R_{q})$ from Section~\ref{sec:FRT construction finite}, which are needed for the proof of an analogue of
Proposition~\ref{prop:double_cartan_RTT_finite}. First, we define the following elements of $\wtd{U}(\wtd{R}_{q}(z))$ and $\wtd{U}'(\wtd{R}_{q}(z))$:
\begin{equation}\label{eq:X+-_generators}
  X_{ij}^{+}[- m] = \ell_{ij}^{+}[-m]\ell_{jj}^{\pm}[0]^{-1}\gamma^{-3m/2}, \quad
  X_{ij}^{-}[m] = (\gamma')^{-3m/2}\ell_{ii}^{-}[0]^{-1}\ell_{ij}^{-}[m] \quad \forall\ 1 \le i,j \le N,\ m \in \bb{Z}_{\ge 0}.
\end{equation}

We also introduce the following subalgebras of $\wtd{U}(\wtd{R}_{q}(z))$ and $\wtd{U}'(\wtd{R}_{q}(z))$:

\begin{itemize}

\item $\wtd{U}_{\ge}(\wtd{R}_{q}(z)) \subset \wtd{U}(\wtd{R}_{q}(z))$ and $\wtd{U}'_{\ge}(\wtd{R}_{q}(z)) \subset \wtd{U}'(\wtd{R}_{q}(z))$,
both generated by $\gamma^{\pm 1/2}$ and all $\ell_{ij}^{+}[-m]$ with $1 \le i,j \le N$ and $m \in \bb{Z}_{\ge 0}$,

\item $\wtd{U}_{\le}(\wtd{R}_{q}(z)) \subset \wtd{U}(\wtd{R}_{q}(z))$ and $\wtd{U}'_{\le}(\wtd{R}_{q}(z)) \subset \wtd{U}'(\wtd{R}_{q}(z))$,
both generated by $(\gamma')^{\pm 1/2}$ and all $\ell_{ij}^{-}[m]$ with $1 \le i, j \le N$ and $m \in \bb{Z}_{\ge 0}$,

\item $\wtd{U}_{\pm}(\wtd{R}_{q}(z)) \subset \wtd{U}(\wtd{R}_{q}(z))$ and $\wtd{U}'_{\pm}(\wtd{R}_{q}(z)) \subset \wtd{U}'(\wtd{R}_{q}(z))$,
both generated by all $X_{ij}^{\pm}[\mp m]$ of~\eqref{eq:X+-_generators} with $1 \le i,j \le N$ and $m \in \bb{Z}_{\ge 0}$,

\item $\wtd{U}_{0}(\wtd{R}_{q}(z)) \subset \wtd{U}(\wtd{R}_{q}(z))$ and $\wtd{U}'_{0}(\wtd{R}_{q}(z)) \subset \wtd{U}'(\wtd{R}_{q}(z))$,
both generated by $\gamma^{\pm 1/2}, (\gamma')^{\pm 1/2}$, and all $\ell_{ii}^{+}[0]^{\pm 1},\ell_{ii}^{-}[0]^{\pm 1}$ with $1 \le i \le N$,

\item $\wtd{U}_{0,+}(\wtd{R}_{q}(z)) \subset \wtd{U}(\wtd{R}_{q}(z))$, generated by $\gamma^{\pm 1/2}$ and all
$\ell_{ii}^{+}[0]^{\pm 1}$ with $ 1 \le i \le N$,

\item $\wtd{U}_{0,-}(\wtd{R}_{q}(z)) \subset \wtd{U}(\wtd{R}_{q}(z))$, generated by $(\gamma')^{\pm 1/2}$ and all
$\ell_{ii}^{-}[0]^{\pm 1}$ with $ 1 \le i \le N$.

\end{itemize}
We then have the following result, which is proved in the same way as Lemma~\ref{lem:cartan_subalgebras_RTT}:

\begin{lemma}
There is an isomorphism from the Laurent polynomial algebra $\bb{K}[x_{0}^{\pm 1}, \ldots ,x_{n}^{\pm 1}]$ to $\wtd{U}_{0,+}(\wtd{R}_{q}(z))$
(resp.\ $\wtd{U}_{0,-}(\wtd{R}_{q}(z))$) given by $x_{0} \mapsto \gamma^{1/2}$ (resp.\ $x_{0} \mapsto (\gamma')^{1/2}$) and $x_{i} \mapsto \ell_{ii}^{+}[0]$
(resp.\ $x_{i} \mapsto \ell_{ii}^{-}[0]$) for $1 \le i \le n$.
\end{lemma}

For any $\mu = \sum_{i = 1}^{n}c_{i}\varepsilon_{i} \in P$, we define
\[
  \ell_{\mu}^{\pm}[0] = \ell_{11}^{\pm}[0]^{c_{1}}\ldots \ell_{nn}^{\pm}[0]^{c_{n}}.
\]
By the above lemma, the elements $\ell_{\mu}^{+}[0]\gamma^{m/2}$ (resp.\ $\ell_{\mu}^{-}[0](\gamma')^{m/2}$), ranging over all $m \in \bb{Z}$ and
$\mu \in P$, form a basis for $\wtd{U}_{0,+}(\wtd{R}_{q}(z))$ (resp.\ $\wtd{U}_{0,-}(\wtd{R}_{q}(z))$). Moreover, we also have the following analogue
of Lemma~\ref{lem:borel_spanning_sets}:

\begin{lemma}\label{lem:borel_spanning_sets_affine}
The elements
  $\{\ell_{\mu}^{+}[0]\gamma^{m/2}X_{i_{1}j_{1}}^{+}[-m_{1}]\ldots X_{i_{k}j_{k}}^{+}[-m_{k}] \,|\,
     k,m_{\ell} \in \bb{Z}_{\ge 0},\, 1 \le i_{\ell},j_\ell \le N,\, m \in \bb{Z},\, \mu \in P \}$
span $\wtd{U}_{\ge}(\wtd{R}_{q}(z))$, while the elements
  $\{X^-_{i_{1}j_{1}}[m_{1}]\ldots X^-_{i_{k}j_{k}}[m_{k}](\gamma')^{m/2}\ell_{\mu}^{-}[0] \,|\,
     k,m_{\ell} \in \bb{Z}_{\ge 0},\, \mu \in P,\, 1 \le i_{\ell},j_\ell \le N,\, m \in \bb{Z}\}$
span $\wtd{U}_{\le}(\wtd{R}_{q}(z))$.
\end{lemma}

\begin{proof}
It is clear that $\wtd{U}_{\ge}(\wtd{R}_{q}(z))$ is generated as an algebra by all $X_{ij}^{+}[-m]$, $\gamma^{\pm 1/2}$, and $\ell_{\mu}^{+}[0]$,
and likewise, $\wtd{U}_{\le}(\wtd{R}_{q}(z))$ is generated as an algebra by all $X_{ij}^{-}[m]$, $(\gamma')^{\pm 1/2}$, and $\ell_{\mu}^{-}[0]$.
Thus, as in the proof of Lemma~\ref{lem:borel_spanning_sets}, it is enough to prove the following formulas:
\begin{equation}\label{eq:RTT_q_commute_affine}
  \ell_{\mu}^{+}[0]u = q^{-(\mu,\lambda + \lambda')}u\ell_{\mu}^{+}[0] \qquad \mathrm{and} \qquad
  \ell_{\mu}^{-}[0]u' = q^{(\mu,\lambda + \lambda')}u'\ell_{\mu}^{-}[0]
\end{equation}
for all $u \in \wtd{U}_{\ge}(\wtd{R}_{q}(z))_{\lambda,\lambda'}$, $u' \in \wtd{U}_{\le}(\wtd{R}_{q}(z))_{\lambda,\lambda'}$.
Since the arguments are very similar to those used in the proof of Lemma~\ref{lem:borel_spanning_sets}, we leave details to the reader.
\end{proof}

Next, we can define a left action of $\wtd{U}_{0,-}(\wtd{R}_{q}(z))$ on $\wtd{U}_{-}(\wtd{R}_{q}(z))$ by
$\ell_{\mu}^{-}[0]\cdot u = q^{(\mu,\lambda + \lambda')}u$ for any $u \in \wtd{U}_{-}(\wtd{R}_{q}(z))_{\lambda,\lambda'}$, and this clearly makes
$\wtd{U}_{-}(\wtd{R}_{q}(z))$ into a left module-algebra over $\wtd{U}_{0,-}(\wtd{R}_{q}(z))$, and thus we can form the left crossed product algebra
$\wtd{U}_{-}(\wtd{R}_{q}(z)) \rtimes \wtd{U}_{0,-}(\wtd{R}_{q}(z))$, which is $\wtd{U}_{-}(\wtd{R}_{q}(z)) \otimes \wtd{U}_{0,-}(\wtd{R}_{q}(z))$
as a vector space, and multiplication is given by~\eqref{eq:left_crossed_product}. Likewise, we can define a right action of $\wtd{U}_{0,+}(\wtd{R}_{q}(z))$
on $\wtd{U}_{+}(\wtd{R}_{q}(z))$ by $u \cdot \ell_{\mu}^{+} = q^{(\mu,\lambda + \lambda')}u$ for all $u \in \wtd{U}_{+}(\wtd{R}_{q}(z))_{\lambda,\lambda'}$.
Since this makes $\wtd{U}_{+}(\wtd{R}_{q}(z))$ into a right module-algebra $\wtd{U}_{0,+}(\wtd{R}_{q}(z))$, we can form the right crossed product algebra
$\wtd{U}_{0,+}(\wtd{R}_{q}(z)) \ltimes \wtd{U}_{+}(\wtd{R}_{q}(z))$, which is $\wtd{U}_{0,+}(\wtd{R}_{q}(z)) \otimes \wtd{U}_{+}(\wtd{R}_{q}(z))$ as
a vector space, and multiplication is given by~\eqref{eq:right_crossed_product}.

We then get the following analogue of Proposition~\ref{prop:RTT_crossed_products}:

\begin{prop}\label{prop:RTT_crossed_products_affine}
The assignments
\begin{align*}
  & \what{\psi}_{+}(X_{ij}^{+}[-m]) = 1 \otimes X_{ij}^{+}[-m], &
  & \what{\psi}_{+}(\ell_{ii}^{+}[0]) = \ell_{ii}^{+}[0] \otimes 1, &
  & \what{\psi}_{+}(\gamma^{\pm 1/2}) = \gamma^{\pm 1/2} \otimes 1, \\
  & \what{\psi}_{-}(X_{ij}^{-}[m]) = X_{ij}^{-}[m] \otimes 1, &
  & \what{\psi}_{-}(\ell_{ii}^{-}[0]) = 1 \otimes \ell_{ii}^{-}[0], &
  & \what{\psi}_{-}((\gamma')^{\pm 1/2}) = 1\otimes (\gamma')^{\pm 1/2},
\end{align*}
for all $1 \le i,j \le N$, $m \in \bb{Z}_{\geq 0}$ give rise to algebra isomorphisms
$\what{\psi}_{+}\colon \wtd{U}_{\ge}(\wtd{R}_{q}(z)) \iso \wtd{U}_{0,+}(\wtd{R}_{q}(z)) \ltimes \wtd{U}_{+}(\wtd{R}_{q}(z))$ and
$\what{\psi}_{-}\colon \wtd{U}_{\le}(\wtd{R}_{q}(z)) \iso \wtd{U}_{-}(\wtd{R}_{q}(z)) \rtimes \wtd{U}_{0,-}(\wtd{R}_{q}(z))$.
\end{prop}

\begin{proof}
Since $\deg(X_{ij}^{+}[-m]) = (\varepsilon_{j} - \varepsilon_{i} + m\delta,0)$ and $\deg(X_{ij}^{-}[m]) = (0,\varepsilon_{j} - \varepsilon_{i} - m\delta)$,
it follows from Lemma~\ref{lem:borel_spanning_sets_affine} that $u(\gamma')^{k/2}\ell_{\lambda}^{-}[0] \in \wtd{U}_{-}(\wtd{R}_{q}(z))$ whenever
$u \in \wtd{U}_{\le}(\wtd{R}_{q}(z))_{\lambda + \frac{k}{2}\delta,\mu + \frac{\ell}{2}\delta}$ and
$\gamma^{-\ell/2}(\ell_{\mu}^{+}[0])^{-1}u \in \wtd{U}_{+}(\wtd{R}_{q}(z))$ whenever
$u \in \wtd{U}_{\ge}(\wtd{R}_{q}(z))_{\lambda + \frac{k}{2}\delta,\mu + \frac{\ell}{2}\delta}$. Thus, the result follows by
using~\eqref{eq:RTT_q_commute_affine} and arguments similar to those in the proof of Proposition~\ref{prop:RTT_crossed_products}.
\end{proof}

As in Section~\ref{sec:FRT construction finite}, we shall now realize $\wtd{U}(\wtd{R}_{q}(z))$ as a Drinfeld double of $\wtd{U}_{\le}(\wtd{R}_{q}(z))$
and $\wtd{U}_{\ge}(\wtd{R}_{q}(z))$. To do so, we first introduce algebras $A_{\pm}(R(z))$ associated to any solution $R(z)$ of the Yang-Baxter equation
with a spectral parameter~\eqref{eq:YBE_affine}. The generators of $A_{\pm}(R(z))$ are
  $\{\gamma_{\pm}^{1/2},\gamma_{\pm}^{-1/2}\} \cup \{t_{ij}^{\pm}[\mp m] \,|\, 1 \le i,j \le N,\, m \in \bb{Z}_{\geq 0}\}$,
where
\begin{equation}\label{eq:central_elements_A(R(z))}
  \gamma_{\pm}^{1/2},\gamma_{\pm}^{-1/2}-\text{central}, \qquad \gamma_{\pm}^{1/2}\gamma_{\pm}^{-1/2} = 1,
\end{equation}
and the remaining defining relations take the form
\begin{equation}\label{eq:RTT_affine}
  R(z/w)T_{1}^{\pm}(z)T_{2}^{\pm}(w) = T_{2}^{\pm}(w)T_{1}^{\pm}(z)R(z/w),
\end{equation}
where $T_{1}^{\pm}(z) = \sum_{i,j=1}^{N} E_{ij} \otimes 1 \otimes t_{ij}^{\pm}(z)$ and
$T_{2}^{\pm}(z) = \sum_{i,j=1}^{N} 1 \otimes E_{ij} \otimes t_{ij}^{\pm}(z)$, with
\[
  t_{ij}^{\pm}(z) = \sum_{m \geq 0} t_{ij}^{\pm}[\mp m]z^{\pm m}.
\]
The algebras $A_{\pm}(R(z))$ have bialgebra structures with coproduct $\Delta_{\pm}$ given by
\begin{equation}\label{eq:affine_A(R)_coproducts}
\begin{split}
  &\Delta_{+}(t_{ij}^{+}(z)) =  \sum_{1\leq k\leq N} t_{ik}^{+}(z(1 \otimes \gamma_{+}^{3/2})) \otimes t_{kj}^{+}(z(\gamma_{+}^{1/2} \otimes 1)), \\
  &\Delta_{-}(t_{ij}^{-}(z)) = \sum_{1\leq k\leq N} t_{ik}^{-}(z(1 \otimes \gamma_{-}^{-1/2})) \otimes t_{kj}^{-}(z(\gamma_{-}^{-3/2} \otimes 1)),
\end{split}
\end{equation}
as well as $\Delta_\pm(\gamma_{\pm}^{1/2})=\gamma_{\pm}^{1/2}\otimes \gamma_{\pm}^{1/2}$, and counit $\epsilon_{\pm}$ given by
$\epsilon_{\pm}(t_{ij}^{\pm}(z)) = \delta_{ij}$, $\epsilon_\pm(\gamma_{\pm}^{1/2})=1$. We shall mainly be working with the quotient algebras
$\wtd{A}_{\pm}(\wtd{R}_{q}(z)) = A_{\pm}(\wtd{R}_{q}(z))/\cal{I}_{\pm}$, where $\cal{I}_{\pm} = \cal{I}_{\pm}' + \cal{I}_{\pm}''$ with $\cal{I}_{\pm}'$
being the ideals generated by the entries of the matrix
\begin{equation}\label{eq:TC_relations_affine}
  T^{\pm}(z)C_{q}T^{\pm}(z\xi)^{\mathsf{t}'}C_{q}^{-1} - I,
\end{equation}
and $\cal{I}_{+}''$ (resp.\ $\cal{I}_{-}''$) is the ideal generated by $\{t_{ij}^{+}[0]\}_{i<j}$ (resp.\ $\{t_{ij}^{-}[0]\}_{i>j}$),
and the additional element $t_{n+1,n+1}^{+}[0] - 1$ (resp.\ $t_{n+1,n+1}^{-}[0] - 1$) if we are in type $B_{n}$. Since one can prove an analogue
of Proposition~\ref{prop:central_element_affine} for the ideals $\cal{I}_{\pm}'$ in $A_{\pm}(\wtd{R}_{q}(z))$, they are coideals.
As $\cal{I}_{\pm}''$ are also coideals, $\wtd{A}_{\pm}(\wtd{R}_{q}(z))$ are both bialgebras. We note that the matrices
$(t_{ij}^{\pm}[0])_{i,j = 1}^{N}$ are invertible in $\wtd{A}_{\pm}(\wtd{R}_{q}(z))$ as $t^\pm_{ii}[0]^{-1}=t^\pm_{i'i'}[0]$,
hence, these bialgebras are actually Hopf algebras, with antipodes determined by
\begin{equation*}
  S_\pm(T^{\pm}(z)) = (T^{\pm}(z\gamma_{\pm}^{\mp 2}))^{-1}, \qquad S_{\pm}(\gamma_{\pm}^{1/2})=\gamma_{\pm}^{-1/2}.
\end{equation*}

For the next result, we first need to record a few identities among $\wtd{R}_{q}(z)$ and $C_{q}$:

\begin{lemma}\label{lem:affine_crossing_symmetries}
The following \textbf{crossing symmetry} identities hold:
\begin{equation}\label{eq:affine_crossing_symmetries_1}
  \wtd{R}_{q}(z)(C_{q})_{1}\wtd{R}_{q}(z\xi)^{\mathsf{t}_{1}'}(C_{q}^{-1})_{1} = \xi^{2} v^{2}I, \qquad
  (C_{q}^{\mathsf{t}'})^{-1}_{1}\wtd{R}_{q}(z\xi)^{-1}(C_{q}^{\mathsf{t}'})_{1}(\wtd{R}_{q}(z)^{-1})^{\mathsf{t}'_1} = \xi^{-2}v^{-2}I,
\end{equation}
\begin{equation}\label{eq:affine_crossing_symmetries_2}
  (C_{q}^{\mathsf{t}'})_{2}^{-1}\wtd{R}_{q}(z\xi^{-1})(C_{q}^{\mathsf{t}'})_{2}\wtd{R}_{q}(z)^{\mathsf{t}'_2} =
  \xi^{2}v^{2}I, \qquad
  \wtd{R}_{q}(z)^{-1}(C_{q})_{2}(\wtd{R}_{q}(z\xi^{-1})^{-1})^{\mathsf{t}_{2}'}(C_{q}^{-1})_{2} = \xi^{-2}v^{-2}I.
\end{equation}
\end{lemma}

We postpone the proof of this Lemma to Appendix~\ref{ssec:app_C}, as it crucially relies on Theorem~\ref{thm:universal_R_evaluation}.

\begin{prop}
There are unique bilinear maps $\sigma,\bar{\sigma}\colon \wtd{A}_{+}(\wtd{R}_{q}(z)) \times \wtd{A}_{-}(\wtd{R}_{q}(z)) \to \bb{K}$ such that
\begin{equation}\label{eq:sigma_generators}
  \sigma(t_{ij}^{+}(z),t_{k\ell}^{-}(w)) = \xi^{-1} v^{-1}(\wtd{R}_{q})^{ik}_{j\ell}(z/w),\qquad
  \sigma(\gamma_{+}^{\pm 1/2},u_{-}) = \epsilon_{-}(u_{-}),\qquad \sigma(u_{+},\gamma_{-}^{\pm 1/2}) = \epsilon_{+}(u_{+}),
\end{equation}
\begin{equation}\label{eq:sigma_bar_generators}
  \bar{\sigma}(t_{ij}^{+}(z),t_{k\ell}^{-}(w)) = \xi v(\wtd{R}_{q}^{-1})^{ik}_{j\ell}(z/w), \qquad
  \bar{\sigma}(\gamma_{+}^{\pm 1/2},u_{-}) = \epsilon_{-}(u_{-}),\qquad \bar{\sigma}(u_{+},\gamma_{-}^{\pm 1/2}) = \epsilon_{+}(u_{+}),
\end{equation}
for all $1 \le i,j \le N$, $u_{\pm} \in \wtd{A}_{\pm}(\wtd{R}_{q}(z))$, and which satisfy~\eqref{eq:skew_pairing_counit}--\eqref{eq:convolution_inverse}.
\end{prop}

\begin{proof}
The arguments showing that there are bilinear maps $\sigma,\bar{\sigma}\colon A_{+}(\wtd{R}_{q}(z)) \times A_{-}(\wtd{R}_{q}(z)) \to \bb{K}$
satisfying~\eqref{eq:skew_pairing_counit}--\eqref{eq:convolution_inverse} and~\eqref{eq:sigma_generators}--\eqref{eq:sigma_bar_generators} are
similar to the proof of Theorem~\ref{thm:A(R)_skew_pairings}, cf.~\cite[Theorem 10.7]{KS}. Indeed, we first consider the free algebras
$A_{\pm} = \bb{K}\langle t_{ij}^{\pm}[\mp m] \,|\, 1 \le i,j \le N,\, m \in \bb{Z}_{\ge 0}\rangle$ generated by all $t_{ij}^{\pm}[\mp m]$.
Then the Laurent polynomial algebras $A_{\pm}[\gamma_{\pm}^{ 1/2},\gamma_{\pm}^{-1/2}]$ have bialgebra structures with counits given by
$\epsilon_{\pm}(t_{ij}^{\pm}(z)) = \delta_{ij}, \epsilon_\pm(\gamma_{\pm}^{1/2})=1$, and coproducts $\Delta_{\pm}$ given by the
formulas~\eqref{eq:affine_A(R)_coproducts} and $\Delta_\pm(\gamma_{\pm}^{1/2})=\gamma_{\pm}^{1/2}\otimes \gamma_{\pm}^{1/2}$.
We also have bialgebra structures on $A_{\pm}$ with counits $\epsilon_{\pm}(t_{ij}^{\pm}(z)) = \delta_{ij}$
and the coproducts $\Delta_{\pm}$ obtained from the formulas~\eqref{eq:affine_A(R)_coproducts} by setting $\gamma_{\pm} = 1$.
We now define a bilinear form $\sigma_{0}\colon A_{+} \times A_{-} \to \bb{K}$ inductively, by first setting
$\sigma_{0}(t_{ij}^{+}[-m],1) = \sigma_{0}(1,t_{ij}^{-}[m]) = \delta_{m,0}\delta_{ij}$ and
$\sigma_{0}(t_{ij}^{+}[-m],t_{k\ell}^{-}[p]) = \delta_{m,p}\xi^{-1} v^{-1}(\wtd{R}_{q})^{ik}_{j\ell}[-m]$, and then
extending $\sigma_{0}$ to all of $A_{\pm}$ using~\eqref{eq:skew_pairing_adjointness}. Next, let
$\varphi_{\pm}\colon A_{\pm}[\gamma_{\pm}^{ 1/2},\gamma_{\pm}^{-1/2}] \to A_{\pm}$ be the algebra homomorphisms defined by
$\gamma_{\pm} \mapsto 1$. It is clear that $(\varphi_{\pm} \otimes \varphi_{\pm})\Delta_{\pm} = \Delta_{\pm}\varphi_{\pm}$, and hence
if we define $\sigma(a,b) = \sigma_{0}(\varphi_{\pm}(a),\varphi_{\pm}(b))$, then $\sigma$ is a bilinear form on
$A_{\pm}[\gamma_{\pm}^{1/2},\gamma_{\pm}^{-1/2}]$ such that~\eqref{eq:skew_pairing_counit},~\eqref{eq:skew_pairing_adjointness},
and~\eqref{eq:sigma_generators} hold. A similar argument shows that there is a bilinear form $\bar{\sigma}$ such
that~\eqref{eq:skew_pairing_counit},~\eqref{eq:inverse_skew_pairing_adjointness}, and~\eqref{eq:sigma_bar_generators} hold. Next, we need
to show that $\sigma$ and $\bar{\sigma}$ induce bilinear forms on $A_{+}(\wtd{R}_{q}(z)) \times A_{-}(\wtd{R}_{q}(z))$ such
that~\eqref{eq:skew_pairing_counit}--\eqref{eq:convolution_inverse},~\eqref{eq:sigma_generators}, and~\eqref{eq:sigma_bar_generators} hold.
First, this requires showing that the ideals $\cal{J}_{\pm}$ generated by the relations~\eqref{eq:RTT_affine}, with $\wtd{R}_{q}(z/w)$ in
place of $R(z/w)$, satisfy $\sigma(\cal{J}_{+},A_{-}(\wtd{R}_{q}(z))) = \sigma(A_{+}(\wtd{R}_{q}(z)),\cal{J}_{-}) = 0$, along with
similar identities for $\bar{\sigma}$. As in the proof of~\cite[Theorem 10.7]{KS}, a computation shows that this is equivalent to the
fact that $\wtd{R}_{q}(z)$ satisfies the Yang-Baxter equation with a spectral parameter~\eqref{eq:YBE_affine}. Moreover, it is easy to see
that~\eqref{eq:convolution_inverse} is equivalent to the identities
$\wtd{R}_{q}(z/w)\wtd{R}_{q}(z/w)^{-1} = \wtd{R}_{q}(z/w)^{-1}\wtd{R}_{q}(z/w) = I$.

Next, we need to prove that
\[
  \sigma(\cal{I}_{+}',A_{-}(\wtd{R}_{q}(z))) = \sigma(A_{+}(\wtd{R}_{q}(z)),\cal{I}_{-}') = 0 =
  \sigma(\cal{I}_{+}'',A_{-}(\wtd{R}_{q}(z))) = \sigma(A_{+}(\wtd{R}_{q}(z)),\cal{I}_{-}''),
\]
along with similar identities for $\bar{\sigma}$. The arguments for $\cal{I}_{\pm}''$ are the same as in the proof of
Proposition~\ref{prop:skew_pairing_descent}, since $\wtd{R}_{q}(0)$ is a multiple of $R_{q}$, and $\sigma(t_{ij}^{+}[-m],t_{k\ell}^{-}[p]) = 0$
unless $m = p$. For $\cal{I}_{\pm}'$, the arguments use the crossing symmetry identities of Lemma~\ref{lem:affine_crossing_symmetries},
similarly to the first part of the proof of Proposition~\ref{prop:skew_pairing_descent}. Because an analogue of
Proposition~\ref{prop:central_element_affine} holds for the ideals $\cal{I}_{\pm}'$, we know that they are generated by the coefficients
of a series $\frak{z}_{q}^{\pm}(z) - 1$, with $\frak{z}_{q}^{\pm}(z)$ given by
\[
  \frak{z}_{q}^{\pm}(z) = \sum_{1\leq k\leq N} c_{ii}c_{kk}^{-1}t_{k'i'}^{\pm}(z\xi)t_{ki}^{\pm}(z) =
  \sum_{1\leq k\leq N} c_{jj}^{-1}c_{kk}t_{jk}^{\pm}(z)t_{j'k'}^{\pm}(z\xi)  \qquad \text{for all}\qquad 1 \le i,j \le N.
\]
As $\frak{z}_{q}^{\pm}(z)$ satisfy the analogue of~\eqref{eq:zq_coproduct},
to prove the vanishing of $\sigma(\cal{I}_{+}',A_{-}(\wtd{R}_{q}(z)))$, it suffices to show
\begin{equation}\label{eq:sigma_I+'_vanishing}
  \delta_{i\ell} = \sigma(\frak{z}_{q}^{+}(z),t_{i\ell}^{-}(w)).
\end{equation}
Since
\begin{align*}
  \sigma(\frak{z}_{q}^{+}(z),t_{i\ell}^{-}(w)) &= \sum_{1\leq k\leq N} c_{jj}^{-1}c_{kk}\sigma(t_{jk}^{+}(z)t_{j'k'}^{+}(z\xi),t_{i\ell}^{-}(w))
  \, = \sum_{1 \le k,p \le N}c_{jj}^{-1}c_{kk}\sigma(t_{jk}^{+}(z),t_{ip}^{-}(w))\sigma(t_{j'k'}^{+}(z\xi),t_{p\ell}^{-}(w)) \\
  &= \sum_{1 \le k,p \le N} c_{jj}^{-1}c_{kk}\xi^{-2}v^{-2}(\wtd{R}_{q})^{ji}_{kp}(z/w)(\wtd{R}_{q})^{j'p}_{k'\ell}(z\xi/w)
   = \delta_{i \ell},
\end{align*}
where the last equality follows from the first identity of~\eqref{eq:affine_crossing_symmetries_1}. Similarly,
\begin{align*}
  \bar{\sigma}(\frak{z}_{q}^{+}(z),t_{i\ell}^{-}(w))
  &= \sum_{1 \le k,p \le N}c_{jj}^{-1}c_{kk}\bar{\sigma}(t_{j'k'}^{+}(z\xi),t_{ip}^{-}(w))\bar{\sigma}(t_{jk}^{+}(z),t_{p\ell}^{-}(w)) \\
  &= \sum_{1 \le k,p \le N}c_{jj}^{-1}c_{kk}\xi^{2}v^{2}
   (\wtd{R}_{q}^{-1})^{j'i}_{k'p}(z\xi/w)(\wtd{R}_{q}^{-1})^{jp}_{k\ell}(z/w) = \delta_{i \ell},
\end{align*}
where the last equality follows from the second identity of~\eqref{eq:affine_crossing_symmetries_1}.

One similarly uses~\eqref{eq:affine_crossing_symmetries_2} to show
$\sigma(t_{i\ell}^{+}(z),\frak{z}_{q}^{-}(w)) = \bar{\sigma}(t_{i\ell}^{+}(z),\frak{z}_{q}^{-}(w)) = \delta_{i\ell}$. This completes the proof.
\end{proof}

As a consequence of the above result, we can form the Drinfeld double $\cal{D}(\wtd{A}_{+}(\wtd{R}_{q}(z)),\wtd{A}_{-}(\wtd{R}_{q}(z));\sigma)$,
in the sense of~\cite[Definition 8.4]{KS}. Then we obtain the following result:

\begin{prop}\label{prop:RTT_drinfeld_double_affine}
There is a Hopf algebra isomorphism $\what{\psi}\colon \wtd{U}(\wtd{R}_{q}(z)) \to  \cal{D}(\wtd{A}_{+}(\wtd{R}_{q}(z)),\wtd{A}_{-}(\wtd{R}_{q}(z));\sigma)$
such that $\ell_{ij}^{+}[-m] \mapsto 1 \otimes t_{ij}^{+}[-m]$, $\ell_{ij}^{-}[m] \mapsto t_{ij}^{-}[m] \otimes 1$,
$\gamma^{\pm 1/2} \mapsto 1 \otimes \gamma_{+}^{\pm 1/2}$, and $(\gamma')^{\pm 1/2} \mapsto \gamma_{-}^{\pm 1/2} \otimes 1$.
\end{prop}

\begin{proof}
Note that $\cal{D}(\wtd{A}_{+}(\wtd{R}_{q}(z)),\wtd{A}_{-}(\wtd{R}_{q}(z));\sigma)$ is isomorphic to the algebra generated by
$t_{ij}^{\pm}[\mp m],\gamma_{\pm}^{1/2},\gamma_{\pm}^{-1/2}$ with the defining
relations~\eqref{eq:central_elements_A(R(z))},~\eqref{eq:RTT_affine},~\eqref{eq:TC_relations_affine},
$t_{ij}^{+}[0] = t_{ji}^{-}[0] = 0$ for $i < j$ (and the additional relation $t_{n+1,n+1}^{\pm}[0] = 1$ in type $B_{n}$),
and the relations (\cite[(22) of \S8.2]{KS})
\begin{equation}\label{eq:double_cross_relations}
  ab = \sum_{(a)(b)}b_{(2)}a_{(2)}\bar{\sigma}(a_{(1)},b_{(1)})\sigma(a_{(3)},b_{(3)}) \qquad
  \text{for all}\qquad a \in \wtd{A}_{+}(\wtd{R}_{q}(z)),\ b \in \wtd{A}_{-}(\wtd{R}_{q}(z)).
\end{equation}
In fact, since the coproducts $\Delta_{\pm}(t_{ij}^{\pm}[\mp m])$, $\Delta_{\pm}(\gamma_{\pm}^{1/2})$, $\Delta_{\pm}(\gamma_{\pm}^{-1/2})$
only involve elements $t_{k\ell}^{\pm}[\mp p]$, $\gamma_{\pm}^{1/2}$, $\gamma_{\pm}^{-1/2}$, it is enough to impose the
relations~\eqref{eq:double_cross_relations} only for $a \in \{t_{ij}^{+}[-m] \,|\, 1 \le i,j \le N,\, m \in \bb{Z}_{\ge 0}\}$
and $b \in \{t_{ij}^{-}[m] \,|\, 1 \le i,j \le N,\, m \in \bb{Z}_{\ge 0}\}$. Thus, it is enough to show that
relations~\eqref{eq:double_cross_relations} for $a=t_{ij}^{+}[-m], b=t_{k\ell}^{-}[p]$ (with $1\leq i,j,k,\ell\leq N$ and $m,p\geq 0$)
are equivalent to the second line of~\eqref{eq:Affine_RLL}. For this, we first note that the second and third equalities
in~\eqref{eq:sigma_generators} and~\eqref{eq:sigma_bar_generators} imply that
$\sigma(u_{+}\gamma_{+}^{d/2},u_{-}\gamma_{-}^{d'/2}) = \sigma(u_{+},u_{-})$ and
$\bar{\sigma}(u_{+}\gamma_{+}^{d/2},u_{-}\gamma_{-}^{d'/2}) = \sigma(u_{+},u_{-})$ for all $d,d' \in \bb{Z}$
and $u_{\pm} \in \wtd{A}_{\pm}(\wtd{R}_{q}(z))$. Then since
\[
  (1 \otimes \Delta_{+})\Delta_{+}(t_{ij}^{+}[-m])
  = \sum^{\substack{d_{1},d_{2},d_{3}\geq 0 \\ d_{1} + d_{2} + d_{3} = m}}_{1 \le k,\ell \le N}
    t_{ik}^{+}[-d_{1}]\gamma_{+}^{(d_{2} + d_{3})/2} \otimes
    t_{k\ell}^{+}[-d_{2}]\gamma_{+}^{(3d_{1} + d_{3})/2} \otimes t_{\ell j}^{+}[-d_{3}]\gamma_{+}^{3(d_{1} + d_{2})/2}
\]
and
\[
  (1 \otimes \Delta_{-})\Delta_{-}(t_{ij}^{-}[m])
  = \sum^{\substack{d_{1},d_{2},d_{3}\geq 0 \\ d_{1} + d_{2} + d_{3} = m}}_{1 \le k,\ell \le N}
    t_{ik}^{-}[d_{1}]\gamma_{-}^{3(d_{2} + d_{3})/2} \otimes
    t_{k\ell}^{-}[d_{2}]\gamma_{-}^{(d_{1} + 3d_{3})/2} \otimes t_{\ell j}^{-}[d_{3}]\gamma_{-}^{(d_{1} + d_{2})/2},
\]
we find that the relation~\eqref{eq:double_cross_relations} with $a = t_{ij}^{+}[-m_{1}]$ and $b = t_{k\ell}^{-}[m_{2}]$ becomes
\begin{align*}
  t_{ij}^{+}[-m_{1}]t_{k\ell}^{-}[m_{2}]
  = \sum t_{p_{2}h_{2}}^{-}[d_{2}']\gamma_{-}^{(d_{1}' + 3d_{3}')/2}t_{p_{1}h_{1}}^{+}[-d_{2}]\gamma_{+}^{(3d_{1} + d_{3})/2}
    \bar{\sigma}(t_{ip_{1}}^{+}[-d_{1}],t_{kp_{2}}^{-}[d_{1}'])\sigma(t_{h_{1}j}^{+}[-d_{3}],t_{h_{2}\ell}^{-}[d_{3}']),
\end{align*}
where the sum above ranges over all $1 \le p_{1},h_{1},p_{2},h_{2} \le N$ and $d_{1},d_{1}',d_{2},d_{2}',d_{3},d_{3}' \ge 0$ with
$d_{1} + d_{2} + d_{3} = m_{1}$ and $d_{1}' + d_{2}' + d_{3}' = m_{2}$. Using the first equalities in~\eqref{eq:sigma_generators}
and~\eqref{eq:sigma_bar_generators}, the above relation is equivalent to
\[
  t_{ij}^{+}[-m_{1}]t_{k\ell}^{-}[m_{2}]
  = \sum t_{p_{2}h_{2}}^{-}[d_{2}']t_{p_{1}h_{1}}^{+}[-d_{2}](\gamma_{-}^{3/2}\gamma_{+}^{1/2})^{d_{3}}(\gamma_{+}^{3/2}\gamma_{-}^{1/2})^{d_{1}}
    (\wtd{R}_{q}^{-1})^{ik}_{p_{1}p_{2}}[-d_{1}](\wtd{R}_{q})^{h_{1}h_{2}}_{j\ell}[-d_{3}],
\]
where the sum is over all $d_1,d_3\geq 0$ such that $d_{2}=m_1-d_1-d_3, d'_2=m_2-d_1-d_3$ are both $\geq 0$.  One can easily check that the
right-hand side of the equality above is precisely the $z^{m_{1}}w^{-m_{2}}E_{ij} \otimes E_{k \ell}$-coefficient of
\[
  \wtd{R}_{q}(z\gamma_{+}(\gamma_{+}\gamma_{-})^{1/2}/w)^{-1}T_{2}^{-}(w)T_{1}^{+}(z)\wtd{R}_{q}(z\gamma_{-}(\gamma_{+}\gamma_{-})^{1/2}/w).
\]
This completes the proof.
\end{proof}

Combining Propositions~\ref{prop:RTT_crossed_products_affine} and~\ref{prop:RTT_drinfeld_double_affine}, we immediately obtain the following result:

\begin{theorem}
The multiplication map
  $\wtd{U}_{-}(\wtd{R}_{q}(z)) \otimes \wtd{U}_{0,-}(\wtd{R}_{q}(z)) \otimes \wtd{U}_{0,+}(\wtd{R}_{q}(z)) \otimes \wtd{U}_{+}(\wtd{R}_{q}(z))
   \to \wtd{U}(\wtd{R}_{q}(z))$
is an isomorphism of vector spaces.
\end{theorem}

As with Theorem~\ref{thm:quadrilangular_decomposition} in the finite case, the above result has a number of consequences, analogous to
Corollaries~\ref{cor:triangular_decomposition_RTT}--\ref{cor:positive_negative_RTT} and Proposition~\ref{prop:triangular_decomposition_RTT'}.
Since the proofs in the affine case are essentially the same as the proofs of the aforementioned results, we shall just record the statements:

\begin{cor}\label{cor:triangular_decomposition_affine_RTT}
The multiplication map $\wtd{U}_{-}(\wtd{R}_{q}(z)) \otimes \wtd{U}_{0}(\wtd{R}_{q}(z)) \otimes \wtd{U}_{+}(\wtd{R}_{q}(z)) \to \wtd{U}(\wtd{R}_{q}(z))$
is an isomorphism of vector spaces.
\end{cor}

\begin{cor}\label{cor:doubled_cartan_subalgebra_affine_RTT}
There is an isomorphism from the Laurent polynomial algebra $\bb{K}[x_{0}^{\pm 1},\ldots ,x_{n}^{\pm 1},y_{0}^{\pm 1},\ldots ,y_{n}^{\pm 1}]$
to $\wtd{U}_{0}(\wtd{R}_{q}(z))$ given by $x_{0} \mapsto \gamma^{1/2}$, $y_{0} \mapsto (\gamma')^{1/2}$, and
$x_{i} \mapsto \ell_{ii}^{+}[0]$, $y_{i} \mapsto \ell_{ii}^{-}[0]$ for all $1 \le i \le n$.
\end{cor}

\begin{cor}\label{cor:positive_negative_affine_RTT}
There are algebra isomorphisms $\wtd{U}_{\pm}(\wtd{R}_{q}(z)) \to \wtd{U}_{\pm}'(\wtd{R}_{q}(z))$ given by
$X_{ij}^{\pm}[\mp m] \mapsto X_{ij}^{\pm}[\mp m]$ for all $1 \le i,j \le N$ and $m \in \bb{Z}_{\ge 0}$.
\end{cor}

\begin{prop}\label{prop:triangular_decomposition_affine_RTT'}
The multiplication map
  $\wtd{U}_{-}'(\wtd{R}_{q}(z)) \otimes \wtd{U}_{0}'(\wtd{R}_{q}(z)) \otimes \wtd{U}_{+}'(\wtd{R}_{q}(z)) \to \wtd{U}'(\wtd{R}_{q}(z))$
is an isomorphism of vector spaces. Moreover, if $J_{0}$ is the ideal of $\wtd{U}_{0}(\wtd{R}_{q}(z))$ generated by
$\{\ell_{ii}^{+}[0] - \ell_{ii}^{-}[0]^{-1}\}_{i = 1}^{n} \cup \{(\gamma')^{1/2} - \gamma^{-1/2}\}$,
then we have $\wtd{U}_{0}'(\wtd{R}_{q}(z)) = \wtd{U}_{0}(\wtd{R}_{q}(z))/J_{0}$.
\end{prop}

We are finally ready to prove the following result, which in conjunction with Proposition~\ref{prop:double_cartan_D} will allow us to upgrade the
isomorphism $U_{q}^{D}(\what{\fg}) \iso \wtd{U}'(\wtd{R}_{q}(z))$ to an isomorphism $U_{q,q^{-1}}^{D}(\what{\fg}) \iso \wtd{U}(\wtd{R}_{q}(z))$:

\begin{prop}\label{prop:double_cartan_RTT_affine}
Let $L_{\delta},L_{\varepsilon_{1}},\ldots ,L_{\varepsilon_{n}}$ be algebraically independent transcendental elements.
There is a unique injective algebra homomorphism
  $\what{\eta}_{R}\colon \wtd{U}(\wtd{R}_{q}(z)) \to
   \wtd{U}'(\wtd{R}_{q}(z)) \otimes \bb{K}[L_{\delta}^{\pm 1},L_{\varepsilon_{1}}^{\pm 1},\ldots ,L_{\varepsilon_{n}}^{\pm 1}]$
given by
\begin{equation}\label{eq:eta-R_assignment}
\begin{split}
  & \what{\eta}_{R}(\gamma^{1/2}) = \gamma^{1/2} \otimes L_{-\delta},\qquad \what{\eta}_{R}((\gamma')^{1/2}) = \gamma^{-1/2} \otimes L_{-\delta},  \\
  & \what{\eta}_{R}(\ell_{ij}^{\pm}[\mp m]) = \ell_{ij}^{\pm}[\mp m] \otimes L_{-\varepsilon_{i} - \varepsilon_{j} - 2m\delta}
    \qquad \text{for all}\qquad 1 \le i,j \le N,\ m \in \bb{Z},
\end{split}
\end{equation}
where we write $L_{\mu} = L_{\delta}^{c_{0}}L_{\varepsilon_{1}}^{c_{1}}\ldots L_{\varepsilon_{n}}^{c_{n}}$ for all
$\mu = c_{0}\delta + \sum_{i = 1}^{n}c_{i}\varepsilon_{i} \in \bb{Z}\delta \oplus P \subset \what{P}$.
\end{prop}

\begin{proof}
Due to Lemma~\ref{lem:affine_RTT_bigradings}(a), the proof that the above assignment~\eqref{eq:eta-R_assignment} gives rise to an algebra
homomorphism $\what{\eta}_{R}$ is the same as that of Proposition~\ref{prop:double_cartan_D}. The proof of injectivity is similar to that of
Propositions~\ref{prop:double_cartan_DJ},~\ref{prop:double_cartan_RTT_finite}, and~\ref{prop:double_cartan_D}: one uses
Corollaries~\ref{cor:triangular_decomposition_affine_RTT} and~\ref{cor:doubled_cartan_subalgebra_affine_RTT} to construct an appropriate
basis of $\wtd{U}(\wtd{R}_{q}(z))$, homogeneous with respect to the $\widehat{P}\times \widehat{P}$-grading, and then using
Corollary~\ref{cor:positive_negative_affine_RTT}, Proposition~\ref{prop:triangular_decomposition_affine_RTT'}, and the definition of
$\what{\eta}_{R}$, one observes that this basis for $\wtd{U}(\wtd{R}_{q}(z))$ is mapped by $\what{\eta}_{R}$ to linearly independent elements
in $\wtd{U}'(\wtd{R}_{q}(z)) \otimes \bb{K}[L_{\delta}^{\pm 1},L_{\varepsilon_{1}}^{\pm 1},\ldots ,L_{\varepsilon_{n}}^{\pm 1}]$.
\end{proof}

   %%%%%%%%%%%%%%%%%%%%%%%%%%%%%%%%%%%%%%%%%%%%%%%%%%%%%%%%%%%%%%%%%%%%%%%%%%

\subsection{Gaussian generators}\label{ssec:gaussian_generators}
\

Finally, using that the $\{\ell_{ii}^{\pm}[0]\}_{i=1}^N$ are invertible (which follows from $\ell_{ii}^{\pm}[0]\ell_{i'i'}^{\pm}[0]=1$) and
$\ell_{ij}^{+}[0] = \ell_{ji}^{-}[0] = 0$ for $i < j$, one can show that the conditions of~\cite[Theorem 4.9.7]{GGRW} hold for $L^{\pm}(z)$,
and therefore we have the \textbf{Gauss decomposition}
\[
  L^{\pm}(z) = F^{\pm}(z)H^{\pm}(z)E^{\pm}(z),
\]
where
\[
  F^{\pm}(z) =
  \begin{bmatrix}
    1 & 0 & \ldots & \ldots & 0 \\
    f_{21}^{\pm}(z) & 1 & 0 & \ldots & 0 \\
    \vdots & f_{32}^{\pm}(z) & 1 & \ddots & \vdots \\
    \vdots & \vdots & \ddots & \ddots & 0 \\
    f_{N1}^{\pm}(z) & f_{N2}^{\pm}(z) & \ldots & f_{N,N-1}^{\pm}(z) & 1
  \end{bmatrix},\qquad
  E^{\pm}(z) =
  \begin{bmatrix}
    1 & e_{12}^{\pm}(z) & \ldots & \ldots & e_{1N}^{\pm}(z) \\
    0 & 1 & e_{23}^{\pm}(z) & \ldots & e_{2N}^{\pm}(z) \\
    \vdots & 0 & \ddots & \ddots & \vdots \\
    \vdots & \vdots & \ddots & 1 & e_{N-1,N}^{\pm}(z) \\
    0 & 0 & \ldots & 0 & 1
  \end{bmatrix},
\]
and $H^{\pm}(z) = \diag(h_{1}^{\pm}(z),\ldots ,h_{N}^{\pm}(z))$. We will refer to the entries of these matrices as
\textbf{Gaussian generators}. We also note that, by construction, the series $h_{i}^{\pm}(z)$ are invertible.
For the entries of these matrices we write
\[
  h_{i}^{\pm}(z) = \sum_{m \geq 0} h_{i}^{\pm}[\mp m]z^{\pm m},\quad
  e_{ij}^{\pm}(z) = \sum_{m \geq 0} e_{ij}^{\pm}[\mp m]z^{\pm m},\quad
  f_{ji}^{\pm}(z) = \sum_{m \geq 0} f_{ji}^{\pm}[\mp m]z^{\pm m} \quad \forall\, 1 \le i < j \le N.
\]

Then we have the following simple lemma:

\begin{lemma}\label{lem:gaussian_generator_degrees}
(a) Relative to the $\what{P}$-bigrading of Lemma~\ref{lem:affine_RTT_bigradings}(a), we have
\[
  \deg(h_{i}^{\pm}[\mp m]) = (-\varepsilon_{i},\varepsilon_{i}) + md_{\pm},\quad
  \deg(e_{ij}^{\pm}[\mp m]) = (0,\varepsilon_{j} - \varepsilon_{i}) + md_{\pm},\quad
  \deg(f_{ji}^{\pm}[\mp m]) = (\varepsilon_{i} - \varepsilon_{j},0) + md_{\pm},
\]
for all $1 \le i < j \le N$ and $m \in \bb{Z}_{\ge 0}$, where
$d_{+} = (-\frac{1}{2}\delta,\frac{3}{2}\delta)$ and $d_{-} = (-\frac{3}{2}\delta,\frac{1}{2}\delta)$.

\medskip
\noindent
(b) Relative to the $P$-bigrading of Lemma~\ref{lem:affine_RTT_bigradings}(b), we have
\[
  \deg(h_{i}^{\pm}[\mp m]) = (-\varepsilon_{i},\varepsilon_{i}),\qquad
  \deg(e_{ij}^{\pm}[\mp m]) = (0,\varepsilon_{j} - \varepsilon_{i}),\qquad
  \deg(f_{ji}^{\pm}[\mp m]) = (\varepsilon_{i}-\varepsilon_{j},0),
\]
for all $1 \le i < j \le N$ and $m \in \bb{Z}_{\ge 0}$.
\end{lemma}

\begin{proof}
As in Lemma~\ref{lem:affine_RTT_bigradings}, part (b) follows immediately from (a). To prove (a), we will show by induction on $k$ that
the claim holds for $h_{k}^{\pm}[\mp m]$, $e_{kj}^{\pm}[\mp m]$, $f_{jk}^{\pm}[\mp m]$ with any $j > k$ and $m\geq 0$. First, we note that
\begin{equation}\label{eq:gauss_decomposition_entries_1}
  \ell_{ij}^{\pm}(z) = \sum_{1\leq p\leq i}f_{ip}^{\pm}(z)h_{p}^{\pm}(z)e_{pj}^{\pm}(z)\qquad \text{for all}\qquad 1 \le i \le j \le N
\end{equation}
and
\begin{equation}\label{eq:gauss_decomposition_entries_2}
  \ell_{ij}^{\pm}(z) = \sum_{1\leq p\leq j}f_{ip}^{\pm}(z)h_{p}^{\pm}(z)e_{pj}^{\pm}(z)\qquad \text{for all}\qquad 1 \le j \le i \le N,
\end{equation}
where we set $f_{ii}^{\pm}(z) = e_{ii}^{\pm}(z) = 1$ for all $1 \le i \le N$. Taking $i = j = 1$ in~\eqref{eq:gauss_decomposition_entries_1},
we get $\ell_{11}^{\pm}(z) = h_{1}^{\pm}(z)$, and hence $\deg(h_{1}^{\pm}[\mp m]) = (-\varepsilon_{1},\varepsilon_{1}) + md_{\pm}$
by Lemma~\ref{lem:affine_RTT_bigradings}(a). Next, if we take $i = 1 < j$
in \eqref{eq:gauss_decomposition_entries_1}, we find that $\ell_{1j}^{\pm}(z) = h_{1}^{\pm}(z)e_{1j}^{\pm}(z)$. Moreover, since
\[
  h_{1}^{\pm}(z) = h_{1}^{\pm}[0] \left (1 + \sum_{m > 0}h_{1}^{\pm}[0]^{-1}h_{1}^{\pm}[\mp m]z^{\pm m}\right ),
\]
and the $z^{\pm m}$-coefficient of the series $1 + \sum_{m > 0}h_{1}^{\pm}[0]^{-1}h_{1}^{\pm}[\mp m]z^{\pm m}$ has degree $md_{\pm}$ for
all $m$, it follows that the $z^{\pm m}$-coefficient of $h_{1}^{\pm}(z)^{-1}$ has degree $(\varepsilon_{1},-\varepsilon_{1}) + md_{\pm}$.
As $e_{1j}^{\pm}(z) = h_{1}^{\pm}(z)^{-1}\ell_{1j}^{\pm}(z)$, we conclude that
  $\deg(e_{1j}^{\pm}[\mp m]) = (\varepsilon_{1},-\varepsilon_{1}) + (-\varepsilon_{1},\varepsilon_{j}) + md_{\pm}
    = (0,\varepsilon_{j} - \varepsilon_{1}) + md_{\pm}$
for all $1 < j \le N$ and $m\geq 0$. Similarly, taking $j = 1 < i$ in~\eqref{eq:gauss_decomposition_entries_2}, we find that
$\ell_{i1}^{\pm}(z) = f_{i1}^{\pm}(z)h_{1}^{\pm}(z)$, so as before we obtain
  $\deg(f_{i1}^{\pm}[\mp m]) = (-\varepsilon_{i},\varepsilon_{1}) + (\varepsilon_{1},-\varepsilon_{1}) + md_{\pm}
   = (\varepsilon_{1} - \varepsilon_{i},0) + md_{\pm}$,
as desired. This completes the proof of the base case.

For the induction step, suppose that $k > 1$, and that for all $t < k$, the degrees of $h_{t}^{\pm}[\mp m]$, $e_{tj}^{\pm}[\mp m]$, and
$f_{jt}[\mp m]$ are as specified whenever $j > t$ and $m \in \bb{Z}_{\ge 0}$. Then taking $i = j = k$ in~\eqref{eq:gauss_decomposition_entries_1},
we get $h_{k}^{\pm}(z) = \ell_{kk}^{\pm}(z) - \sum_{p = 1}^{k - 1}f_{kp}^{\pm}(z)h_{p}^{\pm}(z)e_{pk}^{\pm}(z)$. By the induction hypothesis,
\begin{equation*}
\begin{split}
  \deg(f_{kp}^{\pm}[\mp m_{1}]h_{p}^{\pm}[\mp m_{2}]e_{pk}^{\pm}[\mp m_{3}])
  &= (\varepsilon_{p} - \varepsilon_{k},0) + (-\varepsilon_{p},\varepsilon_{p}) + (0,\varepsilon_{k} - \varepsilon_{p}) + (m_{1} + m_{2} + m_{3})d_{\pm} \\
  &= (-\varepsilon_{k},\varepsilon_{k}) + (m_{1} + m_{2} + m_{3})d_{\pm},
\end{split}
\end{equation*}
so it follows that $\deg(h_{k}^{\pm}[\mp m]) = (-\varepsilon_{k},\varepsilon_{k}) +md_{\pm}$. As above, we then conclude that the
$z^{\pm m}$-coefficient $\wtd{h}_{k}^{\pm}[\mp m]$ of $h_{k}^{\pm}(z)^{-1}$ has degree $(\varepsilon_{k},-\varepsilon_{k}) + md_{\pm}$.
Next, taking $i = k < j$ in~\eqref{eq:gauss_decomposition_entries_1}, we get
\begin{equation}\label{eq:ekj_formula}
  e_{kj}^{\pm}(z) = h_{k}^{\pm}(z)^{-1}\ell_{kj}^{\pm}(z) - h_{k}^{\pm}(z)^{-1}\sum_{1\leq p\leq k - 1}f_{kp}^{\pm}(z)h_{p}^{\pm}(z)e_{pj}^{\pm}(z).
\end{equation}
Then by the induction hypothesis, we have
\begin{equation*}
\begin{split}
  \deg(f_{kp}^{\pm}[\mp m_{1}]h_{p}^{\pm}[\mp m_{2}]e_{pj}^{\pm}[\mp m_{3}])
  &= (\varepsilon_{p} - \varepsilon_{k},0) + (-\varepsilon_{p},\varepsilon_{p}) + (0,\varepsilon_{j} - \varepsilon_{p}) + (m_{1} + m_{2} + m_{3})d_{\pm} \\
  &= (-\varepsilon_{k},\varepsilon_{j}) + (m_{1} + m_{2} + m_{3})d_{\pm},
\end{split}
\end{equation*}
so combining this with the above results and Lemma~\ref{lem:affine_RTT_bigradings}(a), we find
$\deg(e_{kj}^{\pm}[\mp m]) = (0,\varepsilon_{j} - \varepsilon_{k}) + md_{\pm}$ for all $m \in \bb{Z}_{\ge 0}$ and $j > k$.
Finally, by considering~\eqref{eq:gauss_decomposition_entries_2} with $j = k < i$, and proceeding similarly, one derives
$\deg(f_{ik}^{\pm}[\mp m]) = (\varepsilon_{k} - \varepsilon_{i},0) + md_{\pm}$ for all $m \in \bb{Z}_{\ge 0}$ and $i > k$.
This completes the proof.
\end{proof}

We also need a result, which describes the images of the Gaussian generators under $\what{\Psi}'$ of
Proposition~\ref{prop:2_vs_1_parameter_affine_RTT}. To avoid any confusion, we denote the Gaussian generators of $L^{\pm}(z)$
in $\wtd{U}(\wtd{R}_{r,s}(z))$ by $\wtd{h}_{i}^{\pm}(z)$, $\wtd{e}_{ij}^{\pm}(z)$, $\wtd{f}_{ji}^{\pm}(z)$.

\begin{lemma}\label{lem:gaussian_generators_2_param}
For all $1 \le i < j \le N$, we have:
\[
  (\what{\Psi}')^{-1}\colon \quad
  h_{i}^{\pm}(z) \mapsto \wtd{h}_{i}^{\pm}(z),\qquad
  e_{ij}^{\pm}(z) \mapsto \psi_{i}\psi_{j}^{-1}\zeta(\varepsilon_{j},\varepsilon_{i})\wtd{e}_{ij}^{\pm}(z),\qquad
  f_{ji}^{\pm}(z) \mapsto \psi_{i}^{-1}\psi_{j}\zeta(\varepsilon_{j},\varepsilon_{i})\wtd{f}_{ji}^{\pm}(z),
\]
where $\psi_{i}$ are given by~\eqref{eq:psi_values_B} in type $B_{n}$,~\eqref{eq:psi_values_C} in type $C_{n}$,
and~\eqref{eq:psi_values_D} in type $D_{n}$.
\end{lemma}

\begin{proof}
Following a similar method to the proof of Lemma~\ref{lem:gaussian_generator_degrees}, we shall prove by induction on $k$ that the claim holds for
$h_{k}^{\pm}(z)$, $e_{kj}^{\pm}(z)$, $f_{jk}^{\pm}(z)$ with any $j > k$. For the base case $k = 1$, we have $\wtd{h}_{1}^{\pm}(z) = \ell_{11}^{\pm}(z)$,
so the result clearly holds for $\wtd{h}_{1}^{\pm}(z)$. Next, taking $j > i = 1$ in~\eqref{eq:gauss_decomposition_entries_1} yields
\begin{align*}
  (\what{\Psi}')^{-1}(e_{1j}^{\pm}(z))
  &= (\what{\Psi}')^{-1}(\ell_{11}^{\pm}(z)^{-1}\ell_{1j}^{\pm}(z))
  = (\what{\Psi}')^{-1}(\zeta(\varepsilon_{j},\varepsilon_{1})\ell_{11}^{\pm}(z)^{-1} \circ \ell_{1j}^{\pm}(z)) \\
  &= \psi_{1}\psi_{j}^{-1}\zeta(\varepsilon_{j},\varepsilon_{1})\ell_{11}^{\pm}(z)^{-1}\ell_{1j}^{\pm}(z)
  = \psi_{1}\psi_{j}^{-1}\zeta(\varepsilon_{j},\varepsilon_{1})\wtd{e}_{1j}^{\pm}(z),
\end{align*}
as claimed. Similarly, taking $i > j = 1$ in~\eqref{eq:gauss_decomposition_entries_2}, we get
\[
  (\what{\Psi}')^{-1}(f_{i1}^{\pm}(z))
  = (\what{\Psi}')^{-1}(\ell_{i1}^{\pm}(z)\ell_{11}^{\pm}(z)^{-1})
  = (\what{\Psi}')^{-1}(\zeta(\varepsilon_{i},\varepsilon_{1})\ell_{i1}^{\pm}(z) \circ \ell_{11}^{\pm}(z)^{-1})
  = \psi_{i}\psi_{1}^{-1}\zeta(\varepsilon_{i},\varepsilon_{1})\wtd{f}_{i1}^{\pm}(z),
\]
which completes the proof of the base case. For the induction step, suppose that $k > 1$ and the claim holds for all $t < k$. Then taking $i = j = k$
in~\eqref{eq:gauss_decomposition_entries_1} yields $h_{k}^{\pm}(z) = \ell_{kk}^{\pm}(z) - \sum_{p = 1}^{k - 1}f_{kp}^{\pm}(z)h_{p}^{\pm}(z)e_{pk}^{\pm}(z)$,
so by the induction hypothesis and Lemma~\ref{lem:gaussian_generator_degrees}(b), we have
\[
  (\what{\Psi}')^{-1}(f_{kp}^{\pm}(z)h_{p}^{\pm}(z)e_{pk}^{\pm}(z))
  = (\what{\Psi}')^{-1}(\zeta(\varepsilon_{p},\varepsilon_{k})^{2}f_{kp}^{\pm}(z) \circ h_{p}^{\pm}(z) \circ e_{pk}^{\pm}(z))
  = \wtd{f}_{kp}^{\pm}(z)\wtd{h}_{p}^{\pm}(z)\wtd{e}_{pk}^{\pm}(z)
\]
for all $1 \le p \le k - 1$. Thus, $(\what{\Psi}')^{-1}(h_{k}^{\pm}(z)) = \wtd{h}_{k}^{\pm}(z)$, as stated.
Next, using~\eqref{eq:ekj_formula}, the induction hypothesis, and Lemma~\ref{lem:gaussian_generator_degrees}(b), we get
\begin{align*}
  &(\what{\Psi}')^{-1}(e_{kj}^{\pm}(z)) \\
  &= (\what{\Psi}')^{-1}\left (\zeta(\varepsilon_{j},\varepsilon_{k})h^{\pm}_{k}(z)^{-1} \circ \ell_{kj}^{\pm}(z) -
     \zeta(\varepsilon_{j},\varepsilon_{k}) h_{k}^{\pm}(z)^{-1}\circ
     \sum_{1\leq p\leq k - 1}\zeta(\varepsilon_{p},\varepsilon_{k} + \varepsilon_{j})f_{kp}^{\pm}(z) \circ h_{p}^{\pm}(z) \circ e_{pj}^{\pm}(z)\right ) \\
  &= \zeta(\varepsilon_{j},\varepsilon_{k})\left (\psi_{k}\psi_{j}^{-1}\wtd{h}^{\pm}_{k}(z)^{-1}\ell_{kj}^{\pm}(z) -
    \wtd{h}_{k}^{\pm}(z)^{-1}\sum_{1\leq p\leq k-1}\psi_{k}\psi_{j}^{-1}\wtd{f}_{kp}^{\pm}(z)\wtd{h}_{p}^{\pm}(z)\wtd{e}_{pj}^{\pm}(z)\right )
    = \zeta(\varepsilon_{j},\varepsilon_{k})\psi_{k}\psi_{j}^{-1}\wtd{e}_{kj}^{\pm}(z).
\end{align*}
Finally, taking $j = k < i$ in~\eqref{eq:gauss_decomposition_entries_2} and proceeding as above, we get
$(\what{\Psi}')^{-1}(f_{ik}^{\pm}(z)) = \psi_{k}^{-1}\psi_{i}\zeta(\varepsilon_{i},\varepsilon_{k})\wtd{f}_{ik}^{\pm}(z)$ for all $i > k$.
This completes the proof.
\end{proof}

   %%%%%%%%%%%%%%%%%%%%%%%%%%%%%%%%%%%%%%%%%%%%%%%%%%%%%%%%%%%%%%%%%%%%%%%%%%

\subsection{Isomorphism between RTT and new Drinfeld realizations in classical types}\label{ssec:affine_RTT_to_D}
\

In this Subsection, we construct an isomorphism between the two-parameter new Drinfeld and RTT-type two-parameter quantum affine groups in types
$B_{n}$, $C_{n}$, $D_{n}$ by combining the general results above with the corresponding one-parameter isomorphisms of~\cite[Main Theorem]{JLM1}
and~\cite[Main Theorem]{JLM2}.

\begin{remark}\label{rem:rtt=Dr-hatP-compatibility}
All isomorphisms in this Subsection actually intertwine the $\what{P}$-bigradings from Lemmas~\ref{lem:Q_hat_bigrading}
and~\ref{lem:affine_RTT_bigradings}(a), though for our arguments it only suffices to know they intertwine the corresponding $P$-bigradings.
\end{remark}

To state the isomorphisms, it is convenient to package the generators $\{x_{i,m}^{\pm}\}$ of $U_{r,s}^{D}(\fg)$ into the currents
\[
  x_{i}^{\pm}(z) = \sum_{m \in \bb{Z}} x_{i,m}^{\pm}z^{-m} \qquad \text{for all}\qquad 1 \le i \le n.
\]
Similarly, for $U_{q}^{D}(\fg)$, we define
\[
  X_{i}^{\pm}(z) = \sum_{m \in \bb{Z}} X_{i,m}^{\pm}z^{-m},
\]
as well as
\begin{equation}\label{eq:Phi_current}
  \Omega_{i}^{\pm}(z) = \sum_{r \geq 0} \Omega_{i,\pm r}^{\pm}z^{\mp r}
  = K_{i}^{\pm 1} \exp \left( \pm(q_{i} - q_{i}^{-1})\sum_{r>0} H_{i,\pm r}z^{\mp r} \right)
\end{equation}
for all $1 \le i \le n$. We also recall the following standard result:

\begin{lemma}\label{lem:affine_cartan_involution}
There is a unique $\BC(q)$-algebra automorphism $\what{\omega}$ of $U_{q}^{D}(\what{\fg})$ such that
\[
  \gamma^{1/2} \mapsto \gamma^{-1/2},\qquad K_{i} \mapsto K_{i}^{-1},\qquad
  X_{i}^{\pm}(z) \mapsto X_{i}^{\mp}(z^{-1}),\qquad \Omega_{i}^{\pm}(z) \mapsto \Omega_{i}^{\mp}(z^{-1}) \qquad \forall\, 1\leq i\leq n.
\]
\end{lemma}

   %%%%%%%%%%%%%%%%%%%%%%%%%%%%%%%%%%%%%%%%%%%%%%%%%%%%%%%%%%%%%%%%%%%%%%%%%%

\noindent
$\bullet$ \textbf{Type $B_n$.}

Let $R_{q}(z) = \hat{R}_{q}(z) \circ \tau$, where $\hat{R}_{q}(z)$ is the one-parameter $B_{n}$-type affine $R$-matrix, that is $\hat{R}_{q}(z)$ is
the $r=q=s^{-1}$ specialization of \cite[(6.11)]{MT1}, and let $\wtd{R}_{q}(z)$ be its normalization as in~\eqref{eq:f_prefactor}. We first recall
the one-parameter algebra isomorphism of~\cite[Main Theorem]{JLM1}. To state it, we define
\begin{equation}\label{eq:X_series}
  \frak{X}_{i}^{+}(z) = e_{i,i+1}^{+}(z\gamma^{1/2}) - e_{i,i+1}^{-}(z\gamma^{-1/2}),\qquad
  \frak{X}_{i}^{-}(z) = f_{i+1,i}^{+}(z\gamma^{-1/2}) - f_{i+1,i}^{-}(z\gamma^{1/2})\qquad \forall \, 1 \le i \le n,
\end{equation}
where $e_{i,i+1}^{\pm}(z)$, $f_{i+1,i}^{\pm}(z)$, and $h_{i}^{\pm}(z)$ are the Gaussian generators corresponding to $L^{\pm}(z)$, as defined above.

Then we have the following result (see~\cite[Main Theorem]{JLM1}):

\begin{theorem}\label{thm:D=affine_RTT_Btype_1param}
There is a unique algebra isomorphism $\what{\theta}_{q}'\colon U_{q}^{D}(\what{\frak{so}}_{2n+1}) \iso \wtd{U}'(\wtd{R}_{q}(z))$ such that
\begin{align*}
  &\what{\theta}_{q}'(\gamma^{1/2}) = \gamma^{-1/2}, & & ~\\
  &\what{\theta}_{q}'(\Omega_{i}^{\pm}(z)) = h_{i+1}^{\pm}(z^{-1}q^{2i})h_{i}^{\pm}(z^{-1}q^{2i})^{-1} & &  \textit{for} \quad 1 \le i \le n, \\
  &\what{\theta}_{q}'(X_{i}^{\pm}(z)) = \frac{1}{q^{2} - q^{-2}}\frak{X}_{i}^{\mp}(z^{-1}q^{2i}) & & \textit{for} \quad 1 \le i < n, \\
  &\what{\theta}_{q}'(X_{n}^{\pm}(z)) = \frac{q^{\mp 1/2}}{(q - q^{-1})[2]_{q}^{1/2}}\frak{X}_{n}^{\mp}(z^{-1}q^{2n}).
\end{align*}
\end{theorem}

\begin{proof}
After adjusting for the differences between our conventions and those of~\cite{JLM1} (see Remark~\ref{rmk:JLM_comparison_B}), this result
follows by composing the isomorphism of~\cite[Main Theorem]{JLM1} with the map $\what{\omega}$ of Lemma~\ref{lem:affine_cartan_involution}.
We also note that $\what{\theta}_{q}'$ is compatible with $\what{P} \times \what{P}$-grading of both algebras (which is the reason for applying
the map $\what{\omega}$), see Remark~\ref{rem:rtt=Dr-hatP-compatibility}.
\end{proof}

We shall next upgrade this result to a Cartan-doubled version by using Proposition~\ref{prop:double_cartan_D} and
Proposition~\ref{prop:double_cartan_RTT_affine}. For $1 \le i \le n$, we define the generating series
$\wtd{\frak{X}}_{i}^{\pm}(w) \in \wtd{U}(\wtd{R}_{q}(z))[[w,w^{-1}]]$ via
\begin{equation}\label{eq:X_series_double_cartan}
\begin{split}
  & \wtd{\frak{X}}_{i}^{+}(w) = e_{i,i+1}^{+}(w(\gamma')^{-1/2}) - e_{i,i+1}^{-}(w(\gamma')^{1/2}(\gamma\gamma')^{1/2}),\\
  & \wtd{\frak{X}}_{i}^{-}(w) = f_{i+1,i}^{+}(w\gamma^{-1/2}(\gamma\gamma')^{-1/2}) - f_{i+1,i}^{-}(w\gamma^{1/2}).
\end{split}
\end{equation}
Then we have the following result:

\begin{cor}\label{cor:affine_1param_double_Btype}
There is a unique algebra isomorphism $\what{\theta}_{q}\colon U_{q,q^{-1}}^{D}(\what{\frak{so}}_{2n+1}) \iso \wtd{U}(\wtd{R}_{q}(z))$ such that
\begin{align*}
  &\what{\theta}_{q}(\gamma^{1/2}) = (\gamma')^{1/2}, & & \what{\theta}_{q}((\gamma')^{1/2}) = \gamma^{1/2}, \\
  &\what{\theta}_{q}(\omega_{i}(z)) = h_{i+1}^{+}(z^{-1}q^{2i})h_{i}^{+}(z^{-1}q^{2i})^{-1},
  & & \what{\theta}_{q}(\omega_{i}'(z)) = h_{i+1}^{-}(z^{-1}q^{2i})h_{i}^{-}(z^{-1}q^{2i})^{-1} & & 1 \le i \le n, \\
  &\what{\theta}_{q}(x_{i}^{\pm}(z)) = \frac{1}{q^{2} - q^{-2}}\wtd{\frak{X}}_{i}^{\mp}(z^{-1}q^{2i}),
  & & \what{\theta}_{q}(x_{n}^{\pm}(z)) = \frac{q^{\mp 1/2}}{(q - q^{-1})[2]_{q}^{1/2}}\wtd{\frak{X}}_{n}^{\mp}(z^{-1}q^{2n}) & & 1 \le i < n.
\end{align*}
\end{cor}

\begin{proof}
Similarly to~\eqref{eq:g-map}, we define a map
  $\what{g}\colon \bb{K}[L_{\alpha_{0}}^{\pm 1},L_{\alpha_{1}}^{\pm 1},\ldots ,L_{\alpha_{n}}^{\pm 1}] \to
   \bb{K}[L_{\delta}^{\pm 1},L_{\varepsilon_{1}}^{\pm 1},\ldots ,L_{\varepsilon_{n}}^{\pm 1}]$
by $\what{g}(L_{\alpha_{0}}) = L_{\delta- \theta}$ and $\what{g}(L_{\alpha_{i}}) = L_{\alpha_{i}} = L_{\varepsilon_{i} - \varepsilon_{i+1}}$
for $1 \le i \le n$. Then, evoking the algebra embeddings $\eta_{D}$ of Proposition~\ref{prop:double_cartan_D} and $\what{\eta}_{R}$ of
Proposition~\ref{prop:double_cartan_RTT_affine}, along with the algebra isomorphism $\what{\theta}_{q}'$ of Theorem~\ref{thm:D=affine_RTT_Btype_1param},
we can define $\what{\theta}_{q} = \what{\eta}_{R}^{-1}\circ (\what{\theta}_{q}' \otimes \what{g}) \circ \eta_{D}$, provided that
$(\what{\theta}_{q}' \otimes \what{g}) \circ \eta_{D}$ maps $U_{q,q^{-1}}^{D}(\what{\frak{so}}_{2n+1})$ into $\what{\eta}_{R}(\wtd{U}(\wtd{R}_{q}(z)))$.
It suffices to verify the latter on the generators $\gamma^{1/2}$, $(\gamma')^{1/2}$, and the coefficients of $x_{i}^{\pm}(z)$, $\omega_{i}(z)$,
$\omega_{i}'(z)$. The claim is obvious for $\gamma^{1/2},(\gamma')^{1/2}$. To deal with the remaining generators, we note that,
by Proposition~\ref{prop:double_cartan_D}, we have
\[
  \eta_{D}(x_{i}^{\pm}(z)) = X_{i}^{\pm}(z(1 \otimes L_{\mp \delta}))(1 \otimes L_{\alpha_{i}}),
\]
\[
  \eta_{D}(\omega_{i}(z)) = \Omega_{i}^{+}(z(1 \otimes L_{2\delta}))(1 \otimes L_{\alpha_{i}}^{2}),\qquad \eta_{D}(\omega_{i}'(z)) =
  \Omega_{i}^{-}(z(1 \otimes L_{-2\delta}))(1 \otimes L_{\alpha_{i}}^{2}),
\]
for $1 \le i \le n$. Likewise, by Proposition~\ref{prop:double_cartan_RTT_affine} and Lemma~\ref{lem:gaussian_generator_degrees}, we have
\begin{equation}\label{eq:eta_R_Btype-1}
  \what{\eta}_{R}(h_{i}^{\pm}(z)) = h_{i}^{\pm}(z(1 \otimes L_{\mp 2\delta}))(1 \otimes L_{-2\varepsilon_{i}}) ,\qquad
  \what{\eta}_{R}(h_{i}^{\pm}(z)^{-1}) = h_{i}^{\pm}(z(1 \otimes L_{\mp 2\delta}))^{-1}(1 \otimes L_{2\varepsilon_{i}})
\end{equation}
for all $1 \le i \le N$ and
\begin{equation}\label{eq:eta_R_Btype-2}
  \what{\eta}_{R}(e_{ij}^{\pm}(z)) = e_{ij}^{\pm}(z(1 \otimes L_{\mp 2\delta}))(1 \otimes L_{\varepsilon_{i} - \varepsilon_{j}}) ,\qquad
  \what{\eta}_{R}(f_{ji}^{\pm}(z)) = f_{ji}^{\pm}(z(1 \otimes L_{\mp 2\delta}))(1 \otimes L_{\varepsilon_{i} - \varepsilon_{j}})
\end{equation}
for all $1 \le i < j \le N$. Then we find that for $i < n$:
\begin{equation*}
\begin{split}
  (\what{\theta}_{q}' \otimes \what{g})\eta_{D}(x_{i}^{+}(z))
  &= (\what{\theta}_{q}' \otimes \what{g})\left (X_{i}^{+}(z(1 \otimes L_{-\delta}))(1 \otimes L_{\alpha_{i}})\right )   \\
  &= \frac{1}{q^{2} - q^{-2}}
      \left (f_{i+1,i}^{+}(z^{-1}(1 \otimes L_{\delta})\gamma^{-1/2}q^{2i}) - f_{i+1,i}^{-}(z^{-1}(1 \otimes L_{\delta})\gamma^{1/2}q^{2i})\right )
      (1 \otimes L_{\varepsilon_{i} - \varepsilon_{i+1}}).
\end{split}
\end{equation*}
According to~\eqref{eq:eta_R_Btype-2}, we have
\begin{equation*}
\begin{split}
  &f_{i+1,i}^{+}(z^{-1}(1 \otimes L_{\delta})\gamma^{-1/2}q^{2i})(1 \otimes L_{\varepsilon_{i} - \varepsilon_{i+1}}) \\
  &= f_{i+1,i}^{+}(z^{-1}(1 \otimes L_{-2\delta})(\gamma^{1/2} \otimes L_{-\delta})^{-1}(\gamma^{1/2} \otimes
      L_{-\delta})^{-1}(\gamma^{-1/2} \otimes L_{-\delta})^{-1}q^{2i})(1 \otimes L_{\varepsilon_{i} - \varepsilon_{i+1}}) \\
  &= \what{\eta}_{R}(f_{i+1,i}^{+}(z^{-1}\gamma^{-1/2}(\gamma\gamma')^{-1/2}q^{2i})),
\end{split}
\end{equation*}
and
\begin{equation*}
\begin{split}
  f_{i+1,i}^{-}(z^{-1}(1 \otimes L_{\delta})\gamma^{1/2}q^{2i})(1 \otimes L_{\varepsilon_{i} - \varepsilon_{i+1}})
  &= f_{i+1,i}^{-}(z^{-1}(1 \otimes L_{2\delta})(\gamma^{1/2} \otimes L_{-\delta})q^{2i})(1 \otimes L_{\varepsilon_{i} - \varepsilon_{i+1}}) \\
  &= \what{\eta}_{R}(f_{i+1,i}^{-}(z^{-1}\gamma^{1/2}q^{2i})),
\end{split}
\end{equation*}
so that we ultimately get
\[
  (\what{\theta}_{q}' \otimes \what{g})\eta_{D}(x_{i}^{+}(z)) = \what{\eta}_{R} \left(\frac{1}{q^{2} - q^{-2}}\wtd{\frak{X}}_{i}^{-}(z^{-1}q^{2i})\right).
\]
Thus $(\what{\theta}_{q}' \otimes \what{g})\eta_{D}(x_{i,m}^{+}) \in \what{\eta}_{R}(\wtd{U}(\wtd{R}_{q}(z)))$ for all $m$,
and $\what{\theta}_{q}(x_{i}^{+}(z)) = \frac{1}{q^{2} - q^{-2}}\wtd{\frak{X}}_{i}^{-}(z^{-1}q^{2i})$ for $i < n$, as claimed.
The verifications for $x_{n}^{+}(z)$ and $x_{i}^{-}(z)$ with $1 \le i \le n$ are similar.

For $\omega_{i}(z)$, we have
\begin{equation*}
\begin{split}
  (\what{\theta}_{q}' \otimes \what{g})(\eta_{D}(\omega_{i}(z)))
  &= (\what{\theta}_{q}' \otimes \what{g})(\Omega_{i}^{+}(z(1 \otimes L_{2\delta}))(1 \otimes L_{\alpha_{i}}^{2})) \\
  &= h_{i+1}^{+}(z^{-1}(1 \otimes L_{-2\delta})q^{2i})h_{i}^{+}(z^{-1}(1 \otimes L_{-2\delta})q^{2i})^{-1}
     (1 \otimes L_{2\varepsilon_{i} - 2\varepsilon_{i+1}}) \\
  &= \what{\eta}_{R}(h_{i+1}^{+}(z^{-1}q^{2i})h_{i}^{+}(z^{-1}q^{2i})^{-1}),
\end{split}
\end{equation*}
which implies that $\what{\theta}_{q}(\omega_{i}(z)) = h_{i+1}^{+}(z^{-1}q^{2i})h_{i}^{+}(z^{-1}q^{2i})^{-1}$. Finally, we also have
\begin{equation*}
\begin{split}
  (\what{\theta}_{q}' \otimes \what{g})(\eta_{D}(\omega_{i}'(z)))
  &= (\what{\theta}_{q}' \otimes g)(\Omega_{i}^{-}(z(1 \otimes L_{-2\delta}))(1 \otimes L_{\alpha_{i}}^{2})) \\
  &= h_{i+1}^{-}(z^{-1}(1 \otimes L_{2\delta})q^{2i})h_{i}^{-}(z^{-1}(1 \otimes L_{2\delta})q^{2i})^{-1}
     (1 \otimes L_{2\varepsilon_{i} - 2\varepsilon_{i+1}}) \\
  &= \what{\eta}_{R}(h_{i+1}^{-}(z^{-1}q^{2i})h_{i}^{-}(z^{-1}q^{2i})^{-1}),
\end{split}
\end{equation*}
so that $\what{\theta}_{q}(\omega_{i}'(z)) = h_{i+1}^{-}(z^{-1}q^{2i})h_{i}^{-}(z^{-1}q^{2i})^{-1}$, as stated.

Thus, we have a well-defined algebra homomorphism $\what{\theta}_{q}\colon U_{q,q^{-1}}^{D}(\what{\frak{so}}_{2n+1}) \to \wtd{U}(\wtd{R}_{q}(z))$,
which is injective by construction.  It can be easily deduced by looking at the images of the generators of $U_{q,q^{-1}}^{D}(\what{\frak{so}}_{2n+1})$
that all elements $\{h^\pm_\imath[\mp m],e^\pm_{j,j+1}[\mp m], f^\pm_{j+1,j}[\mp m]\}$ with $1\leq \imath\leq n+1, 1\leq j\leq n, m\geq 0\}$ lie in
the image of $\what{\theta}_{q}$.  Then using~\cite[Lemma 4.21, Proposition 4.22]{JLM1}, we see that also such elements with $n+1<\imath\leq 2n+1$
and $n<j\leq 2n$ lie in the image of $\what{\theta}_{q}$. Finally, similarly to~\cite[Proposition A.4]{HT2} one shows that all $e^\pm_{ij}[\mp m]$
with $1\leq i<j\leq 2n+1$ lie in the image of $\what{\theta}_{q}$, and the same reasoning also applies to show that all $f^\pm_{ji}[\mp m]$ with
$1\leq i<j\leq 2n+1$ lie in the image of $\what{\theta}_{q}$. Therefore, $\what{\theta}_{q}$ is indeed an algebra isomorphism.
\end{proof}

Combining this with Lemma~\ref{lem:gaussian_generators_2_param}, we can immediately derive the two-parameter analogue of
Theorem~\ref{thm:D=affine_RTT_Btype_1param}. As above, we shall denote the Gaussian generators of $\wtd{U}(\wtd{R}_{r,s}(z))$ by
$\wtd{e}_{ij}^{\pm}(z)$, $\wtd{f}_{ji}^{\pm}(z)$, $\wtd{h}_{i}^{\pm}(z)$, and we use the same symbol $\wtd{\frak{X}}_{i}^{\pm}(w)$
to denote~\eqref{eq:X_series_double_cartan} while using $\wtd{e}_{i,i+1}^{\pm}(z)$ and $\wtd{f}_{i+1,i}^{\pm}(z)$ in place of
$e_{i,i+1}^{\pm}(z)$ and $f_{i+1,i}^{\pm}(z)$.

\begin{theorem}\label{thm:D=affine_RTT_Btype_2param}
There is a unique algebra isomorphism $\what{\theta}_{r,s}\colon U_{r,s}^{D}(\what{\frak{so}}_{2n+1}) \iso \wtd{U}(\wtd{R}_{r,s}(z))$ such that
\begin{align*}
  &\what{\theta}_{r,s}(\gamma^{1/2}) = (\gamma')^{1/2}, & & \what{\theta}_{r,s}((\gamma')^{1/2}) = \gamma^{1/2}, \\
  &\what{\theta}_{r,s}(\omega_{i}(z)) = \wtd{h}_{i+1}^{+}(z^{-1}r^{i}s^{-i})\wtd{h}_{i}^{+}(z^{-1}r^{i}s^{-i})^{-1},
  & & \what{\theta}_{r,s}(\omega_{i}'(z)) = \wtd{h}_{i+1}^{-}(z^{-1}r^{i}s^{-i})\wtd{h}_{i}^{-}(z^{-1}r^{i}s^{-i})^{-1} & & 1 \le i \le n, \\
  &\what{\theta}_{r,s}(x_{i}^{\pm}(z)) = \frac{1}{r^{2} - s^{2}}\wtd{\frak{X}}_{i}^{\mp}(z^{-1}r^{i}s^{-i}) & & ~ & & 1 \le i < n, \\
  &\what{\theta}_{r,s}(x_{n}^{+}(z)) = \frac{r^{1/2}s}{(r - s)[2]_{r,s}^{1/2}}\wtd{\frak{X}}_{n}^{-}(z^{-1}r^{n}s^{-n}),
  & & \what{\theta}_{r,s}(x_{n}^{-}(z)) = \frac{r^{1/2}}{(r - s)[2]_{r,s}^{1/2}}\wtd{\frak{X}}_{n}^{+}(z^{-1}r^{n}s^{-n}).
\end{align*}
\end{theorem}

\begin{proof}
Note that the isomorphism $\what{\theta}_{q}$ of Corollary~\ref{cor:affine_1param_double_Btype} is compatible with $Q\times Q$-grading
of~\eqref{eq:Drinfeld-grading} and Lemma~\ref{lem:affine_RTT_bigradings}(b), so it induces an algebra isomorphism
$\what{\theta}_{q}\colon U_{q,\zeta}^{D}(\what{\frak{so}}_{2n+1}) \iso \wtd{U}(\wtd{R}_{q}(z))_{\zeta}$. Then we define
$\what{\theta}_{r,s} = (\what{\Psi}')^{-1}\circ \what{\theta}_{q} \circ \varphi_D$, where $\what{\Psi}'$ and $\varphi_D$ are the algebra isomorphisms
of Proposition~\ref{prop:2_vs_1_parameter_affine_RTT} and Theorem~\ref{thm:twisted_algebra_loop}, respectively. Thus, $\what{\theta}_{r,s}$ is also
an algebra isomorphism. The $\what{\theta}_{r,s}$-images of the generators are computed using Lemma~\ref{lem:gaussian_generators_2_param} and the
identities $\psi_{i+1}\psi_{i}^{-1} = (rs)^{-1/2} = \zeta(\varepsilon_{i+1},\varepsilon_{i})$ for $i < n$ and
$\psi_{n+1}\psi_{n}^{-1} = 1= \zeta(\varepsilon_{n+1},\varepsilon_{n})$ (cf.\ the proof of Theorem~\ref{thm:DJ=RTT_Btype_2_param}).
\end{proof}

   %%%%%%%%%%%%%%%%%%%%%%%%%%%%%%%%%%%%%%%%%%%%%%%%%%%%%%%%%%%%%%%%%%%%%%%%%%

\noindent
$\bullet$ \textbf{Type $C_n$.}

Let $R_{q}(z) = \hat{R}_{q}(z) \circ \tau$, where $\hat{R}_{q}(z)$ is the one-parameter $C_{n}$-type affine $R$-matrix, that is the $r=q=s^{-1}$
specialization of~\cite[(6.12)]{MT1}, and let $\wtd{R}_{q}(z)$ be its normalization as in~\eqref{eq:f_prefactor}. Following the procedure carried
out for $B_{n}$-type above, we shall now derive a two-parameter analogue of the one-parameter isomorphism
$U_{q}^{D}(\what{\frak{sp}}_{2n}) \iso \wtd{U}'(\wtd{R}_{q}(z))$ from~\cite[Main Theorem]{JLM2}.
To do so, we first need to enlarge $U_{r,s}^{D}(\what{\frak{sp}}_{2n})$ by adjoining generators $\omega_{n}^{1/2}$ and $(\omega_{n}')^{1/2}$,
with the obvious additional relations imposed. We also consider a similarly defined extension of $U_{q}^{D}(\what{\frak{sp}}_{2n})$. Now,
we consider the Gauss decomposition of $L^{\pm}(z)$ in $\wtd{U}'(\wtd{R}_{q}(z))$, and define $\frak{X}_{i}^{\pm}(z)$ as in~\eqref{eq:X_series}.
Then we have the following result (see~\cite[Main Theorem]{JLM2}):

\begin{theorem}\label{thm:D=RTT_Ctype_1param}
There is a unique algebra isomorphism $\what{\theta}_{q}'\colon U_{q}^{D}(\what{\frak{sp}}_{2n}) \iso \wtd{U}'(\wtd{R}_{q}(z))$ such that
\begin{align*}
  &\what{\theta}_{q}'(\gamma^{1/2}) = \gamma^{-1/2}, & & \what{\theta}_{q}'(K_{n}^{1/2}) = h_{n}^{+}[0]^{-1},  \\
  &\what{\theta}_{q}'(\Omega_{i}^{\pm}(z)) = h_{i+1}^{\pm}(z^{-1}q^{i})h_{i}^{\pm}(z^{-1}q^{i})^{-1},
    & & \what{\theta}_{q}'(X_{i}^{\pm}(z)) = \frac{1}{q - q^{-1}}\frak{X}_{i}^{\mp}(z^{-1}q^{i}) & &  \textit{for} \quad 1 \le i < n, \\
  &\what{\theta}_{q}'(\Omega_{n}^{\pm}(z)) = h_{n+1}^{\pm}(z^{-1}q^{n+1})h_{n}^{\pm}(z^{-1}q^{n+1})^{-1},
   & &\what{\theta}_{q}'(X_{n}^{\pm}(z)) = \frac{1}{q^{2} - q^{-2}}\frak{X}_{n}^{\mp}(z^{-1}q^{n+1}).
\end{align*}
\end{theorem}

\begin{proof}
This follows by composing the isomorphism of~\cite[Main Theorem]{JLM2} with $\what{\omega}$ of Lemma~\ref{lem:affine_cartan_involution}.
\end{proof}

Then, as for type $B_{n}$, we can use Propositions~\ref{prop:double_cartan_D} and~\ref{prop:double_cartan_RTT_affine} to derive a Cartan-doubled
version of the above result, where $\wtd{\frak{X}}_{i}^{\pm}(z)$ are defined exactly as in~\eqref{eq:X_series_double_cartan}:

\begin{cor}\label{cor:D=RTT_Ctype_1param_double_cartan}
There is a unique algebra isomorphism $\what{\theta}_{q}\colon U_{q,q^{-1}}^{D}(\what{\frak{sp}}_{2n}) \iso \wtd{U}(\wtd{R}_{q}(z))$ such that
\begin{align*}
  &\what{\theta}_{q}(\gamma^{1/2}) = (\gamma')^{1/2}, & & \what{\theta}_{q}((\gamma')^{1/2}) = \gamma^{1/2}, \\
  &\what{\theta}_{q}(\omega_{n}^{1/2}) = h_{n}^{+}[0]^{-1}, & & \what{\theta}_{q}((\omega_{n}')^{1/2}) = h_{n}^{-}[0]^{-1}, \\
  &\what{\theta}_{q}(\omega_{i}(z)) = h_{i+1}^{+}(z^{-1}q^{i})h_{i}^{+}(z^{-1}q^{i})^{-1},
   & & \what{\theta}_{q}(\omega_{i}'(z)) = h_{i+1}^{-}(z^{-1}q^{i})h_{i}^{-}(z^{-1}q^{i})^{-1} & & \textit{for} \quad 1 \le i < n, \\
  &\what{\theta}_{q}(\omega_{n}(z)) = h_{n+1}^{+}(z^{-1}q^{n+1})h_{n}^{+}(z^{-1}q^{n+1})^{-1},
   & & \what{\theta}_{q}(\omega_{n}'(z))= h_{n+1}^{-}(z^{-1}q^{n+1})h_{n}^{-}(z^{-1}q^{n+1})^{-1}, \\
  &\what{\theta}_{q}(x_{i}^{\pm}(z)) = \frac{1}{q - q^{-1}}\wtd{\frak{X}}_{i}^{\mp}(z^{-1}q^{i}),
   & & \what{\theta}_{q}(x_{n}^{\pm}(z)) = \frac{1}{q^{2} - q^{-2}}\wtd{\frak{X}}_{n}^{\mp}(z^{-1}q^{n+1}) & & \textit{for} \quad 1 \le i < n.
\end{align*}
\end{cor}

We can now immediately obtain a two-parameter version of Theorem~\ref{thm:D=RTT_Ctype_1param}. As for type $B_{n}$, we shall denote
the Gaussian generators of $\wtd{U}(\wtd{R}_{r,s}(z))$ by $\wtd{e}_{ij}^{\pm}(z)$, $\wtd{h}_{i}^{\pm}(z)$, $\wtd{f}_{ji}^{\pm}(z)$,
and use the same symbol $\wtd{\frak{X}}_{i}^{\pm}(w)$ to denote~\eqref{eq:X_series_double_cartan} but with $\wtd{e}_{i,i+1}^{\pm}(z)$
and $\wtd{f}_{i+1,i}^{\pm}(z)$ in place of $e_{i,i+1}^{\pm}(z)$ and $f_{i+1,i}^{\pm}(z)$.

\begin{theorem}\label{thm:D=RTT_Ctype_2param}
There is a unique algebra isomorphism $\what{\theta}_{r,s}\colon U_{r,s}^{D}(\what{\frak{sp}}_{2n}) \iso \wtd{U}(\wtd{R}_{r,s}(z))$ such that
\begin{equation*}
  \what{\theta}_{r,s}(\gamma^{1/2}) = (\gamma')^{1/2}, \qquad \what{\theta}_{r,s}((\gamma')^{1/2}) = \gamma^{1/2},
\end{equation*}
\begin{equation*}
  \what{\theta}_{r,s}(\omega_{n}^{1/2}) = h_{n}^{+}[0]^{-1},\qquad \what{\theta}_{r,s}((\omega_{n}')^{1/2}) = h_{n}^{-}[0]^{-1},
\end{equation*}
\begin{align*}
  &\what{\theta}_{r,s}(\omega_{i}(z)) =
  \begin{cases}
    \wtd{h}_{i+1}^{+}(z^{-1}r^{i/2}s^{-i/2})\wtd{h}_{i}^{+}(z^{-1}r^{i/2}s^{-i/2})^{-1} & \textit{if}\ 1 \le i < n \\
    \wtd{h}_{n+1}^{+}(z^{-1}r^{(n+1)/2}s^{-(n+1)/2})\wtd{h}_{n}^{+}(z^{-1}r^{(n+1)/2}s^{-(n+1)/2})^{-1} & \textit{if}\ i = n
  \end{cases}, \\
  & \what{\theta}_{r,s}(\omega_{i}'(z)) =
  \begin{cases}
    \wtd{h}_{i+1}^{-}(z^{-1}r^{i/2}s^{-i/2})\wtd{h}_{i}^{-}(z^{-1}r^{i/2}s^{-i/2})^{-1} & \textit{if}\  1 \le i < n \\
    \wtd{h}_{n+1}^{-}(z^{-1}r^{(n+1)/2}s^{-(n+1)/2})\wtd{h}_{n}^{-}(z^{-1}r^{(n+1)/2}s^{-(n+1)/2})^{-1} & \textit{if}\ i = n
  \end{cases}, \\
  &\what{\theta}_{r,s}(x_{i}^{\pm}(z)) = \frac{1}{r - s}\wtd{\frak{X}}_{i}^{\mp}(z^{-1}r^{i/2}s^{-i/2})\qquad \textit{for}\qquad 1 \le i < n, \\
  &\what{\theta}_{r,s}(x_{n}^{+}(z)) = \frac{rs}{r^{2} - s^{2}}\wtd{\frak{X}}_{n}^{-}(z^{-1}r^{(n+1)/2}s^{-(n+1)/2}),\qquad
   \what{\theta}_{r,s}(x_{n}^{-}(z)) = \frac{1}{r^{2} - s^{2}}\wtd{\frak{X}}_{n}^{+}(z^{-1}r^{(n+1)/2}s^{-(n+1)/2}).
\end{align*}
\end{theorem}

\begin{proof}
Since the isomorphism $\what{\theta}_{q}$ of Corollary~\ref{cor:D=RTT_Ctype_1param_double_cartan} preserves the $Q\times Q$-grading,
it induces an algebra isomorphism $\what{\theta}_{q}\colon U_{q,\zeta}^{D}(\what{\frak{sp}}_{2n}) \iso \wtd{U}(\wtd{R}_{q}(z))_{\zeta}$.
We define $\what{\theta}_{r,s} = (\what{\Psi}')^{-1}\circ \what{\theta}_{q} \circ \varphi_D$ with $\what{\Psi}'$ of
Proposition~\ref{prop:2_vs_1_parameter_affine_RTT} and $\varphi_D$ of Theorem~\ref{thm:twisted_algebra_loop}, so that $\what{\theta}_{r,s}$
is an algebra isomorphism. Applying Lemma~\ref{lem:gaussian_generators_2_param} (with $\psi_{i}$ from~\eqref{eq:psi_values_C} though),
the above formulas for $\what{\theta}_{r,s}$-images of the generators are computed as in Theorem~\ref{thm:DJ=RTT_Btype_2_param}.
\end{proof}

   %%%%%%%%%%%%%%%%%%%%%%%%%%%%%%%%%%%%%%%%%%%%%%%%%%%%%%%%%%%%%%%%%%%%%%%%%%

\noindent
$\bullet$ \textbf{Type $D_n$.}

Let $R_{q}(z) = \hat{R}_{q}(z) \circ \tau$, where $\hat{R}_{q}(z)$ is the one-parameter $D_{n}$-type affine $R$-matrix, which is the $r = q = s^{-1}$
specialization of~\cite[(6.13)]{MT1}, and let $\wtd{R}_{q}(z)$ be its normalization as in~\eqref{eq:f_prefactor}. Similarly to type $C_{n}$, we first
need to enlarge $U_{r,s}^{D}(\what{\frak{so}}_{2n})$ by adjoining the generators $(\omega_{n-1}\omega_{n})^{1/2}$, $(\omega_{n-1}'\omega_{n}')^{1/2}$,
$(\omega_{n-1}^{-1}\omega_{n})^{1/2}$, and $((\omega_{n-1}')^{-1}\omega_{n}')^{1/2}$, with the obvious additional relations imposed. We also need a
similarly defined extension of $U_{q}^{D}(\what{\frak{so}}_{2n})$. As in types $B_{n}$ and $C_{n}$, we consider the Gauss decomposition of $L^{\pm}(z)$,
and define $\frak{X}_{i}^{\pm}(z)$ exactly as in~\eqref{eq:X_series} for $i < n$, while for $i = n$ we rather set
\[
  \frak{X}_{n}^{+}(z) = e_{n-1,n+1}^{+}(z\gamma^{1/2}) - e_{n-1,n+1}^{-}(z\gamma^{-1/2}),\qquad
  \frak{X}_{n}^{-}(z) = f_{n+1,n-1}^{+}(z\gamma^{-1/2}) - f_{n+1,n-1}^{-}(z\gamma^{1/2}).
\]
Then we have the following result (see~\cite[Main Theorem]{JLM1}):

\begin{theorem}\label{thm:D=affine_RTT_Dtype_1param}
There is a unique algebra isomorphism $\what{\theta}_{q}'\colon U_{q}^{D}(\what{\frak{so}}_{2n}) \iso \wtd{U}'(\wtd{R}_{q}(z))$ such that
\begin{align*}
  &\what{\theta}_{q}'(\gamma^{1/2}) = \gamma^{-1/2}, & & ~\\
  &\what{\theta}_{q}'((K_{n-1}K_{n})^{1/2}) = h_{n-1}^{+}[0]^{-1}, & & \what{\theta}_{q}'\left ((K_{n-1}^{-1}K_{n})^{1/2}\right ) = h_{n}^{+}[0]^{-1}, \\
  &\what{\theta}_{q}'(\Omega_{i}^{\pm}(z)) = h_{i+1}^{\pm}(z^{-1}q^{i})h_{i}^{\pm}(z^{-1}q^{i})^{-1},
   & & \what{\theta}_{q}'(X_{i}^{\pm}(z)) = \frac{1}{q - q^{-1}}\frak{X}_{i}^{\mp}(z^{-1}q^{i}) & &  \textit{for} \quad 1 \le i < n, \\
  &\what{\theta}_{q}'(\Omega_{n}^{\pm}(z)) = h_{n+1}^{\pm}(z^{-1}q^{n-1})h_{n-1}(z^{-1}q^{n-1})^{-1},
   & &\what{\theta}_{q}'(X_{n}^{\pm}(z)) = \frac{1}{q - q^{-1}}\frak{X}_{n}^{\mp}(z^{-1}q^{n-1}).
\end{align*}
\end{theorem}

\begin{proof}
This follows by composing the isomorphism of~\cite[Main Theorem]{JLM1} with $\what{\omega}$ of Lemma~\ref{lem:affine_cartan_involution}.
\end{proof}

As for types $B_{n}$ and $C_{n}$, we can use Propositions~\ref{prop:double_cartan_D} and~\ref{prop:double_cartan_RTT_affine} to derive
a Cartan-doubled version of the above theorem. To state it, we define $\wtd{\frak{X}}_{i}^{\pm}(w)$ exactly as in~\eqref{eq:X_series_double_cartan}
for $i < n$, while for $i = n$ we rather set
\begin{align*}
  &\wtd{\frak{X}}_{n}^{+}(w) = e_{n-1,n+1}^{+}(w(\gamma')^{-1/2}) - e_{n-1,n+1}^{-}(w(\gamma')^{1/2}(\gamma\gamma')^{1/2}), \\
  &\wtd{\frak{X}}_{n}^{-}(w) = f_{n+1,n-1}^{+}(w\gamma^{-1/2}(\gamma\gamma')^{-1/2}) - f_{n+1,n-1}^{-}(w\gamma^{1/2}).
\end{align*}

\begin{cor}\label{cor:D=RTT_Dtype_1param_double_cartan}
There is a unique algebra isomorphism $\what{\theta}_{q}\colon U_{q,q^{-1}}^{D}(\what{\frak{so}}_{2n}) \iso \wtd{U}(\wtd{R}_{q}(z))$ such that
\begin{align*}
  &\what{\theta}_{q}(\gamma^{1/2}) = (\gamma')^{1/2}, & & \what{\theta}_{q}((\gamma')^{1/2}) = \gamma^{1/2}, \\
  &\what{\theta}_{q}\big( (\omega_{n-1}\omega_{n})^{1/2} \big) = h_{n-1}^{+}[0]^{-1},
  & & \what{\theta}_{q} \big( (\omega_{n-1}'\omega_{n}')^{1/2} \big) = h_{n-1}^{-}[0]^{-1}, \\
  &\what{\theta}_{q} \big( (\omega_{n-1}^{-1}\omega_{n})^{1/2} \big) = h_{n}^{+}[0]^{-1}
  & & \what{\theta}_{q} \big( ((\omega_{n-1}')^{-1}\omega_{n}')^{1/2} \big) = h_{n}^{-}[0]^{-1},  \\
  &\what{\theta}_{q}(\omega_{i}(z)) = h_{i+1}^{+}(z^{-1}q^{i})h_{i}^{+}(z^{-1}q^{i})^{-1},
  & & \what{\theta}_{q}(\omega_{i}'(z)) = h_{i+1}^{-}(z^{-1}q^{i})h_{i}^{-}(z^{-1}q^{i})^{-1} \qquad \textit{for} \quad   1 \le i < n, \\
  &\what{\theta}_{q}(\omega_{n}(z)) = h_{n+1}^{+}(z^{-1}q^{n-1})h_{n-1}^{+}(z^{-1}q^{n-1})^{-1},
  & & \what{\theta}_{q}(\omega_{n}'(z))= h_{n+1}^{-}(z^{-1}q^{n-1})h_{n-1}^{-}(z^{-1}q^{n-1})^{-1}, \\
  &\what{\theta}_{q}(x_{i}^{\pm}(z)) = \frac{1}{q - q^{-1}}\wtd{\frak{X}}_{i}^{\mp}(z^{-1}q^{i}),
  & & \what{\theta}_{q}(x_{n}^{\pm}(z)) = \frac{1}{q - q^{-1}}\wtd{\frak{X}}_{n}^{\mp}(z^{-1}q^{n-1})  \qquad \textit{for} \quad 1 \le i < n.
\end{align*}
\end{cor}

We can now immediately derive a two-parameter version of Theorem~\ref{thm:D=affine_RTT_Dtype_1param}:

\begin{theorem}\label{thm:D=RTT_Dtype_2param}
There is a unique algebra isomorphism $\what{\theta}_{r,s}\colon U_{r,s}^{D}(\what{\frak{so}}_{2n}) \iso \wtd{U}(\wtd{R}_{r,s}(z))$ such that
\begin{equation*}
  \what{\theta}_{r,s}(\gamma^{1/2}) = (\gamma')^{1/2}, \qquad \what{\theta}_{r,s}((\gamma')^{1/2}) = \gamma^{1/2},
\end{equation*}
\begin{equation*}
  \what{\theta}_{r,s} \big( (\omega_{n-1}\omega_{n})^{1/2} \big) = h_{n-1}^{+}[0]^{-1} ,\qquad
  \what{\theta}_{r,s} \big( (\omega_{n-1}'\omega_{n}')^{1/2} \big) = h_{n-1}^{-}[0]^{-1},
\end{equation*}
\begin{equation*}
  \what{\theta}_{r,s} \big( (\omega_{n-1}^{-1}\omega_{n})^{1/2} \big) = h_{n}^{+}[0]^{-1} ,\qquad
  \what{\theta}_{r,s} \big( ((\omega_{n-1}')^{-1}\omega_{n}')^{1/2} \big) = h_{n}^{-}[0]^{-1},
\end{equation*}
\begin{equation*}
  \what{\theta}_{r,s}(\omega_{i}(z)) =
  \begin{cases}
    \wtd{h}_{i+1}^{+}(z^{-1}r^{i/2}s^{-i/2})\wtd{h}_{i}^{+}(z^{-1}r^{i/2}s^{-i/2})^{-1} & \textit{if}\ 1 \le i < n \\
    \wtd{h}_{n+1}^{+}(z^{-1}r^{(n-1)/2}s^{-(n-1)/2})\wtd{h}_{n-1}^{+}(z^{-1}r^{(n-1)/2}s^{-(n-1)/2})^{-1} & \textit{if}\ i = n
  \end{cases},
\end{equation*}
\begin{equation*}
  \what{\theta}_{r,s}(\omega_{i}'(z)) =
  \begin{cases}
    \wtd{h}_{i+1}^{-}(z^{-1}r^{i/2}s^{-i/2})\wtd{h}_{i}^{-}(z^{-1}r^{i/2}s^{-i/2})^{-1} & \textit{if}\  1 \le i < n \\
    \wtd{h}_{n+1}^{-}(z^{-1}r^{(n-1)/2}s^{-(n-1)/2})\wtd{h}_{n-1}^{-}(z^{-1}r^{(n-1)/2}s^{-(n-1)/2})^{-1} & \textit{if}\ i = n
  \end{cases},
\end{equation*}
\begin{equation*}
  \what{\theta}_{r,s}(x_{i}^{\pm}(z)) = \frac{1}{r - s}\wtd{\frak{X}}_{i}^{\mp}(z^{-1}r^{i/2}s^{-i/2})\qquad \textit{for}\qquad 1 \le i < n,
\end{equation*}
\begin{equation*}
  \what{\theta}_{r,s}(x_{n}^{+}(z)) = \frac{1}{r - s}\wtd{\frak{X}}_{n}^{-}(z^{-1}r^{(n-1)/2}s^{-(n-1)/2}),\qquad
  \what{\theta}_{r,s}(x_{n}^{-}(z)) = \frac{rs}{r - s}\wtd{\frak{X}}_{n}^{+}(z^{-1}r^{(n-1)/2}s^{-(n-1)/2}).
\end{equation*}
\end{theorem}

   %%%%%%%%%%%%%%%%%%%%%%%%%%%%%%%%%%%%%%%%%%%%%%%%%%%%%%%%%%%%%%%%%%%%%%%%%%
   %%%%%%%%%%%%%%%%%%%%%%%%%%%%%%%%%%%%%%%%%%%%%%%%%%%%%%%%%%%%%%%%%%%%%%%%%%
   %%%%%%%%%%%%%%%%%%%%%%%%%%%%%%%%%%%%%%%%%%%%%%%%%%%%%%%%%%%%%%%%%%%%%%%%%%

\appendix

   %%%%%%%%%%%%%%%%%%%%%%%%%%%%%%%%%%%%%%%%%%%%%%%%%%%%%%%%%%%%%%%%%%%%%%%%%%
   %%%%%%%%%%%%%%%%%%%%%%%%%%%%%%%%%%%%%%%%%%%%%%%%%%%%%%%%%%%%%%%%%%%%%%%%%%
   %%%%%%%%%%%%%%%%%%%%%%%%%%%%%%%%%%%%%%%%%%%%%%%%%%%%%%%%%%%%%%%%%%%%%%%%%%

\section{A-type counterpart}\label{sec:app_A}

In this Appendix, we present the $A$-type counterparts of the main results from Sections~\ref{sec:twisted_R_matrices},~\ref{sec:FRT construction finite}
and~\ref{sec:affine_RTT}, omitting most of the details, as the arguments are actually simpler than those for $BCD$-types treated above. For convenience
of exposition, we will primarily be working with $U_{q}(\gl_{n+1})$ and $U_{r,s}(\gl_{n+1})$ rather than $U_{q}(\ssl_{n+1})$ and $U_{r,s}(\ssl_{n+1})$,
but the results for (an extended version of) $U_{r,s}(\ssl_{n+1})$ can be easily obtained from those of $U_{r,s}(\gl_{n+1})$.

We first recall the formula for the one-parameter $A_{n}$-type $R$-matrix $\hat{R}_{q}$, which is the $r = q = s^{-1}$-specialization
of~\cite[(5.6)]{BW2}:
\begin{equation*}
  \hat{R}_{q}
  = \sum_{1\leq i\leq n+1} E_{ii} \otimes E_{ii} + q\sum_{1 \le i \neq j \le n + 1}E_{ij} \otimes E_{ji}
  + (1 - q^{2})\sum_{1 \le i < j \le n + 1}E_{jj} \otimes E_{ii}.
\end{equation*}
We note that the $R$-matrix of~\cite[(2.2)]{DF} is different from ours: their $R$ is related to $R_{q} = \hat{R}_{q} \circ \tau$ via
\begin{equation}\label{eq:DF-vs-us}
  q^{-1}R=(R_{q^{-1}})_{21} = (R_{q})_{21}^{-1}.
\end{equation}

Next, we recall the one-parameter isomorphism of~\cite[Theorem 2.1]{DF} (with $\ell_{ij}^{+}$ and $\ell_{ij}^{-}$ interchanged and $q$ replaced
by $q^{-1}$, which is done due to our different conventions and the identification~\eqref{eq:DF-vs-us}). This requires working with the algebra
$U_{q}(\gl_{n+1})$ of~\cite[Definition 2.3]{DF}, but we choose to denote the generators of $U_{q}(\gl_{n+1})$ by
$\{A_{j}^{\pm 1},E_{i},F_{i}\}_{1\leq i\leq n}^{1\leq j\leq n+1}$, in place of $\{q^{\pm H_{j}},e_{i},f_{i}\}_{1\leq i\leq n}^{1\leq j\leq n+1}$
from~\cite[Definition~2.3]{DF}, respectively.
We also equip $U_{q}(\gl_{n+1})$ with a Hopf algebra structure that is co-opposite to that of~\cite[(2.9)]{DF}.

\begin{theorem}\label{thm:DJ=RTT_Atype_1param}
There is a unique Hopf algebra isomorphism $\theta_{q}'\colon U_{q}(\gl_{n+1}) \iso U'(R_{q})$ such that
\[
  \theta_{q}'(A_{j}) = (\ell_{jj}^{+})^{-1}  \qquad \text{for}\quad 1 \le j \le n + 1,
\]
\[
  \theta_{q}'(E_{i}) = \frac{1}{q - q^{-1}}\ell_{i+1,i}^{+}(\ell_{ii}^{+})^{-1},\qquad
  \theta_{q}'(F_{i}) = \frac{1}{q^{-1} - q}(\ell_{ii}^{-})^{-1}\ell_{i,i+1}^{-} \qquad \text{for} \quad 1 \le i \le n.
\]
\end{theorem}

To upgrade this result to the two-parameter setup, we first need to realize the two-parameter $R$-matrix of \cite[(5.6)]{BW2} as a twist
of $\hat{R}_{q}$. Let $P = \bigoplus_{i = 1}^{n+1} \bb{Z}\varepsilon_{i}$, where $\varepsilon_{i}$ is an orthonormal basis of $\bb{R}^{n+1}$ and
$\alpha_{i} = \varepsilon_{i} - \varepsilon_{i+1}$ for $1 \le i \le n$. It is easy to verify that the following formulas define an extension of
the skew bicharacter $\zeta$ of~\eqref{eq:zeta_formula} to a skew bicharacter $\zeta\colon P \times P \to \bb{K}^{*}$:
\begin{equation}\label{eq:zeta_extension_A}
  \zeta(\varepsilon_{i},\varepsilon_{j}) =
  \begin{cases}
    (rs)^{1/4} & \text{if}\ i < j \\
    1 & \text{if}\ i = j \\
    (rs)^{-1/4} & \text{if}\ i > j
  \end{cases} \,.
\end{equation}
Now, let $V = \bigoplus_{i = 1}^{n+1} \bb{K}v_{i}$. Then we have the following result:

\begin{prop}
The following defines a representation $\rho_{q}\colon U_{q,q^{-1}}(\ssl_{n+1}) \to \End(V)$:
\[
  \rho_q(E_{i}) = E_{i,i+1},\qquad \rho_q(F_{i}) = E_{i+1,i},
\]
\[
  \rho_q(\omega_{i}) = qE_{ii} + q^{-1}E_{i+1,i+1}  + \sum_{j \neq i,i+1}E_{jj},\qquad
  \rho_q(\omega_{i}') = q^{-1}E_{ii} + qE_{i+1,i+1} + \sum_{j \neq i,i + 1}E_{jj},
\]
for all $1 \le i \le n$.
\end{prop}

The above representation clearly has a weight space decomposition $V=\bigoplus_{i=1}^{n+1} V[\varepsilon_i]$ with $V[\varepsilon_i]=\bb{K}v_{i}$.
Using the bicharacter~\eqref{eq:zeta_extension_A}, we can twist this representation via Proposition~\ref{prop:twisted_representation} to obtain
a representation $\wtd{\rho}_{r,s}\colon U_{r,s}(\ssl_{n+1}) \to \End(V)$, which is related to the representation $\rho_{r,s}$
of~\cite[Proposition~3.1]{MT1} via
\[
  \wtd{\rho}_{r,s}(x) \circ \psi = \psi \circ \rho_{r,s}(x) \qquad \forall\, x \in U_{r,s}(\ssl_{n+1}),
\]
where the linear map $\psi\colon V \to V$ is defined by $\psi(v_{i}) = \psi_{i}v_{i}$, with
\begin{equation}\label{eq:psi_values_A}
  \psi_{i} = (rs)^{\frac{n + 1 - i}{4}}\qquad \text{for all}\qquad 1 \le i \le n + 1.
\end{equation}

As in Section~\ref{sec:twisted_R_matrices}, it follows that the natural intertwiner for the representation $V \otimes V$ defined by $\rho_{r,s}$ is
\begin{equation}\label{eq:R_rs_twisted}
  \hat{R}_{r,s} = (\psi \otimes \psi)^{-1} \circ \xi \circ \hat{R}_{q} \circ \xi^{-1} \circ (\psi \otimes \psi),
\end{equation}
with $\xi = \xi_{V,V}$ given by~\eqref{eq:xi_formula}. Moreover, it precisely coincides with that of~\cite[(5.6)]{BW2}, cf.~\cite[(4.6)]{MT1}:
\begin{equation}\label{eq:A_R_matrix_2param}
  \hat{R}_{r,s} = \sum_{i} E_{ii} \otimes E_{ii} + r\sum_{i < j} E_{ji} \otimes E_{ij}
  + s^{-1}\sum_{i < j} E_{ij} \otimes E_{ji} + (1 - rs^{-1})\sum_{i< j} E_{jj} \otimes E_{ii}.
\end{equation}
We define $R_{r,s}=\hat{R}_{r,s} \circ \tau$.

Moreover, as in Subsection~\ref{ssec:affine_R_twisting}, we can also twist the one-parameter affine $A_{n}$-type $R$-matrix $\hat{R}_{q}(z)$ given
by~\cite[(3.7)]{DF} composed with $(1-zq^2)\tau$ (which matches with $\hat{R}_{q,q^{-1}}(z)$), to obtain the corresponding two-parameter affine
$R$-matrix $\hat{R}_{r,s}(z)$ of~\cite[(2.4)]{JL2} (cf.~\cite[(6.10)]{MT1}). First, recall that $\hat{R}_{r,s}(z)$ is given~by
\begin{equation}\label{eq:A_AffRMatrix}
\begin{split}
  \hat{R}_{r,s}(z)
  &= (1 -zrs^{-1})\sum_{i} E_{ii} \otimes E_{ii} + (1 - z)r\sum_{i > j} E_{ij} \otimes E_{ji} + (1 - z)s^{-1}\sum_{i < j}E_{ij} \otimes E_{ji} \\
  & \quad + (1 - rs^{-1})\sum_{i > j}E_{ii} \otimes E_{jj} + z(1 - rs^{-1})\sum_{i < j}E_{ii} \otimes E_{jj}.
\end{split}
\end{equation}
Then, we recall from~\cite[(7.2)]{MT1} that
\[
  \hat{R}_{r,s}(z) = \hat{R}_{r,s} - zrs^{-1}\hat{R}_{r,s}^{-1} = \hat{R}_{r,s} - zq^{2}\hat{R}_{r,s}^{-1},
\]
and thus, due to~\eqref{eq:R_rs_twisted}, we also have
\begin{equation}\label{eq:R_rs_twisted_affine}
  \hat{R}_{r,s}(z) = (\psi \otimes \psi)^{-1} \circ \xi \circ \hat{R}_{q}(z) \circ \xi^{-1} \circ (\psi \otimes \psi).
\end{equation}

We may now upgrade Theorem~\ref{thm:DJ=RTT_Atype_1param} to its two-parameter counterpart. To this end, we need to work with $U_{r,s}(\gl_{n+1})$,
which we define as follows (while our definition looks different than that of~\cite[\S1]{BW1} it is actually equivalent by
Remark~\ref{rem:DF-vs-our-Cartan}; however, our choice is more natural from the perspective of Proposition~\ref{prop:twisted_algebra_gln}):

\begin{definition}\label{def:BW-restated}
The two-parameter quantum group of $\gl_{n+1}$ is the associative $\bb{K}$-algebra $U_{r,s}(\gl_{n+1})$ generated by
$\{e_{i},f_{i},a_{j}^{\pm 1},(a_{j}')^{\pm 1}\}_{1\leq i\leq n}^{1\leq j\leq n+1}$, with the defining relations
\[
  [a_{i},a_{j}] = [a_{i},a_{j}'] = [a_{i}',a_{j}'] = 0,\qquad
  a_{i}^{\pm 1}a_{i}^{\mp 1} = 1 = (a_{i}')^{\pm 1}(a_{i}')^{\mp 1} \qquad \text{for} \quad 1 \le i,j \le n+1,
\]
\[
  a_{i}e_{j} = e_{j}a_{i},\qquad a_{i}f_{j} = f_{j}a_{i},\qquad
  a_{i}'e_{j} = e_{j}a_{i}',\qquad a_{i}'f_{j} = f_{j}a_{i}' \qquad \text{if} \quad i < j\ \text{or}\ i > j + 1,
\]
\[
  a_{i}e_{i} = s^{-1}e_{i}a_{i},\qquad a_{i}f_{i} = sf_{i}a_{i},\qquad
  a_{i}'e_{i} = r^{-1}e_{i}a_{i}',\qquad a_{i}'f_{i} = rf_{i}a_{i}' \qquad \text{for}\qquad 1 \le i \le n,
\]
\[
  a_{i+1}e_{i} = r^{-1}e_{i}a_{i+1},\quad a_{i+1}f_{i} = rf_{i}a_{i+1},\quad
  a_{i+1}'e_{i} = s^{-1}e_{i}a_{i+1}',\quad a_{i+1}'f_{i} = sf_{i}a_{i+1}' \quad \text{for}\quad 1 \le i \le n,
\]
\[
  [e_{i},f_{j}] = \delta_{ij}\frac{a_{i}a_{i+1}^{-1} - a_{i}'(a_{i+1}')^{-1}}{r_{i} - s_{i}}\qquad \text{for all}\qquad 1 \le i,j \le n,
\]
and relations~\eqref{eq:R5} for all $i \neq j$.
\end{definition}

\begin{remark}\label{rem:DF-vs-our-Cartan}
Let $U^{\mathrm{BW}}_{r,s}(\gl_{n+1})$ be the algebra of~\cite[\S1]{BW1}, generated by
$\{E_{i},F_{i},A_{j}^{\pm 1},B_{j}^{\pm 1}\}_{1\leq i\leq n}^{1\leq j\leq n+1}$ subject to~\cite[(R1)--(R7)]{BW1}
(with capital letters just to distinguish from the definition above). Then, we have an algebra isomorphism
$\varrho \colon U_{r,s}(\gl_{n+1})\iso U^{\mathrm{BW}}_{r,s}(\gl_{n+1})$ given by
$\varrho(e_i)=E_i, \varrho(f_i)=F_i$ for $1\leq i\leq n$ and
$\varrho(a_j)=1/(B_1\ldots B_j A_{1}\ldots A_{j-1}), \varrho(a'_j)=1/(A_1\ldots A_j B_{1}\ldots B_{j-1})$ for $1\leq j\leq n+1$.
In particular, we have $\varrho(a_i a^{-1}_{i+1})=A_iB_{i+1}, \varrho(a'_i (a'_{i+1})^{-1})=B_iA_{i+1}$ for $1\leq i\leq n$.
We further note that
  $\varrho^{-1}(A_j)=(a'_j)^{-1}\cdot (a_{j-1}a'_{j-1}) \cdot (a_{j-2}a'_{j-2})^{-1} \ldots  (a_{1}a'_{1})^{(-1)^j},
   \varrho^{-1}(B_j)=a_j^{-1}\cdot (a_{j-1}a'_{j-1}) \cdot (a_{j-2}a'_{j-2})^{-1} \ldots  (a_{1}a'_{1})^{(-1)^j}$.
\end{remark}

There is a Hopf algebra structure on $U_{r,s}(\gl_{n+1})$, where $a_{i},a_{i}'$ are grouplike, while the coproduct, counit, and antipode of
$e_{i},f_{i}$ are as for $U_{r,s}(\ssl_{n+1})$ with $\omega_{i} = a_{i}a_{i+1}^{-1}$, $\omega_{i}' = a_{i}'(a_{i+1}')^{-1}$. Furthermore,
the assignment
\[
  \deg(a_{j}) = (\varepsilon_{j},-\varepsilon_{j}),\qquad \deg(a_{j}') = (\varepsilon_{j},-\varepsilon_{j}),\qquad
  \deg(e_{i}) = (\alpha_{i},0),\qquad \deg(f_{i}) = (0,-\alpha_{i})
\]
makes $U_{r,s}(\gl_{n+1})$ into a $P$-bigraded Hopf algebra. Thus, using the bicharacter $\zeta$ of~\eqref{eq:zeta_extension_A}, we may define
the Hopf algebra $U_{q,\zeta}(\gl_{n+1})=(U_{q,q^{-1}}(\gl_{n+1}))_\zeta$, which has multiplication given by~\eqref{eq:twisted_product_general},
and the coproduct, counit, and antipode are as those of $U_{q,q^{-1}}(\gl_{n+1})$. Then, we have the following analogue of
Proposition~\ref{prop:twisted_algebra}:

\begin{prop}\label{prop:twisted_algebra_gln}
There is a Hopf algebra isomorphism $\wtd{\varphi}\colon U_{r,s}(\gl_{n+1}) \iso U_{q,\zeta}(\gl_{n+1})$ given by
\begin{equation}\label{eq:twist_isomorphism_gln}
  \wtd{\varphi}(e_{i}) = e_{i},\quad \wtd{\varphi}(f_{i}) = (rs)^{-1/2}f_{i},\quad
  \wtd{\varphi}(a_{j}) = a_{j},\quad \wtd{\varphi}(a_{j}') = a_{j}' \qquad \forall\, 1\leq i\leq n,\ 1 \le j \le n + 1.
\end{equation}
\end{prop}

\begin{proof}
The proof is straightforward; we leave details to the reader.
\end{proof}

We also have the following analogue of Proposition~\ref{prop:double_cartan_DJ} for $\gl_{n+1}$, with a nearly identical proof:

\begin{prop}\label{prop:double_cartan_gl}
Let $L_{\varepsilon_{1}},\ldots ,L_{\varepsilon_{n+1}}$ be algebraically independent invertible transcendental elements.
Then there is an injective algebra homomorphism
  $\eta\colon U_{q,q^{-1}}(\gl_{n+1}) \to U_{q}(\gl_{n+1}) \otimes \bb{K}[L_{\varepsilon_{1}}^{\pm 1},\ldots ,L_{\varepsilon_{n+1}}^{\pm 1}]$
such that
\[
  \eta(a_{j}) = A_{j} \otimes L_{\varepsilon_{j}}^{2},\qquad \eta(a_{j}') = A_{j}^{-1} \otimes L_{\varepsilon_{j}}^{2},\qquad
  \eta(e_{i}) = E_{i} \otimes L_{\varepsilon_{i} - \varepsilon_{i+1}},\qquad \eta(f_{i}) = F_{i} \otimes L_{\varepsilon_{i} - \varepsilon_{i+1}}
\]
for all $1 \le i \le n$ and $1 \le j \le n + 1$.
\end{prop}

\begin{remark}\label{rem:sl-vs-gl_finiteA}
There is an algebra homomorphism $\iota_{r,s} \colon U_{r,s}(\ssl_{n+1}) \to U_{r,s}(\gl_{n+1})$ defined by
\begin{equation}\label{eq:sl-vs-gl-matching_finiteA}
  e_{i} \mapsto e_{i},\qquad f_{i} \mapsto f_{i},\qquad \omega_{i} \mapsto a_{i}a_{i+1}^{-1},\qquad \omega_{i}' \mapsto a_{i}'(a_{i+1}')^{-1}
  \qquad \forall\, 1\leq i\leq n.
\end{equation}
We note that $\iota_{r,s}$ intertwines the $P \times P$-gradings. We also note that $\iota_{q,q^{-1}}$ is compatible with an analogous embedding
$\iota_q\colon U_{q}(\ssl_{n+1})\to U_q(\gl_{n+1})$ through the embeddings of Propositions~\ref{prop:double_cartan_DJ} and~\ref{prop:double_cartan_gl},
which implies that $\iota_{q,q^{-1}}$ is injective. Furthermore, $\wtd{\varphi}$ of Proposition~\ref{prop:twisted_algebra_gln} is compatible with
$\varphi$ of Proposition~\ref{prop:twisted_algebra}, that is $\wtd{\varphi} \circ \iota_{r,s} = \iota_{q,q^{-1}} \circ \varphi$.
Thus, $\iota_{r,s}$ is injective.
\end{remark}

On the other hand, the algebras $U(R_{q})$ and $U(R_{r,s})$ are also $P$-bigraded Hopf algebras via the assignment
\[
  \deg(\ell_{ij}^{\pm}) = (-\varepsilon_{i},\varepsilon_{j}),
\]
since $(R_{q})^{ij}_{kt} = 0 = (R_{r,s})^{ij}_{kt}$ unless $\{i,j\} = \{k,t\}$, cf.~Propositions~\ref{prop:A(R)_grading} and~\ref{prop:U(R)_grading}.
Moreover, by Corollary~\ref{cor:U(R)_2_vs_1_parameter} (with $U(R_q),U(R_{r,s})$ instead of $\wtd{U}(R_q),\wtd{U}(R_{r,s})$, respectively)
and~\eqref{eq:R_rs_twisted}, we have a Hopf algebra isomorphism $U(R_{r,s}) \iso U(R_{q})_{\zeta}$ given by
$\ell_{ij}^{\pm} \mapsto \psi_{i}^{-1}\psi_{j}\ell_{ij}^{\pm}$ with $\psi_{i}$ of~\eqref{eq:psi_values_A}. We also note
that Proposition~\ref{prop:double_cartan_RTT_affine} still holds in this case (but with $U'(R_{q}), U(R_{q})$ instead of
$\wtd{U}(R_{q}), \wtd{U}'(R_{q})$, respectively). Thus, following the procedure of Section~\ref{sec:FRT construction finite}, we get:

\begin{theorem}\label{thm:DJ=RTT_Atype_2param}
There is a unique Hopf algebra isomorphism $\theta_{r,s}\colon U_{r,s}(\gl_{n+1}) \iso U(R_{r,s})$ such that
\begin{align*}
  &\theta_{r,s}(a_{j}) = (\ell_{jj}^{+})^{-1}, & & \theta_{r,s}(a_{j}') = (\ell_{jj}^{-})^{-1} & & \textit{for} \ \ 1 \le j \le n + 1, \\
  &\theta_{r,s}(e_{i}) = \frac{1}{r - s}\ell_{i+1,i}^{+}(\ell_{ii}^{+})^{-1}, &
  &\theta_{r,s}(f_{i}) = \frac{1}{s - r}(\ell_{ii}^{-})^{-1}\ell_{i,i+1}^{-} & & \textit{for} \ \ 1 \le i \le n.
\end{align*}
\end{theorem}

Combining this result with Remark~\ref{rem:sl-vs-gl_finiteA}, we immediately get:

\begin{cor}
There is a Hopf algebra embedding $\theta_{r,s}\colon U_{r,s}(\ssl_{n+1}) \hookrightarrow U(R_{r,s})$ such that for all $1\leq i\leq n$:
\begin{align*}
  &\theta_{r,s}(\omega_{i}) = \ell^+_{i+1,i+1} (\ell_{ii}^{+})^{-1}, &
  &\theta_{r,s}(\omega'_{i}) = (\ell_{ii}^{-})^{-1} \ell^-_{i+1,i+1}, \\
  &\theta_{r,s}(e_{i}) = \frac{1}{r - s}\ell_{i+1,i}^{+}(\ell_{ii}^{+})^{-1}, &
  &\theta_{r,s}(f_{i}) = \frac{1}{s - r}(\ell_{ii}^{-})^{-1}\ell_{i,i+1}^{-}.
\end{align*}
\end{cor}

\begin{remark}\label{rem:sl-quotient-gl_finite}
(a) In the standard one-parameter setup, it is also known that (an extension of) $U_q(\ssl_{n+1})$ can be realized as a quotient of
$U_q(\gl_{n+1})$, which we recall now. To this end, we shall enlarge $U_{q}(\ssl_{n+1})$ by adjoining additional generators
$(K_{1}^{n} K_{2}^{n-1} \ldots K_n)^{\pm \frac{1}{n+1}}$ subject to obvious additional relations. Then, the assignment akin
to~\eqref{eq:sl-vs-gl-matching_finiteA} gives rise to a Hopf algebra isomorphism
\begin{equation}\label{eq:1param_quotient_finiteA}
  U_{q}(\ssl_{n+1}) \iso U_{q}(\gl_{n+1})/(A_{1}\ldots A_{n+1} - 1).
\end{equation}
Since $\ell_{11}^{\pm}\ldots \ell_{n+1,n+1}^{\pm}$ are grouplike, the quotient
$\wtd{U}'(R_{q}) = U'(R_{q})/(\ell_{11}^{\pm}\ldots \ell_{n+1,n+1}^{\pm} - 1)$ is a Hopf algebra.
Combining Theorem~\ref{thm:DJ=RTT_Atype_1param} and~\eqref{eq:1param_quotient_finiteA}, we obtain
a Hopf algebra isomorphism $\theta_{q}\colon U_{q}(\ssl_{n+1}) \iso \wtd{U}'(R_{q})$ with
\begin{align*}
  &\theta_{q}\Big(\big(K_{1}^{n} K_{2}^{n-1} \ldots K_n\big)^{\frac{1}{n+1}}\Big)=(\ell^+_{11})^{-1}, &
  &\theta_{q}(K_{i}) = \ell_{i+1,i+1}^{+} (\ell_{ii}^{+})^{-1} &
  & \textit{for} \ \ 1 \le i \le n, \\
  &\theta_{q}(E_{i}) = \frac{1}{q - q^{-1}}\ell_{i+1,i}^{+}(\ell_{ii}^{+})^{-1}, &
  &\theta_{q}(F_{i}) = \frac{1}{q^{-1} - q}(\ell_{ii}^{-})^{-1}\ell_{i,i+1}^{-} & & \textit{for} \ \ 1 \le i \le n.
\end{align*}

\noindent
(b) To generalize the above discussion to the two-parameter setup, one actually needs to slightly modify the construction of
Definition~\ref{def:BW-restated} (since the products $a_1\ldots a_{n+1}$ and $a'_{1}\ldots a'_{n+1}$ are not central). Namely, let
$\mathsf{U}_{r,s}(\gl_{n+1})$ be defined similarly to $U_{r,s}(\gl_{n+1})$, but with the following change of the defining relations
(where we use $\lambda=(rs)^{\frac{1}{n+1}}$):
\[
  a_{i}e_{j} = \lambda e_{j}a_{i},\qquad a_{i}f_{j} = \lambda^{-1} f_{j}a_{i},\qquad
  a_{i}'e_{j} = \lambda e_{j}a_{i}',\qquad a_{i}'f_{j} = \lambda^{-1} f_{j}a_{i}' \qquad \text{if} \quad i < j\ \text{or}\ i > j + 1,
\]
\[
  a_{i}e_{i} = \lambda s^{-1}e_{i}a_{i},\qquad a_{i}f_{i} = \lambda^{-1} sf_{i}a_{i},\qquad
  a_{i}'e_{i} = \lambda r^{-1}e_{i}a_{i}',\qquad a_{i}'f_{i} = \lambda^{-1} rf_{i}a_{i}' \qquad \text{if}\ \ 1 \le i \le n,
\]
\[
  a_{i+1}e_{i} = \lambda r^{-1}e_{i}a_{i+1},\quad a_{i+1}f_{i} = \lambda^{-1} rf_{i}a_{i+1},\quad
  a_{i+1}'e_{i} = \lambda s^{-1}e_{i}a_{i+1}',\quad a_{i+1}'f_{i} = \lambda^{-1} sf_{i}a_{i+1}' \quad \text{if}\ \ 1 \le i \le n.
\]
Then, the products $a_1\ldots a_{n+1}$ and $a'_{1}\ldots a'_{n+1}$ are central elements of $\mathsf{U}_{r,s}(\gl_{n+1})$.
Furthermore, the assignment~\eqref{eq:sl-vs-gl-matching_finiteA} still defines an algebra embedding
$\iota_{r,s} \colon U_{r,s}(\ssl_{n+1}) \to \mathsf{U}_{r,s}(\gl_{n+1})$, giving rise to a Hopf algebra isomorphism
(with $U_{r,s}(\ssl_{n+1})$ extended by adjoining $(\omega_{1}^{n} \omega_{2}^{n-1} \ldots \omega_n)^{\frac{1}{n+1}}$
and $((\omega'_{1})^{n} (\omega'_{2})^{n-1} \ldots \omega'_n)^{\frac{1}{n+1}}$):
\begin{equation}\label{eq:2param_quotient_finiteA}
  U_{r,s}(\ssl_{n+1}) \iso \mathsf{U}_{r,s}(\gl_{n+1})/(a_{1}\ldots a_{n+1} - 1,a_{1}'\ldots a_{n+1}' - 1).
\end{equation}

To obtain a two-parameter version of the isomorphism $U_{q}(\ssl_{n+1}) \iso \wtd{U}'(R_{q})$ from part (a), we need to slightly modify
$\zeta$ so that an analogue of Proposition~\ref{prop:twisted_algebra_gln} holds for $\mathsf{U}_{r,s}(\gl_{n+1})$. To this end, we introduce
the skew bicharacter $\wtd{\zeta}\colon P \times P \to \bb{K}^{*}$ given by
\begin{equation}\label{eq:zeta_modified_A}
  \wtd{\zeta}(\varepsilon_{i},\varepsilon_{j}) = \lambda^{(i - j)/2}\zeta(\varepsilon_{i},\varepsilon_{j}),
  \qquad \text{where}\qquad \lambda = (rs)^{\frac{1}{n+1}}.
\end{equation}
Then, the assignment~\eqref{eq:twist_isomorphism_gln} gives rise to a Hopf algebra isomorphism
$\wtd{\varphi}\colon \mathsf{U}_{r,s}(\gl_{n+1}) \iso U_{q,\wtd{\zeta}}(\gl_{n+1})$. Now, let
  $\hat{\mathsf{R}}_{r,s} = (\psi \otimes \psi)^{-1} \circ \wtd{\xi} \circ \hat{R}_{q} \circ \wtd{\xi}^{-1} \circ (\psi \otimes \psi)$
and $\mathsf{R}_{r,s} = \hat{\mathsf{R}}_{r,s} \circ \tau$, where $\wtd{\xi}\colon V \otimes V \to V \otimes V$ is defined by~\eqref{eq:xi_formula}
with $\wtd{\zeta}$ instead of $\zeta$. Then $\hat{\mathsf{R}}_{r,s}$ is an intertwiner for the $U_{r,s}(\ssl_{n+1})$-module $V \otimes V$,
where the $U_{r,s}(\ssl_{n+1})$-module structure $\rho_{r,s}'$ on $V$ is given by
\[
  \rho_{r,s}'(e_{i}) = \lambda^{-1/2}E_{i,i+1},\qquad \rho_{r,s}'(f_{i}) = \lambda^{-1/2}E_{i+1,i},
\]
\[
  \rho_{r,s}'(\omega_{i}) = \lambda^{-1}rE_{ii} + \lambda^{-1}sE_{i+1,i+1} + \lambda^{-1}\sum_{j \neq i,i+1}E_{jj},\qquad
  \rho_{r,s}'(\omega_{i}') = \lambda^{-1}sE_{ii} + \lambda^{-1}rE_{i+1,i+1} + \lambda^{-1}\sum_{j \neq i,i +1}E_{jj}
\]
for all $1 \le i \le n$. Explicitly, the matrix $\hat{\mathsf{R}}_{r,s}$ is given by
\[
  \hat{\mathsf{R}}_{r,s}
  = \sum_{i} E_{ii} \otimes E_{ii} + \sum_{i < j}\lambda^{i - j} r E_{ji} \otimes E_{ij}
  + \sum_{i < j} \lambda^{j - i} s^{-1} E_{ij} \otimes E_{ji} + (1 - rs^{-1})\sum_{i < j} E_{jj} \otimes E_{ii}.
\]
Moreover, combining~\eqref{eq:2param_quotient_finiteA} with an analogue of Corollary~\ref{cor:U(R)_2_vs_1_parameter} for
$\mathsf{U}_{r,s}(\gl_{n+1})$, we get a Hopf algebra isomorphism
  $\theta'_{r,s}\colon U_{r,s}(\ssl_{n+1}) \iso \wtd{U}(\mathsf{R}_{r,s}) = U(\mathsf{R}_{r,s})/(\ell_{11}^{\pm}\ldots \ell_{n+1,n+1}^{\pm} - 1)$
such that for all $1 \le i \le n$:
\begin{align*}
  &\theta'_{r,s}(\omega_{i}) = \ell^+_{i+1,i+1} (\ell_{ii}^{+})^{-1}, &
  &\theta'_{r,s}(\omega'_{i}) = (\ell_{ii}^{-})^{-1} \ell^-_{i+1,i+1}, \\
  &\theta'_{r,s}(e_{i}) = \frac{\lambda^{1/2}}{r - s}\ell_{i+1,i}^{+}(\ell_{ii}^{+})^{-1}, &
  &\theta'_{r,s}(f_{i}) = \frac{\lambda^{1/2}}{s - r}(\ell_{ii}^{-})^{-1}\ell_{i,i+1}^{-}.
\end{align*}
\end{remark}

\begin{remark}
We note that the algebra $A(R_{r,s})$ of Subsection~\ref{ssec:algebra_A(R)} with $R_{r,s}=\hat{R}_{r,s}\circ \tau$,
cf.~\eqref{eq:A_R_matrix_2param}, is isomorphic to the algebra $M_{\alpha,\beta}$ of~\cite{T}, with parameters $\alpha=r$ and $\beta=s^{-1}$,
through the assignment $A(R_{r,s})\ni t_{ij}\mapsto x_{ij}\in M_{r,s^{-1}}$ for all $1\leq i,j\leq n+1$.
\end{remark}

We shall now carry out the affine version of Theorem~\ref{thm:DJ=RTT_Atype_2param}, following the methods of Section~\ref{sec:affine_RTT}.
To this end, we start by recalling the corresponding one-parameter isomorphism of~\cite[Main Theorem]{DF}. Consider
$\bar{R}_{q}(z) = \frac{1}{1 - zq^{2}}R_{q}(z)$, where $R_{q}(z) = \hat{R}_{q}(z) \circ \tau$ with $\hat{R}_{q}(z)=\hat{R}_{q,q^{-1}}(z)$
of~\eqref{eq:A_AffRMatrix}. Explicitly, we have:
\begin{equation*}
  \bar{R}_{q}(z)
  = \sum_{i} E_{ii} \otimes E_{ii} + \frac{z - 1}{zq - q^{-1}}\sum_{i \neq j} E_{ii} \otimes E_{jj}
  + \frac{q - q^{-1}}{zq - q^{-1}}\sum_{j < i}E_{ij} \otimes E_{ji} + \frac{z(q - q^{-1})}{zq - q^{-1}}\sum_{i < j} E_{ij} \otimes E_{ji},
\end{equation*}
which coincides with~\cite[(3.7)]{DF}. Let $U_{q}^{D}(\what{\gl}_{n+1})$ denote the algebra of~\cite[Definition~3.1]{DF},
but using the notation $\gamma^{\pm 1/2}$ instead of $q^{\pm c/2}$, and with the second-to-last relation corrected as follows:
\[
  [X_{i}^{+}(z),X_{j}^{-}(w)] = (q - q^{-1})\delta_{ij}
  \left( \delta(zw^{-1}\gamma^{-1})k_{i+1}^{+}(w\gamma^{1/2})k_{i}^{+}(w\gamma^{1/2})^{-1} -
         \delta(zw^{-1}\gamma)k_{i+1}^{-}(z\gamma^{1/2})k_{i}^{-}(z\gamma^{1/2})^{-1} \right).
\]
Following Subsection~\ref{ssec:gaussian_generators}, we also consider the algebra $U'(\bar{R}_{q}(z))$ and take the Gauss decomposition of its
generating matrices $L^{\pm}(z)$ (note that the resulting Gaussian generators $e_{ij}^{\pm}(z), f_{ji}^{\pm}(z), h_{i}^{\pm}(z)$ correspond
respectively to the Gaussian generators $f_{ij}^{\pm}(z), e_{ji}^{\pm}(z), k_{i}^{\pm}(z)$ of~\cite[(3.22)]{DF}). For $1\leq i\leq n$,
we further set
\[
  \frak{X}_{i}^{+}(z) = e_{i,i+1}^{+}(z\gamma^{1/2}) - e_{i,i+1}^{-}(z\gamma^{-1/2}), \qquad
  \frak{X}_{i}^{-}(z) = f_{i+1,i}^{+}(z\gamma^{-1/2}) - f_{i+1,i}^{-}(z\gamma^{1/2}),
\]
as in Subsection~\ref{ssec:affine_RTT_to_D}, cf.~\eqref{eq:X_series}. Then, according to~\cite[Main Theorem]{DF}, we have the following result:

\begin{theorem}\label{thm:D=RTT_Atype_1param}
There is an algebra isomorphism $\what{\theta}_{q}'\colon U_{q}^{D}(\what{\gl}_{n+1}) \iso U'(\bar{R}_{q}(z))$ such that
\[
  \what{\theta}_{q}'(\gamma^{1/2})=\gamma^{1/2}, \qquad
  \what{\theta}_{q}'(k_{j}^{\pm}(z)) = h_{j}^{\pm}(z),\qquad
  \what{\theta}_{q}'(X_{i}^{\pm}(z)) = \frak{X}_{i}^{\mp}(z)
\]
for all $1 \le j \le n+1$ and $1 \le i \le n$.
\end{theorem}

We shall now derive the two-parameter analogue of the above result. First, we need to introduce the two-parameter quantum affine
group of $\gl_{n+1}$ in the new Drinfeld presentation:

\begin{definition}\label{def:two-param-affine-gl}
The two-parameter quantum group of $\what{\gl}_{n+1}$ in the new Drinfeld presentation is an associative $\bb{K}$-algebra
$U_{r,s}^{D}(\what{\gl}_{n+1})$ generated by
  $\big\{x_{i,m}^{\pm},\psi_{k,\ell},\psi_{k,-\ell}',\psi_{k,0}^{-1},(\psi_{k,0}')^{-1}
   \big\}_{1 \le i \le n, 1 \le k \le n + 1}^{m \in \bb{Z},\ell \in \bb{Z}_{\ge 0}} \cup \{\gamma^{\pm 1/2},(\gamma')^{\pm 1/2}\}$,
with the following defining relations (for all $1 \le i,j \le n$ and $1 \le k,t \le n+1$):
\[
  \gamma^{1/2}\ \text{and}\ (\gamma')^{1/2}\ \text{are central}, \qquad
  \gamma^{\pm 1/2}\gamma^{\mp 1/2} = 1 = (\gamma')^{\pm 1/2}(\gamma')^{\mp 1/2},
\]
\[
  [\psi_{k}(z),\psi_{t}(w)] = [\psi_{k}'(z),\psi_{t}'(w)] = [\psi_{k}(z),\psi_{k}'(w)] = [\psi_{k,0}',\psi_{t,0}] = 0,\quad
  \psi_{k,0}^{\pm 1}\psi_{k,0}^{\mp 1} = 1 = (\psi_{k,0}')^{\pm 1}(\psi_{k,0}')^{\mp 1},
\]
\[
  \frac{1 - wz^{-1}\gamma(\gamma\gamma')^{1/2}}{1 - wz^{-1}\gamma(\gamma\gamma')^{1/2}r^{-1}s}\psi_{k}(z)\psi_{t}'(w)
  = \psi_{t}'(w)\psi_{k}(z)\frac{1 - wz^{-1}\gamma'(\gamma\gamma')^{1/2}}{1 - wz^{-1}\gamma'(\gamma\gamma')^{1/2}r^{-1}s}
  \qquad \text{if}\qquad k < t,
\]
\[
  \frac{1-zw^{-1}\gamma(\gamma\gamma')^{1/2}}{1-zw^{-1}\gamma(\gamma\gamma')^{1/2}rs^{-1}}\psi_{k}'(z)\psi_{t}(w)
  = \psi_{t}(w)\psi_{k}'(z)\frac{1-zw^{-1}\gamma'(\gamma\gamma')^{1/2}}{1-zw^{-1}\gamma'(\gamma\gamma')^{1/2}rs^{-1}}
  \qquad \text{if}\qquad k < t,
\]
\[
  [\psi_{k}'(z),x_{i}^{\pm}(w)] = [\psi_{k}(z),x_{i}^{\pm}(w)] = 0
  \quad \text{if}\quad k - i \le -1 \quad \text{or} \quad k - i \ge 2,
\]
\[
  \psi_{i}(z)x_{i}^{-}(w)
  = \frac{1 - wz^{-1}(\gamma')^{1/2}(\gamma\gamma')^{1/2}}{r - wz^{-1}(\gamma')^{1/2}(\gamma\gamma')^{1/2}s}x_{i}^{-}(w)\psi_{i}(z),\qquad
  \psi_{i}'(z)x_{i}^{-}(w) = \frac{1 - zw^{-1}(\gamma')^{1/2}}{s - zw^{-1}(\gamma')^{1/2}r}x_{i}^{-}(w)\psi_{i}'(z),
\]
\[
  \psi_{i+1}(z)x_{i}^{-}(w)
  = \frac{1 - wz^{-1}(\gamma')^{1/2}(\gamma\gamma')^{1/2}}{s - wz^{-1}(\gamma')^{1/2}(\gamma\gamma')^{1/2}r}x_{i}^{-}(w)\psi_{i+1}(z), \ \
  \psi_{i+1}'(z)x_{i}^{-}(w) = \frac{1 - zw^{-1}(\gamma')^{1/2}}{r - zw^{-1}(\gamma')^{1/2}s}x_{i}^{-}(w)\psi_{i+1}'(z),
\]
\[
  \psi_{i}(z)x_{i}^{+}(w) = \frac{r - wz^{-1}\gamma^{1/2}s}{1 - wz^{-1}\gamma^{1/2}}x_{i}^{+}(w)\psi_{i}(z), \qquad
  \psi_{i}'(z)x_{i}^{+}(w) = \frac{s - zw^{-1}\gamma^{1/2}(\gamma\gamma')^{1/2}r}{1 - zw^{-1}\gamma^{1/2}(\gamma\gamma')^{1/2}}x_{i}^{+}(w)\psi_{i}'(z),
\]
\[
  \psi_{i+1}(z)x_{i}^{+}(w) = \frac{s - wz^{-1}\gamma^{1/2}r}{1 - wz^{-1}\gamma^{1/2}}x_{i}^{+}(w)\psi_{i+1}(z),\qquad
  \psi_{i+1}'(z)x_{i}^{+}(w)
  = \frac{r - zw^{-1}\gamma^{1/2}(\gamma\gamma')^{1/2}s}{1 - zw^{-1}\gamma^{1/2}(\gamma\gamma')^{1/2}}x_{i}^{+}(w)\psi_{i+1}'(z),
\]
\[
  (zr^{\pm 1}s^{\mp 1} - w)x_{i}^{\pm}(z)x_{i}^{\pm}(w) = (z - wr^{\pm 1}s^{\mp 1})x_{i}^{\pm}(w)x_{i}^{\pm}(z),
\]
\[
  (z - w)x_{i}^{+}(z)x_{i+1}^{+}(w) = (zr - ws)x_{i+1}^{+}(w)x_{i}^{+}(z), \qquad
  (z - w)x_{i+1}^{-}(w)x_{i}^{-}(z) = (zr - ws)x_{i}^{-}(z)x_{i+1}^{-}(w),
\]
as well as
\begin{multline*}
  [x_{i}^{+}(z), x_{j}^{-}(w)] \\
  = (s^{-1} - r^{-1})\delta_{ij}
   \left( \delta\left (\frac{z}{w}\gamma^{-1}\right )\psi_{i+1}'(w(\gamma')^{-1/2})\psi_{i}'(w(\gamma')^{-1/2})^{-1}
          - \delta\left (\frac{z}{w}(\gamma')^{-1}\right )\psi_{i+1}(z\gamma^{1/2})\psi_{i}(z\gamma^{1/2})^{-1} \right),
\end{multline*}
and quantum Serre relations~\eqref{eq:D9+}--\eqref{eq:D9-}, where the generating series $x^{\pm}_i(z), \psi_k(z), \psi'_k(z)$ are defined via
\[
  x_{i}^{\pm}(z) = \sum_{m \in \bb{Z}} x_{i,m}^{\pm}z^{-m},\qquad
  \psi_{k}(z) = \sum_{\ell \geq 0} \psi_{k,\ell}z^{-\ell},\qquad
  \psi_{k}'(z) = \sum_{\ell\geq 0} \psi_{k,-\ell}'z^{\ell},
\]
and $\delta(z) = \sum_{m \in \bb{Z}} z^{m}$ denotes the formal delta function.
\end{definition}

\begin{remark}\label{rem:two-vs-one_affine_A}
There is a surjective algebra homomorphism $U_{q,q^{-1}}^{D}(\what{\gl}_{n+1}) \twoheadrightarrow U_{q}^{D}(\what{\gl}_{n+1})$
determined by
\[
  \psi_{t}(z) \mapsto k_{t}^{-}(z),\quad \psi_{t}'(z) \mapsto k_{t}^{+}(z),\quad x_{i}^{\pm}(z) \mapsto X_{i}^{\pm}(z),\quad
   \gamma^{1/2} \mapsto \gamma^{1/2},\quad (\gamma')^{1/2} \mapsto \gamma^{-1/2}.
\]
\end{remark}

Let $\what{P} = P \oplus \bb{Z}\frac{\delta}{2}$. Then we have the following result (the proof is similar to that of Lemma~\ref{lem:Q_hat_bigrading}):

\begin{lemma}\label{lem:gl_n_bigrading}
The following assignment defines a $\what{P} \times \what{P}$-grading on $U_{r,s}^{D}(\what{\gl}_{n+1})$:
\[
  \deg(\gamma^{1/2}) = \deg((\gamma')^{1/2}) = \left (-\frac{1}{2}\delta,\frac{1}{2}\delta\right ),
\]
\[
  \deg(\psi_{k,\ell}) = \left (-\varepsilon_{k} - \frac{3\ell}{2}\delta,\varepsilon_{k} + \frac{\ell}{2}\delta\right ),\qquad
  \deg(\psi_{k,-\ell}') = \left (-\varepsilon_{k} -\frac{\ell}{2}\delta,\varepsilon_{k} + \frac{3\ell}{2}\delta\right ),
\]
\[
  \deg(x_{i,m}^{+}) = (\alpha_{i} - m\delta,0),\qquad \deg(x_{i,m}^{-}) = (0,-\alpha_{i}-m\delta),
\]
for any $1 \le k \le n + 1$, $1 \le i \le n$, $\ell \in \bb{Z}_{\ge 0}$, and $m \in \bb{Z}$.
\end{lemma}

\begin{remark}\label{rem:gln-sln_2param_affineA}
There is an algebra homomorphism $\hat{\iota}_{r,s} \colon U_{r,s}^{D}(\what{\ssl}_{n+1}) \to U_{r,s}^{D}(\what{\gl}_{n+1})$ defined by
\begin{align*}
  &\hat{\iota}_{r,s}(\gamma^{1/2}) = (\gamma')^{1/2}, &
  &\hat{\iota}_{r,s}((\gamma')^{1/2}) = \gamma^{1/2}, \\
  &\hat{\iota}_{r,s}(x_{i}^{+}(z)) = \frac{(rs)^{1/2}}{r - s}x_{i}^{+}(z^{-1}r^{i/2}s^{-i/2}), &
  &\hat{\iota}_{r,s}(x_{i}^{-}(z)) = \frac{(rs)^{1/2}}{r - s}x_{i}^{-}(z^{-1}r^{i/2}s^{-i/2}), \\
  &\hat{\iota}_{r,s}(\omega_{i}(z)) = \psi_{i+1}'(z^{-1}r^{i/2}s^{-i/2})\psi_{i}'(z^{-1}r^{i/2}s^{-i/2})^{-1}, &
  &\hat{\iota}_{r,s}(\omega_{i}'(z)) = \psi_{i+1}(z^{-1}r^{i/2}s^{-i/2})\psi_{i}(z^{-1}r^{i/2}s^{-i/2})^{-1},
\end{align*}
for $1 \le i \le n$. We note that $\hat{\iota}_{r,s}$ intertwines the $\what{P} \times \what{P}$-gradings of
Lemmas~\ref{lem:Q_hat_bigrading} and~\ref{lem:gl_n_bigrading}. In fact, $\hat{\iota}_{r,s}$ is injective, due to
Remark~\ref{rem:gl_n-sl_n_twist_compatibility}.
\end{remark}

Using the $P \times P$-grading on $U_{r,s}^{D}(\what{\gl}_{n+1})$ induced from Lemma~\ref{lem:gl_n_bigrading} via the quotient map
$\what{P} \to \what{P}/\bb{Z}\frac{\delta}{2} = P$, along with the bicharacter $\zeta$ of~\eqref{eq:zeta_extension_A}, we can define the algebra
$U^D_{q,\zeta}(\what{\gl}_{n+1})=(U^D_{q,q^{-1}}(\what{\gl}_{n+1}))_\zeta$, which has multiplication given by~\eqref{eq:twisted_product_general}.
Then we have the following result (the proof is similar to that of Theorem~\ref{thm:twisted_algebra_loop}):

\begin{prop}\label{prop:twisted_algebra_affineA}
There is an algebra isomorphism $\wtd{\varphi}_D\colon U^D_{r,s}(\what{\gl}_{n+1}) \iso U^D_{q,\zeta}(\what{\gl}_{n+1})$ given by
\[
  \wtd{\varphi}_D(\gamma^{1/2}) = \gamma^{1/2},\qquad \wtd{\varphi}_D((\gamma')^{1/2}) = (\gamma')^{1/2},
\]
\[
  \wtd{\varphi}_D(\psi_{t,\ell}) = \psi_{t,\ell},\qquad \wtd{\varphi}_D(\psi_{t,-\ell}') = \psi_{t,-\ell}',\qquad
  \wtd{\varphi}_D(x_{j,m}^{+}) = x_{j,m}^{+},\qquad \wtd{\varphi}_D(x_{j,m}^{-}) = (rs)^{-1/2}x_{j,m}^{-}
\]
for all $1 \le t \le n + 1$, $1 \le j \le n$, $\ell \in \bb{Z}_{\ge 0}$, and $m \in \bb{Z}$.
\end{prop}

We also have the following analogue of Proposition~\ref{prop:double_cartan_D} for $\what{\frak{gl}}_{n+1}$, with a nearly identical proof:

\begin{prop}\label{prop:double_cartan_gl_D}
Let $L_{\delta}, L_{\varepsilon_{1}}, \ldots, L_{\varepsilon_{n+1}}$ be algebraically independent invertible transcendental elements.
Then there is an injective algebra homomorphism
  $\eta_D \colon U_{q,q^{-1}}^{D}(\what{\gl}_{n+1}) \to
   U_{q}^{D}(\what{\gl}_{n+1}) \otimes \bb{K}[L_{\delta}^{\pm 1},L_{\varepsilon_{1}}^{\pm 1},\ldots ,L_{\varepsilon_{n+1}}^{\pm 1}]$
such that
\[
  \eta_D(\gamma^{1/2}) = \gamma^{1/2} \otimes L_{\delta}^{-1},\qquad
  \eta_D((\gamma')^{1/2}) = \gamma^{-1/2} \otimes L_{\delta}^{-1},
\]
\[
  \eta_D(\psi_{t,\ell}) = k_{t,\ell}^{-} \otimes L_{-2\varepsilon_{t} - 2\ell \delta},\qquad
  \eta_D(\psi_{t,-\ell}') = k_{t,-\ell}^{+} \otimes L_{-2\varepsilon_{t} - 2\ell\delta},\qquad
  \eta_D(x_{i,m}^{\pm}) = X_{i,m}^{\pm} \otimes L_{\varepsilon_{i} - \varepsilon_{i+1} \mp m\delta},
\]
for all $1 \le t \le n + 1$, $1 \le i \le n$, $\ell \in \bb{Z}_{\ge 0}$, and $m \in \bb{Z}$.
\end{prop}

\begin{remark}\label{rem:gl_n-sl_n_twist_compatibility}
Let $\hat{\iota}_{q}\colon U_{q}^{D}(\what{\ssl}_{n+1}) \to U_{q}^{D}(\what{\gl}_{n+1})$ denote the embedding
from~\cite[Definition 3.1]{DF}\footnote{In fact, their embedding should be corrected as follows:
$X_{i}^{\pm}(z) \mapsto (q - q^{-1})^{-1}X_{i}^{\mp}(zq^{i})$, $\Omega_{i}^{\pm}(z) \mapsto k_{i+1}^{\mp}(zq^{i})k_{i}^{\mp}(zq^{i})^{-1}$.}.
It is easy to see that it is compatible with $\hat{\iota}_{q,q^{-1}}$ of Remark~\ref{rem:gln-sln_2param_affineA} through the embeddings of
Propositions~\ref{prop:double_cartan_D} and~\ref{prop:double_cartan_gl_D}, thus implying injectivity of $\hat{\iota}_{q,q^{-1}}$. Since we
also have $\wtd{\varphi}_{D} \circ \hat{\iota}_{r,s} = \hat{\iota}_{q,q^{-1}} \circ \what{\omega} \circ \varphi_{D}$, with $\hat{\iota}_{r,s}$
of Remark~\ref{rem:gln-sln_2param_affineA}, $\varphi_{D}$ of Proposition~\ref{thm:twisted_algebra_loop}, and $\what{\omega}$ of
Lemma~\ref{lem:affine_cartan_involution}, we thus conclude that $\hat{\iota}_{r,s}$ is actually injective.
\end{remark}

On the other hand, one can also check that the assignments of Lemma~\ref{lem:affine_RTT_bigradings}(a)
(resp.~Lemma~\ref{lem:affine_RTT_bigradings}(b)) make both $U(\bar{R}_{q}(z))$ and $U(\bar{R}_{r,s}(z))$ into $\what{P}$-bigraded
(resp.\ $P$-bigraded) Hopf algebras, where
\[
  \bar{R}_{r,s}(z) = \frac{1}{1 - zrs^{-1}}R_{r,s}(z) =\frac{1}{1 - zrs^{-1}}\hat{R}_{r,s}(z) \circ \tau,
\]
with $\hat{R}_{r,s}(z)$ of~\eqref{eq:A_AffRMatrix}. We can thus form the algebra $U(\bar{R}_{q}(z))_{\zeta}$, and due
to~\eqref{eq:R_rs_twisted_affine}, we can apply Proposition~\ref{prop:U(R)_twisted_affine} (with $U(\bar{R}_{q}(z)), U(\bar{R}_{r,s}(z))$
instead of $\wtd{U}(\wtd{R}_{q}(z)), \wtd{U}(\wtd{R}_{r,s}(z))$, respectively). We also note that Proposition~\ref{prop:double_cartan_RTT_affine}
still holds in this case (with $U(\bar{R}_{q}(z)),U'(\bar{R}_{q}(z))$ instead of $\wtd{U}(\wtd{R}_{q}(z)),\wtd{U}'(\wtd{R}_{q}(z))$, respectively).
Let $\wtd{e}_{ij}^{\pm}(z)$, $\wtd{h}_{i}^{\pm}(z)$,
$\wtd{f}_{ji}^{\pm}(z)$ denote the Gaussian generators of $L^{\pm}(z)$ in $U(\bar{R}_{r,s}(z))$, and define $\wtd{\frak{X}}_{i}^{\pm}(z)$ as
in~\eqref{eq:X_series_double_cartan}, but with $\wtd{e}_{i,i+1}^{\pm}(z)$ and $\wtd{f}_{i+1,i}^{\pm}(z)$ in place of $e_{i,i+1}^{\pm}(z)$
and $f_{i+1,i}^{\pm}(z)$. Then, following the methods of Section~\ref{sec:affine_RTT}, we obtain:

\begin{theorem}\label{thm:D=RTT_Atype_2param}
There is an algebra isomorphism $\what{\theta}_{r,s}\colon U_{r,s}^{D}(\what{\gl}_{n+1}) \iso U(\bar{R}_{r,s}(z))$ such that
\begin{align*}
  &\what{\theta}_{r,s}(\gamma^{1/2}) = \gamma^{1/2}, &
  &\what{\theta}_{r,s}((\gamma')^{1/2}) = (\gamma')^{1/2}, \\
  &\what{\theta}_{r,s}(\psi_{t}(z)) = \wtd{h}_{t}^{-}(z), &
  &\what{\theta}_{r,s}(\psi_{t}'(z)) = \wtd{h}_{t}^{+}(z) & & \text{for} \quad 1 \le t \le n + 1, \\
  &\what{\theta}_{r,s}(x_{i}^{\pm}(z)) = (rs)^{-1/2}\wtd{\frak{X}}_{i}^{\mp}(z) & & ~ & & \text{for} \quad 1 \le i \le n.
\end{align*}
\end{theorem}

Combining this result with Remark~\ref{rem:gln-sln_2param_affineA}, we immediately get:

\begin{cor}\label{cor:Dr-vs-RTT_affine_sl}
There is an algebra embedding $\what{\theta}_{r,s}\colon U_{r,s}^{D}(\what{\ssl}_{n+1}) \hookrightarrow U(\bar{R}_{r,s}(z))$ such that for all
$1\leq i\leq n$:
\begin{align*}
  &\what{\theta}_{r,s}(\gamma^{1/2}) = (\gamma')^{1/2}, \qquad \what{\theta}_{r,s}((\gamma')^{1/2}) = \gamma^{1/2}, \\
  &\what{\theta}_{r,s}(\omega_{i}(z)) = \wtd{h}_{i+1}^{+}(z^{-1}r^{i/2}s^{-i/2})\wtd{h}_{i}^{+}(z^{-1}r^{i/2}s^{-i/2})^{-1}, \\
  &\what{\theta}_{r,s}(\omega_{i}'(z)) = \wtd{h}_{i+1}^{-}(z^{-1}r^{i/2}s^{-i/2})\wtd{h}_{i}^{-}(z^{-1}r^{i/2}s^{-i/2})^{-1}, \\
  &\what{\theta}_{r,s}(x_{i}^{\pm}(z)) = \frac{1}{r - s}\wtd{\frak{X}}_{i}^{\mp}(z^{-1}r^{i/2}s^{-i/2}).
\end{align*}
\end{cor}

\begin{remark}
(a) In the standard one-parameter setup, it is also known that (an extension of) $U^D_q(\what{\ssl}_{n+1})$ can be realized as a quotient of
$U^D_q(\what{\gl}_{n+1})$, which we recall now. Let $\frak{z}_{q}^{\pm}(z) = \prod_{t=1}^{n+1} k_{t}^{\pm}(zq^{2t}) \in U_{q}^{D}(\what{\gl}_{n+1})$.
Then we have the following equalities:
\[
  [\frak{z}_{q}^{\pm}(z),X_{i}^{\pm}(w)] = [\frak{z}_{q}^{\pm}(z),X_{i}^{\mp}(w)] = [\frak{z}_{q}^{\pm}(z),k_{t}^{\pm}(w)] = 0,
\]
\[
  \frak{z}_{q}^{+}(z)k_{t}^{-}(w) = k_{t}^{-}(w)\frak{z}_{q}^{+}(z)\cdot
  \frac{1-zw^{-1}\gamma^{-1} q^{2}}{1 - zw^{-1}\gamma^{-1}q^{2(n + 1)}}\cdot
  \frac{1 - zw^{-1}\gamma q^{2(n + 1)}}{1-zw^{-1}\gamma q^{2}},
\]
\[
  \frak{z}_{q}^{-}(z)k_{t}^{+}(w) = k_{t}^{+}(w)\frak{z}_{q}^{-}(z)\cdot
  \frac{1-wz^{-1}\gamma^{-1} q^{-2}}{1 - wz^{-1}\gamma^{-1}q^{-2(n + 1)}}\cdot
  \frac{1 - wz^{-1}\gamma q^{-2(n + 1)}}{1-wz^{-1}\gamma q^{-2}}
\]
for any $1\leq i\leq n$ and $1\leq t\leq n+1$. As a consequence of these relations, we see that the coefficients of $\frak{z}_{q}^{\pm}(z)$
commute with the elements of $\mathrm{Im}(\hat{\iota}_q)$ as well as satisfy the following relation:
\begin{equation}\label{eq:almost_commutativity}
  \frak{z}_{q}^{+}(z) \frak{z}_{q}^{-}(w) = \frak{z}_{q}^{-}(w) \frak{z}_{q}^{+}(z)\cdot \frac{G(zw^{-1}\gamma)}{G(zw^{-1}\gamma^{-1})}
  \qquad \mathrm{with} \qquad
  G(u)=\prod_{t=1}^{n+1} \frac{1-uq^{2(n+1-t)}}{1-uq^{2(1-t)}}.
\end{equation}
Let $f(u)$ be the unique formal series with the free term $1$ that satisfies the following difference equation:
\begin{equation}\label{eq:difference_eq1}
  \prod_{1\leq t\leq n+1} f(uq^{2t})=\frac{1-uq^{2(n+1)}}{1-uq^2},
\end{equation}
which also implies
\begin{equation}\label{eq:difference_eq2}
  \prod_{1\leq k,t\leq n+1} f(uq^{2k-2t})=G(u).
\end{equation}
In analogy with~\eqref{eq:f_prefactor}, see also Remark~\ref{rem:normalization-factor}, we consider the normalized $R$-matrix
\begin{equation}\label{eq:f_prefactor_A}
  \wtd{R}_q(z)=f(z)\bar{R}_q(z).
\end{equation}
We shall also consider the algebra $U^{'D}_{q}(\what{\gl}_{n+1})$ defined similarly to $U^{D}_{q}(\what{\gl}_{n+1})$, but with the commutation
relation between $k^+_i(z)$ and $k^-_j(w)$ modified by $\frac{f(zw^{-1}\gamma^{-1})}{f(zw^{-1}\gamma)}$ for all $i,j$, cf.~\cite[(2.4)]{JLM2}.
Then, in analogy with Theorem~\ref{thm:D=RTT_Atype_1param}, we have an algebra isomorphism $U_{q}^{'D}(\what{\gl}_{n+1}) \iso U'(\wtd{R}_{q}(z))$
given by the same formulas. However, combining~\eqref{eq:almost_commutativity} and~\eqref{eq:difference_eq2}, we get
$[\wtd{\frak{z}}_{q}^{+}(z),\wtd{\frak{z}}_{q}^{-}(w)]=0$ with
$\wtd{\frak{z}}_{q}^{\pm}(z) = \prod_{t=1}^{n+1} k_{t}^{\pm}(zq^{2t}) \in U_{q}^{'D}(\what{\gl}_{n+1})$.
Let $\wtd{U}'(\wtd{R}_{q}(z))$ denote the quotient of $U'(\wtd{R}_{q}(z))$ by the relations $\prod_{t = 1}^{n+1}h_{t}^{\pm}(zq^{2t}) = 1$.
Then, evoking the embedding $U_{q}^{D}(\what{\ssl}_{n+1}) \hookrightarrow U_{q}^{'D}(\what{\gl}_{n+1})$ analogous
to that of Remark~\ref{rem:gl_n-sl_n_twist_compatibility}, we obtain an algebra isomorphism
$U^D_{q}(\what{\ssl}_{n+1}) \iso \wtd{U}'(\wtd{R}_{q}(z))$ given by $r=q=s^{-1}$-specialization of the formulas from
Corollary~\ref{cor:Dr-vs-RTT_affine_sl}, where we extend
$U^D_{q}(\what{\ssl}_{n+1})$ by adjoining $(\omega_{1,0}^{n} \omega_{2,0}^{n-1} \ldots \omega_{n,0})^{\frac{1}{n+1}}$ and
$((\omega'_{1,0})^{n} (\omega'_{2,0})^{n-1} \ldots \omega'_{n,0})^{\frac{1}{n+1}}$, as in Remark~\ref{rem:sl-quotient-gl_finite}.

\medskip
\noindent
(b) The above can be naturally generalized to the two-parameter setup, where we define
\[
  \frak{z}_{r,s}(z) = \prod_{1\leq t\leq n+1}\psi_{t}(zr^{t}s^{-t}),\qquad
  \frak{z}_{r,s}'(z) = \prod_{1\leq t\leq n+1}\psi_{t}'(zr^{t}s^{-t}).
\]
Then we have the following relations:
\[
  \frak{z}_{r,s}(z) x_{i}^{\pm}(w) = (rs)^{\pm 1} x_{i}^{\pm}(w)\frak{z}_{r,s}(z),\qquad
  \frak{z}_{r,s}'(z) x_{i}^{\pm}(w) = (rs)^{\pm 1} x_{i}^{\pm}(w)\frak{z}_{r,s}'(z),
\]
\[
  \frak{z}_{r,s}'(z)\psi_{t}(w) = \psi_{t}(w)\frak{z}_{r,s}'(z)
  \cdot \frac{1 - zw^{-1}\gamma'(\gamma\gamma')^{1/2}rs^{-1}}{1 - zw^{-1}\gamma'(\gamma\gamma')^{1/2}r^{n + 1}s^{-n-1}}
  \cdot \frac{1 - zw^{-1}\gamma(\gamma\gamma')^{1/2}r^{n+1}s^{-n-1}}{1 - zw^{-1}\gamma(\gamma\gamma')^{1/2}rs^{-1}},
\]
\[
  \frak{z}_{r,s}(z)\psi'_{t}(w) = \psi'_{t}(w)\frak{z}_{r,s}(z)
  \cdot \frac{1 - wz^{-1}\gamma'(\gamma\gamma')^{1/2}r^{-1}s}{1 - wz^{-1}\gamma'(\gamma\gamma')^{1/2}r^{-n - 1}s^{n+1}}
  \cdot \frac{1 - wz^{-1}\gamma(\gamma\gamma')^{1/2}r^{-n-1}s^{n+1}}{1 - wz^{-1}\gamma(\gamma\gamma')^{1/2}r^{-1}s}
\]
for all $1\leq i\leq n$ and $1\leq t\leq n+1$. In particular, similarly to~\eqref{eq:almost_commutativity}, we obtain:
\begin{equation}\label{eq:almost_commutativity_2param}
  \frak{z}'_{r,s}(z) \frak{z}_{r,s}(w) = \frak{z}_{r,s}(w) \frak{z}'_{r,s}(z) \cdot
  \frac{G(zw^{-1}\gamma(\gamma\gamma')^{1/2})}{G(zw^{-1}\gamma'(\gamma\gamma')^{1/2})}
\end{equation}
with the same $G(u)$ where $q=r^{1/2}s^{-1/2}$. Thus, we choose $f(u)$ to satisfy the same difference equation~\eqref{eq:difference_eq1}.
In analogy with part~(a), let us consider an algebra $\mathsf{U}^{'D}_{r,s}(\what{\gl}_{n+1})$
defined similarly to Definition~\ref{def:two-param-affine-gl}, but with the commutation relation between $\psi'_i(z)$ and $\psi_j(w)$
modified by $\frac{f(zw^{-1}\gamma'(\gamma\gamma')^{1/2})}{f(zw^{-1}\gamma(\gamma\gamma')^{1/2})}$ for all $1\leq i,j\leq n+1$ and with
the commutation relation between $\psi_{t}(z), \psi'_{t}(z)$ and $x^\pm_i(w)$ modified by $\lambda^{-1}=(rs)^{-\frac{1}{n+1}}$ (cf.\ definition
of $\mathsf{U}_{r,s}(\gl_{n+1})$ in Remark~\ref{rem:sl-quotient-gl_finite}(b)). Then, the assignment of Remark~\ref{rem:gln-sln_2param_affineA}
still defines an algebra embedding $\hat{\iota}_{r,s} \colon U_{r,s}^{D}(\what{\ssl}_{n+1}) \to \mathsf{U}^{'D}_{r,s}(\what{\gl}_{n+1})$,
giving rise to an algebra isomorphism
\begin{equation}\label{eq:2param_quotient_affineA}
  U^D_{r,s}(\what{\ssl}_{n+1}) \iso \mathsf{U}^{'D}_{r,s}(\what{\gl}_{n+1})/(\frak{z}_{r,s}(z) - 1, \frak{z}'_{r,s}(z) - 1),
\end{equation}
where we extended $U^D_{r,s}(\what{\ssl}_{n+1})$ by adjoining $(\omega_{1,0}^{n} \omega_{2,0}^{n-1} \ldots \omega_{n,0})^{\frac{1}{n+1}}$
and $((\omega'_{1,0})^{n} (\omega'_{2,0})^{n-1} \ldots \omega'_{n,0})^{\frac{1}{n+1}}$ again.

Now, as in Remark~\ref{rem:sl-quotient-gl_finite}(b), set
\[
  \mathsf{R}_{r,s}(z) = S^{-1}\wtd{D}R_{q}(z)\wtd{D}S, \qquad \text{where}\qquad
  S = \sum_{i,j=1}^{n + 1}\psi_{i}\psi_{j}E_{ii} \otimes E_{jj},\
  D = \sum_{i,j=1}^{n +1}\wtd{\zeta}(\varepsilon_{j},\varepsilon_{i})E_{ii} \otimes E_{jj},
\]
with $\psi_{i}$ given by~\eqref{eq:psi_values_A} and $\wtd{\zeta}$ given by~\eqref{eq:zeta_modified_A}, cf.~\eqref{eq:zeta_extension_A}.
If we further set (cf.~\eqref{eq:f_prefactor_A})
\[
  \wtd{\mathsf{R}}_{r,s}(z) = \frac{f(z)}{1 - zrs^{-1}}\mathsf{R}_{r,s}(z),
\]
and let $\wtd{U}(\wtd{\mathsf{R}}_{r,s}(z))$ denote the quotient of $U(\wtd{\mathsf{R}}_{r,s}(z))$ by the relations
$\prod_{t = 1}^{n+1}\wtd{h}_{t}^{\pm}(zr^{t}s^{-t}) = 1$, then we obtain an algebra isomorphism
$\what{\theta}_{r,s}'\colon U_{r,s}(\ssl_{n+1}) \iso \wtd{U}(\wtd{\mathsf{R}}_{r,s}(z))$ such that for all $1 \le i \le n$:
\begin{align*}
  &\what{\theta}_{r,s}'(\gamma^{1/2}) = (\gamma')^{1/2}, \qquad \what{\theta}_{r,s}'((\gamma')^{1/2}) = \gamma^{1/2}, \\
  &\what{\theta}_{r,s}'(\omega_{i}(z)) = \wtd{h}_{i+1}^{+}(z^{-1}r^{i/2}s^{-i/2})\wtd{h}_{i}^{+}(z^{-1}r^{i/2}s^{-i/2})^{-1}, \\
  &\what{\theta}_{r,s}'(\omega_{i}'(z)) = \wtd{h}_{i+1}^{-}(z^{-1}r^{i/2}s^{-i/2})\wtd{h}_{i}^{-}(z^{-1}r^{i/2}s^{-i/2})^{-1}, \\
  &\what{\theta}_{r,s}'(x_{i}^{\pm}(z)) = \frac{\lambda^{1/2}}{r - s}\wtd{\frak{X}}_{i}^{\mp}(z^{-1}r^{i/2}s^{-i/2}).
\end{align*}
\end{remark}

\begin{remark}
Let us conclude by noting that precisely the same arguments also apply to the two-parameter quantum (affine) superalgebras. In the particular
case of $\widehat{\gl}(m|n)$ this immediately yields and upgrades the main result of the recent note~\cite{HJZ2}. Furthermore, in the
orthosymplectic case $\mathfrak{osp}(V)$ (with any underlying parity sequence) this yields super-analogues of the main results from
Sections~\ref{sec:twisted_R_matrices},~\ref{sec:FRT construction finite} and~\ref{sec:affine_RTT} accordingly (cf.~\cite{HT}).
\end{remark}

   %%%%%%%%%%%%%%%%%%%%%%%%%%%%%%%%%%%%%%%%%%%%%%%%%%%%%%%%%%%%%%%%%%%%%%%%%%
   %%%%%%%%%%%%%%%%%%%%%%%%%%%%%%%%%%%%%%%%%%%%%%%%%%%%%%%%%%%%%%%%%%%%%%%%%%
   %%%%%%%%%%%%%%%%%%%%%%%%%%%%%%%%%%%%%%%%%%%%%%%%%%%%%%%%%%%%%%%%%%%%%%%%%%

\section{Universal R-matrix factorization: finite FRT}\label{ssec:app_B}

In this Appendix, we establish an analogue of the factorization from~\cite{Z} for BD-types, that is, evaluate the
\emph{restricted $R$-matrix} $\Theta_{r,s}$ of $U_{r,s}(\fg)$ in the FRT-realization. Our presentation is slightly different though,
as we derive it by applying the twist to a similar formula for the \emph{restricted $R$-matrix} $\Theta_{q}$ of $U_{q,q^{-1}}(\fg)$.

Following Section~\ref{sec:FRT construction finite}, let $R_{q}$ be the one-parameter $R$-matrix of type $B_{n}$ or $D_{n}$,
let $N = 2n + 1$ in type $B_{n}$ and $N = 2n$ in type $D_{n}$, and let $v = q^{-2}$ in type $B_{n}$ and $v = q^{-1}$ in
type $D_{n}$. Consider the associated algebra $\wtd{U}(R_{q})$ from Subsection~\ref{ssec:algebra_U(R)}.
Generalizing~\eqref{eq:cal-EF-generators}, we define
\begin{equation}\label{eq:root_vectors_RTT}
  \cal{E}_{ij} = \ell_{ji}^{+}(\ell_{ii}^{+})^{-1},\qquad \cal{F}_{ji} = (\ell_{ii}^{-})^{-1}\ell_{ij}^{-}
  \qquad \text{for all}\qquad 1 \le i < j \le N.
\end{equation}
Note that $\deg(\cal{E}_{ij}) = (\varepsilon_{i} - \varepsilon_{j},0)$ and $\deg(\cal{F}_{ji}) = (0,\varepsilon_{j} - \varepsilon_{i})$,
so that $\cal{E}_{ij} \in \wtd{U}_{+}(R_{q})$ and $\cal{F}_{ij} \in \wtd{U}_{-}(R_{q})$ by
Lemma~\ref{lem:borel_spanning_sets}.
Let $\sigma\colon \wtd{U}_{\ge}(R_{q}) \times \wtd{U}_{\le}(R_{q}) \to \bb{K}$ be the bilinear form
induced from that of Proposition~\ref{prop:skew_pairing_descent} via the isomorphism of Proposition~\ref{prop:Drinfeld_double_RTT}.

Recall that the $P \times P$-grading on $\wtd{U}(R_{q})$ induces a $Q$-grading  on $\wtd{U}(R_{q})$ via
$\wtd{U}(R_{q})_{\mu} = \bigoplus_{\lambda_{1} + \lambda_{2} = \mu}\wtd{U}(R_{q})_{\lambda_{1},\lambda_{2}}$. It is easy to see that
the subalgebras $\wtd{U}_{\pm}(R_{q})$ inherit this $Q$-grading, and that $\wtd{U}_{+}(R_{q})_{\mu} = 0 = \wtd{U}_{-}(R_{q})_{-\mu}$
whenever $\mu \notin Q_{+} = \bigoplus_{i = 1}^{n} \bb{Z}_{\ge 0}\alpha_{i}$. We also note that these statements remain true when
$\wtd{U}_{+}(R_{q})$ and $\wtd{U}_{-}(R_{q})$ are replaced by $\wtd{U}_{\ge}(R_{q})$ and $\wtd{U}_{\le}(R_{q})$, respectively.
For any $\mu \in Q$ and $x \in \wtd{U}(R_{q})_{\mu}$, we write $|x| = \mu$.

We then have the following result:

\begin{lemma}\label{lem:degree_zero_pairing}
For any $\mu,\nu \in Q_{+}$ with $\mu \neq \nu$, we have $\sigma(x,y) = 0$ for all $x \in \wtd{U}_{\ge}(R_{q})_{\mu}$ and
$y \in \wtd{U}_{\le}(R_{q})_{-\nu}$.
\end{lemma}

\begin{proof}
It is enough to prove this when $x = \ell_{i_{1}j_{1}}^{+}\ldots \ell_{i_{m}j_{m}}^{+}$ and
$y = \ell_{k_{1}t_{1}}^{-}\ldots \ell_{k_{p}t_{p}}^{-}$, for which we proceed by induction on $m$.

For the base case $m=1$, we need to show that $\sigma(\ell_{ij}^{+},\ell_{k_{1}t_{1}}^{-}\ldots \ell_{k_{p}t_{p}}^{-}) = 0$ whenever
  $\varepsilon_{j} - \varepsilon_{i} \neq
   \varepsilon_{k_{1}} - \varepsilon_{t_{1}} + \varepsilon_{k_{2}} - \varepsilon_{t_{2}} + \ldots + \varepsilon_{k_{p}} - \varepsilon_{t_{p}}$.
We proceed to show this by induction on $p$. Suppose first that $p = 1$. Then we need to show that
$\sigma(\ell_{ij}^{+},\ell_{kt}^{-}) = (R_{q})^{ik}_{jt} = 0$ whenever
$\varepsilon_{j} - \varepsilon_{i} \neq \varepsilon_{k} - \varepsilon_{t}$.
The latter holds as $(R_{q})^{ik}_{jt} \neq 0$ only when $\{i,k\} = \{j,t\}$ or $i = k'$, $j = t'$. Suppose now that $p > 1$; then
  $\sigma(\ell_{ij}^{+},\ell_{k_{1}t_{1}}^{-}\ldots \ell_{k_{p}t_{p}}^{-}) = \sum_{r = 1}^{N}
   \sigma(\ell_{rj}^{+},\ell_{k_{1}t_{1}}^{-}\ldots \ell_{k_{p-1}t_{p-1}}^{-})\sigma(\ell_{ir}^{+},\ell_{k_{p}t_{p}}^{-})$,
and by induction assumption we have $\sigma(\ell_{rj}^{+},\ell_{k_{1}t_{1}}^{-}\ldots \ell_{k_{p-1}t_{p-1}}^{-}) = 0$ unless
$\varepsilon_{j} - \varepsilon_{r} = \varepsilon_{k_{1}} - \varepsilon_{t_{1}} + \ldots + \varepsilon_{k_{p-1}} - \varepsilon_{t_{p-1}}$
and $\sigma(\ell_{ir}^{+},\ell_{k_{p}t_{p}}^{-}) = 0$ unless
$\varepsilon_{r} - \varepsilon_{i} = \varepsilon_{k_{p}} - \varepsilon_{t_{p}}$, which implies that
$\sigma(\ell_{ij}^{+},\ell_{k_{1}t_{1}}^{-}\ldots \ell_{k_{p}t_{p}}^{-}) = 0$ whenever
  $\varepsilon_{k_{1}} - \varepsilon_{t_{1}} + \ldots + \varepsilon_{k_{p}} - \varepsilon_{t_{p}} \neq
   \varepsilon_{j} - \varepsilon_{r} + \varepsilon_{r} - \varepsilon_{i} = \varepsilon_{j} - \varepsilon_{i}$,
as desired.

Finally, assume that $m > 1$. Then we have
\begin{equation}\label{eq:sigma_split}
  \sigma(\ell_{i_{1}j_{1}}^{+}\ldots \ell_{i_{m}j_{m}}^{+},y) =
  \sum_{(y)} \sigma(\ell_{i_{1}j_{1}}^{+}\ldots \ell_{i_{m-1}j_{m-1}}^{+},y_{(1)})\sigma(\ell_{i_{m}j_{m}}^{+},y_{(2)}).
\end{equation}
Because
\begin{equation}\label{eq:coproduct_grading}
  \Delta(\wtd{U}_{\leq}(R_{q})_{\mu}) \subset
  \bigoplus_{\mu_{1} + \mu_{2} = \mu} \wtd{U}_{\leq}(R_{q})_{\mu_{1}} \otimes \wtd{U}_{\leq}(R_{q})_{\mu_{2}},
\end{equation}
the decomposition $\Delta(y) = \sum_{(y)}y_{(1)} \otimes y_{(2)}$ can be regrouped so that each $y_{(1)},y_{(2)}$
is $Q$-homogeneous, and hence, using the induction hypothesis, we conclude that the only nonzero terms in the right-hand side
of~\eqref{eq:sigma_split} are those with
$-|y_{(1)}| = \varepsilon_{j_{1}} - \varepsilon_{i_{1}} + \ldots + \varepsilon_{j_{m-1}} - \varepsilon_{i_{m-1}}$ and
$-|y_{(2)}| = \varepsilon_{j_{m}} - \varepsilon_{i_{m}}$. Then using~\eqref{eq:coproduct_grading} again, we conclude
that~\eqref{eq:sigma_split} is zero unless
  $-|y| = -|y_{(1)}| - |y_{(2)}| = \varepsilon_{j_{1}} - \varepsilon_{i_{1}} + \ldots + \varepsilon_{j_{m}} - \varepsilon_{i_{m}}
   = |\ell_{i_{1}j_{1}}^{+}\ldots \ell_{i_{m}j_{m}}^{+}| = |x|$,
which completes the proof.
\end{proof}

Using the above lemma, we can now prove the following result:

\begin{lemma}\label{lem:single_generator_pairing}
Suppose that $1 \le i < j < i' \le N$, $1 \le k < t < k' \le N$. Then $\sigma(\cal{E}_{ij},\cal{F}_{tk}) = 0$ unless
$i = k$ and $j = t$. Moreover, $\sigma(\cal{E}_{ij},\cal{F}_{ji}) = v - v^{-1}$ whenever $i < j < i'$.
\end{lemma}

\begin{proof}
It is enough to show that, under the above assumptions on $i,j,k,t$, we have
$\varepsilon_{i} - \varepsilon_{j} = \varepsilon_{k} - \varepsilon_{t}$ if and only if $i = k$ and $j = t$.
Clearly if $i = k$ and $j = t$, then this equality holds. For the converse, suppose first that $i \neq t'$ or $j \neq k'$.
Then $\varepsilon_{i} - \varepsilon_{j} = \varepsilon_{k} - \varepsilon_{t}$ is equivalent to
$\varepsilon_{i} + \varepsilon_{t} = \varepsilon_{j} + \varepsilon_{k}$, and due to the assumption that $i \neq t'$ or
$j \neq k'$, we conclude that $\{i,t\} = \{j,k\}$. Furthermore, since $i < j$ and $k < t$, we must have $i = k$ and $j = t$.
If $i = t'$ and $j = k'$, then we have $i' = t < k' = j$, contrary to our assumption.

Moreover,
\begin{align*}
  \sigma(\cal{E}_{ij},\cal{F}_{ji}) &= \sigma(\ell_{ji}^{+}(\ell_{ii}^{+})^{-1},(\ell_{ii}^{-})^{-1}\ell_{ij}^{-})
  = \sum_{i\leq p\leq j}\sigma(\ell_{ji}^{+},(\ell_{ii}^{-})^{-1}\ell_{ip}^{-})
    \sigma((\ell_{ii}^{+})^{-1},(\ell_{ii}^{-})^{-1}\ell_{pj}^{-}) \\
  &= \sum_{i \le m,p \le j}
      \sigma(\ell_{jm}^{+},\ell_{ip}^{-})\sigma(\ell_{mi}^{+},(\ell_{ii}^{-})^{-1})
      \sigma((\ell_{ii}^{+})^{-1},\ell_{pj}^{-})\sigma((\ell_{ii}^{+})^{-1},(\ell_{ii}^{-})^{-1}).
\end{align*}
According to Lemma~\ref{lem:degree_zero_pairing}, we note that the only nonzero terms of the sum above are those with $m = i$
and $p = j$. Therefore, using the fact that the $\ell_{tt}^{\pm}$ are grouplike and the properties of $\sigma$, we get
\begin{align*}
  \sigma(\cal{E}_{ij},\cal{F}_{ji}) &= \sigma(\ell_{ji}^{+},\ell_{ij}^{-})
    \sigma(\ell_{ii}^{+},(\ell_{ii}^{-})^{-1}) \sigma((\ell_{ii}^{+})^{-1},\ell_{jj}^{-}) \sigma((\ell_{ii}^{+})^{-1},(\ell_{ii}^{-})^{-1}) \\
  &= \sigma(\ell_{ji}^{+},\ell_{ij}^{-})\sigma(\ell_{ii}^{+},\ell_{jj}^{-})^{-1} \sigma(1,(\ell_{ii}^{-})^{-1})
   = (R_{q})^{ji}_{ij}((R_{q})^{ij}_{ij})^{-1}.
\end{align*}
Using the explicit formulas for $R_q = \hat{R}_{q}\circ \tau$ with $\hat{R}_{q}$ from~(\ref{eq:1_parameter_R_B},~\ref{eq:1_parameter_R_D}),
we find that when $i < j \neq i'$ we have $(R_{q})^{ji}_{ij} = v - v^{-1}$ and $(R_{q})^{ij}_{ij} = 1$. Thus, the values of
$\sigma(\cal{E}_{ij},\cal{F}_{ji})$ are as stated.
\end{proof}

Combining~\eqref{eq:U(R)_Hopf}, the equalities $\ell_{ji}^{+} = 0 = \ell_{ij}^{-}$ for $i > j$, and formula~\eqref{eq:RTT_q_commute_general},
we get the following result (cf.~notations $\ell^\pm_\mu$ preceding Lemma~\ref{lem:borel_spanning_sets}):

\begin{lemma}\label{lem:coproduct_formulas}
For all $i < j$, we have
\begin{align*}
  \Delta(\cal{E}_{ij}) &= \cal{E}_{ij} \otimes 1
    + \sum_{k = i+1}^{j-1} \cal{E}_{kj}\ell_{\varepsilon_{k} - \varepsilon_{i}}^{+} \otimes \cal{E}_{ik}
    + \ell_{\varepsilon_{j} - \varepsilon_{i}}^{+} \otimes \cal{E}_{ij}, \\
  \Delta(\cal{F}_{ji}) &= 1 \otimes \cal{F}_{ji}
    + \sum_{k = i+1}^{j-1} v\cal{F}_{ki} \otimes \cal{F}_{jk}\ell_{\varepsilon_{k} - \varepsilon_{i}}^{-}
    + \cal{F}_{ji} \otimes \ell_{\varepsilon_{j} - \varepsilon_{i}}^{-}.
\end{align*}
\end{lemma}

Finally, we also need:

\begin{lemma}\label{lem:cartan_elements}
For any $\mu,\nu \in Q$, and any $Q$-homogeneous $x \in \wtd{U}_{+}(R_{q})$, $y \in \wtd{U}_{-}(R_{q})$, we have
\[
  \sigma(x\ell_{\mu}^{+},y\ell_{\nu}^{-}) =
  \sigma(x,y)\sigma(\ell_{\mu}^{+},\ell_{\nu}^{-})\sigma(\ell_{|x|}^{+},\ell_{\nu}^{-})^{-1}\sigma(\ell_{\mu}^{+},\ell_{|y|}^{-}).
\]
\end{lemma}

\begin{proof}
From the properties of $\sigma$, we get
\[
  \sigma(x\ell_{\mu}^{+},y\ell_{\nu}^{-}) = \sum_{(y)}\sigma(x,y_{(1)}\ell_{\nu}^{-})\sigma(\ell_{\mu}^{+},y_{(2)}\ell_{\nu}^{-}) =
  \sum_{(x)(y)} \sigma(x_{(1)},\ell_{\nu}^{-})\sigma(x_{(2)},y_{(1)})\sigma(\ell_{\mu}^{+},\ell_{\nu}^{-})\sigma(\ell_{\mu}^{+},y_{(2)}).
\]
Now, $\sigma(x_{(1)},\ell_{\nu}^{-}) = 0$ unless $|x_{(1)}| = 0$, and since $\wtd{U}_{+}(R_{q})$ is generated by the $\cal{E}_{ij}$'s,
the first formula in Lemma~\ref{lem:coproduct_formulas} implies that $|x_{(1)}| = 0$ only occurs for the terms
$x_{(1)} = (\ell_{|x|}^{+})^{-1}$, $x_{(2)} = x$. Likewise,
$\sigma(\ell_{\mu}^{+},y_{(2)}) = 0$ unless $|y_{(2)}| = 0$, and since $\wtd{U}_{-}(R_{q})$ is generated by the $\cal{F}_{ji}$'s,
the second formula of Lemma~\ref{lem:coproduct_formulas} implies that $|y_{(2)}| = 0$ only occurs for the terms
$y_{(2)} = \ell_{|y|}^{-}$, $y_{(1)} = y$. Thus, the equation above reduces to
\[
  \sigma(x\ell_{\mu}^{+},y\ell_{\nu}^{-}) =
  \sigma((\ell_{|x|}^{+})^{-1},\ell_{\nu}^{-})\sigma(x,y)\sigma(\ell_{\mu}^{+},\ell_{\nu}^{-})\sigma(\ell_{\mu}^{+},\ell_{|y|}^{-}).
\]
This completes the proof as $\sigma((\ell_{|x|}^{+})^{-1},\ell_{\nu}^{-})=\sigma(\ell_{|x|}^{+},\ell_{\nu}^{-})^{-1}$ by
  $1 = \sigma(1,\ell_{\nu}^{-}) = \sigma(\ell_{|x|}^{+},\ell_{\nu}^{-})\sigma((\ell_{|x|}^{+})^{-1},\ell_{\nu}^{-})$.
\end{proof}

Now, we define
  $$\cal{B}_{+} = \big\{\cal{E}_{ij} \,|\, 1 \le i < j < i' \le N\big\}, \qquad
    \cal{B}_{-} = \big\{\cal{F}_{ji} \,|\, 1 \le i < j < i' \le N\big\}.$$
We shall order $\cal{B}_{\pm}$ via:
\begin{equation}\label{eq:root_vector_order}
  \cal{E}_{ij} < \cal{E}_{kt} \ \  \text{and} \ \  \cal{F}_{ji} < \cal{F}_{tk} \qquad
  \text{if}\ i < k,\ \text{or}\ i = k\ \text{and}\  j < t.
\end{equation}

Note that if $i < j < i'$ and $k < t < k'$, then $\cal{E}_{ij} < \cal{E}_{kt}$ (resp.\ $\cal{F}_{ji} < \cal{F}_{tk}$) if and
only if $|\cal{E}_{ij}| < |\cal{E}_{kt}|$ (resp.\ $-|\cal{F}_{ji}| < -|\cal{F}_{tk}|$) with respect to the lexicographical
order $<$ on $\Phi^{+}$, see~\eqref{eq:lyndon_order}, corresponding to the order $\alpha_1<\alpha_2<\dots<\alpha_n$ on $\Pi$,
see~\cite[\S 6--7]{MT2}. This order $<$ is known to be convex, that is for $\alpha<\beta\in \Phi^+$ with $\alpha+\beta\in \Phi^+$ we
have $\alpha<\alpha+\beta<\beta$. Thus, Lemma~\ref{lem:coproduct_formulas} implies the following result:

\begin{cor}\label{cor:coproduct_degrees} For all $i < j< i'$, we have
\begin{equation*}
  \Delta(\cal{E}_{ij}) \in \cal{E}_{ij} \otimes 1 + \ell_{\varepsilon_{j} -\varepsilon_{i}}^{+} \otimes \cal{E}_{ij}
  + \sum_{\substack{\gamma,\gamma' \in \Phi^{+} \\ \gamma + \gamma' = |\cal{E}_{ij}|}}^{\gamma < |\cal{E}_{ij}| < \gamma'}
    \wtd{U}_{\ge}(R_{q})_{\gamma'} \otimes \wtd{U}_{\ge}(R_{q})_{\gamma},
\end{equation*}
\begin{equation*}
  \Delta(\cal{F}_{ji}) \in 1 \otimes \cal{F}_{ji} + \cal{F}_{ji} \otimes \ell_{\varepsilon_{j} - \varepsilon_{i}}^{-}
  + \sum_{\substack{\gamma,\gamma' \in \Phi^{+} \\ \gamma + \gamma' = -|\cal{F}_{ji}|}}^{\gamma < -|\cal{F}_{ji}| < \gamma'}
  \wtd{U}_{\le}(R_{q})_{-\gamma} \otimes \wtd{U}_{\le}(R_{q})_{-\gamma'}.
\end{equation*}
\end{cor}

We are now ready to prove the following result:

\begin{theorem}\label{thm:orthogonal_bases}
(a) The sets
\begin{equation}\label{eq:ordered}
\begin{split}
  \Big\{ \cal{E}_{i_{1}j_{1}}\ldots \cal{E}_{i_{p}j_{p}} \,\Big|\,
         \cal{E}_{i_{r}j_{r}} \in \cal{B}_{+},\ p \ge 0,\ \cal{E}_{i_{1}j_{1}} \ge \ldots \ge \cal{E}_{i_{p}j_{p}} \Big\}, \\
  \Big\{ \cal{F}_{j_{1}i_{1}}\ldots \cal{F}_{j_{p}i_{p}}\, \Big|\,
         \cal{F}_{j_{r}i_{r}} \in \cal{B}_{-},\ p \ge 0,\ \cal{F}_{j_{1}i_{1}} \ge \ldots \ge \cal{F}_{j_{p}i_{p}} \Big\}
  \end{split}
\end{equation}
form bases for $\wtd{U}_{+}(R_{q})$ and $\wtd{U}_{-}(R_{q})$, respectively.

\medskip
\noindent
(b) Let $\frak{E} = \overset{\longleftarrow}{\underset{1 \le i < j < i' \le N}{\prod}}\cal{E}_{ij}^{n_{ij}}$
and $\frak{F} = \overset{\longleftarrow}{\underset{1 \le i < j < i' \le N}{\prod}}\cal{F}_{ji}^{m_{ij}}$ be
elements from the sets above, with $n_{ij},m_{ij} \in \bb{Z}_{\ge 0}$, where the arrow over the product sign
indicates the decreasing order. Then we have
\begin{equation}\label{eq:sigma_pairing_formula}
  \sigma(\frak{E},\frak{F})
  = \underset{1 \le i < j < i' \le N}{\prod}
  \left( \delta_{n_{ij},m_{ij}} v^{\frac{1}{2}m_{ij}(m_{ij} - 1)}[m_{ij}]_{v}!\sigma(\cal{E}_{ij},\cal{F}_{ji})^{m_{ij}} \right).
\end{equation}
\end{theorem}

\begin{proof}
We shall prove part (b) first, arguing by induction on $\sum (n_{ij} + m_{ij})$. The statement is clear if the sum is $0$ or~$1$.
Suppose $\sum (n_{ij} + m_{ij}) > 1$ and that the claim holds when the sum is smaller. Let $\alpha$ be the smallest root with respect
to the lexicographical order on $\Phi^{+}$ such that $\alpha  = |\cal{E}_{ij}|$ for some $i,j$ with $n_{ij} \neq 0$ or
$\alpha = -|\cal{F}_{ji}|$ for some $i,j$ with $m_{ij} \neq 0$. Suppose first that $\alpha = |\cal{E}_{i_{0}j_{0}}|$ with
$n_{i_0j_0}\ne 0$, and write $\frak{E} = \frak{E}'\cal{E}_{i_{0}j_{0}}$. Then we have
\begin{equation}\label{eq:pairing_split}
  \sigma(\frak{E},\frak{F}) = \sum_{(\frak{F})} \sigma(\frak{E}',\frak{F}_{(1)})\sigma(\cal{E}_{i_{0}j_{0}},\frak{F}_{(2)}).
\end{equation}
Due to Lemma~\ref{lem:degree_zero_pairing}, the only nonzero terms in the above sum are those with $|\cal{F}_{(2)}| = \alpha$.
However, since $\frak{F} \in \cal{B}^{-}$ is a product of elements $\cal{F}_{tk}$ with $-|\cal{F}_{tk}| \ge \alpha$,
Corollary~\ref{cor:coproduct_degrees} and the convexity of the lexicographical order implies that the only nonzero terms
in~\eqref{eq:pairing_split} arise from picking the $1 \otimes \cal{F}_{i_{0}j_{0}}$ term in the coproduct from one of the
$\cal{F}_{i_{0}j_{0}}$ factors in $\frak{F}$, and the $\cal{F}_{ji} \otimes \ell_{\varepsilon_{j} - \varepsilon_{i}}^{-}$
term from all other factors. Using~\eqref{eq:RTT_q_commute_general}, we get
\begin{multline*}
  \left (\overset{\longleftarrow}{\underset{(i,j) \neq (i_{0},j_{0})}{\prod}}(\cal{F}_{ji} \otimes \ell_{\varepsilon_{j} - \varepsilon_{i}}^{-})\right )(\cal{F}_{j_{0}i_{0}} \otimes \ell_{\varepsilon_{j_{0}} - \varepsilon_{i_{0}}}^{-})^{a}(1 \otimes \cal{F}_{j_{0}i_{0}})(\cal{F}_{j_{0}i_{0}} \otimes \ell_{\varepsilon_{j_{0}} - \varepsilon_{i_{0}}}^{-})^{m_{i_{0}j_{0}} - a - 1} \\
  = v^{-2a}q^{(\mu,\varepsilon_{j_{0}} - \varepsilon_{i_{0}},)}\frak{F}' \otimes \cal{F}_{j_{0}i_{0}}\ell_{\mu + (m_{i_{0}j_{0}} - 1)(\varepsilon_{j_{0}} - \varepsilon_{i_{0}})}^{-},
\end{multline*}
where $\frak{F} = \frak{F}'\cal{F}_{j_{0}i_{0}}$ and $\mu = \sum_{(i,j) \neq (i_{0},j_{0})}m_{ij}(\varepsilon_{j} - \varepsilon_{i})$.
Then, by Lemma~\ref{lem:cartan_elements} and arguments above, we get
\begin{equation}\label{eq:inductive_step}
\begin{split}
  \sigma(\frak{E},\frak{F})
  &= \sum_{a = 0}^{m_{i_{0}j_{0}} -1} v^{-2a} q^{(\mu,\varepsilon_{j_{0}} - \varepsilon_{i_{0}})}\sigma(\frak{E}',\frak{F}')
     \sigma\left (\cal{E}_{i_{0}j_{0}},\cal{F}_{j_{0}i_{0}}\ell_{\mu + (m_{i_{0}j_{0}} - 1)(\varepsilon_{j_{0}} - \varepsilon_{i_{0}})}^{-}\right ) \\
  &= v^{1 - m_{i_{0}j_{0}}}[m_{i_{0}j_{0}}]_{v} q^{(\mu,\varepsilon_{j_{0}} - \varepsilon_{i_{0}})}
     \sigma\left (\ell_{\varepsilon_{i_{0}} - \varepsilon_{j_{0}}}^{+},\ell_{\mu + (m_{i_{0}j_{0}} - 1)(\varepsilon_{j_{0}} - \varepsilon_{i_{0}})}^{-}\right )^{-1}
     \sigma(\frak{E}',\frak{F}')\sigma(\cal{E}_{i_{0}j_{0}},\cal{F}_{j_{0}i_{0}}).
\end{split}
\end{equation}
Using the explicit formulas for $R_{q} = \hat{R}_{q} \circ \tau$, we find that
$\sigma(\ell_{\lambda}^{+},\ell_{\lambda'}^{-}) = q^{-(\lambda,\lambda')}$ for all $\lambda,\lambda' \in P$, so that
  $\sigma(\ell_{\varepsilon_{i_{0}} - \varepsilon_{j_{0}}}^{+},\ell_{\mu + (m_{i_{0}j_{0}} - 1)(\varepsilon_{j_{0}} - \varepsilon_{i_{0}})}^{-})^{-1} = q^{(\varepsilon_{i_{0}} - \varepsilon_{j_{0}},\mu)}v^{2(m_{i_{0}j_{0}} - 1)}$.
Thus, combining~\eqref{eq:inductive_step} with the induction hypothesis applied to $\sigma(\frak{E}',\frak{F}')$,
we derive formula~\eqref{eq:sigma_pairing_formula} for $\sigma(\frak{E},\frak{F})$, completing the inductive step in this case.
The case $\alpha = -|\cal{F}_{j_{0}i_{0}}|$ with $m_{i_0j_0}\ne 0$ proceeds completely
analogously.

Next, we prove part (a). Part (b) implies that both of the sets in ~\eqref{eq:ordered} consist of linearly independent elements, so we
just need to show that they span $\wtd{U}_{+}(R_{q})$ and $\wtd{U}_{-}(R_{q})$, respectively. Since the elements in the sets
from~\eqref{eq:ordered} are all $Q$-homogeneous, we just need to show that those of degree $\mu$ span $\wtd{U}_{\pm}(R_{q})_{\mu}$.
This follows because, by Corollaries~\ref{cor:DJ=RTT_Btype_1param_double_cartan} and~\ref{thm:DJ=RTT_Dtype_1_param_double_cartan},
there is a $Q$-graded isomorphism $\wtd{U}(R_{q}) \iso U_{q,q^{-1}}(\fg)$ (where $\fg$ is either $\frak{so}_{2n+1}$ or $\frak{so}_{2n}$,
depending on the type of $R_{q}$), and the dimensions of the weight spaces of $U^\pm_{q}(\fg)\simeq U^\pm_{q,q^{-1}}(\fg)$
are given by Kostant's partition function (see~\cite[Proposition 5.19(a)]{J}).
\end{proof}

We shall now use the above Theorem to give a factorization of the \emph{restricted $R$-matrix} $\Theta_{q}$ of $U_{q,q^{-1}}(\fg)$
for $\fg = \frak{so}_{N}$. For the further application to the two-parameter case, we shall assume $U_{q,q^{-1}}(\fg)$ is defined
over $\bb{K}$ and that $q = r^{1/2}s^{-1/2} \in \bb{K}$. We first introduce an involutive $\bb{Q}$-algebra antiautomorphism
$\kappa \colon U_{q,q^{-1}}(\fg) \to U_{q,q^{-1}}(\fg)$ defined by
\[
  \kappa(r) = r^{-1}, \quad \kappa(s) = s^{-1}, \quad \kappa(e_{i}) = e_{i}, \quad \kappa(f_{i}) = f_{i},\quad
  \kappa(\omega_{i}) = \omega_{i},\quad \kappa(\omega_{i}') = \omega_{i}' \qquad \forall\, 1 \le i \le n.
\]
Note that $\kappa(q) = q^{-1}$. We shall also consider the restriction of $\kappa$ to $\bb{K}$, which we also call $\kappa$.
Then the Hopf pairing $(\cdot,\cdot)_{q,q^{-1}}\colon U_{q,q^{-1}}^{\le}(\fg) \times U_{q,q^{-1}}^{\ge}(\fg) \to \bb{K}$ of
Proposition~\ref{prop:pairing_2param} is related to $\sigma$ via
\begin{equation}\label{eq:pairing_comparison}
  \sigma(x,y) = \kappa\left( \left( \kappa\left (\theta_{q}^{-1}(y) \right) , \kappa\left(\theta_{q}^{-1}(x) \right) \right)_{q,q^{-1}} \right)
  \qquad \text{for all}\qquad x \in \wtd{U}_{\ge}(R_{q}),\ y \in \wtd{U}_{\le}(R_{q}),
\end{equation}
where $\theta_{q}$ is one of the isomorphisms from Corollaries~\ref{cor:DJ=RTT_Btype_1param_double_cartan}
and~\ref{thm:DJ=RTT_Dtype_1_param_double_cartan}. Indeed, it is easy to verify that the right-hand side of~\eqref{eq:pairing_comparison}
is a $\bb{K}$-bilinear form that agrees with $\sigma$ when
$x \in \{\cal{E}_{i,i+1},\ell_{ii}^{+} \,|\, 1 \le i \le n\}$ and $y \in \{\cal{F}_{i+1,i},\ell_{ii}^{-} \,|\, 1 \le i \le n\}$
in type $B_{n}$ and when $x \in \{\cal{E}_{i,i+1},\ell_{ii}^{+} \,|\, 1 \le i \le n - 1\} \cup \{\cal{E}_{n-1,n+1},\ell_{nn}^{+}\}$
and $y \in \{\cal{F}_{i + 1,i},\ell_{ii}^{-} \,|\, 1 \le i \le n - 1\}\cup \{\cal{F}_{n + 1,n-1},\ell_{nn}^{-}\}$ in type $D_{n}$.
Then since the set $\{\cal{E}_{i,i+1},\ell_{ii}^{+} \,|\, 1 \le i \le n\}$ (resp.\ $\{\cal{F}_{i+1,i},\ell_{ii}^{-} \,|\, 1 \le i \le n\}$)
generates $\wtd{U}_{\ge}(R_{q})$ (resp.\ $\wtd{U}_{\le}(R_{q})$) in type $B_{n}$ and
$\{\cal{E}_{i,i+1},\ell_{ii}^{+} \,|\, 1 \le i \le n - 1\} \cup \{\cal{E}_{n-1,n+1},\ell_{nn}^{+}\}$
(resp.\ $\{\cal{F}_{i + 1,i},\ell_{ii}^{-} \,|\, 1 \le i \le n - 1\}\cup \{\cal{F}_{n + 1,n-1},\ell_{nn}^{-}\}$)
generates $\wtd{U}_{\ge}(R_{q})$ (resp.\ $\wtd{U}_{\le}(R_{q})$) in type $D_{n}$, the general case follows from the fact that
$\theta_{q}$ is a Hopf algebra isomorphism and $\kappa$ is a $\bb{Q}$-algebra antiautomorphism that satisfies
$\Delta(\kappa(x)) = (\kappa \otimes \kappa)\Delta(x)$ for all $x$. Then, by Theorem~\ref{thm:orthogonal_bases}, the set
\[
  \mathsf{B}_{+} = \Big\{ \kappa\left( \theta_{q}^{-1}(\cal{E}_{i_{1}j_{1}}\ldots \cal{E}_{i_{p}j_{p}}) \right) \,\Big|\,
          \cal{E}_{i_{1}j_{1}} \geq \ldots \geq \cal{E}_{i_{p}j_{p}} \,, \cal{E}_{i_{r}j_{r}} \in \cal{B}_{+}  \Big\}
\]
is a basis of $U_{q,q^{-1}}^{+}(\fg)\simeq U^+_q(\fg)$, and the dual basis of $U_{q,q^{-1}}^{-}(\fg)\simeq U^-_q(\fg)$ with respect
to $(\cdot,\cdot)_{q,q^{-1}}$ is
\[
  \mathsf{B}_{-} =
  \Big\{ \left( \kappa\left (\sigma(\cal{E}_{i_{1}j_{1}}\ldots \cal{E}_{i_{p}j_{p}},\cal{F}_{j_{1}i_{1}}\ldots \cal{F}_{j_{p}i_{p}})\right ) \right)^{-1}
         \kappa\left( \theta_{q}^{-1}(\cal{F}_{j_{1}i_{1}}\ldots \cal{F}_{j_{p}i_{p}}) \right) \,\Big|\,
         \cal{F}_{j_{1}i_{1}}\geq \ldots \geq \cal{F}_{j_{p}i_{p}} \,, \cal{F}_{j_{r}i_{r}} \in \cal{B}_{-} \Big\}.
\]
To simplify notation, we set $\mathsf{E}_{ij} = \kappa\left (\theta_{q}^{-1}(\cal{E}_{ij})\right )$ and
$\mathsf{F}_{ji} = \kappa\left (\theta_{q}^{-1}(\cal{F}_{ji})\right )$. Then, combining the above discussion with
formula~\eqref{eq:sigma_pairing_formula} and Lemma~\ref{lem:single_generator_pairing}, we obtain the following result:

\begin{theorem}\label{thm:R-factorization-RTT_param1}
For $\fg = \frak{so}_{N}$, we have
\[
  \Theta_{q} \, = \overset{\longrightarrow}{\underset{1 \le i < j < i'\le N}{\prod}} \Theta_{q}^{(i,j)},
\]
with the arrow over the product sign indicating the decreasing order, where each factor equals
\begin{equation}\label{eq:Theta_ij_BD_one_param}
  \Theta_{q}^{(i,j)}
  = \sum_{m \ge 0} \frac{v^{\frac{1}{2}m(m-1)}}{[m]_{v}! \left( v^{-1} - v \right)^{m}}
    \mathsf{F}_{ji}^{m} \otimes \mathsf{E}_{ij}^{m}
  =: \exp_{v^{-1}}\left(\frac{\mathsf{F}_{ji} \otimes \mathsf{E}_{ij}}{v^{-1} - v}\right).
\end{equation}
\end{theorem}

To derive the two-parameter version of the above factorization, we first recall the Hopf algebra isomorphism
$\varphi\colon U_{r,s}(\fg) \iso U_{q,\zeta}(\fg)$ from Proposition~\ref{prop:twisted_algebra}. Note that we have
$\kappa(\varphi(x)) = \varphi(\tau(\bar{x}))$ for all $x \in U_{r,s}^{\pm}(\frak{g})$, where $\tau$ (resp.\ $x \mapsto \bar{x}$)
is the $\bb{Q}$-algebra antiautomorphism (resp. $\bb{Q}$-algebra automorphism) of $U_{r,s}(\fg)$ of~\cite[Proposition 2.6(2)]{MT2}
(resp.~\cite[Proposition 2.6(3)]{MT2}).

Next, we recall from~\eqref{eq:pairing-twist-matching} that $(\cdot,\cdot)_{r,s} = (\varphi(\cdot),\varphi(\cdot))_{q,\zeta}$,
with the pairing $(\cdot,\cdot)_{q,\zeta}$ from Proposition~\ref{prop:twisted_pairing}. In particular, it is easy to see that
$(y,x)_{q,\zeta} = (y,x)_{q,q^{-1}}$ for any $y \in U_{q,q^{-1}}^{-}$ and $x \in U_{q,q^{-1}}^{+}$. Thus, the sets
\begin{equation}\label{eq:dual_bases_2_param}
  \wtd{\mathsf{B}}_{+} = \big\{\varphi^{-1}(b) \,|\, b \in \mathsf{B}_{+} \big\} \qquad \text{and}\qquad
  \wtd{\mathsf{B}}_{-} = \big\{\varphi^{-1}(b) \,|\, b \in \mathsf{B}_{-} \big\}
\end{equation}
form a pair of dual bases relative to $(\cdot,\cdot)_{r,s}$ for $U_{r,s}^{+}$ and $U_{r,s}^{-}$, respectively. Consider
$\wtd{\cal{E}}_{ij},\wtd{\cal{F}}_{ji} \in \wtd{U}(R_{r,s})$ defined akin to~\eqref{eq:root_vectors_RTT}, where
$R_{r,s} = \hat{R}_{r,s} \circ \tau$ and $\hat{R}_{r,s}$ is one of the two-parameter $R$-matrices from~\cite[Theorems~4.4, 4.6]{MT1}.
Define
\[
  \wtd{\mathsf{E}}_{ij} = \tau\left (\ol{\theta_{r,s}^{-1}(\wtd{\cal{E}}_{ij})}\right ), \qquad
  \wtd{\mathsf{F}}_{ji} = \tau\left (\ol{\theta_{r,s}^{-1}(\wtd{\cal{F}}_{ji})}\right ),
\]
where $\theta_{r,s}$ is one of the isomorphisms from Theorems~\ref{thm:DJ=RTT_Btype_2_param} and~\ref{thm:DJ=RTT_Dtype_2param}.
Since $\theta_{r,s} = \wtd{\Psi}^{-1} \circ \theta_{q} \circ \varphi$ with $\wtd{\Psi}$ of Corollary~\ref{cor:U(R)_2_vs_1_parameter}, we have
\begin{align*}
  \wtd{\mathsf{E}}_{ij}
  &= \tau\left (\ol{\varphi^{-1}\theta_{q}^{-1}(\wtd{\Psi}(\wtd{\cal{E}}_{ij}))}\right )
  = \tau\left (\ol{\psi_{i}\psi_{j}^{-1}\zeta(\varepsilon_{j},\varepsilon_{i})}\right )
    \tau\left (\ol{\varphi^{-1}\theta_{q}^{-1}(\cal{E}_{ij})}\right ) \\
  &= \psi_{i}^{-1}\psi_{j}\zeta(\varepsilon_{i},\varepsilon_{j})\cdot \varphi^{-1}
     \left (\kappa\left (\theta_{q}^{-1}(\cal{E}_{ij})\right ) \right )
  = \psi_{i}^{-1}\psi_{j}\zeta(\varepsilon_{i},\varepsilon_{j})\cdot \varphi^{-1}(\mathsf{E}_{ij}),
\end{align*}
and similarly
\[
  \wtd{\mathsf{F}}_{ji} = \psi_{i}\psi_{j}^{-1}\zeta(\varepsilon_{i},\varepsilon_{j})\cdot \varphi^{-1}(\mathsf{F}_{ji}).
\]

The above discussion implies the desired factorization for the \emph{restricted $R$-matrix} $\Theta_{r,s}$ of $U_{r,s}(\fg)$.
To this end, recall the $(r,s)$-exponent $\exp_{r,s}(z)=\sum_{m\geq 0} \frac{s^{m(m-1)/2}}{[m]_{r,s}!}z^m$
of~\cite[Remark 5.14]{MT1} and define $\{\sigma_i\}_{i=1}^N$ by
\[
  \sigma_{i} =
  \begin{cases}
    1 & \text{if}\ i < \frac{N + 1}{2} \\
    0 & \text{if}\ i = \frac{N + 1}{2} \quad \mathrm{in\ type}\ B\\
    -1 & \text{if}\ i > \frac{N + 1}{2}
  \end{cases} \,.
\]

\begin{theorem}\label{thm:R-factorization-RTT_param2}
(a) If $\fg = \frak{so}_{2n+1}$, then
\begin{equation}\label{eq:Theta_factorization_2_param}
  \Theta_{r,s} \, = \overset{\longrightarrow}{\underset{1 \le i < j < i'\le N}{\prod}} \Theta_{r,s}^{(i,j)},
\end{equation}
with $N=2n+1$ and each factor equal to
\[
  \Theta_{r,s}^{(i,j)}
  = \sum_{m \ge 0} \frac{s^{m(m-1)}(rs)^{m(1 - \sigma_{i}\sigma_{j})}}{[m]_{r^{2},s^{2}}!(r^{2} - s^{2})^{m}}
    \wtd{\mathsf{F}}^m_{ji} \otimes \wtd{\mathsf{E}}^m_{ij}
  = \exp_{r^{2},s^{2}}\left (\frac{(rs)^{1 - \sigma_{i}\sigma_{j}}}{r^{2} - s^{2}} \wtd{\mathsf{F}}_{ji} \otimes \wtd{\mathsf{E}}_{ij} \right ).
\]

\noindent
(b) If $\fg = \frak{so}_{2n}$, then the factorization~\eqref{eq:Theta_factorization_2_param} holds with $N=2n$ and each factor equal to
\begin{equation}\label{eq:Theta_ij_CD}
  \Theta_{r,s}^{(i,j)}
  = \sum_{m \ge 0} \frac{s^{\frac{1}{2}m(m-1)}(rs)^{\frac{1}{2}m(1 - \sigma_{i}\sigma_{j})}}{[m]_{r,s}!(r - s)^{m}}
    \wtd{\mathsf{F}}_{ji}^m \otimes \wtd{\mathsf{E}}^m_{ij}
  = \exp_{r,s}\left (\frac{(rs)^{\frac{1}{2}(1 - \sigma_{i}\sigma_{j})}}{r - s} \wtd{\mathsf{F}}_{ji} \otimes \wtd{\mathsf{E}}_{ij} \right ).
\end{equation}
\end{theorem}

\begin{proof}
Since the sets $\wtd{\mathsf{B}}_{\pm}$ of~\eqref{eq:dual_bases_2_param} are dual bases for $U_{r,s}^{\pm}(\fg)$ with respect to
$(\cdot,\cdot)_{r,s}$, the equations
  $\varphi^{-1}(\mathsf{E}_{ij}) = \psi_{i}\psi_{j}^{-1}\zeta(\varepsilon_{i},\varepsilon_{j})^{-1}\wtd{\mathsf{E}}_{ij}$
and
  $\varphi^{-1}(\mathsf{F}_{ji}) = \psi_{i}^{-1}\psi_{j}\zeta(\varepsilon_{i},\varepsilon_{j})^{-1}\wtd{\mathsf{F}}_{ji}$
imply that $\Theta_{r,s}$ satisfies~\eqref{eq:Theta_factorization_2_param} with
\begin{equation}\label{eq:two_parameter_Theta_ij}
  \Theta_{r,s}^{(i,j)} = \sum_{m \ge 0}\frac{v^{\frac{1}{2}m(m-1)}\zeta(\varepsilon_{j},\varepsilon_{i})^{2m}}{[m]_{v}!
  \kappa\left(\sigma(\cal{E}_{ij},\cal{F}_{ji}) \right)^{m}} \wtd{\mathsf{F}}_{ji}^{m} \otimes \wtd{\mathsf{E}}_{ij}^{m}.
\end{equation}
One can check that if $i < j < i'$, then $\zeta(\varepsilon_{j},\varepsilon_{i})^{2} = (rs)^{-\sigma_{i}\sigma_{j}}$ if $N$ is odd, and
$\zeta(\varepsilon_{j},\varepsilon_{i})^{2} = (rs)^{-\frac{1}{2}\sigma_{i}\sigma_{j}}$ if $N$ is even.
Combining this with Lemma~\ref{lem:single_generator_pairing} proves both parts.
\end{proof}

\begin{remark}
In analogy to the super $A$-type factorization from~\cite{Z}, one can also derive Theorem~\ref{thm:R-factorization-RTT_param2} rather directly
akin to Theorem~\ref{thm:R-factorization-RTT_param1}; the only changes are to the $\sigma$-pairing of Lemma~\ref{lem:single_generator_pairing}
and formula~\eqref{eq:sigma_pairing_formula}.
\end{remark}

\begin{remark}
In type $C_{n}$, the ordering~\eqref{eq:root_vector_order} does not induce a convex order on
$\Phi^{+}_{C_n}=\{\varepsilon_i - \varepsilon_j \,|\, 1\leq i<j\leq i'\leq N\}$ with $N=2n$, so the proof of
Theorem~\ref{thm:orthogonal_bases} fails on the nose. However, the results of this Appendix can still be extended
to type $C_n$ if we instead replace $\cal{E}_{ii'}$ (resp.\ $\cal{F}_{ii'}$) with $\widehat{\cal{E}}_{ii'}$
(resp.\ $\widehat{\cal{F}}_{i'i}$) that are appropriate $q$-commutators of $\cal{E}_{in}$ and $\cal{E}_{in'}$
(resp.\ $\cal{F}_{ni}$ and $\cal{F}_{n'i}$), akin to the root vectors defined in Section~\ref{sec:pbw}.
\end{remark}

   %%%%%%%%%%%%%%%%%%%%%%%%%%%%%%%%%%%%%%%%%%%%%%%%%%%%%%%%%%%%%%%%%%%%%%%%%%
   %%%%%%%%%%%%%%%%%%%%%%%%%%%%%%%%%%%%%%%%%%%%%%%%%%%%%%%%%%%%%%%%%%%%%%%%%%
   %%%%%%%%%%%%%%%%%%%%%%%%%%%%%%%%%%%%%%%%%%%%%%%%%%%%%%%%%%%%%%%%%%%%%%%%%%

\section{Evaluation of the universal $R$-matrix}\label{ssec:app_C}

In this Appendix, we relate the normalized $R$-matrix of~\eqref{eq:f_prefactor} to the evaluation of universal $R$-matrix for
$U_{q}(\what{\frak{g}})$ on $V \otimes V$, where $V$ is the evaluation representation corresponding to the first fundamental representation
of $BCD$-type $U_{q}(\frak{g})$, see Theorem~\ref{thm:universal_R_evaluation}\footnote{We note that a similar statement was made without any
details after Propositions 5.1 in~\cite{JLM1,JLM2}.}. While being of independent interest, we also use this result to finally establish the
crossing symmetry identities from Lemma~\ref{lem:affine_crossing_symmetries}.
Although the main results (Theorem~\ref{thm:universal_R_evaluation} and Lemma~\ref{lem:affine_crossing_symmetries}) are stated for
$BCD$-types, we shall only present the computation for $\frak{g} = \frak{so}_{2n+1}$, since the arguments for $\frak{g} = \frak{so}_{2n}$
and $\frak{g} = \frak{sp}_{2n}$ are completely analogous. To establish these results, we need to use the algebra isomorphism
$\bar{\Psi}_{q}\colon U_{q}(\what{\frak{g}}) \iso U_{q}^{D}(\what{\frak{g}})$ of~\eqref{eq:Dr_iso} for the following specific choices
of sequences $i_{1},\ldots ,i_{h-1}$ (see the paragraph preceding Theorem~\ref{thm:DJ=D_1_parameter}) in each type\footnote{The
corresponding values of $a_{q}$ given here in types $B_n$ and $C_n$ correct those from~\cite[Theorem 2.2]{Ji}.}:
\begin{itemize}

\item \emph{$B_{n}$-type}
\begin{equation}\label{eq:i_seq_B}
  (i_{1},\ldots ,i_{2n-1}) = (1,2,\ldots ,n,n,n-1,\ldots ,2).
\end{equation}
The values of $\sigma_{k} = (\alpha_{i_{1}} + \ldots + \alpha_{i_{k}},\alpha_{i_{k+1}})$ and the constant $a_{q}$ are then given by
\begin{equation}\label{eq:sigma_and_const_B}
  (\sigma_{1},\ldots ,\sigma_{2n-2}) = (\underbrace{-2,\ldots ,-2}_{n-1\ \text{times}},0,-2,\ldots ,-2),\qquad a_{q} = \frac{1}{[2]_{q}^{2}}.
\end{equation}

\item \emph{$C_{n}$-type}
\[
  (i_{1},\ldots ,i_{2n-1}) = (1,2,\ldots ,n-1,n,n-1,\ldots ,2,1).
\]
The values of $\sigma_{k} = (\alpha_{i_{1}} + \ldots + \alpha_{i_{k}},\alpha_{i_{k+1}})$ and the constant $a_{q}$ are then given by
\[
  (\sigma_{1},\ldots ,\sigma_{2n-2}) = (\underbrace{-1,\ldots ,-1}_{n-2\ \text{times}},-2,-1,\ldots ,-1),\qquad a_{q}
  = \frac{1}{[2]_{q^{1/2}}^{2}}.
\]

\item \emph{$D_{n}$-type}
\[
  (i_{1},\ldots ,i_{2n-3}) = (1, 2,\ldots, n-2, n-1, n, n-2, n-3, \ldots, 2).
\]
The values of $\sigma_{k} = (\alpha_{i_{1}} + \ldots + \alpha_{i_{k}},\alpha_{i_{k+1}})$ and the constant $a_{q}$ are then given by
\[
  (\sigma_{1},\ldots ,\sigma_{2n - 4}) = (-1,\ldots ,-1),\qquad a_{q} = 1.
\]

\end{itemize}

To precisely describe the universal $R$-matrix for $U_{q}(\what{\frak{so}}_{2n+1})$, we need to work in the $h$-adic setting
(cf.~\cite[\S 6.5]{CP}), so we regard $U_{q}(\what{\frak{so}}_{2n+1})$ as an algebra over $\bb{K}[[\hbar]]$, with $q = e^{\hbar}$
and $K_{i} = e^{\hbar d_{i}h_{i}}$ for $0 \le i \le n$, where $d_{i} = \frac{(\alpha_{i},\alpha_{i})}{2}$\footnote{We note that
the elements $h_{i}$ used in~\cite[\S5]{JLM1} and~\cite[\S5]{JLM2} differ from the $h_{i}$ defined here by a factor of $d_{i}$.}.
Writing $\delta = \sum_{i = 0}^{n}\delta_{i}\alpha_{i}$, we also set $c = \sum_{i = 0}^{n}d_{i}\delta_{i}h_{i}$, so that
$\gamma = e^{\hbar c} = q^{c}$. Furthermore, we need to extend $U_{q}(\what{\frak{so}}_{2n+1})$ to $U_{q}(\wtd{\frak{so}}_{2n+1})$
by adding elements $D^{\pm 1} = q^{\pm d}$ satisfying
\[
  DD^{-1} = D^{-1}D = 1 ,\qquad  DK_{i} = K_{i}D ,\qquad
  DE_{i} = q^{\delta_{0i}}E_{i}D ,\qquad  DF_{i} = q^{-\delta_{0i}}F_{i}D \qquad  \forall\, 0 \le i \le n.
\]

Then, $U_{q}(\wtd{\frak{so}}_{2n+1})$ has a universal $R$-matrix
$\scr{R} \in (U_{q}(\wtd{\frak{so}}_{2n+1}) \otimes U_{q}(\wtd{\frak{so}}_{2n+1}))[[\hbar]]$, see~\cite[\S4]{FR}.
Following~\cite[(4.8)]{FR}, we further define
\[
  \cal{R} = q^{-(d \otimes c + c \otimes d)} \scr{R},
\]
so that $\cal{R}$ is an element of $(U_{q}(\what{\frak{so}}_{2n+1}) \otimes U_{q}(\what{\frak{so}}_{2n+1}))[[\hbar]]$. We also recall
from~\cite[(4.3)]{FR} that there is an algebra automorphism $\cal{D}_{z}$ of $U_{q}(\wtd{\frak{so}}_{2n+1}) \otimes \bb{K}[z,z^{-1}]$
defined by
\begin{equation}\label{eq:Dz_aut}
  \cal{D}_{z}(E_{i}) = z^{\delta_{i0}}E_{i} ,\qquad  \cal{D}_{z}(F_{i}) = z^{-\delta_{i0}}F_{i} ,\qquad
  \cal{D}_{z}(K_{i}) = K_{i} ,\qquad  \cal{D}_{z}(D) = D \qquad \forall\, 0 \le i \le n.
\end{equation}
Then, we define
\begin{equation}\label{eq:universal_Rz}
  \cal{R}(z) = (\cal{D}_{z} \otimes \id)(\cal{R}) \in
  (U_{q}(\what{\frak{so}}_{2n+1}) \otimes U_{q}(\what{\frak{so}}_{2n+1}))[[\hbar]] \otimes \bb{K}[[z]].
\end{equation}

To write down a formula for $\cal{R}(z)$, we recall from~\cite{B} that $U_{q}(\what{\frak{g}})$ has root vectors
\[
  \big\{ E_{\alpha + k\delta} \,\big|\, k \ge 0,\, \alpha \in \Phi^{+} \big\} \cup
  \big\{ E_{-\alpha + k\delta} \,\big|\, k > 0,\, \alpha \in \Phi^{+} \big\} \cup
  \big\{ E_{(k\delta,i)} \,\big|\, k >0,\, 1 \le i \le n \big\} \subset U_{q}^{+}(\what{\frak{g}}),
\]
\[
  \big\{ F_{\alpha + k\delta} \,\big|\, k \ge 0,\, \alpha \in \Phi^{+} \big\} \cup
  \big\{ F_{-\alpha + k\delta} \,\big|\, k > 0,\, \alpha \in \Phi^{+} \big\} \cup
  \big\{ F_{(k\delta,i)} \,\big|\, k >0,\, 1 \le i \le n \big\} \subset U_{q}^{-}(\what{\frak{g}}),
\]
such that appropriately ordered products of these elements form bases for $U_{q}^{+}(\what{\frak{g}})$ and $U_{q}^{-}(\what{\frak{g}})$,
respectively. Moreover, due to the results of~\cite{B} and~\cite{BCP}, the following map gives rise to an isomorphism
$U_{q}^{D}(\what{\frak{so}}_{2n+1}) \iso U_{q}(\what{\frak{so}}_{2n+1})$, which is the inverse to the isomorphism
$\bar{\Psi}_{q}\colon U_{q}(\what{\frak{so}}_{2n+1}) \iso U_{q}^{D}(\what{\frak{so}}_{2n+1})$ of~\eqref{eq:Dr_iso}:
\begin{equation}\label{eq:DB_iso}
\begin{split}
  \gamma^{1/2} \mapsto \gamma^{1/2} ,\qquad
  X_{i,k}^{+}\mapsto o(i)^{k}E_{\alpha_{i} + k\delta} ,\qquad
  X_{i,-k}^{-} \mapsto o(i)^{k}F_{\alpha_{i} + k\delta} ,\qquad
  X_{i,-\ell}^{+} \mapsto -o(i)^{\ell}F_{-\alpha_{i} + \ell\delta}K_{i}^{-1}\gamma^{\ell}, \\
  X_{i,\ell}^{-} \mapsto -o(i)^{\ell}\gamma^{-\ell}K_{i}E_{-\alpha_{i} + \ell\delta} ,\qquad
  H_{i,\ell} \mapsto o(i)^{\ell}\gamma^{-\ell/2}E_{(\ell\delta,i)} ,\qquad
  H_{i,-\ell} \mapsto o(i)^{\ell}F_{(\ell \delta,i)}\gamma^{\ell/2},
\end{split}
\end{equation}
for all $1 \le i \le n$, $k \ge 0$, and $\ell > 0$, with an appropriately chosen function
\begin{equation}\label{eq:orientation}
  o\colon I \longrightarrow \{\pm 1\} \qquad \text{such that} \qquad o(i)o(j) = -1\qquad \text{whenever}\qquad a_{ij} < 0.
\end{equation}
We also introduce the following standard matrices:
\[
  B = (B_{ij})_{i,j = 1}^{n} =  (d_{i}a_{ij})_{i,j = 1}^{n} = ((\alpha_{i},\alpha_{j}))_{i,j = 1}^{n} ,\qquad
  B^{-1} = (\wtd{B}_{ij})_{i,j = 1}^{n},
\]
\[
  B(q) = ([B_{ij}]_{q})_{i,j = 1}^{n},\qquad B(q)^{-1} = (\wtd{B}_{ij}(q))_{i,j = 1}^{n}.
\]
We note that both $B^{-1}$ and $B(q)^{-1}$ are symmetric, and are explicitly determined by
\[
  \wtd{B}_{ij} = \frac{j}{2}\qquad \text{for}\qquad j \le i,
\]
\begin{equation}\label{eq:B(q)_inv}
  \wtd{B}_{ij}(q) =
  \begin{cases}
    \frac{[j]_{q^{2}}([n - i]_{q^{2}} - [n - i - 1]_{q^{2}})}{[2]_{q}([n]_{q^{2}} - [n-1]_{q^{2}})} & \text{if}\ j \le i < n \\
    \frac{[j]_{q^{2}}}{([n]_{q^{2}} - [n - 1]_{q^{2}})[2]_{q}} & \text{if}\ i = n
  \end{cases}.
\end{equation}

If we set
\[
  T = \exp \left( -\hbar \sum_{i,j = 1}^{n}\wtd{B}_{ij}d_{i}d_{j}h_{i} \otimes h_{j} \right),
\]
then, as in~\cite{Da}, we can write $\cal{R}(z)$ as
\[
  \cal{R}(z) = \cal{R}^{+}(z)\cal{R}^{0}(z)\cal{R}^{-}(z)
\]
with
\[
  \cal{R}^{+}(z) = \prod_{\alpha \in \Phi^{+}}\prod_{k \ge 0}
  \exp_{q_{\alpha}} \left( (q_{\alpha}^{-1} - q_{\alpha})z^{k}E_{\alpha + k\delta} \otimes F_{\alpha + k\delta} \right),
\]
\[
  \cal{R}^{-}(z) = T^{-1} \left( \prod_{\alpha \in \Phi^{+}}\prod_{k > 0}
  \exp_{q_{\alpha}}\left ((q_{\alpha}^{-1} - q_{\alpha})z^{k}E_{-\alpha + k\delta} \otimes F_{-\alpha + k\delta} \right) \right) T,
\]
and
\begin{equation}\label{eq:R0_formula}
  \cal{R}^{0}(z) = \exp \left( \sum_{k > 0}\sum_{i,j = 1}^{n} \frac{(q_{i}^{-1} - q_{i})(q_{j}^{-1} - q_{j})}{q^{-1} - q}\cdot
  \frac{k}{[k]_{q}}\wtd{B}_{ij}(q^{k})(o(i)o(j))^{k}z^{k}E_{(k\delta,i)} \otimes F_{(k\delta,j)} \right)T,
\end{equation}
with $o$ of~\eqref{eq:orientation} (note that~\eqref{eq:R0_formula} is compatible with the corresponding formula
in~\cite[\S5]{JLM1}, due to~\eqref{eq:DB_iso}). While the ordering of products in $\cal{R}^{\pm}(z)$ follows~\cite[\S 7]{Da},
it will be irrelevant for the arguments below.

Let $V = \bb{K}^{2n+1}$. Similarly to~\cite[Proposition 6.6]{MT1}, we have the following result
(which is essentially the $r = q = s^{-1}$-specialization of~\cite[Proposition 6.6]{MT1} with $u = a = b = 1$):

\begin{prop}\label{prop:affine_rep_B}
The following defines a representation $\rho\colon U_{q}(\what{\frak{so}}_{2n+1}) \to \End(V)$:
\begin{align*}
  \rho(E_{i}) = E_{i,i + 1} - E_{(i + 1)',i'} \quad \mathrm{for}\ \ 1\leq i\leq n \,,
\end{align*}
\begin{align*}
  \rho(F_{i}) &=
  \begin{cases}
    E_{i + 1, i} - E_{i', (i + 1)'} &   1\leq i<n \\
    (q + q^{-1})E_{n + 1,n} - (q + q^{-1})E_{n', n + 1} &  i = n
  \end{cases},
\end{align*}
\begin{align*}
  \rho(K_{i}) =
  \begin{cases}
    \displaystyle{q^{2}E_{ii} + q^{-2}E_{i + 1, i + 1} + q^{-2}E_{i'i'} + q^{2}E_{(i + 1)', (i + 1)'} +
                  \sum_{1\leq j\leq n}^{j \neq i,i + 1} (E_{jj}+E_{j'j'}) + E_{n+1,n+1}}
       &  1\leq i<n \\
    \displaystyle{q^{2}E_{nn} + E_{n + 1,n+1} + q^{-2}E_{n'n'} + \sum_{1\leq j\leq n-1}\left (E_{jj} + E_{j'j'}\right )} &  i = n
  \end{cases},
\end{align*}
\begin{align*}
  & \rho(E_0)=E_{1' 2} -  E_{2' 1} ,\qquad
    \rho(F_0)= E_{2 1'}-E_{1 2'} ,\qquad
    \rho(\gamma)=\Id, \\
  & \rho(K_0)=
    \left( q^{-2}E_{11}+q^{-2}E_{22}+\sum_{i=3}^{n} \big( E_{ii} +  E_{i'i'}\big) +
    q^2E_{2' 2'} + q^{2}E_{1' 1'} + E_{n+1,n+1} \right).
\end{align*}
\end{prop}

If we define $\wtd{\rho} = \rho \circ \omega$, where $\omega$ is the Cartan involution of~\eqref{eq:cartan_involution}
(extended to $U_{q}(\what{\frak{g}})$ by the same formulas), then we have the following result, which is crucially needed
for the proof of Lemma~\ref{lem:affine_crossing_symmetries}:

\begin{lemma}\label{lem:dual_rep}
For all $a \in U_{q}(\what{\frak{so}}_{2n+1})$, we have
\begin{equation}\label{eq:dual_rep_1}
  \wtd{\rho}(S(a)) = C_{q}(\wtd{\rho}(\cal{D}_{\xi}(a)))^{\mathsf{t}'}(C_{q})^{-1},
\end{equation}
\begin{equation}\label{eq:dual_rep_2}
  \wtd{\rho}(S^{-1}(a)) = C_{q}^{\mathsf{t}'}(\wtd{\rho}(\cal{D}_{\xi^{-1}}(a)))^{\mathsf{t}'}(C_{q}^{-1})^{\mathsf{t}'},
\end{equation}
where $\cal{D}_{\xi^{\pm 1}}$ are the automorphisms defined by~\eqref{eq:Dz_aut} with $\xi^{\pm 1} = q^{\pm (2 - 4n)}$ in place of $z$,
and $C_{q}$ is the matrix of~\eqref{eq:Cq_B}.
\end{lemma}

\begin{proof}
It suffices to check~\eqref{eq:dual_rep_1} when $a$ is one of the generators $E_{i},F_{i},K_{i}$ ($0 \le i \le n$).
Since $(\wtd{\rho}(K_{i}))^{\mathsf{t}'} = \wtd{\rho}(K_{i})^{-1}$ for all $i$, the result is immediate for $a = K_{i}$.
The computations for $a = E_{i}$ and $a = F_{i}$ are straightforward, so we shall only present the computation for $a = E_{0}$.
Since $S(E_{0}) = -K_{0}^{-1}E_{0}$, we have
\[
  \wtd{\rho}(S(E_{0})) = -\wtd{\rho}(K_{0}^{-1})\wtd{\rho}(E_{0}) = -\rho(K_{0})\rho(F_{0}) = q^{-2}E_{12'}-q^{-2}E_{21'},
\]
and
\begin{align*}
  C_{q}(\wtd{\rho}(\cal{D}_{\xi}(E_{0})))^{\mathsf{t}'}(C_{q})^{-1} &= C_{q}\left (\xi E_{12'} - \xi E_{21'}\right )(C_{q})^{-1} \\
  &= \xi q^{2(\rho_{1} - \rho_{2'})}E_{12'} - \xi q^{2(\rho_{2} - \rho_{1'})}E_{21'} = q^{-2}E_{12'} - q^{-2}E_{21'},
\end{align*}
so the desired equality~\eqref{eq:dual_rep_1} holds for $a = E_{0}$.

The equality~\eqref{eq:dual_rep_2} follows by using~\eqref{eq:dual_rep_1} with $S^{-1}(\cal{D}_{\xi^{-1}}(a))$ in place of $a$,
together with the fact that $S^{-1}\circ \cal{D}_{\xi^{-1}} = \cal{D}_{\xi^{-1}} \circ S^{-1}$.
\end{proof}

Let $\rho^{D}\colon U_{q}^{D}(\what{\frak{so}}_{2n+1}) \to \End(V)$ be the representation of~\cite[Proposition 5.1]{JLM1} with $q$
replaced by $q^{2}$, where our generators $X_{i,k}^{\pm},H_{i,\ell}, K_{i},\gamma^{1/2}$ correspond respectively to the generators
$x_{i,k}^{\pm},a_{i,\ell},k_{i},q^{c/2}$ of~\cite{JLM1}. Let
$\varsigma \colon U_{q}^{D}(\what{\frak{so}}_{2n+1}) \to U_{q}^{D}(\what{\frak{so}}_{2n+1})$ be the algebra automorphism such that
\begin{equation}\label{eq:varsigma_aut}
\begin{split}
  \varsigma(K_{i}) = K_{i} ,\qquad
  \varsigma(\gamma^{1/2}) = \gamma^{1/2} ,\qquad
  \varsigma(H_{i,\ell}) = (-1)^{n\ell}\xi^{-\ell}H_{i,\ell}, \\
  \varsigma(X_{i,k}^{\pm}) =
  \begin{cases}
    (-1)^{nk + 1}\xi^{-k}X_{i,k}^{\pm} & \text{if}\ 1 \le i < n \\
    (-1)^{nk +1}\xi^{-k}[2]_{q}^{\pm 1/2}X_{n,k}^{\pm} & \text{if}\ i = n
  \end{cases},
\end{split}
\end{equation}
for all $1 \le i \le n$, $k \in \bb{Z}$, and $\ell \in \bb{Z} \setminus \{0\}$.

Then $\rho$ and $\rho^{D}$ are related as follows:

\begin{prop}\label{prop:DJ_to_nD_rep}
We have
\begin{equation}\label{eq:DJ_to_nD_rep}
  \wtd{\rho} = \rho^{D} \circ \varsigma \circ \bar{\Psi}_{q}\colon U_{q}(\what{\frak{so}}_{2n+1}) \longrightarrow \End(V),
\end{equation}
where $\wtd{\rho} = \rho \circ \omega$ with $\rho$ of Proposition~\ref{prop:affine_rep_B}, $\rho^{D}$ is the representation
of~\cite[Proposition 5.1]{JLM1}, $\varsigma$ is given by~\eqref{eq:varsigma_aut}, and
$\bar{\Psi}_{q}\colon U_{q}(\what{\frak{so}}_{2n+1}) \to U_{q}^{D}(\what{\frak{so}}_{2n+1})$ is the algebra isomorphism of~\eqref{eq:Dr_iso}.
\end{prop}

\begin{proof}
It suffices to verify~\eqref{eq:DJ_to_nD_rep} on the generators $E_{i},F_{i},K_{i}$ of $U_{q}(\what{\frak{so}}_{2n+1})$.
For $i \neq 0$, this is immediate from the definitions. Moreover, an easy computation shows that
$\wtd{\rho}(K_{0}) = \rho^{D}(K_{\theta}^{-1})$, so it remains to verify~\eqref{eq:DJ_to_nD_rep} on $E_{0}$ and $F_{0}$.
We shall first present the verification for $E_{0}$. For the choice of a sequence $(i_{1},\ldots ,i_{2n-1})$ specified
in~\eqref{eq:i_seq_B}, we have
\[
  \bar{\Psi}_{q}(E_{0}) =
  [X_{2,0}^{-},X_{3,0}^{-},\ldots ,X_{n,0}^{-},X_{n,0}^{-},X_{n-1,0}^{-},\ldots  ,X_{2,0}^{-},X_{1,1}^{-}]_{(q^{\sigma_{1}},\ldots ,q^{\sigma_{2n-2}})}\gamma K_{\theta}^{-1},
\]
with $\sigma_{i}$ given by~\eqref{eq:sigma_and_const_B}. Thus,
\begin{align*}
  &\rho^{D} \circ \varsigma \circ \bar{\Psi}_{q}(E_{0}) \\
  &= (-1)^{3n - 1}\xi^{-1} [2]_{q}^{-1}
    \left[ \rho^{D}(X_{2,0}^{-}),\ldots ,\rho^{D}(X_{n,0}^{-}),\rho^{D}(X_{n,0}^{-}),\ldots ,\rho^{D}(X_{2,0}^{-}),\rho^{D}(X_{1,1}^{-})
           \right]_{(q^{\sigma_{1}},\ldots ,q^{\sigma_{2n - 2}})}\rho^{D}(K_{\theta}^{-1}).
\end{align*}
Now, using the formulas of~\cite[Proposition 5.1]{JLM1} and the definition of the $q$-bracket~\eqref{eq:q_brackets},
an easy computation shows that
\[
  [\rho^{D}(X_{n,0}^{-}),\ldots ,\rho^{D}(X_{2,0}^{-}),\rho^{D}(X_{1,1}^{-})]_{(q^{-2},\ldots ,q^{-2})} =
  [2]_{q}^{1/2}(q^{4 - 4n}E_{n+1,1'} - q^{-2n}E_{1,n+1}).
\]
Then
\[
  [\rho^{D}(X_{n,0}^{-}),[\rho^{D}(X_{n,0}^{-}),\ldots ,\rho^{D}(X_{2,0}^{-}),\rho^{D}(X_{1,1}^{-})]_{(q^{-2},\ldots ,q^{-2})}]
  = -[2]_{q}(q^{4 - 4n}E_{n1'} - q^{-2n}E_{1n'}),
\]
and similar computations using~\eqref{eq:q_brackets} show that
\[
  [\rho^{D}(X_{2,0}^{-}),\ldots ,-[2]_{q}(q^{4 - 4n}E_{n1'} - q^{-2n}E_{1n'})]_{(q^{-2},\ldots ,q^{-2})}
  = (-1)^{n-1}[2]_{q}q^{4 - 4n}(E_{21'} - E_{12'}).
\]
Thus, because $\rho^{D}(K_{\theta}^{-1}) = \wtd{\rho}(K_{0})$ as noted above, we have
\begin{align*}
  &\rho^{D} \circ \varsigma \circ \bar{\Psi}_{q}(E_{0}) \\
  &= (-1)^{3n - 1}\xi^{-1} [2]_{q}^{-1}
     \left[ \rho^{D}(X_{2,0}^{-}),\ldots, \rho^{D}(X_{n,0}^{-}),\rho^{D}(X_{n,0}^{-}),\ldots ,\rho^{D}(X_{2,0}^{-}),\rho^{D}(X_{1,1}^{-})
          \right]_{(q^{\sigma_{1}},\ldots ,q^{\sigma_{2n - 2}})}\rho^{D}(K_{\theta}^{-1}) \\
  &= E_{21'} - E_{12'} = \wtd{\rho}(E_{0}),
\end{align*}
as desired.

For $F_{0}$, we have
\[
  \bar{\Psi}_{q}(F_{0})
  = \frac{q^{4n - 6}}{[2]_{q}^{2}}\gamma^{-1}K_{\theta}[X_{2,0}^{+},X_{3,0}^{+},\ldots ,X_{n,0}^{+},X_{n,0}^{+},X_{n-1,0}^{+},\ldots ,X_{2,0}^{+},X_{1,-1}^{+}]_{(q^{\sigma_{1}},\ldots ,q^{\sigma_{2n-2}})},
\]
so that
\begin{align*}
  &\rho^{D} \circ \varsigma \circ \bar{\Psi}_{q}(F_{0}) \\
  &= q^{4n-6} [2]^{-1}_{q} \xi(-1)^{n-1}\rho^{D}(K_{\theta})[\rho^{D}(X_{2,0}^{+}),\ldots ,\rho^{D}(X_{n,0}^{+}),\rho^{D}(X_{n,0}^{+}),\ldots ,\rho^{D}(X_{2,0}^{+}),\rho^{D}(X_{1,-1}^{+})]_{(q^{\sigma_{1}},\ldots ,q^{\sigma_{2n-2}})}.
\end{align*}
Then as above, a straightforward computation shows that
\[
  [\rho^{D}(X_{2,0}^{+}),\ldots ,\rho^{D}(X_{n,0}^{+}),\rho^{D}(X_{n,0}^{+}),\ldots ,\rho^{D}(X_{2,0}^{+}),\rho^{D}(X_{1,-1}^{+})]_{(q^{\sigma_{1}},\ldots ,q^{\sigma_{2n-2}})} = [2]_{q}(-1)^{n-1}q^{2}(E_{1'2} - E_{2'1}),
\]
so that
\[
  \rho^{D} \circ \varsigma \circ \bar{\Psi}_{q}(F_{0}) = q^{4n-4}\xi\rho^{D}(K_{\theta})(E_{1'2} - E_{2'1})
  = E_{1'2} - E_{2'1} = \wtd{\rho}(F_{0}).
\]
This completes the proof.
\end{proof}

\begin{cor}\label{cor:imaginary_action}
The imaginary root vectors $E_{(k\delta,i)},F_{(k\delta,i)} \in U_{q}(\what{\frak{so}}_{2n+1})$ act on $V$ via $\wtd{\rho}$ as follows:
\begin{align*}
  &\wtd{\rho}(E_{(k\delta,i)}) \\
  &= (-1)^{nk}\xi^{-k}o(i)^{k}\frac{[k]_{q_{i}}}{k}
     \left( q^{-2k(i + 1)}E_{i+1,i+1}- q^{-2k(i - 1)}E_{ii} + q^{-2k(2n - i)}E_{i'i'} - q^{-2k(2n - 2 - i)}E_{(i + 1)',(i + 1)'} \right), \\
  &\wtd{\rho}(F_{(k\delta,i)}) \\
  &= (-1)^{nk}\xi^{k}o(i)^{k}\frac{[k]_{q_{i}}}{k}
     \left( q^{2k(i + 1)}E_{i+1,i+1}- q^{2k(i - 1)}E_{ii} + q^{2k(2n - i)}E_{i'i'} - q^{2k(2n - 2 - i)}E_{(i + 1)',(i + 1)'} \right),
\end{align*}
for all $1 \le i < n$ and $k > 0$, and
\begin{align*}
  &\wtd{\rho}(E_{(k\delta,n)}) = (-1)^{nk}\xi^{-k}o(n)^{k}\frac{[2k]_{q_{n}}}{k}
   \left( -q^{-2(n-1)k}E_{nn} + (q^{-2nk} - q^{-2(n-1)k})E_{n+1,n+1} + q^{-2nk}E_{n'n'} \right), \\
  &\wtd{\rho}(F_{(k\delta,n)}) = (-1)^{nk}\xi^{k}o(n)^{k}\frac{[2k]_{q_{n}}}{k}
   \left( -q^{2(n-1)k}E_{nn} + (q^{2nk} - q^{2(n-1)k})E_{n+1,n+1} + q^{2nk}E_{n'n'} \right),
\end{align*}
for all $k > 0$.
\end{cor}

\begin{proof}
This is immediate from Proposition~\ref{prop:DJ_to_nD_rep} and the fact that $\bar{\Psi}_{q}(E_{(k\delta,i)}) = o(i)^{k}\gamma^{k/2}H_{i,k}$
and $\bar{\Psi}_{q}(F_{(k\delta,i)}) = o(i)^{k}\gamma^{-k/2}H_{i,-k}$, see~\eqref{eq:DB_iso}.
\end{proof}

This allows us to compute the evaluation of $\cal{R}(z)$ on $V((z)) \otimes V$, where $V((z)) = V \otimes_{\bb{K}}\bb{K}((z))$ is the
$U_{q}(\what{\frak{so}}_{2n+1})((z))$-module defined by $\wtd{\rho}_{z} = \wtd{\rho}\circ \cal{D}_{z}$. To simplify the notation, we set
\begin{equation}\label{eq:universal_R_projection}
  \mathsf{R}(z) = (\wtd{\rho} \otimes \wtd{\rho})(\cal{R}(z)).
\end{equation}

The following result is the key step in this computation:

\begin{lemma}
The operator $\mathsf{R}(z)$ acts on the vector $v_{N} \otimes v_{N} \in V((z)) \otimes V$ as multiplication by the series
$g(z) \in \bb{K}[[z]]$ determined by the following conditions:
\begin{equation}\label{eq:g_qdifference}
  g(0) = q^{-2},\qquad
  g(z)g(z\xi) = q^{-4}\frac{(1 - z)(1 - \xi q^{4}z)}{(1 - q^{-4}z)(1 - \xi z)}.
\end{equation}
\end{lemma}

\begin{proof}
First we note that $v_{N}$ is the highest weight vector for the representation $\wtd{\rho}$, in the sense that
$\wtd{\rho}(u)v_{N} = 0$ whenever $u \in U_{q}(\what{\frak{so}}_{2n+1})_{\alpha + k\delta}$ with $\alpha \in \Phi^{+}$,
since $\wtd{\rho}(u)v_{N} \in V[-\varepsilon_{1} - \alpha] = 0$ for all $\alpha \in \Phi^{+}$. This implies that
$(\wtd{\rho} \otimes \wtd{\rho})(\cal{R}(z))v_{N} \otimes v_{N} = (\wtd{\rho} \otimes \wtd{\rho})(\cal{R}^{0}(z))v_{N} \otimes v_{N}$.
By Corollary~\ref{cor:imaginary_action}, $\wtd{\rho}(E_{(k\delta,i)})$ and $\wtd{\rho}(F_{(k\delta,i)})$ annihilate $v_{N}$ if $i \neq 1$,
and we have
\begin{align*}
  \wtd{\rho}(E_{(k\delta,1)})(v_{N}) = (-1)^{nk}\xi^{-k}o(1)^{k}\frac{[k]_{q^{2}}}{k}q^{-2k(2n-1)}v_{N} ,\qquad
  \wtd{\rho}(F_{(k\delta,1)})(v_{N}) = (-1)^{nk}\xi^{k}o(1)^{k}\frac{[k]_{q^{2}}}{k}q^{2k(2n - 1)}v_{N}.
\end{align*}
Moreover, since $K_{i} = e^{\hbar d_{i}h_{i}}$ and $\wtd{\rho}(K_{i})\cdot v_{N} = q^{-(\alpha_{i},-\varepsilon_{1})}v_{N}$,
we have $\wtd{\rho}(h_{i})\cdot v_{N} = -\frac{2(\alpha_{i},-\varepsilon_{1})}{(\alpha_{i},\alpha_{i})}v_{N} = \delta_{i,1}v_{N}$.
Therefore
\[
  (\wtd{\rho} \otimes \wtd{\rho}) \left( \exp\left (-\hbar\sum_{i,j = 1}^{N}\wtd{B}_{ij}d_{i}d_{j}h_{i} \otimes h_{j}\right ) \right)
  v_{N} \otimes v_{N} = \exp \left( -\hbar \wtd{B}_{11}d_{1}^{2} \right) v_{N} \otimes v_{N} = q^{-2}v_{N} \otimes v_{N}.
\]
Thus, using~\eqref{eq:R0_formula},~\eqref{eq:B(q)_inv}, and the above computations, we get
\begin{align*}
  (\wtd{\rho} \otimes \wtd{\rho})(\cal{R}^{0}(z))&v_{N} \otimes v_{N}
  = q^{-2}\exp \left( \sum_{k > 0}\frac{(q^{-2} - q^{2})^{2}}{q^{-1} - q}\cdot \frac{k}{[k]_{q}}\wtd{B}_{11}(q^{k})z^{k}E_{(k\delta,1)} \otimes F_{(k\delta,1)} \right)v_{N} \otimes v_{N} \\
  &= q^{-2}\exp \left( \sum_{k > 0}\frac{(q^{-2} - q^{2})^{2}}{q^{-1} - q}\cdot \frac{[n-1]_{q^{2k}} - [n-2]_{q^{2k}}}{[2]_{q^{k}}([n]_{q^{2k}} - [n-1]_{q^{2k}})}\cdot \frac{[k]_{q^{2}}^{2}}{k[k]_{q}}z^{k} \right)v_{N} \otimes v_{N} \\
  &= q^{-2}\exp \left( \sum_{k > 0}\frac{(q^{-2k} - q^{2k})(q^{2k(n-1)} - q^{-2k(n - 1)} - q^{2k(n-2)} + q^{-2k(n-2)})}{k(q^{2kn} - q^{-2kn} - q^{2k(n-1)} + q^{-2k(n-1)})}z^{k} \right) v_{N} \otimes v_{N}.
\end{align*}
Now, the series
\[
  \wtd{g}(z) = q^{-2} \exp\left( \sum_{k > 0}
    \frac{(q^{-2k} - q^{2k})(q^{2k(n-1)} - q^{-2k(n - 1)} - q^{2k(n-2)} + q^{-2k(n-2)})}
         {k(q^{2kn} - q^{-2kn} - q^{2k(n-1)} + q^{-2k(n-1)})} z^{k} \right)
\]
clearly satisfies $\wtd{g}(0) = q^{-2}$, and since
\begin{align*}
  q^{2k(n - 1)} - q^{-2k(n-1)} - q^{2k(n-2)} + q^{-2k(n-2)} &= \xi^{k}q^{-4k}(q^{2kn} - q^{2k(n-1)})(q^{4kn} + q^{6k}), \\
  q^{2kn} - q^{-2kn} - q^{2k(n-1)} + q^{-2k(n-1)} &= (q^{2kn} - q^{2k(n-1)})(1 + \xi^{k}),
\end{align*}
we get
\begin{align*}
  \wtd{g}(z)\wtd{g}(z\xi)
  &= q^{-4}\exp\left( \sum_{k > 0}\frac{\xi^{k}q^{-4k}(q^{2kn} - q^{2k(n-1)})(q^{4kn} + q^{6k})(q^{-2k} - q^{2k})(1 + \xi^{k})}{k(q^{2kn} - q^{2k(n-1)})(1 + \xi^{k})}z^{k} \right) \\
  &= q^{-4}\exp \left( \sum_{k > 0}\frac{\xi^{k}q^{-4k}}{k}(q^{4kn} + q^{6k})(q^{-2k} - q^{2k})z^{k}\right ) = q^{-4}\prod_{k > 0}\exp\left (\frac{1}{k}( q^{-4k}- 1 - \xi^{k}q^{4k} + \xi^{k})z^{k} \right).
\end{align*}
Combining this with the identity
\[
  \frac{1}{1 - z} = \prod_{k > 0}\exp \left( \frac{z^{k}}{k} \right)
\]
yields
\[
  \wtd{g}(z)\wtd{g}(z\xi) = q^{-4}\frac{(1 - z)(1 - \xi q^{4}z)}{(1 - q^{-4}z)(1 - \xi z)},
\]
which implies that $\wtd{g}(z) = g(z)$, completing the proof.
\end{proof}

As a consequence, we obtain the main result of this Appendix:

\begin{theorem}\label{thm:universal_R_evaluation}
Let $\frak{g} = \frak{so}_{2n+1}$ (resp.\ $\frak{g} = \frak{sp}_{2n}$, $\frak{g} = \frak{so}_{2n}$), and let $\rho$ be the
$r = q = s^{-1}$-specialization of~\cite[Proposition 6.6]{MT1} (resp.~\cite[Proposition 6.7]{MT1},~\cite[Proposition 6.8]{MT1})
with $u = a = b = 1$. Let $\wtd{\rho} = \rho \circ \omega$ and define $\mathsf{R}(z)$ as in~\eqref{eq:universal_R_projection}.
Then, we have the following equality in $\End\left (V((z)) \otimes V\right )$:
\[
  \mathsf{R}(z) = v^{-1}\xi^{-1} \wtd{R}_{q}(z),
\]
where $\wtd{R}_{q}(z)$ is the type $B_{n}$ (resp.\ $C_{n}$, $D_{n}$) $R$-matrix defined in~\eqref{eq:f_prefactor}.
\end{theorem}

\begin{proof}
We shall only give the proof for $\frak{g} = \frak{so}_{2n+1}$ (as the proof is analogous for the types $C_n$ and $D_n$).
As in~\cite[Proposition 9.5.3]{EFK}, we have the factorization $\mathsf{R}(z) = g(z)\bar{\mathsf{R}}(z)$, where $\bar{\mathsf{R}}(z)$
is the operator uniquely determined by the conditions $\bar{\mathsf{R}}(z)(v_{N} \otimes v_{N}) = 1$ and
$\bar{\mathsf{R}}(z) \circ \tau\colon V((z)) \otimes V \to V \otimes V((z))$ is an isomorphism of
$U_{q}(\what{\frak{so}}_{2n+1})((z))$-modules. We know from~\cite{Jim}, see also~\cite[Theorem 6.11(a)]{MT1},
that $R_{q}(z) \circ \tau\colon V((z)) \otimes V \to V \otimes V((z))$ is an isomorphism of
$U_{q}(\what{\frak{so}}_{2n+1})((z))$-modules, and inspecting the formula~\cite[(6.11)]{MT1} we conclude that
\[
  R_{q}(z) = (z - q^{-4})(z - \xi)\bar{\mathsf{R}}(z).
\]
Using the fact that $f(z)$ is determined by~\eqref{eq:f_qdifference}, we find that $q^{-2}(1 - \xi^{-1}z)(1 - q^{4}z)f(z)$
satisfies~\eqref{eq:g_qdifference}, and hence
\[
  q^{2}\xi^{-1}\wtd{R}_{q}(z) = q^{2}\xi^{-1}f(z)R_{q}(z) =
  q^{-2}(1 - \xi^{-1}z)(1 - q^{4}z)f(z)\bar{\mathsf{R}}(z) =
  g(z)\bar{\mathsf{R}}(z) = \mathsf{R}(z),
\]
as claimed.
\end{proof}

We can now finally present a proof of the crossing symmetry identities from Lemma~\ref{lem:affine_crossing_symmetries}:

\begin{proof}[Proof of Lemma~\ref{lem:affine_crossing_symmetries}]
We shall assume for simplicity that we are in type $B_{n}$, as the arguments in other types are the same. Let us first show that
the second identity of~\eqref{eq:affine_crossing_symmetries_1} can be obtained from the first one. Note that the first identity
of~\eqref{eq:affine_crossing_symmetries_1} is equivalent to
$(C_{q})_{1}(\wtd{R}_{q}(z\xi))^{\mathsf{t}_{1}'}(C_{q}^{-1})_{1} = \xi^{2}v^{2}(\wtd{R}_{q}(z))^{-1}$. Applying $\mathsf{t}_{1}'$
to this equality yields
  $(C_{q}^{\mathsf{t}'})^{-1}_{1}\wtd{R}_{q}(z\xi)(C_{q}^{\mathsf{t}'})_{1} = \xi^{2}v^{2}(\wtd{R}_{q}(z)^{-1})^{\mathsf{t}_{1}'}$,
which is equivalent to
  $\xi^{-2}v^{-2}I = (C_{q}^{\mathsf{t}'})_{1}^{-1}\wtd{R}_{q}(z\xi)^{-1}(C_{q}^{\mathsf{t}'})_{1}(\wtd{R}_{q}(z)^{-1})^{\mathsf{t}_{1}'}$.
The latter is precisely the second identity of~\eqref{eq:affine_crossing_symmetries_1}. A similar argument shows that
the second identity of~\eqref{eq:affine_crossing_symmetries_2} follows from the first one.

Due to Theorem~\ref{thm:universal_R_evaluation}, the first identity of~\eqref{eq:affine_crossing_symmetries_1} will follow if
we show that $\mathsf{R}(z)(C_{q})_{1}\mathsf{R}(z\xi)^{\mathsf{t}'_1}(C_{q}^{-1})_{1} = I$, where $\mathsf{R}(z)$ is given
by~\eqref{eq:universal_R_projection}. For this, recall that the universal $R$-matrix $\cal{R}(z)$ of~\eqref{eq:universal_Rz}
satisfies $(S \otimes 1)\cal{R}(z) = \cal{R}(z)^{-1}$, according to~\cite[Proposition 4.3]{FR}. Therefore,
\[
  \mathsf{R}(z)^{-1} = (\wtd{\rho} \otimes \wtd{\rho})(\cal{R}(z)^{-1}) =
  ((\wtd{\rho} \circ S) \otimes \wtd{\rho})(\cal{R}(z)),
\]
where $\wtd{\rho} = \rho \circ \omega$ for $\rho$ of Proposition~\ref{prop:affine_rep_B}. Thus, by~\eqref{eq:dual_rep_1}, we have
\[
  \mathsf{R}(z)^{-1} =
  (C_{q})_{1}\big( ((\wtd{\rho}\circ \cal{D}_{\xi}) \otimes \wtd{\rho})(\cal{R}(z)) \big)^{\mathsf{t}_{1}'}(C_{q}^{-1})_{1} =
  (C_{q})_{1}\mathsf{R}(z\xi)^{\mathsf{t}_{1}'}(C_{q}^{-1})_{1},
\]
which immediately implies the claimed identity $\mathsf{R}(z)(C_{q})_{1}\mathsf{R}(z\xi)^{\mathsf{t}'_1}(C_{q}^{-1})_{1} = I$.

To prove the first identity of~\eqref{eq:affine_crossing_symmetries_2}, we use~\eqref{eq:dual_rep_2} and the identity
$(1 \otimes S^{-1})\cal{R}(z) = \cal{R}(z)^{-1}$ from~\cite[Proposition 4.3]{FR}. This yields
\begin{align*}
  \mathsf{R}(z\xi^{-1})^{-1} = (\wtd{\rho} \otimes (\wtd{\rho} \circ S^{-1}))(\cal{R}(z\xi^{-1}))
  &= (C_{q}^{\mathsf{t}'})_{2}\left ((\wtd{\rho} \otimes (\wtd{\rho} \circ \cal{D}_{\xi^{-1}}))(\cal{R}(z\xi^{-1}))\right )^{\mathsf{t}_{2}'}(C_{q}^{\mathsf{t}'})_{2}^{-1} \\
  &= (C_{q}^{\mathsf{t}'})_{2}\mathsf{R}(z)^{\mathsf{t}_{2}'}(C_{q}^{\mathsf{t}'})^{-1}_{2},
\end{align*}
which is equivalent to $\mathsf{R}(z\xi^{-1})^{-1}(C_{q}^{\mathsf{t}'})_{2} = (C_{q}^{\mathsf{t}'})_{2}\mathsf{R}(z)^{\mathsf{t}_{2}'}$,
and then finally to $(C_{q}^{\mathsf{t}'})_{2}^{-1}\mathsf{R}(z\xi^{-1})(C_{q}^{\mathsf{t}'})_{2}\mathsf{R}(z)^{\mathsf{t}_{2}'} = I$.
Then, as above, the first identity of~\eqref{eq:affine_crossing_symmetries_2} follows from Theorem~\ref{thm:universal_R_evaluation}.
\end{proof}

   %%%%%%%%%%%%%%%%%%%%%%%%%%%%%%%%%%%%%%%%%%%%%%%%%%%%%%%%%%%%%%%%%%%%%%%%%%
   %%%%%%%%%%%%%%%%%%%%%%%%%%%%%%%%%%%%%%%%%%%%%%%%%%%%%%%%%%%%%%%%%%%%%%%%%%
   %%%%%%%%%%%%%%%%%%%%%%%%%%%%%%%%%%%%%%%%%%%%%%%%%%%%%%%%%%%%%%%%%%%%%%%%%%

\end{document}